\documentclass[11pt,a4paper,usenames,dvipsnames,reqno]{amsart}
\usepackage{amsthm}
\usepackage{amsmath}
\usepackage{amscd}
\usepackage{amsfonts,amssymb}
\usepackage{mathtools}
\usepackage[margin=1in]{geometry}
\usepackage{fancybox}
\usepackage{pdflscape}
\usepackage{rotating}

\usepackage[all,arc]{xy}
\usepackage{enumerate}
\usepackage{mathrsfs}
\usepackage{color}   
\usepackage{hyperref}
\hypersetup{
    colorlinks=true, 
    linktoc=all,     
    linkcolor=blue,  
}

\usepackage{graphicx}
\usepackage{wasysym}
\usepackage{pgf,tikz}
\usetikzlibrary{positioning}
\usepackage{ marvosym }
\usepackage{tikz-cd}
\usepackage{ textcomp }
\usepackage{bbm}
\usepackage{ stmaryrd }
\usepackage{enumitem}
\usepackage{booktabs}
\usepackage{array}

\numberwithin{equation}{section}

\newcommand{\cc}{\mathbb C}

\newcommand{\zz}{\mathbb Z}

\newcommand{\qq}{\mathbb Q}
\newcommand{\rr}{\mathbb R}
\newcommand{\A}{\mathbb A}

\newcommand{\GG}{\mathbb G}

\newcommand{\la}{\langle}
\newcommand{\ra}{\rangle}
\newcommand{\lra}{\longrightarrow}

\newcommand{\ord}{\mathrm{ord}}

\DeclareMathOperator{\GL}{GL}

\DeclareMathOperator{\Ind}{Ind}

\DeclareMathOperator{\SL}{SL}

\newcommand{\lto}{\longrightarrow}

\newcommand\norm[1]{\left\lVert#1\right\rVert}

\newenvironment{psmatrix}
  {\left(\begin{smallmatrix}}
  {\end{smallmatrix}\right)}

\newcommand{\fg}{\mathfrak g}

\newcommand{\fm}{\mathfrak m}
\newcommand{\fn}{\mathfrak n}
\newcommand{\fsl}{\mathfrak{sl}}

\newcommand{\calf}{\mathcal{F}}

\newcommand{\cals}{\mathcal{S}}
\newcommand{\calo}{\mathcal{O}}

\newcommand{\calr}{\mathcal{R}}

\makeatletter
\def\Ddots{\mathinner{\mkern1mu\raise\p@
\vbox{\kern7\p@\hbox{.}}\mkern2mu
\raise4\p@\hbox{.}\mkern2mu\raise7\p@\hbox{.}\mkern1mu}}
\makeatother

\makeatletter
\DeclareRobustCommand\bigop[1]{%
  \mathop{\vphantom{\sum}\mathpalette\bigop@{#1}}\slimits@
}
\newcommand{\bigop@}[2]{%
  \vcenter{%
    \sbox\z@{$#1\sum$}%
    \hbox{\resizebox{\ifx#1\displaystyle.9\fi\dimexpr\ht\z@+\dp\z@}{!}{$\m@th#2$}}%
  }%
}
\makeatother

\newtheorem{Thm}{Theorem}[section]
\newtheorem{Prop}[Thm]{Proposition}

\newtheorem{Lem}[Thm]{Lemma}
\newtheorem{Cor}[Thm]{Corollary}

\newtheorem{Conj}[Thm]{Conjecture}

\theoremstyle{definition}
\newtheorem{Def}[Thm]{Definition}

\theoremstyle{remark}
\newtheorem{Rem}[Thm]{Remark}
\newtheorem{Ex}[Thm]{Example}

\newcommand{\quash}[1]{}

\title[Asymptotics of Schwartz functions]{Asymptotics of Schwartz functions}

\author{Chun-Hsien Hsu}
\address{Department of Mathematics\\
University of Chicago\\
Chicago, IL 60637}
\email{chunhsien@uchicago.edu}


\begin{document}

\begin{abstract}
Let $G$ be a split, simply connected, almost simple algebraic group, and let $P$ be a maximal parabolic subgroup of $G$. Braverman and Kazhdan in \cite{BKnormalized} defined a Schwartz space on the affine closure $X_P$ of $P^{\mathrm{der}}\backslash G$.  An alternate, more analytically tractable definition was  given in \cite{Getz:Hsu:Leslie}, following several earlier works.   When $G$ is a classical group or $G_2$, we show the two definitions coincide and prove several previously conjectured properties of the Schwartz space that will be useful in applications. Along the way, we give an alternative construction of the ring of differential operators on $X_P$ using the Fourier theory. We also establish the Poisson summation formulae in these cases.
\end{abstract}

\maketitle
\setcounter{tocdepth}{1}
\tableofcontents

\section{Introduction}

Let $G$ be a split, simply connected, almost simple algebraic group over a local field $F$. Let $\psi:F\to \cc^\times$ be a nontrivial additive character. For a parabolic subgroup $P$ of $G$, let $X_P^{\circ}:=P^{\mathrm{der}}\backslash G$ and $X_P:=\mathrm{Spec}(F[X_P^\circ])$ be its affine closure. Generalizing earlier work in \cite{BK:basicaffine}, Braverman and Kazhdan in \cite{BKnormalized} constructed unitary operators
    \begin{align*}
        \calf_{P|Q}:=\calf_{P|Q,\psi}:L^2(X_P(F))\tilde{\longrightarrow} L^2(X_Q(F))
    \end{align*}
    such that 
    \begin{align*}
        \calf_{Q|R}\circ \calf_{P|Q}=\calf_{P|R} \quad \textrm{ and }  \quad  \calf_{P|P}=\mathrm{Id}
    \end{align*}
    for any parabolic subgroups $P,Q,R$ of $G$ with the same Levi factor. These operators generalize the standard Fourier transform on vector spaces. A nice exposition of their construction is given in \cite[\S 4]{Shahidi:Generalized}. 
    
    Let $M$ be a Levi subgroup of $G$ and $P$ be a parabolic subgroup of $G$ with Levi factor $M$. When $F$ is nonarchimedean, Braverman and Kazhdan defined the Schwartz space to be 
    \begin{align*}
        \mathcal{S}_{\mathrm{BK}}(X_P(F)):=\sum_{Q} \calf_{Q|P}\left(\mathcal{S}(X_Q^\circ(F))\right)<L^2(X_P(F),
    \end{align*}
    where the sum is taken over all parabolic subgroups $Q$ with Levi factor $M$, and $\mathcal{S}(X_Q^\circ(F))$ is the space of smooth functions on $X_Q^\circ(F)$ with compact support. That is, $\mathcal{S}_{\mathrm{BK}}(X_P(F))< C^\infty(X_P^\circ(F))$ is the smallest space of functions containing  $\mathcal{S}(X_P^\circ(F))$ and stable under Fourier transforms. While this definition of Schwartz space is succinct and natural, it is not known in general how to characterize functions in $\mathcal{S}_{\mathrm{BK}}(X_P(F))$. One main difficulty is that the operator $\calf_{Q|P}$ was initially defined only on an inexplicit subspace of $C^\infty_c(X_Q^\circ (F))$ and extended to an $L^2$-isometry by unitarity. In particular, little analytic information about functions in $\mathcal{S}_{\mathrm{BK}}(X_P(F))$ is understood aside from the case where $X_P$ is a vector space. 
    
    Let $B\le P$ be a Borel subgroup of $G$ and $P^{\mathrm{op}}$ be the opposite of $P$. In the case $G$ is symplectic and $P$ is the Siegel parabolic, using the work of Ikeda \cite{Ikeda:poles:triple}, Getz and Liu defined in \cite{Getz:Liu:BK} another Schwartz space $\mathcal{S}(X_P(F))< C^\infty(X_P^\circ (F)).$  Roughly, it is defined to be a space of functions whose Mellin transforms have poles no worse than those of certain local $L$-functions. This definition was later modified in \cite{Getz:Hsu} and extended by Getz, Leslie, and the author in \cite{Getz:Hsu:Leslie} to all split, simply connected, almost simple algebraic groups $G$ and maximal parabolics $P$.  In loc.~cit. it was shown that functions in $\mathcal{S}(X_P(F))$ lie in $(L^1\cap L^2)(X_P(F))$ and have good analytic behavior towards infinity. For instance, they are compactly supported in $X_P(F)$ if $F$ is nonarchimedean. Moreover, the operator $\calf_{P|P^{\mathrm{op}}}$ descends to an isomorphism
    \begin{align*}
        \calf_{P|P^{\mathrm{op}}}: \mathcal{S}(X_P(F)) \tilde{\longrightarrow} \mathcal{S}(X_{P^{\mathrm{op}}}(F)),
    \end{align*}
    for which an explicit formula is available.

     Nevertheless, determining from the definition whether a smooth function on $X_P^\circ(F)$ lies in $\mathcal{S}(X_P(F))$ can be quite difficult. For instance, it is not clear if the containment $\mathcal{S}(X_P^\circ(F))\le \mathcal{S}(X_P(F))$ holds. Furthermore, many expected properties of $\mathcal{S}(X_P(F))$ remain unverified. One of the properties is \textbf{locality} introduced in \cite[\S 5]{BK-lifting}. 
    
    \begin{Conj}\label{Conj:local}
   \sloppy The Schwartz space $\mathcal{S}(X_P(F))$ is local. That is, $\mathcal{S}(X_P(F))$ contains $\mathcal{S}(X_P^\circ(F))$ and it is closed under function multiplication by tempered functions (resp. compactly supported smooth functions) on $X_P(F)$ when $F$ is archimedean (resp. nonarchimedean).
    \end{Conj}

     Locality was introduced as a key property of conjectural Schwartz spaces in the context of the Poisson summation conjecture. The Poisson summation conjecture, initiated in the seminal paper \cite{BK-lifting}, arose from an observation of Braverman and Kazhdan that the Langlands functoriality conjecture implies the existence of a Fourier theory on $L$-monoids. More precisely, they proposed a vast generalization of Tate's thesis, which would realize local $L$-functions as the greatest common divisors of zeta integrals, subject to the existence of Schwartz spaces and Fourier transforms on $L$-monoids over local fields. If, in addition, Poisson summation formulae are established, this would imply the analytic properties of Langlands $L$-functions, and hence much of functoriality, by the converse theory. The conjecture has since been investigated and refined by Lafforgue \cite{LafforgueJJM}, Ng\^{o} \cite{NgoSums, Ngo:Hankel}, and many others, and extended to the setting of spherical varieties by Sakellaridis \cite{SakellaridisSph}. It has only been established in a handful of cases \cite{BK:basicaffine, Getz:Liu:BK, Getz:Liu:Triple, Choie:Getz, Gu:Autotwist, Getz:Quad, GK:auto}. Braverman-Kazhdan spaces $X_P$ are among the few examples about which we have the best current knowledge.

    In the present paper, we prove Conjecture \ref{Conj:local} and give a detailed description of $\mathcal{S}(X_P(F))$ when $G$ is either classical or $G_2$ and $P$ is a maximal parabolic. When $F$ is archimedean, we show that locality and being stable under the Fourier transform uniquely determine $\mathcal{S}(X_P(F))$. When $F$ is nonarchimedean of characteristic zero, $\mathcal{S}(X_P(F))$ is characterized by the process of semisimplifying degenerate principal series. This paper can be viewed as a continuation of the first part of \cite{Getz:Hsu:Leslie}, which aims to achieve a thorough understanding of the Fourier theory on Braverman-Kazhdan spaces with an expectation to find applications in harmonic analysis and beyond. We also expect that these detailed examples could lead to a better understanding of conjectural Schwartz spaces in the Poisson summation conjecture.

\subsection{Main results}\label{ssec:main}
    
Let $F$ be a local field. Suppose $G$ is classical or $G_2$. Let $P$ be a maximal parabolic subgroup of $G$. We clarify the relation between  $\mathcal{S}(X_P(F))$ and $\mathcal{S}_{\mathrm{BK}}(X_P(F))$ and verify Conjecture \ref{Conj:local}.

\begin{Thm}\label{thm:main:equal}
Suppose $G$ is classical or $G_2$. We have
\begin{align*}
    \mathcal{S}(X_P(F))=\mathcal{S}_{\mathrm{BK}}(X_P(F)).
\end{align*} 
In particular, $\mathcal{S}(X_P^\circ(F))<\mathcal{S}(X_P(F))$. Moreover, $\mathcal{S}(X_P(F))$ is local.
\end{Thm}

The boundary $X_P-X_P^{\circ}$ is a single point at the origin. By Theorem \ref{thm:main:equal}, we can consider the smooth $G(F)$-module $$A_{X_P(F)}:=\mathcal{S}(X_P(F))/\mathcal{S}(X_P^\circ(F)),$$ which we refer to as the \textbf{asymptotics} (toward the origin). Let $I(\chi)=I_P(\chi)=I_{P(F)}^{G(F)}(\chi)$ be the normalized parabolic induction of a quasi-character $\chi:F^\times\to \cc^\times$ in the category of smooth representations. Here we have identified $\chi$ with the quasi-character $\chi\circ\omega_P$ of $P(F),$ where $\omega_P$ is the fundamental weight associated to $P$. Let $\mathcal{R}_{P|P^{\mathrm{op}}}:I_P(\chi)\to I_{P^\mathrm{op}}(\chi^{-1})$ be the intertwining operator associated to the long Weyl element. Define $\mathrm{Supp}(G,P,F)$ to consist of those $\chi$ with $\mathrm{Re}(\chi)<0$ such that $\mathrm{Ker}(\mathcal{R}_{P|P^{\mathrm{op}}})$ is nonzero and proper. The set $\mathrm{Supp}(G,P,F)$ is countable, and it is finite when $F$ is nonarchimedean.

 For an admissible $G(F)$-representation $A$ of finite length, let $\mathrm{Soc}(A)$ be the maximal semisimple $G(F)$-submodule of $A$. Let $\mathrm{Soc}^0(A):=0$ and define inductively for $n\in \zz_{>0}$ the submodule $\mathrm{Soc}^n(A)$ of $A$ so that
\begin{align*}
    \mathrm{Soc}^n(A)/\mathrm{Soc}^{n-1}(A)=\mathrm{Soc}(A/\mathrm{Soc}^{n-1}(A)).
\end{align*}
We show that the asymptotics $A_{X_P(F)}$ can be described in terms of the socles of degenerate principal series.

\begin{Thm}\label{thm:main:filt}
Suppose $F$ is a nonarchimedean local field of characteristic zero, and $G$ is either classical or $G_2$. As $G(F)$-modules, we can write
\begin{align*}
    A_{X_P(F)}= \prod_{\chi\in \mathrm{Supp}(G,P,F)} A(\chi),
\end{align*}
where each $A(\chi)$ admits a natural filtration of $G(F)$-modules
\begin{align*}
    0=A_0(\chi) < A_1(\chi)< A_2(\chi) <\cdots < A_{m(\chi)}(\chi)=A(\chi),
\end{align*}
and we have canonical isomorphisms of $G(F)$-modules 
\begin{align*}
        A_{m(\chi)+1-n}(\chi)/A_{m(\chi)-n}(\chi)\cong \mathrm{Soc}^n(I(\chi))
\end{align*}
for $1\le n\le m(\chi).$ Furthermore,
\begin{align*}
    \mathrm{Soc}^{m(\chi)}(I(\chi))=\mathrm{Ker}(\mathcal{R}_{P|P^{\mathrm{op}}})<I(\chi)=\mathrm{Soc}^{m(\chi)+1}(I(\chi)).
\end{align*}
\end{Thm}

The modules $A(\chi),$ defined in \eqref{eq:Achi} below, describe the analytic behavior of functions in $\mathcal{S}(X_P(F))$ toward the origin $X_P-X_P^\circ$. Unwinding the definition of $A(\chi),$ we have a natural filtration
\begin{align*}
    I(\chi)> A_{1}(\chi)\ge A_2(\chi)/A_1(\chi)\ge \cdots\ge A_{m(\chi)}(\chi)/A_{m(\chi)-1}(\chi)>0,
\end{align*}
which is the Jantzen filtration of $I(\chi)$ \cite[(14.3b)]{AvLTV}. Therefore, Theorem \ref{thm:main:filt} provides the following interpretation of the filtration $(\mathrm{Soc}^n(I(\chi)))$ and the integer $m(\chi)$: for a holomorphic section $f(\chi)$ of $I(\chi),$ the section $\mathcal{R}_{P|P^{\mathrm{op}}}(f(\chi))$ is holomorphic on the left half-plane $\mathrm{Re}(\chi)<0$. Then
    \begin{align}\label{eq:socleequal}
        \mathrm{Soc}^n(I(\chi_0))=\left\{f(\chi_0): \textrm{  }\mathcal{R}_{P|P^{\mathrm{op}}}(f(\chi)) \textrm{ vanishes at } \chi=\chi_0 \textrm{ to order at least $m(\chi)+1-n$}  \right\}.
    \end{align}
We verify in Theorem \ref{thm:reducibility} that $\mathrm{Supp}(G,P,F)$ consists of points of reducibility of degenerate principal series in the left half-plane. Hence, $A_{X_P}(F)$ encodes the process of semisimplifying all degenerate principal series.

The existence of the filtration $(A_{n}(\chi))$ in Theorem \ref{thm:main:filt} is proved for general $F$ in Theorems \ref{thm:exact} and \ref{thm:exactarch}. It also holds for $G$ of type $E$ or $F$ under the assumption $\mathcal{S}(X_P^\circ(F))\subseteq \mathcal{S}(X_P(F))$. However, describing $A_{n}(\chi)$ for $n\ge 2$ requires detailed knowledge of the composition series of degenerate principal series. We conjecture that the socle descriptions \eqref{eq:socleequal} hold in general, as stated in Conjecture \ref{conj:multiplicity}.
 
 The situation is more subtle for archimedean $F$ since $\mathrm{Supp}(G,P,F)$ is infinite, and thus $A_{X_P(F)}$ is a Fr\'echet representation of (countably) infinite length. This is not unexpected since the germ of a smooth function at a point in the nonarchimedean case is constant, whereas in the archimedean case it corresponds to a formal power series. Therefore, the infinite length of $A_{X_P(F)}$ should arise solely from the formal coordinate functions $F[[X_P]]\otimes_\rr \cc$. 

A key ingredient in confirming this expectation is the construction of the Weyl algebra on $X_P$. Since $\mathcal{S}(X_P(F))$ is local by Theorem \ref{thm:main:equal}, it is closed under multiplication by coordinate functions on $X_P(F)$. Say $x_1,\ldots,x_n$ are coordinates on $X_P$. Motivated by the classical fact that the Fourier transform on vector spaces transfers between algebraic functions and differential operators, we define abstractly the operators
\begin{align*}
    \partial_i:=\mathcal{F}_{P^{\mathrm{op}}|P}\circ x_i^{\mathrm{op}}\circ \mathcal{F}_{P|P^{\mathrm{op}}}:\mathcal{S}(X_P(F))\to \mathcal{S}(X_P(F)).
\end{align*}
The Weyl algebra $W_{X_P}$ on $X_P$ is then defined as the algebra generated by coordinate functions and these ``differential operators" $\partial_i$.

We verify in Theorem \ref{thm:Weylrealize} that the operators $\partial_i$ can indeed be realized as differential operators on $X_P^\circ(F)$. In \cite{Hsu:Weyl}, the author proves that $W_{X_P}$ share several key properties with the ring of differential operators on smooth affine varieties. This result allows us to refine our understanding of $A_{X_P(F)}.$

\begin{Thm}\label{thm:main:arch}
   When $F$ is archimedean,  $A_{X_P(F)}$ is a Fr\'echet $G(F)\ltimes W_{X_P}(F\otimes_\rr \cc)$-module of finite length.
\end{Thm}

 Theorem \ref{thm:main:arch} is restated and proved as Corollary \ref{cor:arch:finite} below. It plays an important role in the computation of $A_{X_P(F)}$ and, consequently, reducible degenerate principal series. In particular, it allows us to confirm that Theorem \ref{thm:main:filt} holds for the examples in \S\ref{ssec:arch:ex}.

In the archimedean case, instead of using $\mathcal{S}_{\mathrm{BK}}(X_P(F)),$ Braverman and Kazhdan in \cite{BKnormalized} proposed the following definition of Schwartz spaces:
\begin{align*}
    \mathcal{S}_{\mathcal{D}}(X_P(F)):=\bigg\{f\in L^2(X_P(F)): Df\in L^2(X_P(F))  \textrm{ for every } D\in \mathcal{D}(X_P)(F\otimes_\rr\cc)\bigg\},
\end{align*}
where $\mathcal{D}(X_P)$ is the ring of differential operators on $X_P$ in the sense of Grothendieck. The space $ \mathcal{S}_{\mathcal{D}}(X_P(F))$ comes equipped with a natural Fr\'echet topology. When $X_P$ is a vector space, this is the classical definition of the Schwartz space. They conjectured that $\mathcal{S}_{\mathcal{D}}(X_P(F))$ is preserved by the Fourier transform. In \cite{Hsu:Weyl} we verify that $W_{X_P}=\mathcal{D}(X_P)$, so the conjecture follows. This also implies that the largest local subspace of $L^2(X_P(F))$ that is stable under the Fourier transform is $ \mathcal{S}_{\mathcal{D}}(X_P(F))$. We further prove
\begin{Thm}\label{thm:main:diff}
We have $\mathcal{S}(X_P(F))=\mathcal{S}_{\mathcal{D}}(X_P(F))$ as Fr\'echet spaces.
\end{Thm}

This result can be interpreted as follows: Suppose one can exponentiate the action of $W_{X_P}(F\otimes_\rr \cc)$ to some group $G'(F),$ not necessarily algebraic. Then $\mathcal{S}(X_P(F))$ is the set of smooth vectors of an irreducible representation of $G'(F)$. So far the existence of $G'(F)$ is only known when $\mathcal{S}(X_P(F))$ is a minimal representation. 

Finally, we would like to emphasize again that while so far we have assumed $G$ to be either classical or $G_2,$ most of our results apply to $G$ of type $E$ or $F$ subject to the verification of Theorem \ref{thm:poles} over nonarchimedean local fields, which implies Theorem \ref{thm:main:equal} over any local fields.
\subsection{Poisson summation formulae}

We continue to assume $G$ is either classical or $G_2$ and $P$ is a maximal parabolic. Let $E$ be a global field. Let $\mathcal{S}(X_P(\mathbb{A}_E))$ be the adelic Schwartz space (see \S \ref{sec:Poisson}). For $f\in \mathcal{S}(X_P(\A_E))$ and an idelic character $\omega,$ let
\begin{align*}
    f_{\omega_s}(\cdot):=\int_{M^{\mathrm{ab}}(\A_E)} \delta_P^{1/2}(m)\omega_s(\omega_P(m))f(m^{-1}\cdot)\,dm
\end{align*}
be the Mellin transform of $f,$ where $\delta_P$ is the modular character. It is a meromorphic section of $I_P(\omega_s)$. The associated degenerate Eisenstein series is
\begin{align*}
    \mathrm{Eis}(f_{\omega_s}):=\sum_{\gamma\in P(E)\backslash G(E)} f_{\omega_s}(\gamma).
\end{align*}
The sum converges absolutely for $\mathrm{Re}(s)$ large and defines a meromorphic function on $\cc.$

Let $\kappa_E:=\mathrm{Res}_{s=1}\zeta_E(s)$. Consider the following set of idelic quasi-characters
\begin{align*}
    \mathrm{Supp}(G,P,E):=\bigg\{ \omega=\otimes_v \omega_v: \omega_v\in\mathrm{Supp}(G,P,E_v) \textrm{ for all places } v \textrm{ of } E\bigg\}.
\end{align*} We establish the Poisson summation formula on $X_P$.
\begin{Thm}\label{thm:main:Poisson}
Suppose $G$ is either classical or $G_2.$ Let $f\in \mathcal{S}(X_P(\A_E))$. Then
       \begin{align*}
        &\sum_{x\in X_P^\circ(E)} f(x)-\frac{1}{\kappa_E}\sum_{\omega_{s_0}\in  \mathrm{Supp}(G,P,E)} \mathrm{Res}_{s=s_0}\mathrm{Eis}(f_{\omega_s})\\
        &=\sum_{x\in X_{P^{\mathrm{op}}}^\circ(E)} \mathcal{F}_{P|P^{\mathrm{op}}}(f)(x)-\frac{1}{\kappa_E}\sum_{\omega_{s_0}\in  \mathrm{Supp}(G,P,E)}  \mathrm{Res}_{s=s_0}\mathrm{Eis}(\mathcal{F}_{P|P^{\mathrm{op}}}(f)_{\omega_s}).
    \end{align*}
\end{Thm}
A less precise version of the formula under a mild assumption is already stated in \cite{Choie:Getz}. We verify their assumption and remove trivial terms stated in loc. cit. Theorem \ref{thm:main:Poisson} indicates that global boundary terms should be (restricted) tensor products of coherent\footnote{This terminology was introduced in \cite{KR:Siegel-Weil}.} local boundary terms. This phenomenon has already been examined by \cite{Getz:Quad, GK:auto} when $\mathcal{S}(X_P(F))$ is a minimal representation. In the case where $G=\mathrm{Sp}_{2n}$ and $P$ is the Siegel parabolic, this is made explicit by the regularized Siegel-Weil formulae \cite{KR:Siegel-Weil}.

\subsection{Outline of the paper}
We state our conventions and collect some facts on Schwartz spaces in \S \ref{sec:prelim}. In \S\ref{sec:BKspace} we review the Fourier theory on Braverman-Kazhdan spaces. Assuming the inclusion $\mathcal{S}(X_P^\circ(F))\subseteq\mathcal{S}(X_P(F))$ holds, we describe the asymptotics $A_{X_P(F)}$ in  \S \ref{sec:Sch} and \S \ref{sec:Sch:arch} for nonarchimedean and archimedean $F,$  respectively. The key to such a description is Igusa's work on the Mellin inversion \cite{Igusa:forms}, which relates poles of the Mellin transform of a function to its asymptotics. This approach has already been employed in \cite{Jiang:Luo:Zhang}. Our main results are Theorems \ref{thm:exact} and \ref{thm:exactarch}, which give the filtration in Theorem \ref{thm:main:filt}. Relevant computations that help to read the statement are listed in Appendix \ref{appendix}. We state our expectation of describing $A_{X_P(F)}$ in terms of socles of degenerate principal series as Conjecture \ref{conj:multiplicity}, and confirm it in Theorem \ref{thm:confirmconj} when $G$ is either classical or $G_2$ and $F$ is nonarchimedean of characteristic zero. This completes the proof of Theorem \ref{thm:main:filt}. Several examples of $A_{X_P}(F)$ are computed in \S \ref{ssec:nonarch:ex}, \S \ref{ssec:arch:ex} and \S \ref{ssec:ex:C}. Interestingly, in these examples $A_{X_P(F)}$ can be realized as certain function space on geometric objects defined over $F$.

In \S \ref{sec:casebycase} we justify the assumption $\mathcal{S}(X_P^\circ(F))\subseteq \mathcal{S}(X_P(F))$ when $G$ is classical or $G_2$. It is a consequence of Theorem \ref{thm:poles}, which studies more generally poles of intertwining operators. This result allows us to give an alternative definition of $\mathcal{S}(X_P(F))$ in the sense of \cite{Yamana} and to prove $\mathcal{S}(X_P(F))=\mathcal{S}_{\mathrm{BK}}(X_P(F))$ in Corollary \ref{cor:cont}. Due to the fact that $X_P-X_P^\circ$ is a single point, locality of $\mathcal{S}(X_P(F))$ is a straightforward consequence of the inclusion $\mathcal{S}(X_P^\circ(F))<\mathcal{S}(X_P(F))$ in the nonarchimedean case (see Corollary \ref{cor:localnonarch}). Locality in the archimedean case requires one to use the definition of $\mathcal{S}(X_P(F)),$ which is proved in \S \ref{ssec:localarch}. Then we introduce and study the Weyl algebras $W_{X_P}$ in \S \ref{ssec:Weyl}. The main result is Theorem \ref{thm:Weylrealize} that realizes elements in $W_{X_P}$ as differential operators on $X_P^\circ$. Using several properties of $W_{X_P}$ proved in \cite{Hsu:Weyl}, we refine descriptions of asymptotics $A_{X_P(F)}$ and prove Theorem \ref{thm:main:diff} in \S\ref{ssec:application:Weyl}, and study the harmonic analysis on $X_P$ in \S\ref{ssec:harm}. Finally, we prove Theorem \ref{thm:main:Poisson} in \S \ref{sec:Poisson} and Appendix  \ref{appendix:B} case by case, following the approach in \cite{Ikeda:poles:triple}.

\subsection{Acknowledgments}
Part of this paper is the author's thesis. The author thanks his advisor Jayce Getz for constant encouragement and fruitful discussions. 
The questions addressed in this paper were raised by David Kazhdan (and relayed to the author by Getz). He thanks Zhilin Luo for answering a question on the asymptotics in \cite{Jiang:Luo:Zhang}. He thanks Yupeng Li for a discussion on the combinatorics in \S\ref{sec:casebycase}. He thanks Yeansu Kim,  Baiying Liu, and Ivan Mati\'{c} for answering a question on degenerate principal series of spin groups, and Nadya Gourevich for an informative conversation and suggesting references on degenerate principal series of real Lie groups. He thanks Yiannis Sakellaridis for discussions on Weyl algebras and encouragement.

\section{Preliminaries}\label{sec:prelim}

Throughout the paper $F$ is a local field unless otherwise specified. We let $|\cdot|$ be the number-theorist's norm on $F$. Thus $|\cdot|$ is the usual Euclidean norm if $F=\rr,$ and $|z|=z\bar{z}$ if $F=\cc$. When $F$ is nonarchimedean, we denote by $\calo$ the ring of integers of $F$ and fix once and for all a choice of uniformizer $\varpi$. Then $q=|\varpi^{-1}|$ is the cardinality of the residue field $\calo/\varpi$.

An $F$-variety is a separated reduced scheme of finite type over $F$. For an $F$-variety $X$, let $X^{\mathrm{sm}}$ be its smooth locus and assume $X^{\mathrm{sm}}(F)$ is nonempty. We equip $X(F)$ with the analytic topology \cite[\S 2.6.2]{Poonen}. 

Let $G$ be a split, simply connected, almost simple algebraic group over $F$. Fix throughout the paper a maximal split torus $T<G$ and a pinning $(\Delta, (e_\alpha)_{\alpha\in \Delta})$ of $(G,T)$. Let $B$ be the Borel subgroup of $G$ standard with respect to the pinning. For each $\alpha\in\Delta$, let $\omega_\alpha\in X^\ast(T)$ be the fundamental weight such that
    \begin{align*}
        \langle \omega_\alpha, \beta^\lor\rangle =\delta_{\alpha,\beta}  \textrm { for all } \beta\in \Delta,
    \end{align*}
    where $\delta_{\alpha,\beta}$ is the Kronecker $\delta$.

\subsection{Quasi-characters}\label{ssec:char}

For an abelian topological group $A$, we let
\begin{align*}
    \widehat{A}:=\left\{\chi: A\to \cc^\times : \textrm{ a continuous group  homomorphism}\right\}
\end{align*}
be the space of quasi-characters of $A$. 

Let $K_{\GG_m}$ be the maximal compact subgroup of $F^\times$. Suppose $F$ is archimedean. We have a canonical group isomorphism
\begin{align}\label{eq:splitiso:arch}
\begin{split}
    F^\times&\tilde{\longrightarrow}\, \rr_+\times K_{\GG_m}\\
   a& \mapsto(|a|,\mu(a)),
   \end{split}
\end{align}
 where  $$\mu(a):=a\cdot |a|^{-\frac{1}{[F:\rr]}}.$$ 
 The isomorphism \eqref{eq:splitiso:arch} induces a group isomorphism
 \begin{align*}
     \cc \times \widehat{K}_{\GG_m} &\tilde{\longrightarrow} \, \widehat{F}^\times\\
      (s,\chi)& \mapsto \big(a\mapsto |a|^s\chi(\mu(a))\big).
 \end{align*}
Explicitly, we identify $\widehat{K}_{\GG_m}$ as a subgroup of $\widehat{F}^\times$ via
\begin{align*}
    \widehat{K}_{\GG_m}=\left\{ \mu^\ell: \begin{array}{l}
        \ell\in \zz/2\zz=\{0,1\} \textrm{ if } F=\rr,\\
         \ell\in \zz \textrm{ if } F=\cc
    \end{array}\right\}.
\end{align*}

Suppose $F$ is nonarchimedean, so $K_{\GG_m}=\calo^\times$. Let $\mathrm{ord}:F\to  \zz\cup\{\infty\}$ be the discrete valuation such that $\mathrm{ord}(\varpi)=1$. We have a (noncanonical) group isomorphism  
\begin{align}\label{eq:splitiso}
\begin{split}
    F^\times&\tilde{\longrightarrow}\, \zz\times \calo^\times\\
   a& \mapsto(\mathrm{ord}(a),\tilde{a}),
   \end{split}
\end{align}
 where  $$\tilde{a}:=a\cdot \varpi^{-\mathrm{ord(a)}}.$$ 
 The isomorphism \eqref{eq:splitiso} induces a group isomorphism
 \begin{align*}
     \cc/ \tfrac{2\pi\sqrt{-1}}{\log q}\zz \times \widehat{\calo}^\times &\tilde{\longrightarrow} \, \widehat{F}^\times\\
      (s,\chi)& \mapsto \left(a\mapsto |a|^{s}\chi(\tilde{a})\right).
 \end{align*}
We henceforth identify $\widehat{\calo}^\times$ as a subgroup of $\widehat{F}^\times$ under this isomorphism. Therefore for $\chi\in \widehat{\calo}^\times,$  $\chi(a)=\chi(\tilde{a}),$ and we let $\mathrm{ord}(\chi)$ be the order of $\chi$, i.e., the smallest positive integer $d$ such that $\chi^d=1$.  

In any case, for $\chi\in \widehat{F}^\times$ and $s\in\cc$ define $\chi_s:=\chi|\cdot|^s$. Let $\mathrm{Re}(\chi)$ be the unique real number such that $\chi_{-\mathrm{Re}(\chi)}$ is unitary.

\subsection{Measures}\label{ssec:measure}

We fix once and for all a nontrivial additive character $\psi:F\to \cc^\times$. The Haar measure $da$ on $F$ will always be normalized so that the Fourier transform on $\mathcal{S}(F)$ defined by $\psi$ is self-dual. We set $d^\times a:=\frac{\zeta(1)da}{|a|}$, where $\zeta$ is the usual local zeta function. It is a Haar measure on $F^\times$.

Let $\Phi=\Phi_G=\Phi(G,T)$ be the set of roots of $(G,T)$. Choose a Chevalley basis $(e_\alpha)_{\alpha\in\Phi}$ from the pinning $(e_\alpha)_{\alpha\in \Delta}$ (see \cite[\S 23.h]{Milne:AGbook}).  The Chevalley basis determines group isomorphisms
$\mathbb{G}_a \tilde{\longrightarrow} N_\alpha$,
where $N_\alpha$ is the root subgroup of $\alpha\in\Phi$. We endow $N_\alpha(F)$ with the measure transferred from that on $F=\mathbb{G}_a(F)$ under these isomorphisms. This in turn induces measures on unipotent subgroups of $B(F)$ and its opposite $B^{\mathrm{op}}(F)$.

\subsection{Schwartz spaces}

For each real algebraic variety $X$ (over $\rr$), a Schwartz space $\mathcal{S}_{\mathrm{ES}}(X)$ and a space of tempered functions $\mathcal{T}(X)$ are defined in \cite[Definitions 3.7 and 3.11]{Elazar:Shaviv} (generalizing the previous work in \cite{AG:Nash}). In the case $X$ is affine, one can embed $X$ as a closed subset of $\rr^n$ in the category of real algebraic varieties. Then
\begin{align*}
    \mathcal{S}_{\mathrm{ES}}(X):= \mathcal{S}(\rr^n)/I
\end{align*}
where $\mathcal{S}(\rr^n)$ is the usual space of Schwartz functions on $\rr^n$ and $I \le \mathcal{S}(\rr^n)$ is the (closed) ideal of functions that vanish identically on $X$. The space $\mathcal{S}_{\mathrm{ES}}(X)$ with the quotient topology is a nuclear Fr\'echet space. Its definition and topology do not depend on the choice of the embedding \cite[Lemma 3.6(i)]{Elazar:Shaviv}. Similarly, 
\begin{align*}
    \mathcal{T}(X):= \left\{ f|_{X}: f\in \mathcal{T}(\rr^n)\right\}
\end{align*}
where $\mathcal{T}(\rr^n)$ is the usual space of tempered functions on $\rr^n$. The definition of $\mathcal{T}(X)$ is also independent of the choice of the embedding  \cite[Lemma 3.6(ii)]{Elazar:Shaviv}. We collect the following facts from \cite[\S 3]{Elazar:Shaviv}.
\begin{Thm}\label{thm:ESSchwartz}
    Let $X$ be an affine real algebraic variety (over $\rr$).
    \begin{enumerate}
        \item For an open subset  $U$ of $X,$ we have a closed embedding 
        \begin{align*}
            \mathcal{S}_{\mathrm{ES}}(U)\xhookrightarrow{\,\mathrm{Ext}\,} \mathcal{S}_{\mathrm{ES}}(X)
        \end{align*}
        given by extension by zero.
         
        \item For a closed subset $Z$ of $X,$ we have a continuous surjection
        \begin{center}
        \begin{tikzpicture} 
            \node at (-3,0) (a) {\quad $\mathcal{S}_{\mathrm{ES}}(X)$};
            \node at (0.5,0) (b) {$\mathcal{S}_{\mathrm{ES}}(Z)$\quad \quad \quad \quad \quad  };
            \path [draw,->>] (a) -- (b);
        \end{tikzpicture}
        \end{center}
        given by the restriciton $f\mapsto f|_Z$. 

        \item The space $\mathcal{S}_{\mathrm{ES}}(X)$ is a $\mathcal{T}(X)$-module under function multiplication.
    \end{enumerate}\qed
\end{Thm}

Suppose $X$ is a quasi-projective $F$-variety. When $F$ is archimedean, we view $X(F)=\mathrm{Res}_{F/\rr}X(\rr)$ as an affine real algebraic variety so that $\mathcal{S}_{\mathrm{ES}}(X(F))$ and $\mathcal{T}(X(F))$ are defined as above.

\begin{Lem}\label{lem:germpoint}
    Suppose $F$ is archimedean. Let $X$ be an affine $F$-variety and $p\in X(F)$. Then we have an exact sequence of Fr\'echet spaces
    \begin{align*}
        0\longrightarrow \mathcal{S}_{\mathrm{ES}}(X(F)-p)\xrightarrow{\mathrm{Ext}}{}\mathcal{S}_{\mathrm{ES}}(X(F))\longrightarrow \widehat{F[X]}_p\otimes_\rr \cc\longrightarrow 0,
    \end{align*}
    where the surjection map is given by taking germs of functions at $p,$ and $\widehat{F[X]}_p$ is the  $\mathfrak{m}_p$-adic completion of $F[X]$, where $\mathfrak{m}_p<F[X]$ is the maximal ideal corresponding to $p$.
\end{Lem}

Choose a closed $F$-embedding $X\lhook\joinrel\xrightarrow{\,\,} \mathbb{A}^n$ so that we can identify $p$ with its image $(p_1,\ldots,p_n)$ in $F^n$. If we write $F[X]=F[\mathbb{A}^n]/I$ for some ideal $I$, then $\widehat{F[X]}_p$ is homeomorphic to the quotient of $\widehat{F[\mathbb{A}^n]_p}\cong F[[x_1-p_1,\ldots, x_n-p_n]]$ by its closed ideal generated by $I$. 

\begin{proof}
    It suffices to prove the assertion when $F=\rr, X=\A^n$, and $p$ is the origin. In this case 
    \begin{align*}
       \widehat{F[X]}_p\otimes_\rr\cc=\cc[[X]]
    \end{align*}
    is the ring of formal power series. The lemma is a consequence of Borel's Theorem \cite[Theorem I.1.3]{smoothapprox}.
\end{proof}

If $F$ is nonarchimedean, following the notation in the archimedean case, we define for a quasi-projective $F$-variety $X$, 
$$\mathcal{S}_{\mathrm{ES}}(X(F)):=C^\infty_c(X(F))$$ to be the space of locally constant functions on $X(F)$ with compact support. If $X$ is affine, then for any closed $F$-embedding $X\lhook\joinrel\xrightarrow{\,\,} \mathbb{A}^n,$ we have
\begin{align*}
    \mathcal{S}_{\mathrm{ES}}(X(F))=\mathcal{S}(F^n)\big|_{X(F)}.
\end{align*}
Furthermore, for $p\in X(F)$ we have an exact sequence given by the evaluation map at $p$:
\begin{align*}
    0\longrightarrow \mathcal{S}_{\mathrm{ES}}(X(F)-p)\longrightarrow \mathcal{S}_{\mathrm{ES}}(X(F))&\longrightarrow \cc\longrightarrow 0\\
     f&\longmapsto f(p).
\end{align*}

In any case, when $X$ is a smooth variety, we write
\begin{align*}
    \mathcal{S}(X(F)):=\mathcal{S}_{\mathrm{ES}}(X(F)).
\end{align*}
When $X$ is singular with nonempty $X^{\mathrm{sm}}(F),$ in general one does not have a good definition of Schwartz space $\mathcal{S}(X(F))$ that takes into account of the singularity of $X$. In fact, it is an important open problem to obtain a good definition of Schwartz spaces when $X$ is spherical \cite{SakellaridisSph}. In this paper, a Schwartz space of $X(F)$ will always be subspaces of $C^\infty(X^\mathrm{sm}(F)),$ so a Schwartz function on $X(F)$ needs not be defined on all $X(F)$. If $F$ is archimedean, it  will always be a Fr\'echet space.  In the case $\mathcal{S}(X(F))$ is defined, two spaces $\mathcal{S}(X(F))$ and $\mathcal{S}_{\mathrm{ES}}(X(F))$ are usually different. Indeed, a main result of this paper is to justify the relation $\mathcal{S}_{\mathrm{ES}}(X(F)) \le \mathcal{S}(X(F))$ for Braverman-Kazhdan spaces. The containment is strict if $X$ is not smooth.

Finally, we recall the definition of locality of a function space introduced in \cite[\S 5]{BK-lifting}.
\begin{Def}
    Let $F$ be a local field and $X$ be a quasi-projective $F$-variety with nonempty $X^{\mathrm{sm}}(F)$. A function space $\mathcal{S} \subseteq C^\infty(X^{\mathrm{sm}}(F))$ of $X(F)$ is said to be \textbf{local} if $\mathcal{S}(X^{\mathrm{sm}}(F)) \subseteq \mathcal{S},$ and $\mathcal{S}$ is a $\mathcal{T}(X(F))$-module (resp. a $C^\infty(X(F))$-module.) if $F$ is archimedean (resp. nonarchimedean).
\end{Def}

\noindent  For instance, the space $\mathcal{S}(X^{\mathrm{sm}}(F))$ is local. When $X^{\mathrm{sm}}(F)$ is dense in $X(F)$, we can identify $\mathcal{S}_{\textrm{ES}}(X(F))$ as a subspace of $C^\infty(X^\mathrm{sm}(F))$ via restriction. In this case $\mathcal{S}_{\textrm{ES}}(X(F))$ is also local.

\subsection{Analytic notation}\label{ssec:analytic}

For a domain $\Omega\subseteq \cc$, we let $H(\Omega)$ be the space of holomorphic functions on $\Omega$ equipped with the topology of uniform convergence on compact subsets. For a meromorphic function $f$ on $\Omega$ and $s_0\in \Omega,$ let  $\mathrm{ord}_{s=s_0} f(s)$
denote the order of the pole of $f$ at $s=s_0$. Given a discrete subset $S$ and a function $\delta:S\to \zz_{> 0}$, let $M(\Omega,\delta)$ be the (closed) subspace of $H(\Omega-S)$ consisting of meromorphic functions $f$ such that $ \mathrm{ord}_{s=s_0} f(s)\le \delta(s_0)$ for all $s_0\in S\cap \Omega$. Note that for any $n\in \zz_{\ge 0}$ the map
\begin{align*}
    f\mapsto \mathrm{Res}_{s=s_0} (s-s_0)^nf(s)
\end{align*}
is a continuous linear functional on $M(\Omega,\delta)$.

For nonzero real numbers $a,b$ and a (finite) set of elements $?$, we write $a\ll_? b$ if there exists a constant $C_?>0$ depending on $?$ such that $a\le C_? b$. 

We denote the usual norm on $\cc$ by $|\cdot |$. This creates the possibility of confusion when
we have chosen an identification $F= \cc$. When $F$ is denoted by $\cc$, we use the standard norm; otherwise, we use the number-theorist’s norm. Thus if $f:X\to \cc$ is a complex-valued function on a set $X$, then $|f(x)| = (f(x)\overline{f(x)})^{1/2}$ for $x\in X$.

\subsection{Multisets}\label{ssec:multi}
A multiset $L$ consists of a set $U$ and a function $m(L):U\to \zz_{>0}$. Equivalently, we can understand $L$ as the set
\begin{align*}
    \{(r,m_r): r\in U\}
\end{align*}
where $m_r:=m_r(L)\in \zz_{>0}$ is the multiplicity of $r$ in $L$. We write $\mathrm{Supp}(L):=U$ as the underlying set of $L$. Given two multisets $L_1,L_2$,  their sum $L_1+L_2$ is the multiset such that
\begin{align*}
    m_r(L_1+L_2)=m_r(L_1)+m_r(L_2),
\end{align*}
where the domain of $m(L_i)$ is extended to $\mathrm{Supp}(L_1)\cup \mathrm{Supp}(L_2)$ by setting $m_r(L_i)=0$ if $r\not\in \mathrm{Supp}(L_i)$.

\section{Fourier theory on Braverman-Kazhdan spaces}\label{sec:BKspace}

 We briefly recall the Schwartz space and the Fourier transform on a Braverman-Kazhdan space associated to a maximal parabolic. For a detailed discussion, we refer one to \cite[\S 3-\S 6]{Getz:Hsu:Leslie}, and we follow the notation therein below.

Let $P\ge B$ be a maximal parabolic subgroup of $G$ with a Levi decomposition $P=MN_P$ such that $T\le M$. The attached Braverman-Kazhdan space
\[
X_P^{\circ}:=P^{\mathrm{der}}\backslash G 
\]
admits a natural right $M^{\mathrm{ab}}\times G$-action, given on points in an $F$-algebra $R$ by 
\begin{align*} 
X_P^\circ(R) \times M^{\mathrm{ab}}(R) \times G(R) &\lto X_P^\circ(R)\\
    (x,m,g) & \longmapsto x.(m,g):=m^{-1}xg.
\end{align*}
This yields an action of $M^{\mathrm{ab}}(F)\times G(F)$ on $C^\infty(X_P^\circ(F)):$
\begin{align}\label{fun:act}
    L(m)R(g)f(x):=f(m^{-1}xg),
\end{align}
under which it is a smooth representation. When $F$ is archimedean, consider 
\begin{align*}
    \fm^{\mathrm{ab}}\oplus \fg:=\mathrm{Lie}(M^{\mathrm{ab}}(F)\times G(F)),
\end{align*}
viewed as a real Lie algebra. It acts on $C^\infty(X_{P}^\circ(F))$ via the differential of the action \eqref{fun:act}. Hence we obtain a natural action of $U(\fm^{\mathrm{ab}}\oplus \fg),$  the universal enveloping algebra of the complexification $(\fm^{\mathrm{ab}}\oplus \fg)_\cc$.
 
 Let $\omega_P$ be the fundamental weight corresponding to the simple root associated to $P$. Let $V_P$ be the right $G$-representation of (anti-dominant) highest weight $-\omega_P$, and fix a highest weight vector $v_P\in V_P(F)$. By \cite[Lemma 3.4]{Getz:Hsu:Leslie} we have a well-defined $F$-embedding
\begin{align*}
    \mathrm{Pl}:=\mathrm{Pl}_{v_P}:X_P^\circ\lhook\joinrel\xrightarrow{\quad} V_P
\end{align*}
induced by the map (given on points in an $F$-algebra $R$ by)
\begin{align*}
    G(R)&\longrightarrow  V_P(R)\\
    g&\longmapsto v_Pg.
\end{align*}
The map $\omega_P$, originally a character of $T$, extends to a character of $P$ and induces an isomorphism \begin{align*}
    \omega_P: M^{\mathrm{ab}}=P^{\mathrm{ab}} \tilde{\longrightarrow}\, \GG_m.
\end{align*}
For $m\in M^\mathrm{ab}(R)$, one has 
\begin{align*}
        \mathrm{Pl}(m^{-1}g)=\omega_P(m)\mathrm{Pl}(g).
\end{align*}
Let $$X_P:=\mathrm{Spec}(F[X_P^\circ])$$ be the affine closure of $X_P^\circ$.  By \cite[Theorems 1 and 2]{Popov:Vinberg} the embedding $\mathrm{Pl} $ extends to a closed immersion $X_P\hookrightarrow V_P$, and $X_P-X_P^\circ$ is a closed point mapped to the origin under $\mathrm{Pl}$.
    
Let $V_P^\lor$ be the dual representation of $V_P$ and $P^{\mathrm{op}}$ be the parabolic subgroup opposite to $P$. Let $v_{P^{\mathrm{op}}}^*\in V_P^{\lor}(F)$ be the lowest weight vector such that $\langle v_P,v_{P^{\mathrm{op}}}^*\rangle=1 $. Similarly, we have a  $G$-equivariant closed embedding $\mathrm{Pl}_{v^\ast_{P^{\mathrm{op}}}}: X_{P^{\mathrm{op}}}\lhook\joinrel\xrightarrow{\quad}V_P^\lor$. The restriction of the canonical pairing between $V_P$ and $V_P^\lor$ yields a pairing 
\begin{align}\label{eq:pairing}
   \langle \cdot ,\cdot \rangle_{P|\mathrm{P^{op}}} :X_P\times X_{P^{\mathrm{op}}}\longrightarrow \mathbb{G}_a 
\end{align}
via the embeddings $\mathrm{Pl}_{v_P},\mathrm{Pl}_{v_{P^{\mathrm{op}}}^\ast}$. This pairing is independent of the choice of $v_P$. We drop the subscript and simply write $\langle \cdot ,\cdot \rangle$ whenever the context is clear.

\subsection{Induced representations and intertwining operators}

For a quasi-character $\chi\in \widehat{F}^\times,$ let
\begin{align*}
I(\chi):=I_P(\chi):=\Ind_{P(F)}^{G(F)}(\chi \circ \omega_P), \quad \overline{I}(\chi):=\overline{I}_{P^{\mathrm{op}}}(\chi):=\Ind_{P^{\mathrm{op}}(F)}^{G(F)}(\chi \circ \omega_P)
\end{align*}
be the normalized inductions in the category of smooth representations. Let $\delta_P$ be the modular character of $P$.
The Mellin transforms along $\chi$ are defined as
\begin{align*} 
\begin{split}
\mathcal{S}(X_P^{\circ}(F)) &\lra I(\chi) \\
f &\longmapsto f_{\chi}(\cdot):=f_{\chi,P}(\cdot):=\int_{M^{\mathrm{ab}}(F)}\delta_P^{1/2}(m)\chi(\omega_P(m))f(m^{-1}\cdot)dm, \\
\mathcal{S}(X_{P^{\mathrm{op}}}^\circ (F)) &\lra \overline{I}(\chi) \\
f &\longmapsto f_{\chi}^{\mathrm{op}}(\cdot):=f_{\chi,P^{\mathrm{op}}}^{\mathrm{op}}(\cdot):=\int_{M^{\mathrm{ab}}(F)}\delta_{P^{\mathrm{op}}}^{1/2}(m)\chi(\omega_P(m))f(m^{-1} \cdot)dm.
\end{split}
\end{align*}
Here $dm$ is the Haar measure on $M^{\mathrm{ab}}(F)$ transferred from that on $F^\times$ via $\omega_P$. In the notation $\overline{I}_{P^{\mathrm{op}}}(\chi)$ and $f_{\chi,P^{\mathrm{op}}}^{\mathrm{op}},$ the bar and the superscript $\mathrm{op}$ indicate that we are inducing from $\chi\circ\omega_P$ instead of $\chi\circ \omega_{P^{\mathrm{op}}}$. The same notation is used for functions in $C^\infty(X_P^\circ(F))$ and $C^\infty(X_{P^{\mathrm{op}}}^\circ(F))$ whenever the integrals defining $f_{\chi},f_{\chi}^{\mathrm{op}}$ exist for $\mathrm{Re}(\chi)$ in a subset of $\rr,$ and in some cases are extended to larger complex domains by analytic continuation.  

For $\chi\in \widehat{F}^\times$, when $F$ is archimedean (resp. nonarchimedean), a section $f^{(s)} \in I(\chi_s)$ is said to be \textbf{holomorphic} if for all $g \in G(F)$ the complex function
\begin{align*}
    \cc &\lto \cc\\
    s &\longmapsto f^{(s)}(g) 
\end{align*}
is an entire function (resp. lies in $\cc[q^{-s},q^s]$). It is \textbf{meromorphic} if there exists a nonzero entire function $a(s)$ (resp. $a(s)\in \cc[q^{-s},q^{s}]$) such that $a(s)f^{(s)}$ is a holomorphic section.

    For a function $f \in C^\infty(X_P^{\circ}(F))$ and $x^\ast=P^{\mathrm{op},\mathrm{der}}(F)g\in X_{P^{\mathrm{op}}}^{\circ}(F),$ we define the unnormalized intertwining operator
\begin{align*}
\mathcal{R}_{P|P^{\mathrm{op}}}(f)(x^\ast):=\int_{N_{P^{\mathrm{op}}}(F)}f\left(ug\right)du=\int_{N_{P^{\mathrm{op}}}(F)}f(ux^\ast)du
\end{align*}
whenever this integral is absolutely convergent (or obtained via some regularization procedure). Here the measure $du$ on $N_{P^{\mathrm{op}}}(F)$ is defined as in \S \ref{ssec:measure}. For $\chi\in\widehat{F}^\times$ and $\mathrm{Re}(\chi_s)$ sufficiently large, the integral above  converges absolutely for $f\in I(\chi_s)$, and it defines a map
\begin{align*}
    \mathcal{R}_{P|P^{\mathrm{op}}}:=\mathcal{R}_{P|P^{\mathrm{op}}}(\chi_s): I(\chi_s)\lto \overline{I}(\chi_s).
\end{align*}
The map extends meromorphically in $s$ to $\cc$. 

\subsection{The Schwartz space} \label{ssec:the:Schwartz:space}
We have a group homomorphism of the complex dual groups
\begin{align*}
    \GG_m(\cc)\times \SL_2(\cc)\longrightarrow \widehat{M}\lhook\joinrel\xrightarrow{\quad} \widehat{G}
\end{align*}
such that $\GG_m(\cc)\to Z(\widehat{M})=\widehat{M}^{\mathrm{ab}}$ is the isomorphism dual to $\omega_P$, and the derivaton of $\SL_2(\cc)\to \widehat{M}$ is a principal $\fsl_2$-triple of $\mathrm{Lie}(\widehat{M})$. This morphism is a distinguished morphism in the sense of \cite[Theorem 2.2.3]{SakVenkperiods}. It is unique up to conjugation by $\widehat{T}$.

Consider the adjoint representation $\widehat{\fn}_P:=\mathrm{Lie}(\widehat{N}_P)$ of $\widehat{M}$. Restricting to $\SL_2(\cc),$ we can decompose 
\begin{align*}
    \widehat{\fn}_P=\bigoplus_{i=1}^k \mathrm{Sym}^{2s_i}(\cc^2).
\end{align*}
for some $2s_i\in \zz_{\ge 0}.$ Let $\lambda_i\in \zz_{\ge 0}$ be the $\GG_m$-weight of a highest weight vector in $\mathrm{Sym}^{2s_i}(\cc^2)$. Consider the multiset $\Lambda:=\{(s_i,\lambda_i)\}_{1\le i\le k}$. By \cite[Lemma A.1]{Getz:Hsu:Leslie} we have $\lambda_i> 0$ for all $i$. Thus $\Lambda$ admits a \textbf{good ordering}, i.e., it can be reindexed so that
    \begin{align*}
        \frac{s_{i+1}}{\lambda_{i+1}}\ge \frac{s_i}{\lambda_i}
    \end{align*}
    for all $i$.  We henceforth assume $\Lambda$ is in such an ordering. It is shown in \cite[Appendix A]{Getz:Hsu:Leslie} that the highest data $(s_k,\lambda_k)$ is unique in the following sense.

\begin{Prop}\label{prop:lambda=1}
We have $\lambda_k=1$ and $s_k>\tfrac{s_i}{\lambda_i}$ for all $i<k$.\qed
\end{Prop}

\noindent
Moreover, by \cite[Proposition 6.2]{Getz:Hsu:Leslie} we have
\begin{align}\label{eq:sk}
\delta_P=|\omega_P|^{2(s_k+1)}.
\end{align}

Fix a maximal compact subgroup $K$ of $G(F)$ such that the Iwasawa decomposition $G(F)=P(F)K$ holds. 

\begin{Def} \label{defn:na}  Assume $F$ is nonarchimedean. The Schwartz space $\cals(X_P(F))$ consists of smooth functions $f \in C^\infty(X_P^{\circ}(F))$ satisfying the following two conditions:
\begin{enumerate}[label={(\arabic*na)}]
    \item  The function $f$ is right $K$-finite. For each $g\in G(F)$ and  $\chi\in\widehat{\calo}^\times$, the integral defining $f_{\chi_s}(g)$ is absolutely convergent for $\mathrm{Re}(s)$ sufficiently large (independent of $g$), and the section
\begin{align*}
        \frac{f_{\chi_s}}{\prod_{i=1}^k L(s_i+1,\chi_s^{\lambda_i})}  
\end{align*}
    is holomorphic. \label{cond:orig}
    \item  The section
\begin{align*}
 \frac{\mathcal{R}_{P|P^{\mathrm{op}}}(f)_{\chi_s}^{\mathrm{op}}}{\prod_{i=1}^{k}L(-s_i,\chi_s^{\lambda_i})} 
\end{align*}
is holomorphic for each $\chi\in\widehat{\calo}^\times$.\label{cond:origradon}
\end{enumerate} 
\end{Def}

For real numbers $\sigma_{1}<\sigma_2,$ let
\begin{align*}
    V_{\sigma_1,\sigma_2}:=\left\{ s\in \cc: \sigma_1<\mathrm{Re}(s)<\sigma_2\right\}.
\end{align*}
For a meromorphic function $h$ on $V_{\sigma_1,\sigma_2}$ and a polynomial $Q\in \cc[s],$ define
\begin{align*}
    |h|_{\sigma_1,\sigma_2,Q}:=|hQ|_{\sigma_1,\sigma_2}:=\sup_{s\in V_{\sigma_1,\sigma_2}} |h(s)Q(s)|.
\end{align*}

\begin{Def}
    Assume $F$ is archimedean. The Schwartz space $\mathcal{S}(X_P(F))$ consists of smooth functions $f \in C^\infty(X_P^{\circ}(F))$ satisfying the following two conditions:
\begin{enumerate}[label={(\arabic*a)}]
    \item For each $g\in G(F),\chi\in \widehat{K}_{\GG_m},$ and $D\in U(\fm^{\mathrm{ab}}\oplus \fg)$, the integral defining $(D.f)_{\chi_s}(g)$ is absolutely convergent for $\mathrm{Re}(s)$ sufficiently large (independent of $g$), the section
\begin{align*}
        \frac{(D.f)_{\chi_s}}{\prod_{i=1}^k L(s_i+1,\chi_s^{\lambda_i})}  
\end{align*}
    is holomorphic, and for all real numbers $\sigma_1<\sigma_2,$ polynomials $Q(s)=Q_{P|P}(s)\in \cc[s]$ such that $Q(s)\prod_{i=1}^k L(s_i+1,\chi_s^{\lambda_i})$ is holomorphic on $V_{\sigma_1,\sigma_2}$ for all $\chi\in \widehat{K}_{\GG_m},$ and compact subsets $\Omega\subset X_P^\circ(F),$
    \begin{align*}
    |f|_{\sigma_1,\sigma_2,Q_{P|P},\Omega,D}:=\sum_{\chi\in \widehat{K}_{\GG_m}} \sup_{g\in \Omega} |(D.f)_{\chi_s}(g)|_{\sigma_1,\sigma_2,Q}<\infty.
    \end{align*} \label{cond:orig:arch} 
    \item The section
\begin{align*}
 \frac{(\mathcal{R}_{P|P^{\mathrm{op}}}(D.f))_{\chi_s}^{\mathrm{op}}}{\prod_{i=1}^{k}L(-s_i,\chi_s^{\lambda_i})} 
\end{align*}
is holomorphic, and for all real numbers $\sigma_1<\sigma_2,$ polynomials $Q(s)=Q_{P|P^{\mathrm{op}}}(s)\in \cc[s]$ such that $Q(s)\prod_{i=1}^k L(-s_i,\chi_s^{\lambda_i})$ is holomorphic on $V_{\sigma_1,\sigma_2}$ for all $\chi\in \widehat{K}_{\GG_m},$ and compact subsets $\Omega\subset X_P^\circ(F),$
    \begin{align*}
    |f|_{\sigma_1,\sigma_2,Q_{P|P^{\mathrm{op}}},\Omega,D}:=\sum_{\chi\in \widehat{K}_{\GG_m}} \sup_{g\in \Omega} |(\mathcal{R}_{P|P^{\mathrm{op}}}(D.f))_{\chi_s}^{\mathrm{op}}(g)|_{\sigma_1,\sigma_2,Q}<\infty.
    \end{align*}  \label{cond:origradon:arch}
\end{enumerate}
\end{Def}
\noindent These seminorms endow $\mathcal{S}(X(F))$ a structure of Fr\'echet space \cite[Lemma 3.2]{Getz:Hsu}.

For $F$ archimedean or nonarchimedean, by \cite[Lemma 5.9]{Getz:Hsu:Leslie}  for $\mathrm{Re}(\chi_s)$ sufficiently large we have a commutative diagram 
\begin{equation}\label{eq:Rcom}
\begin{tikzcd}
 \mathcal{S}(X_P(F)) \ar[r,"\mathcal{R}_{P|P^{\mathrm{op}}}"]\ar[d,"(\cdot)_{\chi_s}"] & \mathcal{R}_{P|P^{\mathrm{op}}}(\mathcal{S}(X_P(F)))\ar[d,"(\cdot)_{\chi_s}^{\mathrm{op}}"] \le C^\infty(X_{P^{\mathrm{op}}}^\circ(F))\\
 I(\chi_s) \ar[r,"\mathcal{R}_{P|P^{\mathrm{op}}}"] & \overline{I}(\chi_s).
\end{tikzcd}
\end{equation}

\begin{Rem}\label{rem:com}
    The commutative diagram \eqref{eq:Rcom} holds more generally for smooth functions in $L^1(X_P(F)),$ where the measure on $X_P(F)$ is given by \eqref{eq:dxmeasure} below. To see this, according to the proof of \cite[Lemma 5.9]{Getz:Hsu:Leslie}, it suffices to show that if $f\in L^{1}(X_P(F))$ then the integral defining $\mathcal{R}_{P|P^{\mathrm{op}}}(f)$ is absolutely convergent. This follows from the proof of \cite[Theorem 6.5]{Getz:Hsu:Leslie}.
\end{Rem}

\begin{Rem}
    In \cite{Getz:Hsu:Leslie} the conditions \ref{cond:origradon} and \ref{cond:origradon:arch} in the definition of  Schwartz spaces use $\mathcal{R}_{P|P^{\mathrm{op}}}(f_{\chi_s})$ instead of $\mathcal{R}_{P|P^{\mathrm{op}}}(f)_{\chi_s}^{\mathrm{op}}.$ However, they are all equivalent by Remark \ref{rem:com} and Proposition \ref{prop:asympinfty} below.
\end{Rem}

Consider the box norm on $V_P(F)\cong F^{\dim V_P}$ given by
\begin{align*}
    |(y_1,\ldots, y_{\dim V_P})|:=\max_{1\le i\le \dim V_P }|y_i|.
\end{align*}
Note that when $F= \cc,$ this is not a norm in the classical sense. The box norm restricts to a function on $X_P(F)$
\begin{align*}
|\cdot|:X_P(F) &\lto \rr_{\ge 0}\\
x &\longmapsto |\mathrm{Pl}(x)|.
\end{align*}
We normalize it so that
\begin{align*}
    |mk|=|\omega_P(m)|^{-1} \quad \textrm{ for } m\in P(F), k\in K. 
\end{align*}    
For later use, we define
$$X_P^1:=\{x\in X_P(F) : |x|=1\}.$$
It is a compact subset of $X_P^\circ(F)$.

Since $X_P^\circ(F)$ is homogeneous by \cite[Corollary 3.3]{Getz:Hsu:Leslie}, it admits a right $G(F)$-invariant Radon measure, unique up to scalar. We let $dx$ be the measure on $X_P^\circ(F)$ such that 
 \begin{align}\label{eq:dxmeasure}
     d(mu)=\frac{\delta_{P^{\mathrm{op}}}(m)dm du}{\zeta(1)} \quad \textrm{ for } (m,u)\in M^\mathrm{ab}(F)\times N_{P^\mathrm{op}}(F).
 \end{align}
Extend it by zero to a measure on $X_P(F).$

\begin{Prop}\label{prop:asympinfty}
      Let $f\in C^\infty(X_P^\circ(F)).$ 
      \begin{enumerate}
          \item Suppose $F$ is nonarchimedean.  If $f$ satisfies \ref{cond:orig}, then $f$ has compact support in $X_P(F),$ i.e., $f(x)=0$ for $|x|\gg_f 1$ and $f\in  (L^1\cap L^2)(X_P(F)).$
          \item Suppose $F$ is archimedean.  If $f$ satisfies \ref{cond:orig:arch}, then $f$ is rapidly decreasing toward the infinity and $f\in  (L^1\cap L^2)(X_P(F)).$
      \end{enumerate}
\end{Prop}

\begin{proof}
    This follows from the proof of \cite[Lemma 5.7]{Getz:Hsu:Leslie} and the Iwasawa decomposition.
\end{proof}

The following is straightforward from the definition.
\begin{Lem}\label{lem:cc}
    The Schwartz space $\mathcal{S}(X_P(F))$ is stable under taking  the complex conjugate $f\to \bar{f}$. Moreover, this $\rr$-linear map  is continuous when $F$ is archimedean.\qed
\end{Lem}

\subsection{The Fourier transform $\mathcal{F}_{P|P^{\mathrm{op}}}$}

\sloppy In \cite{BKnormalized} Braverman and Kazhdan constructed a Fourier transform
\begin{align*}
    \calf_{P|P^{\mathrm{op}}}:L^2(X_{P}(F))\longrightarrow L^2(X_{P^{\mathrm{op}}}(F))
\end{align*}
such that $\calf_{P^{\mathrm{op}}|P}\circ \calf_{P|P^{\mathrm{op}}}=\mathrm{Id}$. The \textbf{BK-Schwartz space} is
\begin{align*}
   \mathcal{S}_{\mathrm{BK}}(X_P(F)):=\mathcal{S}(X_P^\circ(F))+\calf_{P^{\mathrm{op}}|P}(\mathcal{S}(X_{P^{\mathrm{op}}}^\circ(F))),
\end{align*}
i.e., it is the smallest subspace of $C^\infty(X_P^\circ(F))$ containing $\mathcal{S}(X_P^\circ(F))$ and stable under the Fourier transform. Their construction in this particular setting is refined in \cite{Getz:Hsu:Leslie}.

For each $(s,\lambda)\in \cc\times\zz_{>0}$,  define a linear map
\begin{align*}
  \lambda_!(\mu_s):  \mathcal{S}(X_{P^{\mathrm{op}}}^\circ(F)) &\lto C^\infty(X_{P^{\mathrm{op}}}^\circ(F))
  \end{align*}
  given by 
\begin{align*}
\lambda_!(\mu_s)(f)(x):=\int_{M^{\mathrm{ab}}(F)}\psi(\omega_P(m))|\omega_P(m)|^{s+1}\delta_{P^{\mathrm{op}}}^{\lambda /2}(m)f(m^{-\lambda}x)\frac{dm}{\zeta(1)}.
\end{align*}
The same notation is used if the integral exists by regularization. We define operators
\begin{align*}
    \mu_{P}^{\mathrm{aug}}:=\lambda_{1!}(\mu_{s_1}) \circ \dots \circ \lambda_{(k-1)!}(\mu_{s_{k-1}}) \quad\text{and}\quad\mu_P^{\mathrm{geo}}:=1_{!}(\mu_{s_k}).
\end{align*}
The following is \cite[Theorems 5.12 and 6.5]{Getz:Hsu:Leslie}.
 
\begin{Thm} \label{Thm: Fourier formula}
The Fourier transform $\mathcal{F}_{P|P^{\mathrm{op}}}$ descends to an isomorphism
\begin{align*}
    \calf_{P|P^{\mathrm{op}}}:\mathcal{S}(X_P(F))\longrightarrow \mathcal{S}(X_{P^{\mathrm{op}}}(F))
\end{align*}
that is continuous in the archimedean case. It is given by
\begin{align*}
    \mathcal{F}_{P|P^{\mathrm{op}}}=\mu_{P}^{\mathrm{aug}} \circ \mathcal{F}_{P|P^{\mathrm{op}}}^{\mathrm{geo}},
\end{align*}
where 
\begin{align*}
     \calf_{P|P^{\mathrm{op}}}^{\mathrm{geo}}:=\mu_{P}^{\mathrm{geo}}\circ\calr_{P|P^{\mathrm{op}}}.
 \end{align*}
Moreover, for $f\in  \mathcal{S}(X_P(F))$ and $x^\ast\in X_{P^{\mathrm{op}}}^\circ(F)$ we have
\[
\calf_{P|P^{\mathrm{op}}}^{\mathrm{geo}}(f)(x^\ast)=
\int_{X_{P}^\circ(F)}f(x)\psi\left(\la x,x^\ast\ra_{P|P^{\mathrm{op}}}\right)dx.
\]
Here the pairing $\la \cdot,\cdot\ra_{P|P^{\mathrm{op}}}$ is defined as in \eqref{eq:pairing}.

The following diagram 
\begin{center}
\begin{tikzcd}
    \mathcal{S}(X_P(F))\ar[r,"\mathcal{F}_{P|P^{\mathrm{op}}}"]\ar[d,"(\cdot)_{\chi}"] & \mathcal{S}(X_{P^{\mathrm{op}}}(F))\ar[d,"(\cdot)_\chi^{\mathrm{op}}=(\cdot)_{\chi^{-1}}"]\\
    I(\chi) \ar[r,"\mu_{\Lambda}(\chi)\mathcal{R}_{P|P^{\mathrm{op}}}"] & \overline{I}(\chi)
\end{tikzcd}
\end{center}
commutes, where
\begin{align*}
    \mu_{\Lambda}(\chi):=\prod_{i=1}^k \gamma(-s_i,\chi^{\lambda_i},\psi)
\end{align*}
is a product of Tate $\gamma$-factors.
\qed
\end{Thm}

One can further write $\calf_{P|P^{\mathrm{op}}}(f)$ as an iterated integral for $f\in \mathcal{S}(X_P(F))$. For some detailed examples, we refer one to \cite[\S 6.3]{Getz:Hsu:Leslie}.  However, it is cumbersome to check if a given smooth function on $X_P^\circ(F)$ belongs to $\mathcal{S}(X_P(F))$ from the definition. For instance, it is hard to see if the containment $\mathcal{S}(X_P^\circ(F))\subseteq \mathcal{S}(X_P(F))$ is true from the definition, and thus the relation between $\mathcal{S}(X_P(F))$ and $\mathcal{S}_{\mathrm{BK}}(X_P(F))$ is unclear. See \cite[\S 5.4]{Getz:Hsu:Leslie} for a discussion. In \S \ref{sec:casebycase} we will justify the inclusion $\mathcal{S}(X_P^\circ(F))\subseteq \mathcal{S}(X_P(F))$ and further show that $\mathcal{S}(X_P(F))=\mathcal{S}_{\mathrm{BK}}(X_P(F))$ when $G$ is either classical or $G_2$. 

\subsection{A discussion on Schwartz spaces over archimedean local fields.}

Suppose $F$ is archimedean. Let $\mathcal{S}_\Lambda(X_P^\circ(F))$ be the subset of $C^\infty(X_P^\circ(F))$ consisting of functions satisfying \ref{cond:orig:arch}. By the Mellin inversion, the seminorms $|\cdot|_{\sigma_1,\sigma_2,Q_{P|P},\Omega,D}$ endow $\mathcal{S}_\Lambda(X_P^\circ(F))$ with a structure of a Fr\'echet space as argued in \cite[Lemma 3.2]{Getz:Hsu}.

\begin{Prop}\label{prop:holomorphy}
    Let $f\in \mathcal{S}_\Lambda(X_P^\circ(F))$. Then $f\in \mathcal{S}(X_P(F))$ if and only if the section
\begin{align*}
 \frac{(\mathcal{R}_{P|P^{\mathrm{op}}}(D.f))_{\chi_s}^{\mathrm{op}}}{\prod_{i=1}^{k}L(-s_i,\chi_s^{\lambda_i})} 
\end{align*}
is holomorphic for all $D\in U(\mathfrak{m}^\mathrm{ab}\oplus \mathfrak{g})$ and $\chi\in \widehat{K}_{\GG_m}$. In particular, the Schwartz space $\mathcal{S}(X_P(F))$ is a closed subspace of $\mathcal{S}_\Lambda(X_P^\circ(F))$. 
\end{Prop}

\begin{proof}
    Let 
    \begin{align*}
        \mathcal{S}:=\left\{f\in \mathcal{S}_{\Lambda}(X_P^\circ(F)):\frac{(\mathcal{R}_{P|P^{\mathrm{op}}}(D.f))_{\chi_s}^{\mathrm{op}}}{\prod_{i=1}^{k}L(-s_i,\chi_s^{\lambda_i})} \textrm{ is holomorphic for all } D\in U(\mathfrak{m}^\mathrm{ab}\oplus \mathfrak{g}),\chi\in \widehat{K}_{\GG_m} \right\}
    \end{align*}
    be equipped with the subspace topology. Note that $\mathcal{S}$ is by definition closed under the action of $U(\mathfrak{m}^\mathrm{ab}\oplus \mathfrak{g})$.
    Clearly, we have a continuous inclusion $\mathcal{S}(X_P(F))\hookrightarrow \mathcal{S}$. Let us show for $f\in \mathcal{S}$ the seminorms $|f|_{\sigma_1,\sigma_2,Q_{P|P^{\mathrm{op}}},\mathrm{Id},\Omega}$ are finite.
    
    By \cite[Lemma 10.1.2]{Wallach:RGII} there is an absolute constant $C$ such that for any $f\in \mathcal{S}_{\Lambda}(X_P^\circ(F)),$ $\chi\in \widehat{K}_{\GG_m}$ and $\sigma_2>\sigma_1\ge C$,
    \begin{align*}
        \sup_{g\in \Omega} |(\mathcal{R}_{P|P^{\mathrm{op}}}(f))_{\chi_s}^{\mathrm{op}}(g)|_{\sigma_1,\sigma_2}<\infty.
    \end{align*}
    Since $\mathcal{S}_{\Lambda}(X_P^\circ(F))$ is closed under the action of $U(\mathfrak{g}),$ by \cite[Theorem 10.1.5]{Wallach:RGII} there is a nonzero entire function $a(s)$ such that for any $f\in \mathcal{S}_{\Lambda}(X_P^\circ(F))$, $a(s)(\mathcal{R}_{P|P^{\mathrm{op}}}(f))_{\chi_s}^{\mathrm{op}}$ is holomorphic. Furthermore, if we let $S$ denote the set of zeros of $a(s)$, for any $\sigma_2>\sigma_1$ and $\epsilon>0$
     \begin{align*}
        \sup_{g\in \Omega} \sup_{\substack{\sigma_1<\mathrm{Re}(s)<\sigma_2\\ |s-z|>\epsilon, z\in S}}|(\mathcal{R}_{P|P^{\mathrm{op}}}(f))_{\chi_s}^{\mathrm{op}}(g)|<\infty.
    \end{align*}
    As $\mathcal{S}$ is closed under the action of $U(\mathfrak{m}^{\mathrm{ab}}),$ by the standard trick of integration by parts (see e.g., \cite[Lemmas 5.8 and  5.9]{Getz:Liu:BK}) and the definition of $\mathcal{S}$, one has $|f|_{\sigma_1,\sigma_2,Q_{P|P^{\mathrm{op}}},\mathrm{Id},\Omega}<\infty$.

    For $n\in \zz_{> 0}$ and $s_0\in \cc$, the maps
    \begin{align*}
        T_{s_0,n,D,\chi,g}(f):=\mathrm{Res}_{s=s_0} (s-s_0)^n\frac{(\mathcal{R}_{P|P^{\mathrm{op}}}(D.f))_{\chi_s}^{\mathrm{op}}(g)}{\prod_{i=1}^{k}L(-s_i,\chi_s^{\lambda_i})}
    \end{align*}
    are continuous on $\mathcal{S}_{\Lambda}(X_P^\circ(F)).$ Therefore the space $\mathcal{S},$ being the intersection of  kernels of $T_{s_0,n,D,\chi,g},$ is a closed subspace of $\mathcal{S}_{\Lambda}(X_P^\circ(F)),$ and hence it is a Fr\'echet space. We conclude that the inclusion $\mathcal{S}(X_P(F))\hookrightarrow\mathcal{S}$ is a continuous bijection of Fr\'echet spaces, and hence a homeomorphism by the open mapping theorem.
\end{proof}

For later use, we record the following.
\begin{Lem}\label{lem:archfil}
Let $\chi\in \widehat{K}_{\GG_m}$.  For any $n\ge 0$ and $s_0\in \cc$, the set
\begin{align*}
    \left\{ f\in\mathcal{S}(X_P(F)): \mathrm{ord}_{s=s_0} f_{\chi_s}(g)\le n \textrm{ for all } g\in G(F)\right\}
\end{align*}
is closed in $\mathcal{S}(X_P(F)).$ \qed 
\end{Lem}

\section{Nonarchimedean asymptotics}\label{sec:Sch}

In this and the following section, we give an asymptotic description of $\mathcal{S}(X_P(F))$ assuming the set-theoretic inclusion 
\begin{align*}
    \mathcal{S}(X_P^\circ(F)) \subseteq \mathcal{S}(X_P(F))
\end{align*}
holds. The assumption is verified by Corollary \ref{cor:cont} below if $G$ is either classical or $G_2$.
 
 Fix a choice of section $m:\mathbb{G}_m\to M$ of $\omega_P$. We describe a smooth function on $X_P^\circ(F)$ by analyzing its behavior under the action of $M^{\mathrm{ab}}(F)$ and $G(F)$. To do this we will proceed in two steps. First, observe that by Proposition \ref{prop:asympinfty} the space $A_{X_P(F)}=\mathcal{S}(X_P(F))/\mathcal{S}(X_P^\circ(F))$ consists of germs of Schwartz functions at the origin. Therefore, we will restate \ref{cond:orig} and \ref{cond:orig:arch} in terms of asymptotics toward the origin following the idea of Igusa \cite[\S 1]{Igusa:forms}. This approach is already taken in \cite{Jiang:Luo:Zhang}. Then we reinterpret \ref{cond:origradon} and \ref{cond:origradon:arch} representation-theoretically.

As in \cite{Igusa:forms}, we treat the nonarchimedean case and the archimedean case separately. Throughout this section $F$ is nonarchimedean. Since the Schwartz space is independent of $\psi$, we may assume the conductor of $\psi$ is $\mathcal{O},$ i.e., $\psi$ is trivial on $\mathcal{O}$ but not on $\varpi^{-1}\mathcal{O}$.

\subsection{Work of Igusa: nonarchimedean}

 We first review \cite[\S 1.5]{Igusa:forms}.
Let $L$ be a finite multiset whose underlying set is a subset of $\cc/\tfrac{2\pi \sqrt{-1}}{\log q}\zz$. We define $C^\infty_L(F^\times)$ to be the subspace of $C^\infty(F^\times)$ that consists of functions $f$ satisfying the following assumptions.
\begin{enumerate}
    \item The function $f$ has compact support in $F$, and
    \item for $|a|\ll_f 1$, we have
    \begin{align*}
        f(a)=\sum_{r\in \mathrm{Supp}(L)}\sum_{j=1}^{m_r(L)} c_{r,j}(\tilde{a})|a|^{r}\ord(a)^{j-1}
    \end{align*}
for some $c_{r,j}\in C^\infty( \mathcal{O}^\times)$.
\end{enumerate}
Note that we can further write $$c_{r,j}=\sum_{\chi\in\widehat{\calo}^\times} c_{r,j,\chi} \chi$$ for some constants $c_{r,j,\chi}\in\cc$ that are zero for all but finitely many $\chi$. On the other hand, let $\mathcal{Z}_L(\widehat{F}^\times)$  be the set of complex-valued functions $Z$ on $\widehat{F}^\times$ such that
\begin{enumerate}
    \item for every $\chi\in \widehat{\calo}^\times$, there exist constants $b_{r,j,\chi}\in\cc$ such that
    \begin{align}\label{eq:laurent}
        Z(\chi_s)-\sum_{r\in\mathrm{Supp}(L)}\sum_{j=1}^{m_r(L)} b_{r,j,\chi}\zeta(s+r)^j
    \end{align}
    is a function in $\cc[q^{-s},q^s]$, and
    \item for all but finitely many $\chi\in \widehat{\calo}^\times$, $Z(\chi_s)=0$ for all $s$.
\end{enumerate}

Igusa showed in \cite[Theorem 1.5.3]{Igusa:forms} the following:
\begin{Thm}\label{thm:Igusa}
For $f\in C^\infty_L(F^\times),$ the Mellin transform
    $$M:f\mapsto \left(\chi \mapsto \int_{F^\times} f(a)\chi(a)d^\times a\right),$$
originally defined on $\chi\in \widehat{F}^\times$ with $\mathrm{Re}(\chi)\gg_{L} 1$, extends analytically to whole $\widehat{F}^\times$ and gives rise to an isomorphism between $C^\infty_L(F^\times)$ and $\mathcal{Z}_L(\widehat{F}^\times)$. Moreover, the Mellin inversion $M^{-1}$ is given as follows: Given $Z\in \mathcal{Z}_L(\widehat{F}^\times)$, for each $\chi\in \widehat{\calo}^\times$ let $Z_\chi(z)$ be the complex function obtained from $Z(\chi_s)$ by substituting $z$ for $q^{-s}$. Then for $a\in F^\times$,
\begin{align}\label{eq:invchar}
    M^{-1}(Z)(a)= \sum_{\chi\in \widehat{\calo}^\times}\mathrm{Res}_{z=0} (Z_\chi(z)z^{-\mathrm{ord}(a)-1})\chi^{-1}(a).
\end{align}\qed
\end{Thm}
\noindent Note that the sum in \eqref{eq:invchar} is a finite sum.

\subsection{Asymptotics toward the origin}\label{ssec:nec}

Recall that we have fixed a good ordering of $\Lambda=\{(s_i,\lambda_i)\}_{1\le i\le k}$ and thus $\lambda_k=1$. For each $d\in \zz_{>0}$, define (finite) multisets, whose underlying sets are subsets of $\cc/\tfrac{2\pi\sqrt{-1}}{\log q} \zz$:
\begin{align}\label{eq:levelset}
\begin{split}
    L(d):=&\sum_{\substack{i: \lambda_i=d}}\left\{\tfrac{s_i+1}{\lambda_i}-(s_k+1)+\tfrac{\mu}{\lambda_i}\tfrac{2\pi\sqrt{-1}}{\log q}: 0\le \mu<\lambda_i\right\},\\
    L_d:=&\sum_{d|\lambda} L(\lambda),
\end{split}
\end{align}
where sums are taken in the sense of multisets (see \S \ref{ssec:multi} for conventions). Note that for $r\in \cc/\tfrac{2\pi\sqrt{-1}}{\log q} \zz,$ its real part $\mathrm{Re}(r)$ is well defined. From Proposition \ref{prop:lambda=1} we deduce

\begin{Lem}\label{lem:basic}
For $r\in \mathrm{Supp}(L_1)$, $0\ge \mathrm{Re}(r)>-(s_k+1)$. Moreover,  $\mathrm{Re}(r)=0$ if and only if $r=0$, and in this case $m_0(L_1)=1$.\qed 
\end{Lem}

\begin{Prop}\label{prop:globalchar}
    Let $f\in C^\infty(X_P^\circ(F))$. Then $f$ satisfies \ref{cond:orig} if and only if 
    \begin{enumerate}
        \item the function $f$ has compact support in $X_P(F)$, and 
        \item for $n$ large enough (depending on $f$) and $x\in X_P^1$, we have
    \begin{align}\label{eq:globalchar}
        f(m(\varpi^{n})^{-1}x)=\sum_{d=1}^\infty \sum_{\substack{\chi\in \widehat{\calo}^\times\\\ord(\chi)=d}}\sum_{r\in \mathrm{Supp}(L_d)}\sum_{j=1}^{m_r(L_d)} c_{r,j,\chi}(x)q^{-nr}n^{j-1},
    \end{align}
    for some $c_{r,j,\chi}\in C^\infty(X_P^1)$ satisfying $c_{r,j,\chi}(m(a)^{-1}x)=\chi(a)c_{r,j,\chi}(x)$ for $a\in \calo^\times$.
    \end{enumerate}
\end{Prop}

\begin{Rem}\label{rem:d}
Since $\Lambda$ is a finite multiset, the sum in \eqref{eq:globalchar} is actually finite. Indeed, $L(d)$ and $L_d$ are empty for $d>6$ (see Appendix \ref{appendix}).
\end{Rem}

\begin{proof}
    By Proposition \ref{prop:asympinfty} both assumptions imply $f$ is right $K$-finite and has compact support in $X_P(F),$ which we assume henceforth. 
    
    Let $K'\le K$ be a compact open subgroup such that $f$ is right $K'$-invariant. Let $\{g_i\}$ be a (finite) set of representatives of left cosets of $K'$ in $K$. By the Iwasawa decomposition, \ref{cond:orig} is satisfied if and only if they are satisfied for all $g_i$. On the other hand, since the set $\{g_i\}$ is finite, asymptotics \eqref{eq:globalchar} can be checked at each $g_i$ separately. 
        
    Therefore, it suffices to fix $g\in G(F)$ and study the smooth function 
\begin{align*}
    f_g: F^\times&\longrightarrow \cc \\
        a&\mapsto f(m(a)^{-1}g),
\end{align*}
Note that $f_g$ has compact support in $F$ and is $\calo^\times$-finite by our assumption. Moreover, for $\chi\in\widehat{\calo}^\times$ by \eqref{eq:sk} we can write
\begin{align*}
        f_{\chi_s}(g)=M(f_g)(\chi_{s+s_k+1}).
\end{align*}
Then \ref{cond:orig} is reduced to
\begin{enumerate}
    \item [(1na')]  The integral defining $M(f_g)(\chi_s)$ is absolutely convergent for $\mathrm{Re}(s)$ sufficiently large, and the complex function 
    \begin{align*}
    \frac{M(f_g)(\chi_{s})}{\prod_{i=1}^k L(s_i+1-\lambda_i(s_k+1),\chi_s^{\lambda_i})} 
\end{align*}
is a function in $\cc[q^{-s},q^s]$.
\end{enumerate}
    To prove the proposition, we need to show $f_g$ satisfies (1na') if and only if there exist constants $c_{r,j,\chi}\in\cc$ such that for $|a|\ll_{f_g} 1$ we have
    \begin{align*}
    f_g(a)=\sum_{d=1}^\infty \sum_{\substack{\chi\in \widehat{\calo}^\times\\\ord(\chi)=d}}\sum_{r\in \mathrm{Supp}(L_d)}\sum_{j=1}^{m_r(L_d)} c_{r,j,\chi}\chi(a)|a|^r\mathrm{ord}(a)^{j-1}.
    \end{align*}

Observe that for $\chi\in\widehat{\calo}^\times$ of order $d$, we have
    \begin{align*}
        \prod_{i=1}^k L(s_i+1-\lambda_i(s_k+1),\chi_s^{\lambda_i})=\prod_{r\in \mathrm{Supp}(L_d)} \zeta(s+r)^{m_r(L_d)}.
    \end{align*}
Therefore, the set
\begin{align*}
    \big\{M(f_g): f_g \textrm{ satisfies (1na')} \big\}
\end{align*}
is the subspace of functions in $\mathcal{Z}_{L_1}(\widehat{F}^\times)$ such that the constants $b_{r,j,\chi}$ in \eqref{eq:laurent} are zero unless $r\in \mathrm{Supp}(L_{\mathrm{ord}(\chi)})$ and $j\le m_r(L_{\mathrm{ord}(\chi)})$.
By Theorem \ref{thm:Igusa} and \eqref{eq:invchar}, this is equivalent to requiring the constants $c_{r,j,\chi^{-1}}$ of $f_g\in C^\infty_{L_1}(F^\times)$ are zero unless $r\in \mathrm{Supp}(L_{\mathrm{ord}(\chi)})$ and $j\le m_r(L_{\mathrm{ord}(\chi)})$. The assertion follows as $\mathrm{ord}(\chi)=\mathrm{ord}(\chi^{-1})$. 
\end{proof}

 To proceed  we first give an alternative definition of $\mathcal{S}(X_P(F))$ in terms of $\mathcal{F}^{\mathrm{geo}}_{P|P^\mathrm{op}}$ instead of $\mathcal{R}_{P|P^{\mathrm{op}}}$.
    
\begin{Lem}\label{lem:newdef}
Suppose $f\in C^\infty(X_P^\circ(F))$ satisfies \ref{cond:orig}. Then \ref{cond:origradon} is equivalent to the following: For each $\chi\in \widehat{\calo}^\times$ and  $g\in G(F)$, the integral defining
        $\calf^{\mathrm{geo}}_{P|P^{\mathrm{op}}}(f)_{\chi_s}^{\mathrm{op}}(g)$
    is absolutely convergent for $\tfrac{s_{k-1}}{\lambda_{k-1}}<\mathrm{Re}(s)<s_k+1$, and the section
    \begin{align*}
          \frac{\calf^{\mathrm{geo}}_{P|P^{\mathrm{op}}}(f)_{\chi_s}^{\mathrm{op}}}{L(s_k+1,\chi_s^{-1})\prod_{i=1}^{k-1}L(-s_i,\chi^{\lambda_i}_s)}
    \end{align*}
    is holomorphic.
\end{Lem}

\begin{proof}
 Recall by definition $\mathcal{F}_{P|P^{\mathrm{op}}}^{\mathrm{geo}}=1_!(\mu_{s_k})\circ \mathcal{R}_{P|P^{\mathrm{op}}}$. The sufficiency follows from the work in \cite[\S 4]{Getz:Hsu:Leslie}. For the converse, in view of \cite[Proposition 4.8]{Getz:Hsu:Leslie} and its proof, it suffices to check the integral defining $\mathcal{R}_{P|P^\mathrm{op}}(f)^{\mathrm{op}}_{\chi_s}$ is holomorphic for $\mathrm{Re}(s)>s_k+1-\epsilon$ for some $\epsilon>0$ for all $\chi\in\widehat{\calo}^\times$. Since $f\in L^1(X_P(F))$ by Proposition \ref{prop:asympinfty}, by Fubini-Tonelli theorem one may reverse the proof of \cite[Theorem 6.5]{Getz:Hsu:Leslie} to see the integral defining $\mathcal{R}_{P|P^\mathrm{op}}(f)^{\mathrm{op}}_{\chi_s}$ is absolutely convergent for $\mathrm{Re}(s)=s_k+1$. Since it is also absolutely convergent for $\mathrm{Re}(s)\gg 0$, it is absolutely convergent for $\mathrm{Re}(s)\ge s_k+1$. As $\mathcal{R}_{P|P^\mathrm{op}}(f)^{\mathrm{op}}_{\chi_s}$ is meromorphic and vanishes for all but finitely many $\chi,$ there exists $\epsilon>0$ such that it is holomorphic for $\mathrm{Re}(s)>s_k+1-\epsilon$ for all $\chi$. 
\end{proof}

For $j\in\zz_{>0}$, let $h_j\in\zz[z]$ be the monic polynomial of degree $j-1$ such that 
\begin{align*}
    \frac{h_j(z)}{(1-z)^j}=\sum_{i=0}^\infty i^{j-1}z^i. 
\end{align*}
\begin{Prop}\label{prop:implicit}
 Let $f\in C^\infty(X_P^\circ (F))$ satisfy the equivalent conditions in Proposition \ref{prop:globalchar}. Then $f\in \mathcal{S}(X_P(F))$ if and only if for every $d\in \zz_{>0}$, $r\in \mathrm{Supp}(L_d), \chi\in\widehat{\calo}^\times$ of order $d$, and $x^\ast\in X_{P^{\mathrm{op}}}^\circ(F)$, the meromorphic function
    \begin{align*}
    \sum_{j=1}^{m_r(L_d)} h_j(q^{-s})\zeta(s)^{j}\int_{F^\times} |t|^{s-r-2s_k-2}\int_{X_P^1} c_{r,j,\chi}(x)\psi(\varpi^{-\mathrm{ord}(t)}\langle  x,x^\ast\rangle)dx d^\times t
\end{align*}
is holomorphic at $s=0$.
\end{Prop}

\begin{proof}
By the assumption $\mathcal{S}(X_P^\circ(F))\subseteq  \mathcal{S}(X_P(F))$ and linearity, to prove the assertion we can fix a triple $(d,r,\chi)$ with $r\in \mathrm{Supp}(L_d)$ and $\mathrm{ord}(\chi)=d,$ and assume for all $x\in X_P^1$
\begin{align*}
        f(m(\varpi^{n})^{-1}x)=\sum_{j=1}^{m_r(L_d)} c_{r,j,\chi}(x)q^{-nr}n^{j-1}.
\end{align*}
for $n\ge 0,$ and $f(m(\varpi^{n})^{-1}x)$ vanishes for $n<0$.

 Let $x^\ast=P^{\mathrm{op},\mathrm{der}}(F)g\in X_{P^{\mathrm{op}}}^\circ(F)$. Viewing  $c_{r,j,\chi}$ as a function in $\mathcal{S}(X_P^\circ(F))$ supported on $X_P^1$, we have by Theorem \ref{Thm: Fourier formula} 
 \begin{align*}
    \mathcal{F}_{P|P^{\mathrm{op}}}^{\mathrm{geo}}(c_{r,j,\chi})(x^\ast)=\int_{X_P^1} c_{r,j,\chi}(x)\psi(\langle  x,x^\ast\rangle)dx.
 \end{align*}
 Since $\mathrm{Re}(r+s_k+1)>0$ by Lemma \ref{lem:basic}, changing variables $t\mapsto at$, the integral 
\begin{align*}
    &\int_{F^\times}\int_{|a|\le 1} |t|^{s-s_k-1} |a|^{r+2s_k+2}\mathrm{ord}(a)^{j-1}\left|\mathcal{F}_{P|P^{\mathrm{op}}}^{\mathrm{geo}}(c_{r,j,\chi})(m(a^{-1}t)^{-1}x^\ast)\right| d^\times a d^\times t\\
    &=\int_{|a|\le 1}|a|^{s+r+s_k+1}\mathrm{ord}(a)^{j-1}\int_{F^\times} |t|^{s-s_k-1} \left|\mathcal{F}_{P|P^{\mathrm{op}}}^{\mathrm{geo}}(c_{r,j,\chi})(m(t)^{-1}x^\ast)\right| d^\times td^\times a
\end{align*}
converges for $\frac{s_{k-1}}{\lambda_{k-1}}<\mathrm{Re}(s)<s_k+1$ by Lemma \ref{lem:newdef}.

Thus by Fubini-Tonelli theorem and the Iwasawa decomposition for $\eta\in \widehat{\calo}^\times$
\begin{align*}
     &\calf^{\mathrm{geo}}_{P|P^{\mathrm{op}}}(f)_{\eta_s}^{\mathrm{op}}(g)\\
     &=\sum_{j=1}^{m_r(L_d)}\int_{F^\times} \eta(t)|t|^{s-s_k-1}\int_{|a|\le 1} \chi(a)|a|^{r+2s_k+2}\mathrm{ord}(a)^{j-1}\mathcal{F}_{P|P^{\mathrm{op}}}^{\mathrm{geo}}(c_{r,j,\chi})(m(a^{-1}t)^{-1}x^\ast)d^\times a d^\times t\\
     &=\sum_{j=1}^{m_r(L_d)}\int_{|a|\le 1} (\eta\chi)(a)|a|^{s+r+s_k+1}\mathrm{ord}(a)^{j-1}\int_{F^\times} \eta(t)|t|^{s-s_k-1}\mathcal{F}_{P|P^{\mathrm{op}}}^{\mathrm{geo}}(c_{r,j,\chi})(m(t)^{-1}x^\ast) d^\times td^\times a,
\end{align*}
which is zero unless $\eta\chi=1$; if $\eta=\overline{\chi}$, we have
\begin{align*}
   \calf^{\mathrm{geo}}_{P|P^{\mathrm{op}}}(f)_{\overline{\chi}_s}^{\mathrm{op}}(g)&= \sum_{j=1}^{m_r(L_d)}h_j(q^{-s-r-s_k-1})\zeta(s+r+s_k+1)^j\mathcal{F}_{P|P^{\mathrm{op}}}^{\mathrm{geo}}(c_{r,j,\chi})^{\mathrm{op}}_{\overline{\chi}_{s}}(g).
\end{align*}
Therefore, by Lemmas \ref{lem:basic} and \ref{lem:newdef}
\begin{align*}
    \frac{\calf^{\mathrm{geo}}_{P|P^{\mathrm{op}}}(f)_{\overline{\chi}_s}^{\mathrm{op}}(g)}{L(s_k+1,\overline{\chi}_s^{-1})\prod_{i=1}^{k-1}L(-s_i,\overline{\chi}^{\lambda_i}_s)}
\end{align*}
lies in $\cc[q^{-s},q^s]$ if and only if
\begin{align*}
    &\sum_{j=1}^{m_r(L_d)} h_j(q^{-s-r-s_k-1})\zeta(s+r+s_k+1)^{j}\int_{F^\times} |t|^{s-s_k-1} \int_{X_P^1} c_{r,j,\chi}(x)\psi(\varpi^{-\mathrm{ord}(t)}\langle  x,x^\ast\rangle)dx d^\times t
\end{align*}
is holomorphic at $s=-r-s_k-1$.
\end{proof}

For later use we write
\begin{align}\label{eq:radonnew}
    \Phi(s)(c_{r,j,\chi})(\cdot):= \int_{F^\times} |t|^{s-r-2s_k-2}\int_{X_P^1} c_{r,j,\chi}(x)\psi(\varpi^{-\mathrm{ord}(t)}\langle  x,\cdot \rangle)dx d^\times t.
\end{align}
If follows from the proof above and Lemma \ref{lem:newdef} that $\Phi(s)(c_{r,j,\chi})$ is a meromorphic section holomorphic at $s=0$.

\subsection{Representation-theoretic characterization}\label{ssec:explicit}

Let $(d,r,\chi)$ be a triple, where $d\in \zz_{>0}$ with $L_d$ nonempty, $r\in \mathrm{Supp}(L_d),$ and $\chi\in \widehat{\calo}^\times$ of order $d$.  For $n\in\zz_{>0},$ consider smooth $G(F)$-modules $I_{r,n,\chi}$ defined by
\begin{align*}
        \left\{f\in C^\infty(X_P^\circ(F)) : \begin{array}{c}
    f(x)=0 \textrm{ for } |x|\gg_f 1,\\
    f(x)=\sum_{j=1}^{n}|x|^{r}\log_{1/q}^{j-1}(|x|)c_{r,j,\chi}(m(\varpi^{\log_q |x|})^{-1}x) \\\textrm{ for } |x|\ll_f 1\end{array}
    \right\}\Big/\mathcal{S}(X_P^\circ(F)).
\end{align*}
For ease of notation, we will write interchangeably $f\in I_{r,n,\chi}$ as the vector $(c_{r,j,\chi})_{1\le j\le n}$. For convenience, we set $I_{r,0,\chi}:=0$.

\begin{Lem}\label{lem:easy}
For $n\in\zz_{>0}$, we have a canonical $G(F)$-equivariant isomorphism 
\begin{align*}
        \varphi_{r,n,\chi}: I_{r,n,\chi}/I_{r,n-1,\chi}\tilde{\longrightarrow} I_P(\overline{\chi}_{-r-s_k-1})
    \end{align*}
by sending $f$ to the unique function in $I_P(\overline{\chi}_{-r-s_k-1})$ that equals $c_{r,n,\chi}$ on $X_P^1$. 
\end{Lem}

\begin{proof}
The only unclear statement is the $G(F)$-equivariance.  Given $f=(c_{r,j,\chi})_{1\le j\le n}\in I_{r,n,\chi}$ and $g\in G(F)$, we have 
\begin{align*}
    R(g)f(x)&=f(xg)\\
    &=\sum_{j=1}^{n}|x|^{r}\left(\log_{1/q}(|x|)+\log_{1/q}\left(\frac{|xg|}{|x|}\right)\right)^{j-1}\left(\frac{|xg|}{|x|}\right)^r c_{r,j,\chi}(m(\varpi^{\log_q |xg|})^{-1}xg)\\
    &=\sum_{j=1}^{n}|x|^{r}\log_{1/q}^{j-1}(|x|)\sum_{i=j}^n \binom{i-1}{j-1}\log_{1/q}^{i-j}\left(\frac{|xg|}{|x|}\right)\left(\frac{|xg|}{|x|}\right)^r c_{r,i,\chi}(m(\varpi^{\log_q |xg|})^{-1}xg)
\end{align*}
for $|x|$ sufficiently small. In other words, $R(g)f=(\tilde{c}_{r,j,\chi})_{1\le j\le n}$, where
\begin{align}\label{eq:gtranslate}
    \tilde{c}_{r,j,\chi}(x):=\sum_{i=j}^n \binom{i-1}{j-1}\log_{1/q}^{i-j}\left(\frac{|xg|}{|x|}\right)\left(\frac{|xg|}{|x|}\right)^r c_{r,i,\chi}(m(\varpi^{\log_q |xg|})^{-1}xg).
\end{align}
Therefore for $x\in X_P^1,$ 
$$\varphi_{r,n,\chi}(R(g)f)(x)=|xg|^r c_{r,n,\chi}(m(\varpi^{\log_q |xg|})^{-1}xg)=R(g)\varphi_{r,n,\chi}(f)(x).$$ 
\end{proof}

We can understand \eqref{eq:radonnew} as a map
\begin{align*}
    \Phi(s):I_{r,1,\chi}\cong I_{r,n,\chi}/I_{r,n-1,\chi}&\longrightarrow I_{P^{\mathrm{op}}}(\chi_{-s+r+s_k+1})\\
    c_{r,n,\chi}&\mapsto  \Phi(s)(c_{r,n,\chi})
\end{align*}
 that produces meromorphic sections holomorphic at $s=0$. 
 \begin{Lem} \label{lem:equivariant}
 We have a commutative diagram
\begin{center}
    \begin{tikzcd}
            I_{r,1,\chi}\ar[d,"\varphi_{r,1,\chi}"]\ar[dr, "c\cdot \Phi(0)"]& \\
            I_P(\overline{\chi}_{-r-s_k-1}) \ar[r," \mathcal{R}_{P|P^{\mathrm{op}}}"] & I_{P^{\mathrm{op}}}(\chi_{r+s_k+1}),
    \end{tikzcd}
\end{center}
where $c$ is a nonzero constant. In particular, $\Phi(0)$ is $G(F)$-equivariant.
\end{Lem}

\begin{proof}
By \cite[Proposition 4.8]{Getz:Hsu:Leslie} and \eqref{eq:Rcom}, we have an identity of meromorphic sections
\begin{align*}
     &\int_{F^\times} \overline{\chi}(t)|t|^{s-r-2s_k-2}\int_{X_P^1} c_{r,1,\chi}(x)\psi(t^{-1}\langle  x,\cdot \rangle)dx d^\times t\\
     &=\left(1_!(\mu_{s_k})\circ \mathcal{R}_{P|P^{\mathrm{op}}}\left(c_{r,1,\chi}\mathbf{1}_{X_P^1}\right)\right)_{\overline{\chi}_{s-r-s_k-1}}^{\mathrm{op}}\\
     &=\gamma(-s_k,\overline{\chi}_{s-r-s_k-1},\psi)^{-1}\mathcal{R}_{P|P^{\mathrm{op}}}\left(\left(c_{r,1,\chi}\mathbf{1}_{X_P^1}\right)_{\overline{\chi}_{s-r-s_k-1}}\right).
\end{align*}
 By Lemma \ref{lem:basic} $\gamma(-s_k,\overline{\chi}_{-r-s_k-1},\psi)$ is a nonzero constant, and evaluating at $s=0$ we have
\begin{align*}
    \Phi(0)(c_{r,1,\chi})=\gamma(-s_k,\overline{\chi}_{-r-s_k-1},\psi)^{-1}\mathcal{R}_{P|P^{\mathrm{op}}} \circ \varphi_{r,1,\chi}(c_{r,1,\chi}).
\end{align*} 
\end{proof}

 For integers $m_r(L_d)\ge n\ge 1$, define $G(F)$-submodules 
\begin{align}\label{eq:Achi}
    \nonumber A_{r,n,\chi}&:=\left\{(c_{r,j,\chi})_{1\le j\le n} \in I_{r,n,\chi} : \sum_{j=1}^{n} h_j(q^{-s})\zeta(s)^j\Phi(s)(c_{r,j,\chi}) \textrm{ is holomorphic at } s=0\right\},\\
    A_{r,\chi}&:=A_{r,m_r(L_d),\chi}.
\end{align}
Now we are ready to state our main result.
\begin{Thm}\label{thm:exact}
Suppose $F$ is a nonarchimedean local field with $\mathcal{S}(X_P^\circ(F))\subseteq \mathcal{S}(X_P(F))$. We have an exact sequence of smooth $G(F)$-modules
\begin{align*}
    0\longrightarrow \mathcal{S}(X_P^\circ(F))\longrightarrow \mathcal{S}(X_P(F))\longrightarrow \bigoplus_{d\ge 1, L_d\neq \emptyset } \bigoplus_{\substack{\chi\in \widehat{\calo}^\times\\\ord(\chi)=d}}\bigoplus_{r\in \mathrm{Supp}(L_d)} A_{r,\chi}\longrightarrow 0.
\end{align*}
Each $A_{r,\chi}$ admits a natural filtration of $G(F)$-submodules
\begin{align*}
    0=A_{r,0,\chi}< A_{r,1,\chi}< \cdots < A_{r,m_r(L_d),\chi}=A_{r,\chi}
\end{align*}
together with canonical $G(F)$-equivariant injections
\begin{align}\label{eq:injseq}
 0 \neq A_{r,m_r(L_d),\chi}/A_{r,m_r(L_d)-1,\chi} \hookrightarrow \cdots  \hookrightarrow A_{r,2,\chi}/A_{r,1,\chi} \hookrightarrow A_{r,1,\chi},
\end{align}
and
$A_{r,1,\chi}\cong\mathrm{Ker}(\mathcal{R}_{P|P^{\mathrm{op}}})< I_P(\overline{\chi}_{-r-s_k-1})$ is a proper submodule. 
\end{Thm}

 In comparison with notations in Theorem \ref{thm:main:filt}, we have $A_{r,n,\chi}=A_{n}(\overline{\chi}_{-r-s_k-1})$ and $m(\overline{\chi}_{-r-s_k-1})=m_r(L_{\mathrm{ord}(\chi)}).$

\begin{proof}
The exact sequence follows directly from Proposition \ref{prop:implicit} and the definition of $A_{r,\chi}$.  For integers $j\ge 0,$ let
\begin{align}\label{eq:Mj}
   M_j:=\big\{ c_{r,1,\chi}\in I_{r,1,\chi}: \zeta(s)^j\Phi(s)(c_{r,1,\chi}) \textrm{ is holomorphic at } s=0\big\}
\end{align}
Then we have a descending chain of $G(F)$-modules
\begin{align}\label{eq:alternativedef}
    I_{r,1,\chi}=M_0\ge M_1\ge M_2\ge \cdots.
\end{align}
 By definition $h_j(1)\neq 0$ and since $\Phi(0)$ is $G(F)$-equivariant by Lemma \ref{lem:equivariant}, we have
\begin{align*}
    A_{r,n,\chi}=\big\{(c_{r,j,\chi})_{1\le j\le n} \in I_{r,n,\chi} : c_{r,j,\chi}\in M_j\big\}.
\end{align*}
Consequently, for $1\le n\le m_r(L_d)$ we have a canonical $G(F)$-equivariant isomorphism
\begin{align}\label{eq:successivequotient}
    A_{r,n,\chi}/A_{r,n-1,\chi}\cong M_n.
\end{align}
Then \eqref{eq:injseq} follows from \eqref{eq:alternativedef}. The nontriviality statement 
\begin{align*}
A_{r,m_r(L_d),\chi}/A_{r,m_r(L_d)-1,\chi}\neq 0
\end{align*}
will be proved in Corollary \ref{cor:nontrivial}. 

We are left to inspect the module $A_{r,1,\chi}\cong M_1,$ which is isomorphic to $\mathrm{Ker}(\mathcal{R}_{P|P^{\mathrm{op}}})$ by Lemma \ref{lem:equivariant}. To prove it is proper, we show there exists $c_{r,1,\chi}\in I_{r,1,\chi}$ such that $\Phi(0)(c_{r,1,\chi})(v_{P^\mathrm{op}}^\ast)$ is nonzero. Choose a compact open subgroup $K'\le K$ sufficiently small such that $\omega_P(K'\cap P(F))\le (\calo^\times)^d$ and $\langle v_Pg, v_{P^{\mathrm{op}}}^\ast\rangle \in \mathrm{ker}(\chi)\cap \calo^\times$ for all $g\in K'$. Let $c_{r,1,\chi}$ be the unique function supported on $P^{\mathrm{der}}(F)m(\calo^\times)K'$ such that for $a\in \calo^\times$ and $g\in K'$,
\begin{align*}
    c_{r,1,\chi}(m(a)^{-1}g)=\chi(a).
\end{align*}
Recall that we have assumed the conductor of $\psi$ is $\mathcal{O}.$ Then for $t\in F^\times$
\begin{align*}
    &\int_{X_P^1} c_{r,1,\chi}(x)\psi(\varpi^{-\mathrm{ord}(t)}\langle  x,v_{P^\mathrm{op}}^\ast\rangle)dx
\end{align*}
is up to a nonzero constant
\begin{align*}
 &\int_{\calo^\times}\chi(a)\psi(\varpi^{-\mathrm{ord}(t)} a)d^\times a
\end{align*}
which is the well-known Gauss sum
\begin{align*}
    &\mathfrak{G}(\varpi^{-\mathrm{ord}(t)},\chi) = \left\{\begin{array}{ll}
      1   &  \textrm{ if } \chi=1, |t|\ge 1, \\
      \zeta(-1)   & \textrm{ if } \chi=1, |t|= q^{-1},\\
     \neq 0 & \textrm{ if } c(\chi)=\mathrm{ord}(t)>0,\\
      0 & \textrm{ otherwise},
    \end{array}\right.
\end{align*}
where $c(\chi)$ is the conductor of $\chi$.
Therefore, $\Phi(0)(c_{r,1,\chi})(v_{P^\mathrm{op}}^\ast)$ is up to a nonzero constant
\begin{align*}
    \left\{\begin{array}{ll}
      \zeta(r+2s_k+2)+\zeta(-1)q^{r+2s_k+2}   &  \textrm{ if } \chi=1, \\ q^{-c(\chi)(-r-2s_k-2)}\mathfrak{G}(\varpi^{-c(\chi)},\chi) & \textrm{ if } \chi\neq 1,\\
    \end{array}\right.
\end{align*}
which is nonzero since $\mathrm{Re}(r+2s_k+2)>1$.
\end{proof} 

\begin{Rem}\label{rem:vanish}
    We claim $M_{m_r(L_d)+1}=0$. If not, then there exists $f\in \mathcal{S}(X_P^\circ(F))$ such that $f_{\overline{\chi}_{-r-s_k-1}}$ corresponds to a nonzero vector in $M_{m_r(L_d)+1}.$ Then by the identity \eqref{eq:Fouriercomp} below, we would arrive at $f_{\overline{\chi}_{-r-s_k-1}}=0,$ a contradiction. 
\end{Rem}

While we do not know how to further describe $A_{r,n,\chi}$ for $n\ge 2$, we propose the following conjectural description: Recall that the socle $\mathrm{Soc}(A)$ of an admissible $G(F)$-representation $A$ of finite length is the maximal semisimple $G(F)$-submodule of $A$. Let $\mathrm{Soc}^0(A):=0$ and define inductively for $n\in \zz_{>0}$ the submodule $\mathrm{Soc}^n(A)$ of $A$ so that
\begin{align*}
    \mathrm{Soc}^n(A)/\mathrm{Soc}^{n-1}(A)=\mathrm{Soc}(A/\mathrm{Soc}^{n-1}(A)).
\end{align*}

\begin{Conj}\label{conj:multiplicity}
Let $(d,r,\chi)$ be a triple. Then for $1\le n\le m_r(L_d)$
\begin{align*}
        A_{r,m_r(L_d)+1-n,\chi}/A_{r,m_r(L_d)-n,\chi}\cong \mathrm{Soc}^n(I_P(\overline{\chi}_{-r-s_k-1})), 
\end{align*}
and 
 $m_r(L_d)=\mathrm{ord}_{s=-(r+s_k+1)} \prod_{i=1}^k L(s_i+1,\chi_s^{\lambda_i})$ is the smallest positive integer such that $$\mathrm{Soc}^{m_r(L_d)+1}(I_P(\overline{\chi}_{-r-s_k-1}))=I_P(\overline{\chi}_{-r-s_k-1}).$$
\end{Conj}

By Remark \ref{rem:vanish} we set $A_{r,m_r(L_d)+1,\chi}:=A_{r,m_r(L_d),\chi}.$ The following proposition will be useful in verifying Conjecture \ref{conj:multiplicity}.
\begin{Prop}\label{prop:condconj}
Let $(d,r,\chi)$ be a triple. 
\begin{enumerate}
    \item The module  $I(\overline{\chi}_{-r-s_k-1})/\mathrm{Ker}(\mathcal{R}_{P|P^{\mathrm{op}}})$ is self-dual.
    \item Suppose $I(\overline{\chi}_{-r-s_k-1})$ is multiplicity-free. Then all successive quotients in \eqref{eq:injseq} are self-dual, i.e.,
\begin{align*}
    \big(A_{r,n,\chi}/A_{r,n-1,\chi}\big)/\big(A_{r,n+1,\chi}/A_{r,n,\chi}\big)
\end{align*}
is self-dual for every $1\le n\le m_{r}(L_d).$ 
\end{enumerate}
\end{Prop}

\begin{proof}
 For any $f\in \mathcal{S}(X_{P}^\circ(F)), h\in\mathcal{S}(X_{P^\mathrm{op}}^\circ(F)),$ we have a well-defined meromorphic function
    \begin{align*}
        &\int_{P^{\mathrm{op}}(F)\backslash G(F)} \mathcal{R}_{P|P^{\mathrm{op}}}(f_{\chi_s,P})(g)h_{\chi_s,P^{\mathrm{op}}}(g) \,d\dot{g}=\int_{P^{\mathrm{op}}(F)\backslash G(F)} \mathcal{R}_{P|P^{\mathrm{op}}}(f)_{\chi_s}^{\mathrm{op}}(g)h_{\chi_s}(g) \,d\dot{g}.
    \end{align*}   
    By definition and the Fubini-Tonelli theorem, the meromorphic function equals 
    \begin{align*}
    &\int_{P^{\mathrm{op}}(F)\backslash G(F)\times M^{\mathrm{ab}}(F)\times M^{\mathrm{ab}}(F)} \mathcal{R}_{P|P^{\mathrm{op}}}(f)(m^{-1}g)h(m^{\prime-1}g)\delta^{1/2}_{P^{\mathrm{op}}}(mm')\chi_s(\omega_P(mm'^{-1}))dmdm'd\dot{g}.
    \end{align*}
    Note that this integral converges absolutely for $\mathrm{Re}(s)$ large. Changing variables $m\mapsto mm'$, we arrive at
    \begin{align*}
        &\int_{P^{\mathrm{op}}(F)\backslash G(F)\times M^{\mathrm{ab}}(F)\times M^{\mathrm{ab}}(F)} \mathcal{R}_{P|P^{\mathrm{op}}}(f)(m^{-1}m^{\prime-1}g)h(m^{\prime-1}g)\\
        &\quad\quad\quad\quad\quad\quad\quad\quad\quad\quad\quad\quad\quad\times\delta^{1/2}_{P^{\mathrm{op}}}(m)\chi_s(\omega_P(m))dm\delta_{P^{\mathrm{op}}}(m')dm'd\dot{g}\\
        &=\int_{X_{P^{\mathrm{op}}}^{\circ}(F)} \mathcal{R}_{P|P^{\mathrm{op}}}(f)_{\chi_s}^{\mathrm{op}}(x^\ast)h(x^\ast) dx^\ast.
    \end{align*}
    By Theorem \ref{Thm: Fourier formula} we can write this as
    \begin{align*}
        &\mu_{\Lambda}(\chi_s)^{-1}\int_{X_{P^{\mathrm{op}}}^{\circ}(F)} \mathcal{F}_{P|P^{\mathrm{op}}}(f)_{\chi_s}^{\mathrm{op}}(x^\ast)h(x^\ast) dx^\ast\\
        &=\mu_{\Lambda}(\chi_s)^{-1}\int_{M^{\mathrm{ab}}(F)}\delta^{1/2}_{P^{\mathrm{op}}}(m)\chi_s(\omega_P(m))\int_{X_{P^{\mathrm{op}}}^{\circ}(F)} \mathcal{F}_{P|P^{\mathrm{op}}}(f)(m^{-1}x^\ast)h(x^\ast) dx^\ast dm.
    \end{align*}
    Using the twisted equivariance $L(m)\circ \mathcal{F}_{P|P^{\mathrm{op}}}=\delta_P(m)\cdot \mathcal{F}_{P|P^{\mathrm{op}}}\circ L(m)$ and the fact that $\mathcal{F}_{P|P^{\mathrm{op}}}$ is unitary, the meromorphic function is
    \begin{align*}
        &\mu_{\Lambda}(\chi_s)^{-1}\int_{M^{\mathrm{ab}}(F)}\delta^{1/2}_{P}(m)\chi_s(\omega_P(m))\int_{X_{P}^{\circ}(F)} f(m^{-1}x)\overline{\mathcal{F}_{P^{\mathrm{op}}|P}(\overline{h})}(x) dx dm\\
        &=\mu_{\Lambda}(\chi_s)^{-1}\int_{M^{\mathrm{ab}}(F)}\delta^{1/2}_{P}(m)\chi_s(\omega_{P^{\mathrm{op}}}(m))\int_{X_{P}^{\circ}(F)} f(x)\overline{\mathcal{F}_{P^{\mathrm{op}}|P}(\overline{h})}(m^{-1}x) dx dm
    \end{align*}
    Reversing the procedure, we conclude
    \begin{align}\label{eq:contra:pairing}
        \int_{P^{\mathrm{op}}(F)\backslash G(F)} \mathcal{R}_{P|P^{\mathrm{op}}}(f_{\chi_s,P})(g)h_{\chi_s,P^{\mathrm{op}}}(g) \,d\dot{g}=\int_{P(F)\backslash G(F)} f_{\chi_s,P}(g)\overline{\mathcal{R}_{P^{\mathrm{op}}|P}(\overline{h}_{\chi_s,P^{\mathrm{op}}})}(g)\,d\dot{g}.
    \end{align}
    
    Recall the definition of $M_j$  in \eqref{eq:Mj}. We identify $M_j$ with its image in $I_P(\overline{\chi}_{-r-s_k-1})\cong I_{P^{\mathrm{op}}}(\overline{\chi}_{-r-s_k-1})$. Choose $f$ such that $f_{\overline{\chi}_{-r-s_k-1}}\in M_{n}.$ Note that the pairing 
    \begin{align*}
       \langle f,h\rangle:= \zeta(s+r+s_k+1)^{n}\int_{P^{\mathrm{op}}(F)\backslash G(F)} \mathcal{R}_{P|P^{\mathrm{op}}}(f_{\overline{\chi}_s,P})(g)h_{\overline{\chi}_s,P^{\mathrm{op}}}(g) \,d\dot{g}\bigg|_{s=-r-s_k-1}
    \end{align*}
    is well defined. For a given $f$ such that $f_{\overline{\chi}_{-r-s_k-1}}\not\in M_{n+1},$ there exists $h$ such that the pairing is nonzero. Furthermore, using \eqref{eq:contra:pairing} we see that the pairing is zero if  $f_{\overline{\chi}_{-r-s_k-1}}$ or $h_{\overline{\chi}_{-r-s_k-1}}$ belongs to $M_{n+1}$. We conclude that we have a $G(F)$-equivariant injection
    \begin{align}\label{eq:map}
        M_n/M_{n+1}\hookrightarrow (M_0/M_{n+1})^\lor.
    \end{align}
    In particular, $M_0/M_1$ is self-dual. This proves (1).
    
    Now assume $I(\overline{\chi}_{-r-s_k-1})$ is multiplicity-free. We prove that $M_n/M_{n+1}$ is self-dual by induction on $n$. Suppose $M_i/M_{i+1}$ is self-dual for all $i\le n-1.$ Then \eqref{eq:map} must induce an injection
    \begin{align*}
        M_n/M_{n+1}\to (M_n/M_{n+1})^\lor,
    \end{align*}
    which has to be an isomorphism. Hence (2) follows from \eqref{eq:successivequotient}. 
\end{proof}

Let 
\begin{align*}
    R(G,P,F):=\bigg\{(s,\chi)\in  \cc/ \tfrac{2\pi\sqrt{-1}}{\log q}\zz \times \widehat{\calo}^\times: I(\chi_s) \textrm{ is reducible and }\mathrm{Re}(s)<0\bigg \}.
\end{align*}

\begin{Thm}\label{thm:reducibility}
Suppose $F$ is a nonarchimedean local field of characteristic zero. Let $(s,\chi)\in \cc/ \tfrac{2\pi\sqrt{-1}}{\log q}\zz \times \widehat{\calo}^\times$ with $\mathrm{Re}(s)<0$ and $d=\ord(\chi)$. Then 
$$(s,\chi)\in R(G,P,F) \, \Longleftrightarrow \,  L_d\neq\emptyset \textrm{ and } -(s+s_k+1)\in \mathrm{Supp}(L_{d}).$$ Consequently
\begin{align*}
    R(G,P,F)=\bigg\{ (s,\chi)\in  \cc/ \tfrac{2\pi\sqrt{-1}}{\log q}\zz \times \widehat{\calo}^\times : \mathcal{R}_{P|P^{\mathrm{op}}}(\chi_s) \textrm{ has nontrivial kernel and }\mathrm{Re}(s)<0 \bigg\}.
\end{align*}
\end{Thm}

Recall the set $\mathrm{Supp}(G,P,F)$ defined in \S \ref{ssec:main}. Theorem \ref{thm:reducibility} can be stated as follows:  under the isomorphism $\cc/ \tfrac{2\pi\sqrt{-1}}{\log q}\zz \times \widehat{\calo}^\times \to \widehat{F}^\times,$ the set $R(G,P,F)$ is mapped to $\mathrm{Supp}(G,P,F).$

\begin{proof}
In the following we use the Bourbaki numbering of the Dynkin diagram of $G$ and let $P=P_\ell$ be the maximal parabolic associated to the $\ell$th node.

\textbf{Type $A$:} We can extend the natural $\SL_n$-action on $X_P$ to an action of $\GL_n$ by
\begin{align*}
    x. g:=\begin{psmatrix}
        I_{n-1} & \\
        & \det g^{-1}
    \end{psmatrix}xg.
\end{align*}
Then induced representations in this case are of the form $\chi\times 1$ in \cite{G:A}, and thus the assertion follows from \cite[Theorem 1.1(i)]{G:A}.

\textbf{Type $B_n$:} Let $\varphi:G=\mathrm{Spin}_{2n+1}\to \mathrm{SO}_{2n+1}$ be the degree $2$ isogeny. Note that $\varphi(P_\ell)$ is a maximal parabolic subgroup of $\mathrm{SO}_{2n+1}$ corresponding to the same node. Consider first $\ell<n$. In this case $\omega_{P_\ell}$ lies inside the character group $X^\ast(\varphi(T))$, i.e., the $G$-representation $V_P$ descends to a representation of $\mathrm{SO}_{2n+1}$.  Therefore, we can take $G$ to be $\mathrm{SO}_{2n+1}.$ Points of reducibility are described in \cite[Theorem 4.1]{J:BC}. 

Suppose $\ell=n$. Since $\varphi$ is of degree $2$, the assertion follows from \cite[Theorem 4.1]{J:BC} if $\mathrm{ord}(\chi)$ is odd. Note that the points of reducibility in loc.~cit. are a half of the claimed values; the factor $2$ comes from the fact that $2\omega_{P_n}$ lies inside $X^\ast(\varphi(T))$. Therefore, it remains to show $I(\chi_s)$ is irreducible if $\mathrm{ord}(\chi)$ is even. One can, as in the case of type $A,$ extend the action of $G$ to $\textrm{GSpin}_{2n+1}$ and appeal to \cite{KYM:Gspin}. We give a direct proof instead utilizing the approach in \cite[\S 3]{HS:E6}. 

The case $B_2=C_2$ is proved below, so we assume $n\ge 3$. Let $W=W_G$ be the Weyl group of $(G,T)$. We claim $\mathrm{Stab}_{W}(\chi)$ is trivial for $\chi$ of even order. This would then imply $I_P(\chi_s)$ is regular for any real $s$. By the fourth point in \cite[Remark 3.1]{HS:E6}, it suffices to show for any $w\neq \mathrm{Id}$ such that if we write
\begin{align*}
    \omega_P-w\omega_P=\sum_{\beta\in \Delta} n_\beta\omega_\beta,
\end{align*}
then $n_\beta\in \zz$ is odd for some $\beta$. We only need to check for $w\in W/W_M$, i.e.,  representatives of minimal length in each left coset of $W_M$. We proceed by induction on the length $\ell(w)$ of $w$. Let $\alpha_i$ denote the simple root associated to the $i$th node, and $s_i$ be the corresponding simple reflection. We have
\begin{align*}
    n_{\alpha_i}&=\delta_{in}-\langle \omega_P,w^{-1}\alpha_i^\lor\rangle.
\end{align*}
 If $\ell(w)=1,$ then $w=s_n$ and $n_{\alpha_{n-1}}=-1$. For $\ell(w)\ge 2,$ if $s_1w$ is a reduced expression in $W/W_M$, then $w^{-1}\alpha_1^\lor\in (\Phi_G^\lor)^+-(\Phi_M^\lor)^+$ by Lemma \ref{lem:positive} below, so $n_{\alpha_1}=-1$ by the list of coroots below \eqref{eq:Bnw0n}. Otherwise, we can write $w=s_1w'$ for some $w'\in W/W_M$ and we get 
\begin{align*}
    n_{\alpha_n}&=\delta_{jn}-\langle \omega_P,w'^{-1}\alpha_j^\lor\rangle, \quad 2<j\le n,\\
     n_{\alpha_2}&=-\langle \omega_P,w'^{-1}\alpha_1^\lor+w'^{-1}\alpha_2^\lor\rangle.\\
    n_{\alpha_1}&=-\langle \omega_P,-w'^{-1}\alpha_1^\lor\rangle.
\end{align*}
By the induction hypothesis, the assertion holds for $w'$, so one of the terms above is odd.

Now since $I(\chi_s)$ is regular, to prove irreducibility we only need to show that for every $\beta^\lor\in (\Phi_{G}^\lor)^+-(\Phi_M^\lor)^+$ and $s\in \rr$
\begin{align*}
    \chi_{s+s_k+1}\circ \omega_P\circ \beta^\lor \neq|\cdot|^{\pm 1}\cdot (\delta_B^{1/2}\circ\beta^\lor).
\end{align*}
As $\chi\neq 1$ and $\langle \omega_P,\beta^\lor\rangle=1$, the inequality is clear.

\textbf{Type $C_n$:} The assertion follows from \cite[Theorem 4.3]{J:BC}.

\textbf{Type $D_n$:} For the same reason as in  the case of type $B,$ for $\ell<n-1$ or $\chi$ of odd order, we can reduce it to the case of $\mathrm{SO}_{2n}$, and the assertion in this case follows from \cite[\S 5]{BJ:D}. For the remaining cases, by symmetry we may assume $\ell=n.$ Suppose $\mathrm{ord}(\chi)$ is even. We utilize the same argument as the case where $G=\mathrm{Spin}_{2n+1}$ and $P=P_n$. If $\ell(w)=1,$ then $w=s_n$ and $n_{\alpha_{n-2}}=-1$. For $\ell(w)\ge 2$, if $s_{n-1}w$ is a reduced expression in $W/W_M$, then $w^{-1}\alpha_{n-1}^\lor\in (\Phi_G^\lor)^+-(\Phi_M^\lor)^+$, so $n_{\alpha_{j}}=-1$ for some $j<n$ by the list of coroots at the end of \S \ref{ssec:Dn}. Otherwise, we can write $w=s_{n-1}w'$ for some $w'\in W/W_M$, and we get 
\begin{align*}
    n_{\alpha_j}&=\delta_{jn}-\langle \omega_P,w'^{-1}\alpha_j^\lor\rangle,\quad j\neq n-1,n-2,\\
     n_{\alpha_{n-2}}&=-\langle \omega_P,w'^{-1}\alpha_{n-1}^\lor+w'^{-1}\alpha_{n-2}^\lor\rangle,\\
    n_{\alpha_{n-1}}&=-\langle \omega_P,-w'^{-1}\alpha_{n-1}^\lor\rangle.
\end{align*}
By the induction hypothesis, the assertion holds for $w'$, so one of the terms above is odd. As $\langle \omega_P,\beta^\lor\rangle=1$ for all $\beta^\lor\in (\Phi_G^\lor)^+-(\Phi_M^\lor)^+$, we conclude $I(\chi_s)$ is regular and irreducible.

\textbf{Exceptional types:} Each case is treated separately in \cite{Muic:G2,CJ:F, HS:E6, HS:E7, HS:E8}.

\end{proof}

\begin{Thm}\label{thm:confirmconj}
   Suppose $F$ is a nonarchimedean local field of characteristic zero.  Conjecture \ref{conj:multiplicity} holds if $G$ is either classical or $G_2$.
\end{Thm}

\begin{proof}
\sloppy It is straightforward to check $m_r(L_d)$ is the smallest positive integer such that $\mathrm{Soc}^{m_r(L_d)+1}(I_P(\overline{\chi}_{-r-s_k-1}))=I_P(\overline{\chi}_{-r-s_k-1})$ in each case using \cite[Theorem 1.2, Lemma 4.2]{G:A}, \cite[Theorems 6.1 and 7.1, Corollary 5.7]{J:BC}, \cite[Theorems 6.2 and 7.2, Corollary 5.8]{J:BC},  \cite[\S 5]{BJ:D} and \cite[\S 4]{Muic:G2}. More explanation is needed for the case $G=G_2$, $\chi=1$ and $\mathrm{Re}(r)\neq 0$ since \cite[Proposition 4.3]{Muic:G2} only states composition factors. From the proof in loc.~ cit. together with \cite[Proposition 1.4]{GinzburgJiang}, one can conclude, using the notation in \cite{Muic:G2},  $I_{P_2}(1_{1/2})$ is of length $3$ with a unique irreducible subrepresentation $\pi(1),$ a unique irreducible quotient $J_\beta(1,\pi(1,1))$ and an irreducible subquotient $J_\beta(\frac{1}{2},\delta(1))$. This justifies the case for $P=P_2$ by taking contragredient. For $P=P_1$, one can use Frobenius reciprocity to conclude $J_\beta(\frac{1}{2},\delta(1))$ is the unique irreducible subrepresentation of $I_{P_1}(1_{1/2})$ and hence a unique irreducible quotient of $I_{P_1}(1_{-1/2}).$ Now suppose on the contrary that $\mathrm{Soc}^2(I_{P_1}(1_{-1/2}))$ is proper. Then either $I_{P_1}(1_{-1/2})/\mathrm{Ker}(\mathcal{R}_{P|P^{\mathrm{op}}})$ or $\mathrm{Ker}(\mathcal{R}_{P|P^{\mathrm{op}}})$ is not semisimple and hence not self-dual. This contradicts Proposition \ref{prop:condconj}.

In all considered cases, $I_P(\overline{\chi}_{-r-s_k-1})$ has a unique irreducible quotient and is multiplicity-free. Then Proposition \ref{prop:condconj}(1) implies $\mathrm{Ker}(\mathcal{R}_{P|P^{\mathrm{op}}})=\mathrm{Soc}^{m_r(L_d)}(I_P(\overline{\chi}_{-r-s_k-1})).$ This justifies the assertion for $m_r(L_d)=1$. Now suppose $m_r(L_d)=2$. In this case $I_P(\overline{\chi}_{-r-s_k-1})$ also has a unique irreducible subrepresentation, so $A_{r,2,\chi}/A_{r,1,\chi}$ is the unique irreducible subrepresentation by Proposition \ref{prop:condconj}(2). 

\end{proof}

\subsection{Examples}\label{ssec:nonarch:ex} Assume $F$ has characteristic zero. We illustrate Theorem \ref{thm:exact} using examples considered in \cite[\S 6.3]{Getz:Hsu:Leslie}. We will freely use the computation in Appendix \ref{appendix}.  

\subsubsection{Line bundles over Grassmannians}

For $n\ge 2$ and $1\le \ell<n$, let $G=\mathrm{SL}_n$ and let $P$ be a maximal parabolic stabilizing an $\ell$-plane. Then
\begin{align*}
    L_1=L(1)=\left\{ -i: 0\le i<\min(\ell,n-\ell)\right\}, \quad s_k=\frac{n-2}{2}.
\end{align*}
We have by \cite[Lemma 4.2]{G:A}
\begin{align*}
     A_{X_P(F)}\cong \bigoplus_{i=0}^{\min(\ell,n-\ell)-1} A_{-i,1},
\end{align*}
where $A_{-i,1}$ is the unique irreducible $G(F)$-subrepresentation of $I(1_{i-n/2}).$

More explicitly, for $0\le  i<\min(\ell,n-\ell),$ let $\Omega_i\subset N_{P_\ell}\cong M_{\ell,n-\ell}$ be the (reduced) closed subscheme consisting of matrices of rank $\le i$, and let $\Omega_i^\circ:=\Omega_i-\Omega_{i-1}.$ Then using the notion of $N$-rank in \cite[Definition 6.1.2 and Proposition 6.3.1]{W1}, one can modify the argument in \cite[\S 3]{SW:minimal} to conclude that $A_{-i,1}$ can be realized as certain space of functions $\mathcal{S}(\Omega_i(F))\subseteq C^\infty(\Omega_i^\circ(F)).$ Therefore, 
\begin{align*}
     A_{X_P(F)}\cong \cc\oplus \bigoplus_{i=1}^{\min(\ell,n-\ell)-1} \mathcal{S}(\Omega_i(F)).
\end{align*}

\subsubsection{Odd-dimensional cones} For $n\ge 2$, consider the cone $X=X_n$ in $V_n=\mathbb{A}^{2n}$ defined by the vanishing locus of the polynomial
\begin{align*}
    x_{1}x_{2n}+x_{2}x_{2n-1}+\ldots+ x_{n}x_{n+1}.
\end{align*}
For $n\ge 3$, $X_n^\circ \cong P_n^{\mathrm{der}}\backslash \mathrm{SO}_{2n}$ where $P_n$ is the stabilizer of a line spanned by an isotropic vector. We have
\begin{align*}
    L(1)&=\left\{0,-(n-2)\right\}, \quad s_k=n-2.
\end{align*}
One can compare formulae of Fourier transforms in \cite{Getz:Hsu:Leslie} and \cite{GK:cone} to conclude $\mathcal{S}(X_n(F))$ is the minimal representation of $\mathrm{SO}_{2n+2}(F)$. By \cite{W1,SW:minimal}
\begin{align*}
    \mathcal{S}(X_n(F))/\mathcal{S}(X_n^\circ(F))\cong \cc\oplus \mathcal{S}(X_{n-1}(F)).
\end{align*}
Note that $X_2$ is not a Braverman-Kazhdan space, and here $\mathcal{S}(X_2(F))$ is the unitarizable minimal representation of $\mathrm{SL}_4(F)$ contained in $I_{P}^{\mathrm{SL}_{4}}(1_{-1})$, where $P$ is a maximal parabolic stabilizing a plane (see \cite{SW:minimal}). This is a special case proved in \cite{GK:cone}. 

\subsubsection{The Lagrangian Grassmannian}\label{ssec:LG} Let  $G=\mathrm{Sp}_{2n}$ and $P$ be the Siegel parabolic. Then
\begin{align*}
    L(1)&=\left\{0\right\}, &L(2)&=\left\{-i+j\tfrac{\pi\sqrt{-1}}{\log q}: 1\le i\le \lfloor\tfrac{n}{2}\rfloor, 0\le j\le 1\right\}, &s_k&= \tfrac{n-1}{2}.
\end{align*}
Therefore, $L_1$ and $L_2$ are genuine sets. By the work of Kudla and Rallis \cite{KR:spn} we have
\begin{align*}
    A_{X_P(F)}\cong \bigoplus_{d=1}^2 \bigoplus_{\substack{\chi\in \widehat{\calo}^\times\\\ord(\chi)=d}}\bigoplus_{r\in L_d}\mathrm{Ker}(\mathcal{R}_{P|P^{\mathrm{op}}}(\overline{\chi}_{-r-\tfrac{n+1}{2}}))\cong \bigoplus_{V}\mathcal{S}(V^n(F))_{O(V)(F)},
\end{align*}
where $V$ ranges over all equivalence classes of nondegenerate even-dimensional quadratic spaces over $F$ of dimension not larger than $n$, and $\mathcal{S}(V^n(F))_{O(V)(F)}$ is the space of $O(V)(F)$-coinvariants of the Weil representation of $G=\mathrm{Sp}_{2n}$ realized on $\mathcal{S}(V^n(F))$. 

Let $\mathcal{X}_i$ be the smooth stack $(\mathbb{A}^{2i})^n/O_{2i},$ where $O_{2i}$ is the split orthogonal group acting on each copy of $\mathbb{A}^{2i}$ via the standard action. Using the definition of the Schwartz space on $\mathcal{X}_i(F)$ \cite{SakStack}, we can further write
\begin{align*}
    \bigoplus_{V}\mathcal{S}(V^n(F))_{O(V)(F)}=\bigoplus_{i=0}^{\lfloor n/2\rfloor} \mathcal{S}(\mathcal{X}_i(F))=\cc\oplus \bigoplus_{i=1}^{\lfloor n/2\rfloor} \mathcal{S}(\mathcal{X}_i(F)).
\end{align*}

\section{Archimedean asymptotics}\label{sec:Sch:arch}

Let $F$ be an archimedean local field. We continue to assume the set-theoretic inclusion $\mathcal{S}(X_P^\circ(F))\subseteq \mathcal{S}(X_P(F))$ holds.

\begin{Lem}\label{lem:inclusmooth}
    The inclusion $\mathcal{S}(X_P^\circ(F))\hookrightarrow \mathcal{S}(X_P(F))$
    is a homeomorphism onto its image, where $\mathcal{S}(X_P^\circ(F))$ is equipped with its usual topology of Schwartz space.
\end{Lem}

\begin{proof}
     Let $\tau_1,\tau_2$ be the topology on $\mathcal{S}(X_P^\circ(F))$ given respectively by the topology of $\mathcal{S}_{\mathrm{ES}}(X_P^\circ(F))$ and by the subspace topology of $\mathcal{S}(X_P(F))$. Note that $(\mathcal{S}(X_P^\circ(F)),\tau_2)$ is a closed subspace of $\mathcal{S}(X_P(F))$ by Lemma \ref{lem:archfil} and hence a Fr\'echet space. We claim the identity map $(\mathcal{S}(X_P^\circ(F)),\tau_1)\to (\mathcal{S}(X_P^\circ(F)),\tau_2)$ is continuous, and thus the assertion follows from the open mapping theorem.

     By the closed graph theorem, we only need to show the set
     \begin{align*}
         \bigcap_{x\in X_P^\circ(F)} \{(f_1,f_2)\in (\mathcal{S}(X_P^\circ(F)),\tau_1)\times(\mathcal{S}(X_P^\circ(F)),\tau_2):f_1(x)=f_2(x)\}.
     \end{align*}
     is closed in $(\mathcal{S}(X_P^\circ(F)),\tau_1)\times(\mathcal{S}(X_P^\circ(F)),\tau_2)$. This is clear since evaluation maps are continuous linear functionals on $\mathcal{S}(X_P^\circ(F))$ with respect to both $\tau_1,\tau_2$.
\end{proof}

\subsection{Work of Igusa: Archimedean}

Let $\{\phi_r\}_{r\ge 0}$ be a sequence of functions on the set of positive real numbers $\rr^+$ satisfying the following conditions for every $r\ge 0$:
\begin{enumerate}
    \item Each function $\phi_r$ is nonvanishing in a small neighborhood of $0$.
    \item We have $\phi_{r+1}(t)=o(\phi_{r}(t))$ for $t$ sufficiently close to $0$, i.e., for a given $\epsilon>0$, there is a sufficiently small neighborhood $U_{\epsilon}$ of $0$ such that 
    \begin{align*}
        |\phi_{r+1}(t)|\le \epsilon |\phi_{r}(t)|\quad  \textrm{ for } t\in U_{\epsilon}. 
    \end{align*}
\end{enumerate}
We say a function $f$ on $\rr^+$ has an \textbf{asymptotic expansion as $t$ tends to $0$} with respect to $\{\phi_r\}$ if there is a sequence $\{c_r\}$ of complex numbers such that for each $r$ and all $t$ sufficiently close to $0,$
\begin{align*}
    \left|f(t)-\sum_{j=0}^r c_j\phi_j(t)\right|\ll |\phi_{r+1}(t)|.
\end{align*}
We denote this by
\begin{align*}
    f(t)\approx \sum_{r=0}^\infty c_r\phi_r(t) \quad \textrm{as } t \to 0.
\end{align*}
Note that the coefficients $c_r$ are unique. In the case coefficients $c_r=c_r(x)$ depend on a parameter $x$ and all implicit constants can be made independent of $x$, we say the asymptotic expansion is \textbf{uniform}.

If all $f$ and $\phi_r$ are smooth, and for every $n\ge 0$
\begin{align*}
    f^{(n)}(t)\approx \sum_{r=0}^{\infty} c_r \phi_r^{(n)}(t)  \quad \textrm{as } t \to 0,
\end{align*}
where we assume that for each $n$, $\{\phi_r^{(n)}\}$ satisfies conditions (1)(2)  above after removing a finite set from the sequence, then we say that the asymptotic expansion of $f$ as $t$ tends to $0$ is \textbf{termwise differentiable} or \textbf{commutes with} $U(\mathfrak{gl_1})$. The notion of \textbf{uniform termwise differentiability} is similarly defined.

For a multiset $L$ whose underlying set $\mathrm{Supp}(L)$ is a discrete subset of $\rr$ that is bounded below, $\mathrm{Supp}(L)$ is countable and naturally ordered by $<.$ Consider smooth functions on $\rr^+$:
\begin{align*}
    t^{r}(\log t)^{j-1},    \quad r\in \mathrm{Supp}(L), 1\le j\le  m_r(L).
\end{align*}
These functions and their derivatives satisfy conditions (1)(2) above.

Let $F$ be an archimedean local field. If $F=\rr,$ let $\psi(a)=e^{\sqrt{-1}a}$ and $H:=\zz/2\zz=\{0,1\};$ if $F=\cc,$ let $\psi(a)=e^{2\sqrt{-1}\mathrm{Re}(a)}$ and $H:=\zz$. For $f\in C^\infty(F^\times)$ by the Fourier expansion we can write
\begin{align}\label{eq:Fourier}
    f(a)=\sum_{\ell\in H} \mu^\ell(a)f_\ell(|a|),
\end{align}
where $f_\ell\in C^\infty(\rr_+)$. We henceforth identify $f$ with $(f_\ell)_{\ell\in H}.$ Moreover for any $D\in U(\mathfrak{gl}_1)$ and real numbers $\sigma_1<\sigma_2,$
\begin{align*}
     \sup_{\ell\in H}\sup_{\sigma_1<|a|<\sigma_2} |(D.f)_\ell(a)|<\infty,
\end{align*}
and the identity \eqref{eq:Fourier} is termwise differentiable.

 Let $L=\{L_\ell\}_{\ell\in H}$ be a collection of multisets such that $\cup_{\ell\in H}\mathrm{Supp}(L_\ell)\subset \rr$ is discrete and bounded below, $L_\ell=L_{-\ell}$ and $\sup_{\ell\in H} m_r(L_\ell)<\infty$ for all $r\in \rr$. We define $C^\infty_{L}(F^\times)$ to be the subspace of $C^\infty(F^\times)$ that consists of functions $f=(f_\ell)_{\ell\in H}$ satisfying the following: 
\begin{enumerate}
    \item For $\sigma\in \rr$ and $m,n\in \zz_{\ge0}$, one has
    \begin{align*}
        \sup_{\ell\in H, \,t\ge 1} t^\sigma|\ell^m f_\ell^{(n)}(t)|<\infty.
    \end{align*}
    Equivalently, for $D\in U(\mathfrak{gl}_1)$
    \begin{align*}
        \sup_{\ell\in H, \,t\ge 1} t^\sigma |(D.f)_\ell(t)|<\infty.
    \end{align*}
    
    \item We have
    \begin{align*}
        f_\ell(t)\approx  \sum_{r\in \mathrm{Supp}(L_\ell)}\sum_{j=1}^{m_r(L_\ell)} c_{r,j,\ell}t^{r}(\log t)^{j-1} \quad \textrm{ as } t\to 0
    \end{align*}
for some (unique) $c_{r,j,\ell}\in \cc$, and this asymptotic expansion is termwise differentiable. Moreover, for every $s\in \rr,m,n\in \zz_{\ge 0}$ and $\sigma>-s$, one has
\begin{align*}
    \sup_{\ell\in H,\, 0<t\le 1} t^{\sigma+n}|\ell^m (R_{s})_{\ell}^{(n)}(t)|<\infty,
\end{align*}
where
\begin{align*}
    R_{s}(a):=f(a)-\sum_{\substack{r\in \cup_{\ell\in H}\mathrm{Supp}(L_\ell)\\ r<s}} \sum_{\ell:r\in \mathrm{Supp}(L_\ell)}\sum_{j=1}^{m_r(L_\ell)}\mu^\ell(a)c_{r,j,\ell}|a|^r(\log |a|)^{j-1}.
\end{align*}
Equivalently, for $D\in U(\mathfrak{gl}_1)$
\begin{align*}
    \sup_{\ell\in H,\, 0<t\le 1} t^{\sigma}|(D.R_{s})_\ell(t)|<\infty.
\end{align*}
\end{enumerate}

 On the other hand, let $\mathcal{Z}_L(\widehat{F}^\times)$  be the set of complex-valued functions $Z$ on $\widehat{F}^\times$ such that for each $\ell\in H$,
\begin{enumerate}
    \item the function $Z(\mu^\ell |\cdot|^s)$ is holomorphic on $\cc-\mathrm{Supp}(L_\ell)$, and
    
    \item there are constants $b_{r,j,\ell}\in \cc$ such that for each $r\in \mathrm{Supp}(L_\ell)$
    \begin{align*}
        Z(\mu^\ell |\cdot|^s)-\sum_{j=1}^{m_r(L_\ell)} \frac{b_{r,j,\ell}}{(s+r)^{j}}
    \end{align*}
    is holomorphic in a neighborhood of $s=-r$. Moreover, for any real numbers $\sigma_1<\sigma_2, m\in\zz_{\ge 0}$  and polynomials $Q\in \cc[s]$ such that $(s+r)^{\max_{\ell\in H}m_r(L_\ell)}|Q(s)$ for all $r\in \cup_{\ell\in H} \mathrm{Supp}(L_\ell)\cap V_{\sigma_1,\sigma_2}$
    \begin{align*}
       \sup_{\ell\in H}|\ell^m Z(\mu^\ell |\cdot|^s)|_{\sigma_1,\sigma_2,Q} <\infty.
    \end{align*}
\end{enumerate}

The following is a consequence of \cite[Theorems 1.4.2 and 1.4.3]{Igusa:forms}:
\begin{Thm}\label{thm:Igusa:real}
 For $f\in C^\infty_L(F^\times),$ the Mellin transform
    $$M:f\mapsto \left(\chi \mapsto \frac{1}{\mathrm{vol}(K_{\mathbb{G}_m})}\int_{F^\times} f(a)\chi(a)d^\times a \right),$$
originally defined on $\chi\in \widehat{F}^\times$ with $\mathrm{Re}(\chi)\gg_{L} 1$, extends analytically to whole $\widehat{F}^\times$ and gives rise to an isomorphism between $C^\infty_L(F^\times)$ and $\mathcal{Z}_L(\widehat{F}^\times)$. Moreover, the Mellin inversion $M^{-1}$ is given as follows: Given $Z\in \mathcal{Z}_L(\widehat{F}^\times)$, for $a\in F^\times$
\begin{align*}
    M^{-1}(Z)(a)=\sum_{\ell\in H}\mu^{-\ell}(a)\frac{1}{2\pi \sqrt{-1}} \int_{\sigma-\sqrt{-1}\infty}^{\sigma+\sqrt{-1}\infty}Z(\mu^\ell |\cdot|^s)|a|^{-s} ds
\end{align*}
defines a function $f\in C^\infty_L(F^\times)$ independent of $\sigma\gg_L 1.$ Furthermore, 
\begin{align*}
    b_{r,j,\ell}=(-1)^{j-1}(j-1)!c_{r,j,-\ell}.
\end{align*}
for every $r,j,\ell$.\qed
\end{Thm}

\subsection{Asymptotics toward the origin}
If $F=\rr,$ define multisets
\begin{align*}
    L_{0}&:=\sum_{i=1}^k \left(\frac{s_i+1}{\lambda_i}-(s_k+1)+\frac{2}{\lambda_i}\zz_{\ge 0} \right),\\
      L_{1}&:=\sum_{i:\lambda_i\textrm{odd}} \left(\frac{s_i+2}{\lambda_i}-(s_k+1)+\frac{2}{\lambda_i}\zz_{\ge 0} \right)+\sum_{i:\lambda_i\textrm{even}} \left(\frac{s_i+1}{\lambda_i}-(s_k+1)+\frac{2}{\lambda_i}\zz_{\ge 0} \right).
\end{align*}
If $F=\cc,$ for each $\ell\in H$ define multisets 
\begin{align*}
    L_\ell:=\sum_{i=1}^k \left(\frac{s_i+1}{\lambda_i}-(s_k+1)+\frac{|\ell|}{2}+\frac{1}{\lambda_i}\zz_{\ge 0}\right).
\end{align*}

By the Fourier expansion, for $f\in C^\infty(X_P^\circ(F))$ we can write 
\begin{align*}
    f=\sum_{\ell\in H} f_\ell,
\end{align*}
where $f_\ell\in C^\infty(X_P^\circ(F))$ satisfies $f_\ell(m(a)^{-1}x)=\mu^\ell(a)f_\ell(x)$ for all $a\in K_{\GG_m}$ and $x\in X_P^\circ(F)$. The identity commutes with the action of $U(\mathfrak{m}^{\mathrm{ab}}\oplus \mathfrak{g}).$

\begin{Prop}\label{prop:globalchar:real}
    Let $f\in C^\infty(X_P^\circ(F))$. Then $f$ satisfies \ref{cond:orig:arch} if and only if the following holds.
    \begin{enumerate}
        \item For $\sigma\in \rr$ and $D\in U(\mathfrak{m}^{\mathrm{ab}}\oplus \mathfrak{g})$, one has
    \begin{align*}
        \sup_{\substack{\ell\in H\\|a|\ge 1,\,x\in X_P^1}} \left|a|^\sigma |D.f_\ell(m(a)^{-1}x)\right|<\infty.
    \end{align*}
        \item For $x\in X_P^1,$ $f_\ell$ has a uniform asymptotic expansion
    \begin{align*}
        f_\ell(m(a)^{-1}x)\approx \mu^{\ell}(a)\sum_{r\in \mathrm{Supp}(L_{\ell})}\sum_{j=1}^{m_r(L_{\ell})} c_{r,j,\ell}(x) |a|^r(\log |a|)^{j-1}\quad \textrm{ as } |a|\to 0,
    \end{align*}
    for some (unique) $c_{r,j,\ell}\in C^\infty(X_P^1)$ satisfying $c_{r,j,\ell}(m(a)^{-1}x)=\mu^{\ell}(a)c_{r,j,\ell}(x)$ for $a\in K_{\GG_m}$. Moreover, the asymptotic expansion commutes with the action of $U(\fm^{\mathrm{ab}}\oplus\fg)$ uniformly, and for every $D\in U(\fm^{\mathrm{ab}}\oplus\fg), s\in \rr$ and $\sigma>-s$ we have
    \begin{align*}
        \sup_{\substack{ \ell\in H\\0<|a|\le 1,\, x\in X_P^1}} |a|^{\sigma}|D.R_{s,\ell}(m(a)^{-1}x)|<\infty,
    \end{align*}
    where
    \begin{align*}
        R_{s,\ell}(m(a)^{-1}x):=f_\ell(m(a)^{-1}x)-\mu^\ell(a)\sum_{\substack{r\in \mathrm{Supp}(L_\ell)\\r< s}}\sum_{j=1}^{m_r(L_\ell)} c_{r,j,\ell}(x)|a|^{r}(\log |a|)^{j-1}.
    \end{align*}
    \end{enumerate}
\end{Prop}
 Here the action of $\mathfrak{m}^{\mathrm{ab}}\oplus\fg$ on the asymptotic expansion is the differential of the action induced from the action \eqref{fun:act} (see also \eqref{eq:gtranslate}): for $x\in X_P^1$
\begin{align*}
     R(g)(c_{r,j,\ell})(x)&=\left(\frac{|xg|}{|x|}\right)^r\sum_{i=j}^{m_r(L_\ell)} \binom{i-1}{j-1}\log^{i-j}\left(\frac{|xg|}{|x|}\right) c_{r,i,\ell}(m(|xg|)xg),\\
     L(m)(c_{r,j,\ell})(x)&=\mu^\ell(\omega_P(m))|\omega_P(m)|^r\sum_{i=j}^{m_r(L_\ell)} \binom{i-1}{j-1} \log^{i-j}|\omega_P(m)|c_{r,i,\ell}(x).
\end{align*}

\begin{proof}
     Let $g\in G(F)$ and $\chi\in \widehat{K}_{\GG_m}$. By Theorem \ref{thm:Igusa:real} and the fact that $\cup_{\ell\in H}\mathrm{Supp}(L_\ell)$ is bounded below, the integral defining $f_{\chi_s}(g)$ is absolutely convergent for $\mathrm{Re}(s)\gg 1$ if and only if the integral defining $(f_\ell)_{\chi_s}(g)$ is absolutely convergent for $\mathrm{Re}(s)\gg 1$ assuming that $f_\ell$ satisfies condition (2) for all $\ell\in H$. Moreover, we have
    \begin{align*}
        f_{\chi_s}(g)&=\int_{M^{\mathrm{ab}}(F)} \delta_P^{1/2}(m)\chi(\omega_P(m))|\omega_P(m)|^s \sum_{\ell\in H}f_\ell(m^{-1}g)dm\\
        &=\sum_{\ell\in H}\int_{F^\times} \chi_{s+s_k+1}(a)f_\ell(m(a)^{-1}g)d^\times a.
    \end{align*}
    If $\chi=\mu^{-i},$ then the integrals indexed by $\ell$ above vanish unless $\ell=i,$ in which case it is up to a nonzero constant equal to $M(R(g)f_{i})(\chi_{s+s_k+1}).$ Therefore, to prove the proposition we fix an $\ell$ such that $\chi=\mu^{-\ell}$ and assume $f=f_\ell.$ By the Iwasawa decomposition \ref{cond:orig:arch} is reduced to

    \begin{enumerate}
    \item [(1a')] For $g\in K$ and $D\in U(\mathfrak{gl}_1)$, the integral defining $M(D.R(g)f)(\chi_s)$ is absolutely convergent for $\mathrm{Re}(s)$ sufficiently large, and the complex function
    \begin{align*}
    \frac{M(D.R(g)f)(\chi_{s})}{\prod_{i=1}^k L(s_i+1-\lambda_i(s_k+1),\chi_s^{\lambda_i})} 
    \end{align*}
    is holomorphic, and for all real numbers $\sigma_1<\sigma_2$ and polynomials $Q(s)\in \cc[s]$ such that $Q(s)\prod_{i=1}^k L(s_i+1-\lambda_i(s_k+1),\chi_s^{\lambda_i})$ is holomorphic on $V_{\sigma_1,\sigma_2},$
    \begin{align*}
 \sup_{g\in K} |M(D.R(g)f)(\chi_{s})|_{\sigma_1,\sigma_2,Q}<\infty.
    \end{align*} 
    \end{enumerate}

    By the definition of Tate's $L$-function, $L(s_i+1-\lambda_i(s_k+1),\chi_s^{\lambda_i})$ is up to a nowhere vanishing function equal to
    \begin{align*}
      \begin{cases}
             \Gamma\left(\frac{\lambda_is+s_i+1-\lambda_i(s_k+1)}{2}\right)& \textrm{ if } F \textrm{ is real, } \lambda_i\ell \textrm{ is even}, \\
              \Gamma\left(\frac{\lambda_is+s_i+1-\lambda_i(s_k+1)+1}{2}\right)& \textrm{ if } F \textrm{ is real, } \lambda_i\ell \textrm{ is odd}, \\
             \Gamma\left(\lambda_is+s_i+1-\lambda_i(s_k+1)+|\ell|\lambda_i/2\right)& \textrm{ if } F \textrm{ is complex.}
        \end{cases}
    \end{align*}
Thus poles of $\prod_{i=1}^k L(s_i+1-\lambda_i(s_k+1),\chi_s^{\lambda_i})$ are at $s=-r$ where $r\in L_\ell$ (counting multiplicities). The proposition then follows from Theorem \ref{thm:Igusa:real}.

\end{proof}

\subsection{Representation-theoretic description}

Recall that $\mathcal{S}_\Lambda(X_P^\circ(F))$ consists of functions in $C^\infty(X_P^\circ(F))$ satisfying the equivalent conditions in Proposition \ref{prop:globalchar:real}. By Lemma \ref{lem:inclusmooth}, $\mathcal{S}(X_P^\circ(F))$ is a closed subspace of $\mathcal{S}_{\Lambda}(X_P^\circ(F)).$ Let $\ell\in H$ and $r\in \mathrm{Supp}(L_\ell)$. For $n\in\zz_{>0},$ consider smooth $G(F)$-modules $I_{r,n,\ell}$ defined by
\begin{align*}
        \left\{f\in \mathcal{S}_\Lambda(X_P^\circ(F)) :
    f(x)\approx \sum_{j=1}^{n}|x|^{r}\log^{j-1}(|x|)c_{r,j,\ell}(m(|x|^{\frac{-1}{[F:\rr]}})x) \textrm{ as } |x|\to 0\right\}\Big/\mathcal{S}(X_P^\circ(F)).
\end{align*}
For ease of notation, we will write interchangeably $f\in I_{r,n,\ell}$ as the vector $(c_{r,j,\ell})_{1\le j\le n}$. For convenience, we set $I_{r,0,\ell}:=0$. The proof of Lemma \ref{lem:easy} carries over to the archimedean case.

\begin{Lem}
For $n\in\zz_{>0}$, we have a canonical $G(F)$-equivariant isomorphism.
\begin{align*}
        \varphi_{r,n,\ell}: I_{r,n,\ell}/I_{r,n-1,\ell}\tilde{\longrightarrow} I_P((\mu^{-\ell})_{-r-s_k-1}).
    \end{align*}
by sending $f$ to the unique function in $I_P((\mu^{-\ell})_{-r-s_k-1})$ that equals $c_{r,n,\ell}$ on $X_P^1$. \qed
\end{Lem}

We equip a degenerate principal series with its natural topology so that it is a nuclear Frech\'et space by \cite[Lemma 10.1.1]{Wallach:RGII} and \cite[Corollary 5.6]{BK:globalization}. We endow successively $I_{r,n,\ell}$ with the topology so that $\varphi_{r,n,\ell}$ is a homeomorphism. For each $\ell,r$ we have a natural surjection
\begin{align*}
    \mathcal{S}_\Lambda(X_P^\circ(F))/\mathcal{S}(X_P^\circ(F))\longrightarrow I_{r,m_r(L_\ell),\ell}.
\end{align*}
It is straightforward to check the map is continuous by the closed graph theorem. Therefore, we have a continuous injection 
\begin{align}\label{eq:germ:arch}
    \mathcal{S}_\Lambda(X_P^\circ(F))/\mathcal{S}(X_P^\circ(F))\longrightarrow \prod_{\ell\in H}\prod_{r\in \mathrm{Supp}(L_\ell)} I_{r,m_r(L_\ell),\ell}.
\end{align}

\begin{Lem}\label{lem:nuclear}
    The natural map \eqref{eq:germ:arch} is a homeomorphism. In particular, both $\mathcal{S}_{\Lambda}(X_P^\circ(F))$ and $\mathcal{S}(X_P(F))$ are nuclear spaces.
\end{Lem}

\begin{proof}
    Since a countable product of Fr\'echet spaces is Fr\'echet, to prove the map \eqref{eq:germ:arch} is a homeomorphism, by the open mapping theorem it suffices to show surjectivity. By the Iwasawa decomposition we have a homeomorphism
    \begin{align*}
        X_P^\circ(F)\cong \rr^{+}\times X_P^1.
    \end{align*}
    Arguing analogously as in the proof Borel's theorem \cite[Theorem I.1.3]{smoothapprox} by using a smooth cutoff function at the origin, for a given 
    \begin{align*}
        (c_{r,j,\ell})\in \prod_{\ell\in H}\prod_{r\in \mathrm{Supp}(L_\ell)} I_{r,m_r(L_\ell),\ell}=\prod_{r\in \cup_{\ell\in H}\mathrm{Supp}(L_\ell)} \prod_{\ell :r\in\mathrm{Supp} (L_\ell)}I_{r,m_r(L_\ell),\ell}
    \end{align*}
    one can find $f=(f_\ell)\in \mathcal{S}_{\Lambda}(X_P^\circ(F))$ supported on $\{x\in X_P(F), |x|< 1\}$  such that the germ of $f_\ell$ at the origin is $(c_{r,j,\ell})$. 

    Finally, since both $\mathcal{S}(X_P^\circ(F))$ and $\prod_{\ell,r} I_{r,m_r(L_\ell),\ell}$ are nuclear Fr\'echet spaces, by \cite[Proposition 3.8]{threespace} $\mathcal{S}_\Lambda(X_P^\circ(F))$ and hence $\mathcal{S}(X_P(F))$ are nuclear.
    
\end{proof}

Finally we state the analogue of Theorem \ref{thm:exact} in the Archimedean case, which can be proved similarly.

\begin{Thm}\label{thm:exactarch}
Suppose $F$ is an archimedean local field and $\mathcal{S}(X_P^\circ(F))\subseteq \mathcal{S}(X_P(F))$. We have an exact sequence of smooth Fr\'echet $G(F)$-modules
\begin{align*}
    0\longrightarrow\mathcal{S}(X_P^\circ(F))\longrightarrow \mathcal{S}(X_P(F))\longrightarrow \prod_{\ell\in H}\prod_{r\in \mathrm{Supp}(L_\ell)} A_{r,\ell} \longrightarrow 0,
\end{align*}
where each $A_{r,\ell}$ admits a natural filtration of (Fr\'echet) $G(F)$-submodules
\begin{align*}
    0=A_{r,0,\ell}< A_{r,1,\ell}< \cdots < A_{r,m_r(L_\ell),\chi}=A_{r,\ell}
\end{align*}
together with canonical $G(F)$-equivariant (continuous) injections
\begin{align*}
 0 \neq A_{r,m_r(L_\ell),\ell}/A_{r,m_r(L_\ell)-1,\ell} \hookrightarrow \cdots  \hookrightarrow A_{r,2,\ell}/A_{r,1,\ell} \hookrightarrow A_{r,1,\ell},
\end{align*}
and
$A_{r,1,\ell}\cong \mathrm{Ker}(\mathcal{R}_{P|P^{\mathrm{op}}})< I_P((\mu^{-\ell})_{-r-s_k-1})$ is a proper submodule.\qed
\end{Thm}

\noindent We remark that an analogue of Conjecture \ref{conj:multiplicity} can be made and Proposition \ref{prop:condconj} remains valid.
\subsection{Examples over $\rr$ }\label{ssec:arch:ex} We compute $A_{X_P(F)}=\mathcal{S}(X_P(F))/\mathcal{S}(X_P^\circ(F))$ when $F=\rr$ for examples considered in \S \ref{ssec:nonarch:ex}, and follow the notation therein. We use freely the Weyl algebra on $X_P$ introduced in \S \ref{ssec:Weyl} and its applications in \S \ref{ssec:application:Weyl}. We refer one to $\S \ref{sec:Weyl}$ for any unexplained notation below. Also note that the following examples confirm an analogue of Conjecture \ref{conj:multiplicity}.

\subsubsection{Line bundles over Grassmannians}
The case for $\min(\ell,n-\ell)=1$ is well known. Suppose $\min(\ell,n-\ell)\ge 2$. Then 
\begin{align*}
    L_{0}&:=\sum_{i=0}^{\min(\ell,n-\ell)-1} \left(-i+2\zz_{\ge 0} \right),&L_{1}:=\sum_{i=0}^{\min(\ell,n-\ell)-1} \left(-i+1+2\zz_{\ge 0} \right).
\end{align*}
 It follows from \cite[Theorems 3.4.2, 3.4.3 and 3.4.4]{HL:GL} that for integers $i\le \min(\ell,n-\ell)-1$, 
\begin{align*}
    \textrm{length of } I(1_{i-n/2})=\textrm{length of } I(\mu_{i-1-n/2})=\begin{cases}
        \lfloor \frac{\min(\ell,n-\ell)+1}{2}\rfloor-\max(i/2,0)+1 & \textrm{ if } i \textrm{ even},\\
         \lfloor \frac{\min(\ell,n-\ell)}{2}\rfloor-\max(\frac{i+1}{2},0)+1 & \textrm{ if } i \textrm{ odd}.
    \end{cases} 
\end{align*}
Furthermore, by \cite[Theorem 1.2]{G:A} each representation admits a unique composition series. The Gelfand-Kirillov dimensions of composition factors are in strictly increasing order.

Using \cite[Theorem A]{AGS}, one can mimic the argument in \cite[\S 3]{SW:minimal} to conclude for each $0\le i\le \min(\ell,n-\ell)-1,$ the unique irreducible subrepresentation $I(1_{i-n/2})$ can be realized as some function space $\mathcal{S}(\Omega_i(\rr))\subseteq C^\infty(\Omega_i^\circ(\rr))$. Therefore, we see that $\prod_{r,\ell}A_{r,1,\ell}$ contains all $\mathcal{S}(\Omega_i(\rr))\cc[[X_P]]$. By counting length, we have
\begin{align*}
     \textrm{length of } I(1_{i-n/2})-\textrm{length of } A_{-i,1,0}=1=\textrm{length of } I(\mu_{i-1-n/2})-\textrm{length of } A_{-i+1,1,1},
\end{align*}
and 
\begin{align*}
    \prod_{r,\ell}A_{r,1,\ell}\cong \sum_{i=0}^{\min(\ell,n-\ell)-1} \mathcal{S}(\Omega_i(\rr))\cc[[X_P]]=  \cc[[X_P]]+\sum_{i=1}^{\min(\ell,n-\ell)-1} \mathcal{S}(\Omega_i(\rr))\cc[[X_P]].
\end{align*}
Arguing as the discussion for cones below, we have for $\lfloor\frac{\min(\ell,n-\ell)+1}{2}\rfloor\ge  j\ge 2$
\begin{align*}
     \prod_{r,\ell}A_{r,j,\ell}/A_{r,j-1,\ell}\cong  \bigg(\sum_{i=0}^{\min(\ell,n-\ell)-2(j-1)-1} \mathcal{S}(\Omega_i(\rr))\cc[[X_P]]\bigg)\log^{j-1}|x|.
\end{align*}
Here $\log^{j-1}|x|$ is only symbolic that keeps track of the index $j$. We do the same for other examples below. Suppose $2\ell\le n$. We conclude that $A_{X_P(F)}$ has a quotient module
\begin{align*}
    \begin{cases}
     \bigoplus_{i=0}^{\ell/2-1} \bigg(\mathcal{S}(\Omega_{2i}(\rr))\oplus \mathcal{S}(\Omega_{2i+1}(\rr))\bigg)\log^{\frac{\ell}{2}-1-i} |x|   &  \textrm{ if } \ell \textrm{ even},\\
      \cc\log^{\frac{\ell-1}{2}}|x|\oplus \bigoplus_{i=1}^{(\ell-1)/2} \bigg(\mathcal{S}(\Omega_{2i}(\rr))\oplus \mathcal{S}(\Omega_{2i-1}(\rr))\bigg)\log^{\frac{\ell-1}{2}-i} |x|   &  \textrm{ if } \ell \textrm{ odd}.
    \end{cases}
\end{align*}

\subsubsection{Odd-dimensional cones}\label{sssec:cones}
    For $n\ge 3,$ we have 
    \begin{align*}
        L_0&=(-(n-2)+2\zz_{\ge0})+(2\zz_{\ge 0}), &L_1=(1-(n-2)+2\zz_{\ge0})+(1+2\zz_{\ge 0}).
    \end{align*}
As explained in Example \ref{ex:cone}, $\mathcal{S}(X_n(\rr))$ is the minimal representation of $\mathrm{SO}(n+1,n+1)$ contained in $I_{P_{n+1}}^{\mathrm{SO}(n+1,n+1)}(1_{-1})$. For $n$ odd, both $L_0$ and $L_1$ are genuine sets, and each $A_{r,\ell}=A_{r,1,\ell}$ is irreducible by \cite[\S 2 Case OO]{HT:cone}. We conclude that
\begin{align*}
    \mathcal{S}(X_n(\rr))/\mathcal{S}(X_n^\circ(\rr))\cong \cc[[X_n]]\oplus \mathcal{S}(X_{n-1}(\rr))\cc[[X_n]].
\end{align*}
Here $\mathcal{S}(X_2(\rr))$ is the minimal representation of $\mathrm{SL}_4(\rr)$ contained in $I_{P}^{\mathrm{SL}_{4}}(1_{-1})$, where $P$ is a maximal parabolic stabilizing a plane.

Suppose $n$ is even. We have $A_{r,1,\ell}$ is irreducible if $r<0$ and has length at most $2$ for $r\ge 0$ by \cite[\S 2 Case EE]{HT:cone}, and thus $A_{r,1,\ell}$ must have length $2$ for $r\ge 0$ and
    \begin{align*}
         \prod_{r,\ell} A_{r,1,\ell}\cong \cc[[X_n]]+ \mathcal{S}(X_{n-1}(\rr))\cc[[X_n]].
    \end{align*}
We claim $A_{r,2,\ell}/A_{r,1,\ell}\neq A_{r,1,\ell}$ for $r\ge 0,$ so
    \begin{align*}
        \mathcal{S}(X_n(\rr))/\mathcal{S}(X_n^\circ(\rr))\cong \cc[[X_n]]+ \bigg(\mathcal{S}(X_{n-1}(\rr))\cc[[X_n]]+ \cc[[X_n]]\log |x|\bigg),
\end{align*}
and it has a quotient module 
\begin{align*}
    \cc\log|x|\oplus \mathcal{S}(X_{n-1}(\rr)).
\end{align*}
    
Suppose the claim is false. Let $\widehat{I}_{m}$ be the closed ideal in $\cc[[X_n]]$ generated by homogeneous polynomials of degree $m$. Then $$\prod_{r\ge 0,\ell}A_{r,2,\ell}/A_{r,1,\ell}=\bigg(\cc[[X_n]]+\mathcal{S}(X_{n-1}(\rr))\widehat{I}_{m}\bigg)\log |x|$$
is a $W_{X_n}(\cc)$-module for some $m\ge n-2> 0.$ \sloppy
Applying differential operators $\partial_i$ to $\mathcal{S}(X_{n-1}(\rr))\widehat{I}_{m}\log |x|$, by the product rule one has $\mathcal{S}(X_{n-1}(\rr))\widehat{I}_{m-1}\log |x|$ is contained in $\prod_{r\ge 0,\ell}A_{r,2,\ell}/A_{r,1,\ell},$ a contradiction.

\subsubsection{Lagrangian Grassmannian}
We have
\begin{align*}
    L_{0}&:=(2\zz_{\ge 0})+\sum_{i=1}^{\lfloor n/2\rfloor} \left(-i+\zz_{\ge 0} \right),&L_{1}&:=(1+2\zz_{\ge 0})+\sum_{i=1}^{\lfloor n/2\rfloor} \left(-i+\zz_{\ge 0} \right).
\end{align*}
Suppose $n=2m$ is even. By \cite[Theorem 5.2]{Lee:degeneratepssp} for integers $i\le m$
\begin{align*}
    \textrm{length of }  I(1_{i-\frac{n+1}{2}})=\begin{cases}
        \frac{(2(m+1)-\min(m-i+1,m))(\min(m-i+1,m)+1)}{2}& \textrm{ if } 2\nmid i,\\
        2m+1& \textrm{ if } 2|i, i=m,\\
        m+\frac{(2(m+1)-\min(m-i,m))(\min(m-i,m)+1)}{2} & \textrm{ if } 2| i, i<m.
    \end{cases} 
\end{align*}
Furthermore, by the description of socles
\begin{align*}
    \textrm{length of } A_{-i,1,0}\le \begin{cases}
        \frac{(2m+1-\min(m-i+1,m))\min(m-i+1,m)}{2}& \textrm{ if } 2\nmid i,\\
        m+1& \textrm{ if } 2|i, i=m,\\
        \frac{(2(m+1)-\min(m-i,m))(\min(m-i,m)+1)}{2} & \textrm{ if } 2| i, i<m.
    \end{cases} 
\end{align*}
Similarly, 
\begin{align*}
     \textrm{length of } A_{-i,1,1}\le \begin{cases}
        \frac{(2m+3-\min(m-i+1,m))\min(m-i+1,m)}{2}+\delta_{i<0}& \textrm{ if } 2\nmid i,\\
        m& \textrm{ if } 2|i, i=m,\\
        \frac{(2m-\min(m-i,m))(\min(m-i,m)+1)}{2} & \textrm{ if } 2| i, i<m.
    \end{cases}
\end{align*}

If follows from \cite{KR:Sp2nR} or \cite[Proposition 2.8(iii)]{KR:Siegel-Weil} that 
\begin{align*}
    \bigoplus_{i=0}^{n/2}\bigoplus_{V} \mathcal{S}(V^n(\rr))_{O(V)(\rr)}\leq \bigoplus_{i=0}^{n/2} A_{-i,1,0}\oplus A_{-i,1,1}
\end{align*}
as representations of $\mathrm{Sp}_{2n}(\rr).$ Here $V$ ranges over all equivalence classes of nondegenerate $2i$-dimensional quadratic spaces over $\rr$, and each coinvariant space is irreducible. By counting length, we conclude that 
\begin{align*}
    \prod_{r,\ell} A_{r,1,\ell}= \sum_{V} \mathcal{S}(V^n(\rr))_{O(V)(\rr)}\cc[[X_P]],
\end{align*}
where $V$ ranges over all equivalence classes of nondegenerate even-dimensional quadratic spaces over $\rr$ of dimension at most $n$. The same expression holds for $n$ odd using \cite[Theorems 5.4 and 5.6]{Lee:degeneratepssp}.

In any case, we can use the definition of Schwartz spaces on smooth stacks in \cite{SakStack} to write
\begin{align*}
    \prod_{r,\ell}A_{r,1,\ell}=\sum_{i=0}^{\lfloor n/2\rfloor} \mathcal{S}(\mathcal{X}_i(\rr))\cc[[X_P]].
\end{align*}
Arguing as previous two examples, we have for $2\le j\le \lfloor n/2\rfloor+1$
\begin{align*}
    \prod_{r,\ell}A_{r,j,\ell}/ A_{r,j-1,\ell}\cong \sum_{i=0}^{\lfloor n/2\rfloor+1-j} \mathcal{S}(\mathcal{X}_i(\rr))\cc[[X_P]]\log^{j-1}|x|,
\end{align*}
and $A_{X_P(F)}$ has a quotient module
\begin{align*}
    \bigoplus_{i=0}^{\lfloor n/2\rfloor}  \mathcal{S}(\mathcal{X}_i(\rr))\log^{\lfloor n/2\rfloor-i}|x|.
\end{align*}

\subsection{Examples over $\cc$}\label{ssec:ex:C}
In this subsection we illustrate how one can compute $A_{X_{P}(F)}$ for $F=\cc$ when only $A_{r,j,0}$ are known. Consider the odd-dimensional cones $X_n$ for $n\ge 3$. In this case for $\ell\in \zz$
\begin{align*}
   L_\ell= \left(\frac{|\ell|}{2}+\zz_{\ge 0}\right)+\left(-(n-2)+\frac{|\ell|}{2}+\zz_{\ge 0}\right).
\end{align*}
As explained in Example \ref{ex:cone} $\mathcal{S}(X_n(\cc))$ is the minimal representation of $\mathrm{SO}_{2n+2}(\cc)$ contained in $I_{P_{n+1}}^{\mathrm{SO}_{2n+2}}(1_{-1})$. Therefore, $\cc\oplus \mathcal{S}(X_{n-1}(\cc))$ is contained in $A_1:=\prod_{r,\ell} A_{r,1,\ell}$. Here $\mathcal{S}(X_2(\cc))$ is the minimal representation of $\mathrm{SL}_4(\cc)$ contained in $I_{P}^{\mathrm{SL}_{4}}(1_{-1})$, where $P$ is a maximal parabolic stabilizing a plane.

Since $A_1$ is a $W_{X_n(\cc)}$-module, $U:=(\cc+\mathcal{S}(X_{n-1}(\cc)))\cc[[\mathrm{Res}_{\cc/\rr}X_{n}]]$ is a $P_{n}^{\mathrm{ab}}(\cc)\times\mathrm{SO}_{2n}(\cc)$-submodule of $A_1$. We claim $U=A_{1}$. Then by an argument analogous to Example \ref{sssec:cones} when $F=\rr$ and $\ell$ is even, we have
\begin{align*}
    A_{X_P(\cc)}\cong \bigg(\cc+\mathcal{S}(X_{n-1}(\cc))+\cc\log |x|\bigg)\cc[[\mathrm{Res}_{\cc/\rr}X_n]],
\end{align*}
and it has a quotient module 
\begin{align*}
    \cc\log|x|\oplus \mathcal{S}(X_{n-1}(\cc)).
\end{align*}

We first show that $U$ is stable under $W_{X_{n}(\cc)}$. It suffices to show for $f\in \mathcal{S}(X_{n-1}(\cc))$ and any monomial $p\in \cc[V_{n}],$ we have $\partial_i pf\in  U$ for any $i$. We proceed the proof by induction on the total degree of $p$. When $p$ is constant, since $\mathcal{S}(X_{P_{n-1}}(\cc))\le A_{2-n,0}$ and $\partial_iA_{2-n,0}=0,$ we have $\partial_i f=0$. Suppose $\deg p>0$. By \cite[Lemma 2.4.8]{Kobayashi:Mano} (or direct computation) for any $j,$ $[\partial_i,x_j]$ is an element in $\cc+\mathrm{Lie}(P_n^{\mathrm{ab}}(\cc)\times \mathrm{SO}_{2n}(\cc)).$ Therefore, if we write $p=x_jp'$ for some $j$ and monomial $p',$ then
\begin{align*}
    \partial_i pf=[\partial_i, x_j]p'f+x_i\partial_jp'f.
\end{align*}
By the induction hypothesis $x_i\partial_jp'f\in U$. On the other hand, $[\partial_i, x_j]p'f\in U$ as $U$ is a $P_{n}^{\mathrm{ab}}(\cc)\times\mathrm{SO}_{2n}(\cc)$-module. This justifies the claim. 

By \cite[Theorem 3C]{Sahi} we have 
\begin{align*}
    \prod_{r}A_{r,1,0}=\sum_{\ell \ge 0} \left(\cc+\mathcal{S}(X_{P_{n-1}}(\cc))\right)\mathrm{Sym}^\ell V_n^\lor(\cc)\mathrm{Sym}^\ell\overline{V_n^\lor}(\cc).
\end{align*}
Here we have identified $\mathrm{Sym}^\ell V_n^\lor(\cc)$ and $\mathrm{Sym}^\ell\overline{V_n^\lor}(\cc)$ with their restrictions to $X_n(\cc).$ Now suppose on the contrary that $A_1$ properly contains $U$. Then there is $0\neq v\in A_{r,1,\ell}-U$ for some $r,\ell$. By taking the complex conjugate, we may assume $\ell \ge 0$. For any $p\in \mathrm{Sym}^{\ell}V_n^\lor(\cc),$ we have $pv\in \prod_{r}A_{r,1,0}$. Since $U$ is a $W_{X_n(\cc)}$-module, the set
\begin{align*}
    \{ D\in W_{X_n}(\cc): Dv\in U\} 
\end{align*}
is a left ideal containing $b(\mathbf{x})$ and $b^{\mathrm{op}}(\partial)$ for all homogeneous polynomials $b,b^{\mathrm{op}}\in \cc[X_n]$ of degree $m$ for some large $m\ge \ell$. It follows from Theorem \ref{thm:Weylprop}(2) that the left ideal is $W_{X_n}(\cc)$ so $v\in U$, a contradiction.

\section{Poles of intertwining operators}\label{sec:casebycase}

In this section we prove the set-theoretic inclusion $\mathcal{S}(X_P^\circ(F))\subseteq\mathcal{S}(X_P(F))$ by studying more generally poles of intertwining operators following the approach of \cite[\S 1]{Ikeda:poles:triple}. We first recall some general results on intertwining operators discussed in \S 1.2 in loc.~cit. Our main references are \cite{Shahidi:certainL, Shahidi:Eisen}.

Let $B=TN$ be the standard Borel subgroup with respect to our fixed pinning in \S\ref{sec:prelim}, let $N^{\mathrm{op}}$ be the opposite of $N$, and let $W=W_G$ be the Weyl group of $(G,T)$. Let $\Phi^\lor=\Phi^\lor_G$ be the set of coroots of $(G,T),$ and let $(\Phi^\lor)^+$ (resp. $(\Phi^\lor)^-$) be the set of positive (resp. negative) coroots.

For a quasi-character $\tilde{\chi}$ of $T(F)$, let $I_B(\tilde{\chi}):=\mathrm{Ind}_{B(F)}^{G(F)}(\tilde{\chi})$ be the normalized induced representation in the category of smooth representations. Here we have identified $\tilde{\chi}$ as a quasi-character of $B(F)$. For each $w\in W$, we choose a representative in $K$, which we continue to denote by $w$. Let $\tilde{\chi}^w(t):=\tilde{\chi}(w^{-1} t w)$ and $N_w:=N\cap wN^{\mathrm{op}}w^{-1}$.
The (unnormalized) intertwining operator $M_w$ is defined as
\begin{align*}
    M_w(\tilde{\chi}): I_B(\tilde{\chi})&\longrightarrow I_{B}(\tilde{\chi}^w)\\
            f&\mapsto  \left( g\mapsto\int_{N_w(F)} f(w^{-1}ug)du\right).
\end{align*}
Here the measure $du$ on $N_w$ is defined as in \S \ref{ssec:measure}.
The integral defining $M_w(\tilde{\chi})f(g)$ converges absolutely for all $g\in G(F)$ and $\tilde{\chi}$ in some open cone, and extends meromorphically to all $\tilde{\chi}$.

 For $w\in W$, let $\ell(w)$ be the length of $w$. When considering the quotient $W/W_M$, we always choose representatives of minimal length in each left coset of $W_M$. 
Let $w_0\in W/W_M$ be (the representative of) the long Weyl element. For $w\in W/W_M$, $\chi\in\widehat{F}^\times$ and $s\in \cc$, define
\begin{align*}
     \Phi^{\lor}_w&:= \big\{\beta^\lor \in (\Phi^{\lor})^{+} : w\beta^{\lor}\in (\Phi^{\lor})^- \big\}\subseteq (\Phi^\lor)^{+}-(\Phi^{\lor}_M)^+,\\
     m_w(h,\lambda)&:=\bigg|\left\{\beta^\lor \in \Phi_w^\lor : \sum_{\alpha\in\Delta} \langle \omega_\alpha, \beta^\lor \rangle=h, \langle \omega_P,\beta^\lor\rangle=\lambda\right \}\bigg| \textrm{ for } h, \lambda \in \zz,\\
     c_w(\chi_s)&:=\prod_{\beta^\lor \in \Phi_w^{\lor}} \frac{L\left(- \sum_{\alpha\in \Delta}\langle \omega_\alpha,\beta^\lor\rangle ,\chi_{s+s_k+1}^{\langle \omega_P, \beta^\lor \rangle}\right)}{L\left(1- \sum_{\alpha\in \Delta}\langle \omega_\alpha,\beta^\lor\rangle ,\chi_{s+s_k+1}^{\langle \omega_P, \beta^\lor \rangle}\right)}\\
     &=\prod_{\lambda=1}^{\infty}\prod_{h=1}^{\infty} \frac{ L\left(-h,\chi_{s+s_k+1}^{\lambda}\right)^{\max(0,m_w(h,\lambda)-m_w(h+1,\lambda))}}{ L\left(1-h,\chi_{s+s_k+1}^{\lambda}\right)^{\max(0,m_w(h,\lambda)-m_w(h-1,\lambda))}},\\
     a_w(\chi_s)&:=\prod_{\lambda=1}^{\infty}\prod_{h=1}^{\infty}  L\left(-h,\chi_{s+s_k+1}^{\lambda}\right)^{\max(0,m_w(h,\lambda)-m_w(h+1,\lambda))}.
\end{align*}

Note that $I_P(\chi_s)\le I_B(\tilde{\chi}_s)$, where 
\begin{align}\label{eq:normalizedchi}
    \tilde{\chi}_s:=\delta_P^{1/2}\delta_B^{-1/2}\cdot \chi_s\circ \omega_P =\left(\prod_{\alpha\in \Delta}| \omega_\alpha|^{-1}\right)\cdot \chi_{s+s_k+1}\circ \omega_P, 
\end{align}
and the restriction of $M_w(\tilde{\chi}_s)$ to $I_P(\chi_s)$, denoted by $M_{w}(\chi_s)$, is well defined for $s$ outside a discrete subset of $\cc$. 

Our main result is

\begin{Thm}\label{thm:poles}
Suppose $G$ is either classical or $G_2$. The intertwining operator 
\begin{align}\label{eq:prime}
    M'_w(\chi_s):=a_w(\chi_s)^{-1} M_{w}(\chi_s)
\end{align}
is holomorphic, i.e., if $f^{(s)}$ is a holomorphic section of $I_P(\chi_s)$, then $M'_w(\chi_s)f^{(s)}$ is holomorphic. Moreover for all $h,\lambda\in \zz_{>0}$,
\begin{align}\label{eq:comb-}
     \max(0,m_w(h,\lambda)-m_w(h-1,\lambda))\le \max(0,m_{s_\alpha w}(h,\lambda)-m_{s_{\alpha}w}(h-1,\lambda))
\end{align}
for every $w\in W/W_M$ and $\alpha\in \Delta$ such that $s_\alpha w$ is a reduced expression in $W/W_M$. Consequently, 
\begin{align*}
        d(\chi_s)&:=\prod_{\lambda=1}^{\infty}\prod_{h=1}^{\infty}  L\left(1-h,\chi_{s+s_k+1}^{\lambda}\right)^{\max(0,m_{w_0}(h,\lambda)-m_{w_0}(h-1,\lambda))}
\end{align*}
is the least common denominator of $c_w(\chi_s)$ for all $w\in W/W_M$. 
\end{Thm}

For a simple root $\alpha\in \Delta$, let $\iota_\alpha:\SL_2\to G$ be the homomorphism determined by the fixed pinning. Let $\tilde{w}_0$ be the nonidentity element in $W_{\SL_2}$. We may assume $\iota_\alpha(\tilde{w}_0)=s_\alpha$ is the simple reflection defined by $\alpha$. Then for $f\in I_B(\tilde{\chi})$ we have
\begin{align}\label{eq:sl2}
    \iota_{\alpha}^\ast(M_{s_\alpha}(\tilde{\chi})f)=M_{\tilde{w}_0}(\iota_\alpha^\ast\tilde{\chi})(\iota_\alpha^\ast f).
\end{align}
 If $w_1,w_2\in W$ satisfies $\ell(w_1w_2)=\ell(w_1)+\ell(w_2)$, then $M_{w_1w_2}(\tilde{\chi})=M_{w_1}(\tilde{\chi}^{w_2})\circ M_{w_2}(\tilde{\chi})$. Note that by definition $M_w$ commutes with right translations, and thus to study poles of $M_w$, it suffices to have a good understanding for the case $G=\SL_2$. Our proof of Theorem \ref{thm:poles} requires case-by-case discussions; we will prove the theorem for classical groups and $G=G_2$ in \S \ref{ssec:proof}- \S \ref{ssec:Dn} when $F$ is nonarchimedean, and complete the discussion for the archimedean case in \S \ref{ssec:modarch}.

 Theorem \ref{thm:poles} together with the following lemma has several consequences. 

\begin{Lem}\label{lem:direct}
We have
\begin{align*}
    a_{w_0}(\chi_s) &=\prod_{i=1}^k L(-s_i,\chi_s^{\lambda_i}), \quad\quad d(\chi_s) =\prod_{i=1}^k L(1+s_i,\chi_s^{\lambda_i}).
\end{align*}
\end{Lem}
\begin{proof}
 Both identities can be easily verified which we leave to reader. For a list of positive (co)roots see e.g., \cite{Bour46} for exceptional groups, and see \S\ref{ssec:An}-\S\ref{ssec:Dn} for classical groups; for a list of multiset $\Lambda=\{(s_i,\lambda_i)\}$ see \cite[Appendix A]{Getz:Hsu:Leslie}. We alert the reader that the result in loc.~cit. is stated in the dual side. 
\end{proof}


First, we can rephrase the definition of ($K$-finite) good sections in \cite[\S 5.2]{Getz:Hsu:Leslie} in the sense of \cite[Definition 3.1]{Yamana}. A ($K$-finite) meromorphic section $f^{(s)}\in I(\chi_s)$ is a \textbf{good section} in the sense of \cite{Ikeda:poles:triple, Getz:Hsu:Leslie} if both
\begin{align*}
    \frac{f^{(s)}}{d(\chi_s)} \quad \textrm{ and }\quad\frac{M_{w_0}(\chi_s)f^{(s)}}{a_{w_0}(\chi_s)}
\end{align*}
are holomorphic.

\begin{Cor}\label{cor:Yamana}
Let $\chi$ be a unitary character of $F^\times$. A meromorphic section $f^{(s)}\in I(\chi_s)$ is good if and only if  $f^{(s)}$ has no poles for $\mathrm{Re}(s)>-\frac{1}{6}$ and $d(\chi_s^{-1})a_{w_0}(\chi_s)^{-1}M_{w_0}(\chi_s)f^{(s)}$ has no poles for $\mathrm{Re}(s)<0$.
\end{Cor}

\begin{proof}
Sufficiency follows from the fact that $\lambda_i\le 6$ for all $i$. For the converse, by the assumption of $f^{(s)}$ and Theorem \ref{thm:poles}, the meromorphic section
\begin{align*}
 \frac{M_{w_0}(\chi_s)f^{(s)}}{a_{w_0}(\chi_s)}
\end{align*}
has no poles for $\mathrm{Re}(s)>-\frac{1}{6}$ or $\mathrm{Re}(s)<0$. Thus it is holomorphic. On the other hand, by \eqref{eq:sl2},  Theorem \ref{thm:sl2}\ref{itm:5} and Remark \ref{rem:gammaL} below, we can write
\begin{align}\label{eq:Fouriercomp}
    f^{(s)}= h(s)\frac{d(\chi_s)d(\chi_s^{-1})}{a_{w_0}(\chi_s^{-1})a_{w_0}(\chi_s)}M_{w_0}(\chi_s^{-1})M_{w_0}(\chi_s)f^{(s)}
\end{align}
for some nowhere vanishing holomorphic function $h(s)$, which lies in $\cc[q^{-s},q^s]$ if $F$ is nonarchimedean. Then by the assumption of $f^{(s)}$ and Theorem \ref{thm:poles}
\begin{align*}
    \frac{f^{(s)}}{d(\chi_s)}=h(s) \frac{M_{w_0}(\chi_s^{-1})}{a_{w_0}(\chi_s^{-1})}\frac{d(\chi_s^{-1})}{a_{w_0}(\chi_s)}M_{w_0}(\chi_s)f^{(s)}
\end{align*}
has no poles for $\mathrm{Re}(s)<0$. Thus  $d(\chi_s)^{-1}f^{(s)}$ is holomorphic, and $f^{(s)}$ is a good section.
\end{proof}

Recall that $\mathcal{S}_{\mathrm{BK}}(X_P(F))=\mathcal{S}(X_P^\circ(F))+\mathcal{F}_{P^{\mathrm{op}}|P}\left(\mathcal{S}(X_{P^{\mathrm{op}}}^\circ(F))\right).$
\begin{Cor} \label{cor:cont}
We have $\mathcal{S}(X_P^\circ(F))<\mathcal{S}(X_P(F))$ and  $\mathcal{S}(X_P(F))=\mathcal{S}_{\mathrm{BK}}(X_P(F))$.
\end{Cor}

\begin{proof} 
Observe that $M_{w_0}$ is, up to a conjugation by an element of $K$, essentially equal to $\mathcal{R}_{P|P^{\mathrm{op}}}$. By the Mellin inversion, a function whose Mellin transform along each character is holomorphic if and only if it lies in $\mathcal{S}(X_P^\circ(F))$. The inclusion
$\mathcal{S}(X_P^\circ(F))<\mathcal{S}(X_P(F))$ follows from Theorem \ref{thm:poles} and Lemma \ref{lem:direct}. By Corollary \ref{cor:Yamana}, the identity $\mathcal{S}(X_P(F))=\mathcal{S}_{\mathrm{BK}}(X_P(F))$ is a consequence of Theorem \ref{Thm: Fourier formula}, \cite[Proposition 3.1(4)]{Yamana} and continuity.
\end{proof}

    The following corollary generalizes \cite[Lemma 1.2]{Ikeda:poles:triple}.
    
\begin{Cor}\label{cor:more}
For $f\in \mathcal{S}(X_P(F))$, the section
\begin{align}\label{eq:normalizedMw}
    \left(d(\chi_s)c_w(\chi_s)\right)^{-1}M_w(\chi_s)f_{\chi_s}
\end{align}
is holomorphic for any $w\in W/W_M$ and $\chi\in \widehat{\calo}^\times$.
\end{Cor}

\begin{proof}
Replace $\chi_s$ with $\chi_s^{-1}$. By Theorems \ref{Thm: Fourier formula} and \ref{thm:poles}  and Corollary \ref{cor:cont}, it suffices to consider sections of $I(\chi_s^{-1})$ of the form $$\frac{d(\chi_s^{-1})}{a_{w_0}(\chi_s)}M_{w_0}(\chi_s)f_{\chi_s} \textrm{ where } f\in \mathcal{S}(X_{w_0Pw_0^{-1}}^\circ (F)).$$ 
Write $w_0=ww'$ such that $\ell(w_0)=\ell(w)+\ell(w')$. Then we can rewrite \eqref{eq:normalizedMw} as
\begin{align*}
    &\frac{1}{c_w(\chi_s^{-1})a_{w_0}(\chi_s)}M_w(\chi_s^{-1})M_{w_0}(\chi_s)f_{\chi_s}\\
    &=\frac{1}{d(\chi_s)c_w(\chi_s^{-1})c_{w_0}(\chi_s)}M_w(\chi_s^{-1})M_{w}(\chi_s^{w'})M_{w'}(\chi_s)f_{\chi_s}
\end{align*}
By \eqref{eq:sl2} and Theorem \ref{thm:sl2}\ref{itm:5}, this section has the same poles (counting multiplicities) as 
\begin{align*}
    \frac{1}{d(\chi_s)c_{w'}(\chi_s)}M_{w'}(\chi_s)f_{\chi_s},
\end{align*}
which is holomorphic by Theorem \ref{thm:poles}.
\end{proof}

\begin{Cor}\label{cor:nontrivial}
    \begin{enumerate}
        \item  Suppose $F$ is nonarchimedean.  The module $A_{r,m_r(L_d),\chi}/A_{r,m_r(L_d)-1,\chi}$ is nonzero for all $\mathrm{ord}(\chi)=d$ and $r\in \mathrm{Supp}(L_d)$.
        \item Suppose $F$ is archimedean. For $\ell\in H,$ the module $A_{r,m_r(L_\ell),\ell}/A_{r,m_r(L_\ell)-1,\ell}$ is nonzero for all $r\in \mathrm{Supp}(L_\ell)$.
    \end{enumerate}
\end{Cor}

\begin{proof} 
By Corollary \ref{cor:Yamana} this follows from \cite[Proposition 3.1 (3)]{Yamana}. 
\end{proof}

\begin{Cor}\label{cor:trivcont}
    We have $\cc\le A_{0,1,\mathrm{triv}}$. 
\end{Cor}

\begin{proof}
     Assume first $F$ is nonarchimedean. Since $\mathcal{S}(X_P(F))$ is independent of $\psi$, we may assume the conductor of $\psi$ is $\calo$. Let $f$ be the unique right $G(\calo)$-invariant function in $C^\infty(X_P^\circ(F))$ such that $(f)_{1_s}|_{G(\calo)}=d(1_s)$. By Gindikin-Karpelevi\v{c} formula \cite[Proposition 4.6]{Lai} we have $f\in \mathcal{S}(X_P(F))$. Since $d(1_0)\neq 0$, the image of $f$ in $A_{0,1,1}$ under the exact sequence in Theorem \ref{thm:exact} is nonzero. As $f$ is right $G(\calo)$-invariant, the image is a nonzero constant. The archimedean case is similar using \cite{Knapp:GKformula}.
\end{proof}

\begin{Cor}\label{cor:localnonarch}
    Suppose $F$ is nonarchimedean. The space $\mathcal{S}(X_P(F))$ is local and  $\mathcal{S}_{\mathrm{ES}}(X_P(F))\le  \mathcal{S}(X_P(F))$.
\end{Cor}

\begin{proof}
     Recall that germs $\mathcal{S}_{\mathrm{ES}}(X_P(F))/\mathcal{S}(X_P^\circ(F))=\cc$ are constant functions at the origin. Therefore, $\mathcal{S}_{\mathrm{ES}}(X_P(F))\le \mathcal{S}(X_P(F))$ by Corollary \ref{cor:trivcont}. Since functions in $\mathcal{S}(X_P(F))$ are compactly supported in $X_P(F),$ we have
$C^\infty(X_P(F))\cdot \mathcal{S}(X_P(F))=\mathcal{S}_{\mathrm{ES}}(X_P(F))\cdot \mathcal{S}(X_P(F)).$ As $\mathcal{S}(X_P^\circ(F))\cdot \mathcal{S}(X_P(F))=\mathcal{S}(X_P^\circ(F))$, the assertion that $\mathcal{S}(X_P(F))$ is local is equivalent to the module $\mathcal{S}(X_P(F))/\mathcal{S}(X_P^\circ(F))$ being a $\cc$-vector space. 
\end{proof}

\subsection{Proof outline of Theorem \ref{thm:poles}}\label{ssec:proof}

In this subsection, we discuss in detail the proof of Theorem \ref{thm:poles} for the case $G=G_2$ and explain how to adapt it for classical groups. As mentioned, we need decent understanding for the case of $\SL_2$:

\begin{Thm}\label{thm:sl2} Let $G=\SL_2$ and $w_0=\begin{psmatrix}
 & 1\\
-1 & 
\end{psmatrix}$. Then the intertwining operator $M_{w_0}(\chi_s):I_B(\chi_s)\to I_{B}(\chi_{s}^{-1})$ has following properties.
\begin{enumerate}[label={(\roman*)}]
    \item The function $L(0,\chi_{s})^{-1}M_{w_0}(\chi_s)$ is holomorphic. \label{itm:1}
    \item  The kernel of $M_{w_0}(1_{-1})$ is the trivial representation. \label{itm:2}
\item The map $\mathrm{Res}_{s=0}M_{w_0}(1_s)$ is a nonzero scalar multiplication.  \label{itm:3}
\item The image of $M_{w_0}(1_{1})$ is the trivial representation. \label{itm:4}
\item The operator $\gamma(s,\tilde{\chi},\psi)\gamma(-s,\chi^{-1},\psi)M_{w_0}(\chi_s^{-1})M_{w_0}(\chi_s)$ is the identity for all $s$. \label{itm:5}
\item  Suppose $F$ is nonarchimedean and $\psi$ has conductor $\mathcal{O}$. Let $\phi_s$ be the unique right $\SL_2(\calo)$-invariant section in $I_B(1_s)$ such that $\phi_s|_{\SL_2(\calo)}=1$. Then $$M_{w_0}(1_s)\phi_s=\frac{L(0,1_s)}{L(1,1_s)}\phi_{-s}.$$
\label{itm:6}
\end{enumerate}
\end{Thm}

\begin{proof}
See \cite[\S 1.2]{Ikeda:poles:triple}. Note that there is a typo in (1.2.3) in loc.~cit. We refer one to \cite[Lemma 3.11]{Getz:Hsu} for a corrected statement.
\end{proof}

\noindent  We will be constantly referring to this theorem in the rest of the section, so for simplicity instead of writing Theorem \ref{thm:sl2}\ref{itm:1}, we will just write \ref{itm:1} and do the same for \ref{itm:2}-\ref{itm:6}.

\begin{Rem}\label{rem:gammaL}
Since $\gamma(s,\chi,\psi)$ equals  $L(1,\chi_s^{-1})/L(0,\chi_s)$ up to a nowhere vanishing function, \ref{itm:5} implies
\begin{align*}
   \frac{L(1,\chi_s^{-1})L(1,\chi_s)}{L(0,\chi_s)L(0,\chi_s^{-1})}M_{w_0}(\chi_s^{-1})M_{w_0}(\chi_s)
\end{align*}
is an $\SL_2(F)$-equivariant isomorphism at all $s$.
\end{Rem}

We record the following lemma and notations for later use (see e.g., \cite{Bour46}).

\begin{Lem}\label{lem:positive}
 Let $w\in W/W_M$ and $w=w_m\cdots w_2w_1$ be a reduced expression. For each $i$, let $\alpha_{(i)}\in \Delta$ be the simple root such that $w_i$ is the corresponding reflection. Then
\begin{align*}
   \Phi^{\lor}_w=\{\tilde{\alpha}_{(i)}^{\lor} := w_1 \cdots w_{i-2}w_{i-1}\alpha_{(i)}^\lor : 1\le i\le m\}.
\end{align*}
\qed
\end{Lem}
\noindent We say the coroot $\tilde{\alpha}_{(i)}^\lor$ corresponds to the simple reflection $w_i$ (under the reduced expression).

Note that for every $w\in W/W_M$ there is a reduced expression $w_m\cdots w_2w_1$ of $w_0$ such that $w=w_r\cdots w_2w_1$ for some $r\le m$, and a reduced expression can be transformed to another expression by performing a sequence of defining relations of the Weyl group $W$. Explicitly, for two distinct simple (co)roots $\alpha$ and $\alpha'$, let $s_\alpha, s_{\alpha'}$ be the corresponding simple reflections and let $n_{\alpha\alpha'}$ be the number of edges between the corresponding nodes in the Dynkin diagram. 
\begin{enumerate}[label=(\alph*)]
    \item If $n_{\alpha\alpha'}=0$, then replace $s_{\alpha}s_{\alpha'}$ with $s_{\alpha'}s_{\alpha}$; \label{itm:a}
    \item if $n_{\alpha\alpha'}=1$, then replace  $s_{\alpha}s_{\alpha'}s_{\alpha}$ with $s_{\alpha'}s_{\alpha}s_{\alpha'}$;  \label{itm:b}
    \item if $n_{\alpha\alpha'}=2$, then replace  $(s_{\alpha}s_{\alpha'})^2$ with $(s_{\alpha'}s_{\alpha})^2$;  \label{itm:c}
    \item if $n_{\alpha\alpha'}=3$, then
    replace $(s_{\alpha}s_{\alpha'})^3$ with $(s_{\alpha'}s_{\alpha})^3$.  \label{itm:d}
\end{enumerate}

 Assume $F$ is nonarchimedean until the end of \S \ref{ssec:Dn}. We prove Theorem \ref{thm:poles} for the nonarchimedean case first and explain how the proof can be modified for the archimedean case in \S \ref{ssec:modarch}.  Recall the operator $M'_w(\chi_s)$ defined in \eqref{eq:prime} for $w\in W/W_M$.

\begin{Lem}\label{lem:holo}
Suppose $F$ is nonarchimedean. Let $w\in W/W_M$. Let $\alpha$ be a simple root such that $s_\alpha w$ is a reduced expression in $W/W_M$, and $(h,\lambda)$ be the unique pair such that $m_{s_\alpha w}(h,\lambda)-m_w(h,\lambda)=1$.
 Suppose $M'_{w}(\chi_s)$ is holomorphic.

Assume
\begin{align}\label{eq:comb+}
    m_w(h,\lambda)\ge m_w(h+1,\lambda).
\end{align}
Then $M'_{s_\alpha w}(\chi_s)$ is holomorphic if either $\chi^\lambda\neq 1$ or $\chi^\lambda=1$ and one of the following holds.
 \begin{enumerate}[label={(\Roman*)}]
     \item $m_w(h,\lambda)\ge m_w(h-1,\lambda)$. \label{itm:I}
     \item $M'_w(\chi_s)f^{(s)}\big|_{\lambda(s+s_k+1)=h-1}$ is left $\iota_{\alpha}(\mathrm{SL}_2)$-invariant for all holomorphic sections $f^{(s)}$ of $I_P(\chi_s)$. \label{itm:II}
 \end{enumerate}
\end{Lem}

\begin{proof}
By \eqref{eq:normalizedchi}, \eqref{eq:sl2} and \ref{itm:1},
\begin{align*}
&L\left(-\sum_{\beta \in \Delta} \langle w\omega_\beta ,\alpha^{\lor}\rangle,\chi_{s+s_k+1}^{\langle w\omega_P,\alpha^\lor \rangle}\right)^{-1}M_{s_{\alpha}}(\chi_s^w)\left(M'_{w}(\chi_s)f^{(s)}\right)
\end{align*}
is holomorphic. 
By Lemma \ref{lem:positive}
\begin{align*}
    L\left(-\sum_{\beta\in \Delta} \langle w\omega_\beta ,\alpha^{\lor}\rangle,\chi_{s+s_k+1}^{\langle w\omega_P,\alpha^\lor \rangle}\right)=L\left(-\sum_{\beta\in \Delta} \langle \omega_\beta ,w^{-1}\alpha^{\lor}\rangle,\chi_{s+s_k+1}^{\langle \omega_P,w^{-1}\alpha^\lor \rangle}\right)=L(-h,\chi_{s+s_k+1}^{\lambda}).
\end{align*}
Thus by the definition of $a_{s_\alpha w}(\chi_s)$ we can assume $\chi^\lambda=1$ and \ref{itm:I} fails, so that by assumption \eqref{eq:comb+}
\begin{align*}
    a_{s_\alpha w}(\chi_s)=\frac{\zeta\left(\lambda(s+s_k+1)-(h-1)\right)}{\zeta\left(\lambda(s+s_k+1)-h\right)}a_{w}(\chi_s).
\end{align*}
By \ref{itm:II} and \ref{itm:2},
\begin{align*}
    M_{s_\alpha}(\chi_s^w)M'_w(\chi_s)f^{(s)}\big|_{\lambda(s+s_k+1)=h-1}=0,
\end{align*}
and thus $M'_{s_\alpha w}(\chi_s)f^{(s)}$ is holomorphic.
\end{proof}

\noindent Similar to Theorem \ref{thm:sl2}, we write \ref{itm:I} and \ref{itm:II} when referring to conditions in Lemma \ref{lem:holo}.

 The strategy to prove Theorem \ref{thm:poles} is as follows. We will choose a reduced expression $w_m\cdots w_1$ of $w_0$ and list corresponding coroots $\tilde{\alpha}_{(i)}^\lor$ (see Lemma \ref{lem:positive}). For each case, we will first verify the combinatorics inequalities $\eqref{eq:comb-}$ and $\eqref{eq:comb+}$ for any $w\in W/W_M$ by studying the effects of operations above on the order of $\tilde{\alpha}_{(i)}^{\lor}$. Then we apply Lemma \ref{lem:holo} repeatedly to show, inductively on the length of $w$, that the invariance property \ref{itm:II} holds whenever \ref{itm:I} fails. 
 
 In the rest of the section, we will use the Bourbaki numbering of the Dynkin diagram and (co)roots. Let $P=P_\ell$ be the maximal parabolic subgroup associated to the $\ell$th node of the Dynkin diagram of $(G,B,T)$. Let $\alpha_i$ denote the simple root attached to the $i$th node and let $s_{i}$ be the corresponding simple reflection. However, note that we also write $s_k$ for the highest data. This should lead to little or no confusion as the subscript $k$ is only used as the size of the multiset $\Lambda$ throughout the paper.

\begin{proof}[Proof of Theorem \ref{thm:poles} for $G=G_2$] 
Observe that $M^{\mathrm{der}}\cong \SL_2$ for both cases and $|W/W_{M}|=6=\ell(w_0)+1$, so that $w_0\in W/W_{M}$ has a unique reduced expression.  Inequalities \eqref{eq:comb-} and \eqref{eq:comb+} will be clear from the order of coroots below.

For $\ell=1$, the reduced expression of $w_0$ is 
\begin{align*}
   s_1s_2s_1s_2s_1;
\end{align*}
corresponding coroots $\tilde{\alpha}_{(i)}^\lor$ are
\begin{align*}
    \alpha_1^\lor, \alpha_1^\lor+\alpha_2^{\lor},  2\alpha_1^\lor+3\alpha_2^{\lor}, \alpha_1^\lor+2\alpha_2^{\lor}, \alpha_1^\lor+3\alpha_2^\lor.
\end{align*}
Since \ref{itm:I} holds for $w=\textrm{Id}, s_2s_1$, we may assume $\chi=1$, and we are left to verify \ref{itm:II} holds for $w=s_1,s_1s_2s_1, s_2s_1s_2s_1$.

A holomorphic section $f^{(s)}$ of $I_{P_1}(\chi_s)$ is left $\iota_{\alpha_2}(\SL_2)$-invariant, and so is 
\begin{align*}
    M'_{s_1}(1_s)f^{(s)}\big|_{s+s_k+1=1} 
\end{align*}
by \ref{itm:3},  which justifies the case $w=s_1$. For $w=s_1s_2s_1$, by \ref{itm:4} 
\begin{align*}
    M'_{s_1}(\chi_s)f^{(s)}\big|_{s+s_k+1=2}
\end{align*}
is left $\iota_{\alpha_1}(\SL_2)$-invariant, and so is
\begin{align*}
     M'_{s_2s_1}(\chi_s)f^{(s)}\big|_{s+s_k+1=2}
\end{align*}
by \ref{itm:3}. As \ref{itm:I} holds for $s_2s_1$, by \ref{itm:2}
\begin{align*}
    M'_{s_1s_2s_1}(1_s)f^{(s)}\big|_{s+s_k+1=2}=0.
\end{align*}
This proves the case $w=s_1s_2s_1$. Finally for $w=s_2s_1s_2s_1$, by \ref{itm:4}
\begin{align*}
    M'_{s_1s_2s_1}(1_s)f^{(s)}\big|_{s+s_k+1=3} 
\end{align*}
is left $\iota_{\alpha_1}(\SL_2)$-invariant, and thus so is
\begin{align*}
    M'_{s_2s_1s_2s_1}(1_s)f^{(s)}\big|_{s+s_k+1=3} 
\end{align*}
by \ref{itm:3}.

For $\ell=2$, the reduced expression of $w_0$ is
\begin{align*}
    s_2s_1s_2s_1s_2;
\end{align*}
corresponding coroots $\tilde{\alpha}_{(i)}^\lor$ are
\begin{align*}
    \alpha_2^\lor, \alpha_1^\lor+3\alpha_2^{\lor},  \alpha_1^\lor+2\alpha_2^{\lor}, 2\alpha_1^\lor+3\alpha_2^{\lor}, \alpha_1^\lor+\alpha_2^\lor.
\end{align*}
 Since \ref{itm:I} holds for $w=\textrm{Id},s_2,s_1s_2$, we may assume $\chi^3=1$, and we are left to verify \ref{itm:II} holds for $w=s_2s_1s_2, s_1s_2s_1s_2$. 

A holomorphic section $f^{(s)}$ of $I_{P_2}(\chi_s)$ is left $\iota_{\alpha_1}(\SL_2)$-invariant. Since 
\begin{align*}
    L(-4,\chi_{s+s_k+1}^3)^{-1}M_{s_1}(\chi_s^{s_2})\big|_{3(s+s_k+1)=4}
\end{align*}
is a nonzero scalar multiplication by \ref{itm:3} and
$s_2s_2=\mathrm{Id}$, we have
\begin{align*}
    M'_{s_2s_1s_2}(\chi_s)f^{(s)}\big|_{3(s+s_k+1)=4}
\end{align*}
is left $\iota_{\alpha_1}(\SL_2)$-invariant, which justifies the case $w=s_2s_1s_2$. For $w=s_1s_2s_1s_2$, it suffices to consider $\chi=1$. By \ref{itm:3}
\begin{align*}
    M'_{s_2}(\chi_s)f^{(s)}\big|_{s+s_k+1=1}
\end{align*}
is left $\iota_{\alpha_1}(\SL_2)$-invariant. As \ref{itm:I} holds for $s_2$ and $s_1s_2$,  by \ref{itm:2}
\begin{align*}
    M'_{s_1s_2}(\chi_s)f^{(s)}\big|_{s+s_k+1=1}=0,
\end{align*}
and hence
\begin{align*}
M'_{s_1s_2s_1s_2}(\chi_s)f^{(s)}\big|_{s+s_k+1=1}=0.
\end{align*}
\end{proof}

Arguments are more complicated for other groups as in general there are more than one reduced expression of $w_0$. We close this subsection by studying the effects of \ref{itm:a}, \ref{itm:b} and \ref{itm:c} on the order of coroots $\tilde{\alpha}_{(i)}^\lor$. Let $w_0=w_m\cdots w_1$ be a reduced expression.
\begin{enumerate}[label=(\alph*$^\prime$)]
    \item If $n_{\alpha_{(i+1)}\alpha_{(i)}}=0$ and $w_{i+1}w_{i}$ is replaced with $w_iw_{i+1}$, then 
    \begin{align*}
        \textrm{ the position of the coroots } \tilde{\alpha}_{(i)}^\lor \textrm{ and } \tilde{\alpha}_{(i+1)}^\lor \textrm{ are swapped. } 
    \end{align*}\label{itm:a'}
    \item If $n_{\alpha_{(i+1)}\alpha_{(i)}}=1$, $\alpha_{(i)}=\alpha_{(i+2)}$, and $w_{i+2}w_{i+1}w_i$ is replaced with $w_{i+1}w_{i}w_{i+1}$, then 
        \begin{align*}
        \textrm{ the position of the coroots } \tilde{\alpha}_{(i)}^\lor \textrm{ and } \tilde{\alpha}_{(i+2)}^\lor \textrm{ are swapped. } 
    \end{align*}
    Moreover, $ \alpha_{(i+1)}^\lor=\alpha_{(i+2)}^\lor+\alpha_{(i)}^\lor$, so
    $$\langle \omega_P, \alpha_{(i+1)}^\lor\rangle\ge 2$$ 
    and thus such operations will not occur if $G$ is of type $A$.\label{itm:b'}
    \item  If $n_{\alpha_{(i+1)}\alpha_{(i)}}=2$, $\alpha_{(i)}=\alpha_{(i+2)}$ and $\alpha_{(i+1)}=\alpha_{(i+3)}$, then
    \begin{align*}
        \Phi^\lor_{w_{i+3}\cdots w_1}-\Phi^\lor_{w_{i-1}\cdots w_1}=
    \{\beta^\lor, \beta^{\lor}+\gamma^\lor,\beta^\lor+2\gamma^\lor,\gamma^\lor\}
    \end{align*}
    for some coroots $\beta^\lor,\gamma^\lor\in\Phi^+-\Phi^+_M$. In particular, we have $\langle \omega_P, \beta^\lor+2\gamma^\lor\rangle \ge 3$, which is impossible for classical groups. \label{itm:c'}
\end{enumerate}

\subsection{Type $A_n$ }\label{ssec:An}
A reduced expression of $w_0$ is
\begin{align}\label{eq:specific}
     (s_{n-\ell+1}\cdots s_{n-1}s_n)\cdots (s_{r-\ell+1}\cdots s_{r-1} s_{r})\cdots (s_1\cdots s_{\ell-1}s_\ell)
\end{align}
Here the parentheses are only present in this expression so that the reader can follow the pattern. The corresponding order of coroots $\tilde{\alpha}^{\lor}_{(i)}$ is
\begin{center}
    \begin{tabular}{ccccc}
$\alpha_\ell^\lor,$                & $\cdots$ & $\sum_{j=t}^\ell \alpha_{j}^\lor,$ & $\cdots$ & $\sum_{j=1}^\ell \alpha_{j}^\lor,$ \\
                $\vdots$                       &           & $\vdots$                          &           &     $\vdots$                                   \\
$\sum_{j=\ell}^r \alpha_{j}^\lor,$ & $\cdots$ & $\sum_{j=t}^r \alpha_{j}^\lor,$    & $\cdots$ & $\sum_{j=1}^{r} \alpha_{j}^\lor,$  \\
                    $\vdots$                   &           & $\vdots$                          &           &                 $\vdots$                       \\
$\sum_{j=\ell}^n \alpha_{j}^\lor,$ & $\cdots$ & $\sum_{j=t}^n \alpha_{j}^\lor,$    & $\cdots$ & $\sum_{j=1}^n \alpha_{j}^\lor.$   
\end{tabular}
\end{center}
The $i$th row of coroots (read from the top left) corresponds to the $i$th parenthesis of \eqref{eq:specific} (read from the right). Each positive coroot in $\Phi_{w_0}^{\lor}$ appears exactly once and there are no other coroots in $\Phi^\lor-\Phi_{w_0}^{\lor}$ in the above list, so \eqref{eq:specific} is indeed a reduced expression of $w_0$. The same justification of reduced expressions will be used in the rest of the section without mention.

We denote
\begin{align*}
    \beta^\lor \longleftrightarrow \beta^{\prime\lor}
\end{align*}
if the order of two coroots $\beta^\lor$ and $\beta^{\prime\lor}$ can be reversed under a series of operations \ref{itm:a'} and \ref{itm:b'}. Denote coroots
\begin{align*}
\sum_{j=t}^r \alpha^\lor_j \quad   \textrm{ for } 1\le t \le \ell\le r\le n
\end{align*}
by $(r,t)$. Since operations \ref{itm:b'} cannot be applied if $G$ is of type $A$, via operations \ref{itm:a'} we have
\begin{align}\label{eq:R1}
\textrm{ for } r\le r', (r,t) \longleftrightarrow (r',t') \textrm{ if and only if } r<r' \textrm{ and } t<t'.\tag{R1}
\end{align}
 Note that coroots share the same $h$-value if they lie on the same $45^\circ$ line in the table above. Consequently for $w\in W/W_M$, $m_w(h,1)\ge c$ only if $ m_w(h-1,1)\ge c$ unless $$m_w(h,1)=m_{w_0}(h,1)=m_{w_0}(h-1,1)+1=m_{w}(h-1,1)+1.$$
This verifies inequalities \eqref{eq:comb-} and \eqref{eq:comb+}.

 We prove the holomorphy of $M'_w(\chi_s)$ by induction on both $n$ and length of $w$. It suffices to prove the holomorphy when $\chi=1$. For $\ell=1$, the reduced expression \eqref{eq:specific} is unique. As \ref{itm:I} only holds for $\mathrm{Id}$, we need to check
 \begin{align*}
     M'_{s_{r}\cdots s_1}(1_s)f^{(s)}\big|_{s+s_k+1=r} \textrm{ is left } \iota_{\alpha_{r+1}}(\SL_2)\textrm{-invariant}.
 \end{align*}
 for $1\le r<n$. Since the coroot $(r,1)$ has $h$-value $r$ and $f^{(s)}$ is left $\iota_{\alpha_{r+1}}(\SL_2)$-invariant, the assertion follows from \ref{itm:3}. 
 
 For the general case, by symmetry we may assume $n\ge 3$ and $\lceil n/2 \rceil \ge \ell\ge 2$. By induction hypothesis, we may assume  $(r,1)\in \Phi^\lor_w$ for some $r$.  Choose $r$ to be maximal. By \eqref{eq:R1} we can write $w=s_{r-\ell+1}w_1$ such that $s_{r-\ell+1}$ corresponds to the coroot $(r,1)$. Assume first $r>\ell$. We need to show 
\begin{align}\label{eq:Abeta1}
    M'_{w_1}(1_s)f^{(s)}\big|_{s+s_k+1=r-1} \textrm{ is left } \iota_{\alpha_{r-\ell+1}}(\SL_2)\textrm{-invariant}.
\end{align} 
Write $w_1=s_{r-\ell}w'$ such that $s_{r-\ell}$ corresponds to the coroot $(r-1,1)$. We have 
\begin{align*}
    w^{\prime -1}\alpha_{r-\ell+1}=w_1^{-1}\alpha_{r-\ell+1}-w^{\prime -1}\alpha_{r-\ell}=\sum_{j=1}^r \alpha_j-\sum_{j=1}^{r-1} \alpha_j=\alpha_r.
\end{align*}
Since $f^{(s)}$ is left $\iota_{\alpha_r}(\SL_2)$-invariant, $M_{w'}'(1_s)f^{(s)}$ is left $\iota_{\alpha_{r-\ell+1}}(\SL_2)$-invariant. As the coroot $(r-1,1)$ has $h$-value $r-1$,  \eqref{eq:Abeta1} follows by \ref{itm:3}.

Suppose $r=\ell$. By the induction hypothesis, we may assume $(n,t)\in  \Phi_w^\lor$ for some $t$. Choose $t$ to be minimal. A similar argument as above justifies the holomorphy of $M'_w(1_s)$ if $t<\ell$, and thus we can assume $t=\ell$. By \eqref{eq:R1} we can also write $w=s_nw_1$ such that $s_n$ corresponds to the coroot $(n,\ell)$, which has $h$-value $n-\ell+1$. Therefore by the induction hypothesis and \ref{itm:1}, if $n-\ell\neq \ell-1$ then $M'_{w}(1_s)f^{(s)}=M'_{s_nw_1}(1_s)f^{(s)}$ is holomorphic at $s+s_k+1=\ell-1$ and hence is holomorphic at all $s$. For the same reason, to justify $M'_{w}(1_s)f^{(s)}$ is holomorphic at $s+s_k+1=n-\ell=\ell-1$, it suffices to consider the case where no coroot in $\Phi_w^\lor$ has $h$-value greater than $n-\ell+1=\ell$. 

Note that when $n=2\ell-1$, \ref{itm:I} fails except when $m_w(\ell,1)=\ell$. Assume $m_w(\ell,1)<\ell$ and thus there exists an integer $0\le i<\ell-2$ such that $(n-i,\ell-i)\in \Phi_w^{\lor}$ but $(n-i-1,\ell-i-1)\not\in \Phi_w^{\lor}.$ By \eqref{eq:R1} we can rewrite $w=s_{n-2i}s_{n-2i-1}w'$ such that $s_{n-2i-1}$ and $s_{n-2i}$ correspond to coroots $(n-i-1,\ell-i)$ and $(n-i,\ell-i)$ respectively. We need to show 
\begin{align*}
    M'_{s_{n-2i-1}w'}(1_s)f^{(s)}\big|_{s+s_k+1=n-\ell} \textrm{ is left } \iota_{\alpha_{n-2i}}(\SL_2)\textrm{-invariant}.
\end{align*}
We have \begin{align*}
    w^{\prime -1}\alpha_{n-2i}=(s_{n-2i-1}w')^{-1}\alpha_{n-2i}-w^{\prime-1}\alpha_{n-2i-1}=\sum_{j=\ell-i}^{n-i} \alpha_j-\sum_{j=\ell-i}^{n-i-1} \alpha_j=\alpha_{n-i}.
\end{align*}
Since $f^{(s)}$ is left $\iota_{\alpha_{n-i}}(\SL_2)$-invariant and the coroot $(n-i-1,\ell-i)$ has $h$-value $n-\ell$, the assertion follows by \ref{itm:3}.

 \qed

\subsection{Type $B_n$ ($n\ge 2$)}\label{ssec:Bn} 

We break down the discussion into two cases: $n=\ell$ and $n>\ell$. Consider first $1\le \ell<n$. A reduced expression of $w_0$ is
\begin{small}
\begin{align}\label{eq:Bnw0}
\begin{split}
    &\big((s_\ell \cdots s_{n-1})(s_n)\big)\cdots \big((s_{\ell-r+1}\cdots s_{n-r})(s_{n-r+1}\cdots s_n)\big)\cdots \big((s_1\cdots s_{n-\ell})(s_{n-\ell+1}\cdots s_n)\big)\cdot\\
    &(s_{n-\ell}\cdots s_{n-2}s_{n-1})\cdots (s_{r-\ell+1} \cdots s_{r-1} s_{r})\cdots(s_1\cdots s_{\ell-1} s_\ell);
\end{split}
\end{align}
\end{small}
corresponding coroots $\tilde{\alpha}_{(i)}^\lor$ are
\begin{align}
    \begin{tabular}{ccccc}
$\alpha_\ell^\lor,$                & $\cdots$ & $\sum_{j=t}^\ell \alpha_{j}^\lor,$ & $\cdots$ & $\sum_{j=1}^\ell \alpha_{j}^\lor,$ \\
                $\vdots$                       &           & $\vdots$                          &           &     $\vdots$                                   \\
$\sum_{j=\ell}^r \alpha_{j}^\lor,$ & $\cdots$ & $\sum_{j=t}^r \alpha_{j}^\lor,$    & $\cdots$ & $\sum_{j=1}^{r} \alpha_{j}^\lor,$  \\
                    $\vdots$                   &           & $\vdots$                          &           &                 $\vdots$                       \\
$\sum_{j=\ell}^{n-1} \alpha_{j}^\lor,$ & $\cdots$ & $\sum_{j=t}^{n-1} \alpha_{j}^\lor,$    & $\cdots$ & $\sum_{j=1}^{n-1} \alpha_{j}^\lor,$   
\end{tabular}\tag{$\ast$}\label{tab:upper1}
\end{align}
\begin{center}
\begin{tabular}{>{\tiny}c>{\tiny}c>{\tiny}c>{\tiny}c>{\tiny}c}
$\alpha_n^\lor+2\sum_{j=\ell}^{n-1}\alpha^\lor_j,$  & $\cdots$ & $\alpha_n^\lor+2\sum_{j=\ell}^{n-1}\alpha^\lor_j+\sum_{j=t}^{\ell-1}\alpha^\lor_j,$ & $\cdots$ & $\alpha_n^\lor+2\sum_{j=\ell}^{n-1}\alpha^\lor_j+\sum_{j=1}^{\ell-1}\alpha^\lor_j,$    \\
$\alpha_n^\lor+\sum_{j=\ell}^{n-1} \alpha_j^\lor, $ & $\cdots$ & $\alpha_n^\lor+2\sum_{j=m}^{n-1} \alpha_j^\lor+\sum_{j=\ell}^{m-1} \alpha_j^\lor, $ & $\cdots$ & $\alpha_n^\lor+2\sum_{j=\ell+1}^{n-1} \alpha_j^\lor+ \alpha_\ell^\lor, $               \\
$\vdots$                                            &    $\ddots$      & $\vdots$                                                                            &   $\ddots$       & $\vdots$                                                                               \\
    &  & $\alpha_n^\lor+2\sum_{j=r}^{n-1}\alpha^\lor_j+\sum_{j=t}^{r-1}\alpha^\lor_j,$     & $\cdots$    & $\alpha_n^\lor+2\sum_{j=r}^{n-1}\alpha^\lor_j+\sum_{j=1}^{r-1}\alpha^\lor_j,$          \\
$\alpha_n^\lor+\sum_{j=r}^{n-1} \alpha_j^\lor, $    & $\cdots$ & $\alpha_n^\lor+2\sum_{j=m}^{n-1} \alpha_j^\lor+\sum_{j=r}^{m-1} \alpha_j^\lor, $    & $\cdots$ & $\alpha_n^\lor+2\sum_{j=\ell+1}^{n-1} \alpha_j^\lor+\sum_{j=r}^{\ell} \alpha_j^\lor, $ \\
$\vdots$                                            &          & $\vdots$                                                                            &     $\ddots$     & $\vdots$                                                                               \\
    &          &                                                                                     &          &                           $\alpha_n^\lor+2\sum_{j=1}^{n-1}\alpha^\lor_j,$                                                          \\
$\alpha_n^\lor+\sum_{j=1}^{n-1} \alpha_j^\lor, $    & $\cdots$ & $\alpha_n^\lor+2\sum_{j=m}^{n-1} \alpha_j^\lor+\sum_{j=1}^{m-1} \alpha_j^\lor, $    & $\cdots$ & $\alpha_n^\lor+2\sum_{j=\ell+1}^{n-1} \alpha_j^\lor+\sum_{j=1}^{\ell} \alpha_j^\lor.$  
\end{tabular}
\end{center}
\smallskip
\noindent Here \eqref{tab:upper1} corresponds to the second row of \eqref{eq:Bnw0} (see \S \ref{ssec:An}). For rows of coroots below \eqref{tab:upper1}, the $i$th odd (resp. even) row corresponds to the former (resp. latter) parenthesis in the $i$th big parenthesis (reading from the right).

We retain the terminology in \S \ref{ssec:An}. Denote coroots
\begin{align*}
    &\alpha_n^\lor+2\sum_{j=r}^{n-1}\alpha_j^\lor+\sum_{j=t}^{r-1} \alpha_j^\lor \quad\,\, \textrm{ for } 1\le t\le r\le \ell, \textrm{ and } \\
    &\alpha_n^\lor+2\sum_{j=m}^{n-1} \alpha_j^\lor+\sum_{j=r}^{m-1} \alpha_j^\lor \quad \textrm{ for } 1\le r\le \ell<m\le n
\end{align*}
by $(2,r,t)$ and $(1,m,r)$ respectively. Note that neither the sum nor the difference of two coroots of the same type is a coroot in $\Phi_{w_0}^{\lor}$. Therefore, \ref{itm:b'} has no effects on the order of two coroots of the same type, so by \ref{itm:a'} additional to \eqref{eq:R1} we have
\begin{align}
    &\textrm{ for } r\ge r', (2,r,t)\longleftrightarrow (2,r',t') \textrm{ if and only if } r>r' \textrm{ and } t'>t,  \tag{R2}\label{eq:R2}\\
    &\textrm{ for } r\ge r', (1,m,r)\longleftrightarrow (1,m',r') \textrm{ if and only if } r>r' \textrm{ and } m'>m.\tag{R3}\label{eq:R3}
\end{align}
Arguing similarly as in \S \ref{ssec:An}, we have \eqref{eq:R2} implies \eqref{eq:comb-} and \eqref{eq:comb+} for $\lambda=2$. We also remark that since \ref{itm:c'} is not applicable, the coroot  $(2,r,r)$ always corresponds to the $(\ell-r+1)$th $s_n$ in a reduced expression of $w_0$.

Observe that $(2,r,t)$ can be written as the sum
\begin{align*}
    (1,m',r')+(r'',t'') 
\end{align*}
if and only if either $r'=r$ and $t''=t$ or $r'=t$ and $t''=r$. We claim the operation \ref{itm:b'} can only be applied to triples $(2,r,t), (1,m,r), (m-1,t)$ for $t<r$ due to the occurrence of $s_n$. Indeed, if $r=t=r'$, then $(2,r,r)$ corresponds to $s_n$; if $r>t=r'$, then by \eqref{eq:R2} the coroot $(2,t,t)$, which corresponds to $s_n$, must come after $(2,r,t)$, while the original position of $(1,m,t)$ is behind $(2,t,t)$.

Consequently, together with operations \ref{itm:a'} we have
\begin{align}
    &(2,r,t)\longleftrightarrow (r',t') \textrm{ if and only if } t'<r, \tag{B1}\label{eq:B1}\\
    &(1,m,r)\longleftrightarrow (r',t') \textrm{ if and only if } t'<r,\quad \tag{B2}\label{eq:B2}\\
    &(2,r,t)\longleftrightarrow (1,m',r') \textrm{ if and only if }  r< r' \textrm{ or } t<r=r' . \tag{B3}\label{eq:B3}
\end{align}
Inequalities \eqref{eq:comb-} and \eqref{eq:comb+} for $\lambda=1$ follow from \eqref{eq:R1}, \eqref{eq:R3}, and \eqref{eq:B2} by an argument similar to that in \S \ref{ssec:An}. To see this, flip the above list of coroots of type $(1,m,r)$ along the diagonal and attach it below \eqref{tab:upper1}, so one obtains a table of coroots (with $\lambda$-value equal to $1$) of size $2(n-\ell)$ by $\ell$. Then rules \eqref{eq:R1}, \eqref{eq:R3}, and \eqref{eq:B2} combined are essentially the same as the rule \eqref{eq:R1} for $G$ of type $A_{2n-\ell-1}$ with node $\ell$. 

 We prove the holomorphy of $M'_w(\chi_s)$ by induction on both $n$ and length of $w$. Only holomorphy of $M'_w(\chi_s)$ for $\chi^2=1$ requires a proof. Suppose $\ell=1$ and $n\ge 2$. Then the reduced expression \eqref{eq:Bnw0} of $w_0$ is unique. Note that \ref{itm:I} only holds for $\mathrm{Id}$ and $s_{n-1}s_{n-2}\cdots s_1$. As type $A$ is justified,  we may assume $w=s_{m-1}\cdots s_{n-1}s_ns_{n-1}\cdots s_1$ for some $2\le m\le n$. We need to show \begin{align*}
    M'_{s_{m}\cdots s_{n-1}s_n\cdots s_2s_1}(1_s)f^{(s)}\big|_{s+s_k+1=2n-m-1} \textrm{ is left } \iota_{\alpha_{m-1}}(\SL_2)\textrm{-invariant}.
\end{align*}
When $m<n$, ince $M'_{s_{m+1}\cdots s_ns_{n-1}\cdots s_1}f^{(s)}$ is left $\iota_{\alpha_{m-1}}(\SL_2)$-invariant and the coroot $(1,m+1,1)$ has $h$-value $2n-m-1$, this follows from \ref{itm:3}. For $m=n$, note that $f^{(s)}$ is left $\iota_{\alpha_n}(\SL_2)$-invariant. Since the coroot $(n-1,1)$ has $h$-value $n-1$, by \ref{itm:3} $ M'_{s_{n-1}\cdots s_2s_1}(1_s)f^{(s)}\big|_{s+s_k+1=n-1}$ is both left $s_{n-1}\iota_{\alpha_n}(\SL_2)$-invariant and left $\iota_{\alpha_n}(\SL_2)$-invariant. As
\begin{align*}
    s_n\alpha_{n-1}=(\alpha_n+\alpha_{n-1})+\alpha_n=s_{n-1}\alpha_n+\alpha_n,
\end{align*}
we have $s_n\iota_{\alpha_{n-1}}(\SL_2)$ is contained in the group generated by $s_{n-1}\iota_{\alpha_n}(\SL_2)$ and $\iota_{\alpha_n}(\SL_2)$, and thus
\begin{align*}
    M'_{s_ns_{n-1}\cdots s_2s_1}(1_s)f^{(s)}\big|_{s+s_k+1=n-1} \textrm{ is left } \iota_{\alpha_{n-1}}(\SL_2)\textrm{-invariant}.
\end{align*}

Now consider general $n>\ell\ge 2$. We will use rules \eqref{eq:R1}-\eqref{eq:R3} and \eqref{eq:B1}-\eqref{eq:B3} without further mention below. As type A is justified, by the induction hypothesis it suffices to justify the assertion for $w$ such that  \begin{align*}\Phi^\lor_{w}\cap \{(r,1), (1,m,1), (2,r',1) :\ell\le r\le n-1, \ell<m\le n, 1\le r' \le \ell\}
\end{align*}
is nonempty and $\Phi^\lor_w$ does not only consist of coroots of type $(r,t)$.  Write $w=s_\alpha w_1$ and let $\beta^\lor$ be the coroot in $\Phi^\lor_{w}-\Phi^\lor_{w_1}$.

\begin{Lem}\label{lem:ind}
There exists a reduced expression $w=s_\alpha w_1$ such that $\beta^\lor$ is either $(r,1), (1,m,1),$ or $(2,r',1)$. 
\end{Lem}

\begin{proof}
Observe that the coroot $(1,m,1)$ is not a part of any triple of coroots on which an operation \ref{itm:b'} can be applied. Therefore, if $(1,m,1)\in \Phi_{w}^\lor$ for some $m$, then by performing a series of \ref{itm:a'}, one can take $\beta^\lor$ to be $(1,m',1)\in \Phi_w^\lor$ with $m'$ minimal. 

Thus we assume no coroots $(1,m,1)$ are contained in $\Phi_w^\lor$. Consider first the case that $(r,1)\in \Phi_w^\lor$ for some $r$. Choose $r$ to be maximal. Let $c_1$ (resp. $c_2$) be the number of coroots of the form $(2,r',1)$ (resp. $(1,r+1,r')$) contained in $\Phi_w^\lor$. Note that coroots $(2,r',1), (1,r+1,r'), (r,1)$ form a triple on which the operation \ref{itm:b'} can be applied. Therefore, if $c_1\le c_2$ (resp. $c_1>c_2$), then $\beta^\lor$ can be taken as $(r,1)$ (resp. $(2,\ell-c_2+1,1)$. The case $(2,r',1)\in \Phi_w^\lor$ for some $r'$ can be argued similarly.

\end{proof}

Suppose $\beta^\lor=(2,r,1)$. Note that \ref{itm:I} fails for $w_1$ unless 
\begin{align}\label{eq:exceptionB}
    \ell=r \textrm{ is odd and } (2,\ell-c,1+c)\in \Phi_w^\lor \textrm{ for all } 0\le c\le (\ell-1)/2.
\end{align}
 We need to show
\begin{align}\label{eq:Bchi}
    M'_{w_1}(\chi_s)f^{(s)}\big|_{2(s+s_k+1)=2(n-1)+1-r} \textrm{ is left } \iota_{\alpha}(\SL_2)\textrm{-invariant}
\end{align}
except when \eqref{eq:exceptionB} occurs. By \ref{itm:1} and \eqref{eq:R2} we can assume coroots in $\Phi_w^\lor$ with $\lambda$-value $2$ have $h$-value at most $2(n-1)+2-r$.

\textbf{Case $r=1$:} We have $\alpha=\alpha_n$ and we can write $w_1=w''s_{n-1}s_nw'$ where $s_n$ and $s_{n-1}$ correspond to coroots $(2,2,2)$ and $(2,2,1)$ respectively, and $\Phi_{w_1}^\lor-\Phi_{s_{n-1}s_nw'}^\lor$ consists of coroots of the form $(1,m,2)$. Since coroots $(2,2,2)$ and $(2,2,1)$ have $(h,\lambda)$-value $(2(n-2)+1,2)$ and $(2(n-1),2)$ respectively, we have 
\begin{align*}
    M'_{s_{n-1}s_nw'}(\chi_s)f^{(s)}\big|_{2(s+s_k+1)=2(n-1)} \textrm{ is left } \iota_{\alpha_n}(\SL_2)\textrm{-invariant}
\end{align*}
by \ref{itm:3} and \ref{itm:4}. As coroots in $\Phi_{w_1}^\lor-\Phi_{s_{n-1}s_nw'}^\lor$ correspond to reflections $s_{\alpha_i}$ where $i<n-1$, to prove \eqref{eq:Bchi} by \ref{itm:1} it suffices to show $$\frac{a_{s_{n-1}s_nw'}(\chi_s)}{a_{w_1}(\chi_s)}M_{w''}(\chi_s^{s_{n-1}s_nw'})$$ is holomorphic at $s+s_k+1=n-1$. Observe that the coroot $(1,n-1,2)$ has $h$-value $n$, and the coroot $(1,n,2)$ must come before $(1,n-1,2)$. Thus the assertion follows from the proof of type $A$ with node $1$.

\textbf{Case $r$ is even, $\ell>r$:} Since we are to check holomorphy of $M'_w(\chi_s)$ at $2(s+s_k+1)=2(n-1)+1-r,$ which is an odd integer, by \ref{itm:1} and \ref{itm:5} we may add to or remove from $\Phi_w^\lor$ coroots with $\lambda$-value $1$ as long as rules \eqref{eq:R1}-\eqref{eq:R3} and \eqref{eq:B1}-\eqref{eq:B3} are obeyed. In particular, we can take $w$ such that $\alpha=\alpha_{n-r+1}$ and $w_1=s_{n-r}w'$ where $s_{n-r}$ corresponds to the coroot $(2,r+1,1)$. Then we have
\begin{align*}
    w^{\prime -1}\alpha_{n-r+1}&=w_1^{-1}\alpha_{n-r+1}-w^{\prime-1}\alpha_{n-r}\\
    &=\left(2\sum_{j=r}^{n}\alpha_j+\sum_{j=1}^{r-1}\alpha_j\right)-\left(2\sum_{j=r+1}^{n}\alpha_j+\sum_{j=1}^{r}\alpha_j\right)=\alpha_{r}.
\end{align*}
As the coroot $(2,r+1,1)$ has $h$-value $2(n-1)+1-r$ and $f^{(s)}$ is left $\iota_{\alpha_{r}}(\SL_2)$-invariant, we deduce \eqref{eq:Bchi} from \ref{itm:3}.

\textbf{Case $r=\ell$ is even:}  Let $0\le c\le \ell/2-1$ be the largest integer such that $(2,\ell-c,1+c)\in \Phi_w^\lor$. As in the previous case, we can assume $w$ can be rewritten as $w=s_{n-\ell+1+2c}s_{n-\ell+2+2c}w'$, where $s_{n-\ell+2+2c}$ and $s_{n-\ell+1+2c}$ correspond to coroots $(2,\ell-c,c+2)$ and $(2,\ell-c,c+1)$ respectively. We claim
\begin{align*}
     M'_{s_{n-\ell+2+2c}w'}(\chi_s)f^{(s)}\big|_{2(s+s_k+1)=2(n-1)+1-\ell} \textrm{ is left } \iota_{\alpha_{n-\ell+1+2c}}(\SL_2)\textrm{-invariant}.
\end{align*}
We have
\begin{align*}
    w^{\prime-1}\alpha_{n-\ell+1+2c}&=(s_{n-\ell+2+2c}w')^{-1}\alpha_{n-\ell+1+2c}-(1+\delta_{c,\frac{\ell}{2}-1}) w^{\prime-1}\alpha_{n-\ell+2+2c}\\
    &=\bigg(2\sum_{j=\ell-c}^{n}\alpha_j+\sum_{j=c+1}^{\ell-c-1}\alpha_j\bigg)-\bigg(2\sum_{j=\ell-c}^{n}\alpha_j+\sum_{j=c+2}^{\ell-c-1}\alpha_j\bigg)=\alpha_{c+1}.
\end{align*}
Since $(2,\ell-c,c+2)$ has $h$-value $2(n-1)+1-\ell$ and $f^{(s)}$ is left $\iota_{\alpha_{c+1}}(\SL_2)$-invariant, our claim follows from \ref{itm:3}.

\textbf{Case $r$ is odd, $\ell> r>1$:} Say $\alpha=\alpha_{n-r+1-i}$ where $n-\ell\ge i\ge 0$ is the number of coroots $(m,1)$ not in $\Phi_w^\lor$. By our assumption on coroots, we have $(2,t,t)\not\in \Phi_w^\lor$ for $2t<r+1$ and thus $(1,m,t)\not\in \Phi_w^\lor$ for $t< (r+1)/2$. Write $w_1=w''s_{n-r-i}w'$ where $s_{n-r-i}$ corresponds to the coroot $(2,r+1,1)$, and $\Phi_{w_1}^\lor-\Phi_{s_{n-r-i}w'}^\lor$ consists of coroots of the form $(1,m,t)$ where $t\ge (r+1)/2 $. Choose a reduced expression such that reflections corresponding to coroots in $\Phi_{w_1}^\lor-\Phi_{s_{n-r-i}w'}^\lor$ fixes $\alpha_{n-r+1-i}$. Then we have
\begin{align*}
    w^{\prime -1}\alpha_{n-r+1-i}&=
    (s_{n-r-i}w')^{-1}\alpha_{n-r+1-i}-w^{\prime-1}\alpha_{n-r-i}\\
    &=w_1^{-1}\alpha_{n-r+1-i}-w^{\prime-1}\alpha_{n-r-i}\\
    &=\left(2\sum_{j=r}^{n}\alpha_j+\sum_{j=1}^{r-1}\alpha_j\right)-\left(2\sum_{j=r+1}^{n}\alpha_j+\sum_{j=1}^{r}\alpha_j\right)=\alpha_{r}.
\end{align*}
As the coroot $(2,r+1,1)$ has $h$-value $2(n-1)+1-r$ and $f^{(s)}$ is left $\iota_{\alpha_{r}}(\SL_2)$-invariant, we conclude by \ref{itm:3}
\begin{align*}
    M'_{s_{n-r-i}w'}(\chi_s)f^{(s)}\big|_{2(s+s_k+1)=2(n-1)+1-r} \textrm{ is left } \iota_{\alpha_{n-r+1-i}}(\SL_2)\textrm{-invariant}.
\end{align*} 
Therefore \eqref{eq:Bchi} follows if $$\frac{a_{s_{n-r-i}w'}(\chi_s)}{a_{w_1}(\chi_s)}M_{w''}(\chi_s^{s_{n-r-i}w'})$$ is holomorphic at $2(s+s_k+1)=2(n-1)+1-r$. By the proof of type $A$, this is true if $\chi\neq 1$ or $\chi= 1$ and the number of coroots of the form $(1,m,t)$ in $\Phi_w^\lor$ with $h$-value $\frac{2(n-1)+1-r}{2}+1$ is at most that with $h$-value $\frac{2(n-1)+1-r}{2}$.
    
We assume this is not the case, so $\chi=1$, $2(n-\ell)\ge  \frac{2(n-1)+1-r}{2}+1,$ and $(1,n-c,\frac{r+1}{2}+c)\in\Phi_w^\lor$ for all $c$. Since  $(2,\frac{r+1}{2},\frac{r-1}{2})\not\in\Phi_w^\lor$ by assumption and $(1,n,\frac{r+1}{2})\in\Phi_w^\lor$, we have $(n-1,\frac{r-1}{2})\not\in\Phi_w^\lor$ but $(n-1,\frac{r+1}{2})\in \Phi_w^\lor$. Rewrite $w=s_{n-1}v's_{n-\ell+\frac{(r-1)}{2}}v$ such that $s_{n-1}$ and $s_{n-\ell+\frac{(r-1)}{2}}$ correspond to coroots $(1,n,\frac{r+1}{2})$ and $(n-1,\frac{r+1}{2})$ respectively, and coroots of type $(m,t)$ in $\Phi_w^\lor$ all lie in $\Phi_{v}^\lor$. We need to show
\begin{align}\label{eq:Bodd}
    M'_{v's_{n-\ell+\frac{(r-1)}{2}}v}(1_s)f^{(s)}\big|_{s+s_k+1=\frac{2(n-1)+1-r}{2}} \textrm{ is left } \iota_{\alpha_{n-1}}(\SL_2)\textrm{-invariant.}
\end{align}

\begin{Lem}\label{lem:Bind}
We have $$M'_{v's_{n-\ell+\frac{(r-1)}{2}}v}(1_s)f^{(s)}\big|_{s+s_k+1=\frac{2(n-1)+1-r}{2}}$$ is left $v'v\iota_{\alpha_n}(\SL_2)$-invariant. 
\end{Lem} 

\begin{proof}
Since the coroot $(n-1,\frac{r+1}{2})$ has $h$-value $\frac{2(n-1)+1-r}{2}$ and $f^{(s)}$ is left $\iota_{\alpha_n}(\SL_2)$-invariant, by \ref{itm:3} to justify the lemma it suffices to show
$$\frac{a_{s_{n-\ell+\frac{(r-1)}{2}}v}(1_s)}{a_{v's_{n-\ell+\frac{(r-1)}{2}}v}(1_s)}M_{v'}(1_s^{s_{n-\ell+\frac{(r-1)}{2}}v})$$ is holomorphic at $s+s_k+1=\frac{2(n-1)+1-r}{2}$.
 By \eqref{eq:R2} and \eqref{eq:R3}, the number of coroots in $\Phi_{v's_{n-\ell+\frac{(r-1)}{2}}v}^\lor-\Phi_{s_{n-\ell+\frac{(r-1)}{2}}v}^\lor$ with $(h,\lambda)$-value $(2(n-1)+2-r,2)$ (resp. $(\frac{2(n-1)+1-r}{2}+1,1)$) is at most that with $(h,\lambda)$-value $(2(n-1)+1-r,2)$ (resp. $(\frac{2(n-1)+1-r}{2},1)$). Therefore, the holomorphy follows by the induction on length (and proofs).
\end{proof}

Clearly, $M'_{v's_{n-\ell+\frac{(r-1)}{2}}v}(1_s)f^{(s)}$ is left $v's_{n-\ell+\frac{(r-1)}{2}}v\iota_{\alpha_n}(\SL_2)$-invariant. We have
\begin{align*}
    v\alpha_n=\alpha_n+\alpha_{n-1}+\ldots+\alpha_{n-\ell+\frac{(r
    +1)}{2}},
\end{align*}
and
\begin{align*}
    v'v\alpha_n=s_n\alpha_n,\quad  
    v's_{n-\ell+\frac{(r-1)}{2}}v\alpha_n=s_n(\alpha_{n-1}+\alpha_n).
\end{align*}
Since $s_n\alpha_n+s_n(\alpha_{n-1}+\alpha_n)=s_ns_n\alpha_{n-1}=\alpha_{n-1},$ we deduce \eqref{eq:Bodd}.

\textbf{Case $\ell=r$ is odd:}  Since \eqref{eq:exceptionB} does not hold, there exists smallest $0\le c\le (\ell-1)/2-1$ such that $(2,\ell-c-1,2+c)\not\in\Phi_w^\lor$. We can rewrite $w=s_{n-\ell+1+2c-i}s_{n-\ell+2+2c-i}w'$ for some $i\ge 0$, where $s_{n-\ell+2+2c-i}$ and $s_{n-\ell+1+2c-i}$ correspond to coroots $(2,\ell-c,c+2)$ and $(2,\ell-c,c+1)$ respectively. A similar argument as in the case $r=\ell$ even justifies the holomorphy. 

\bigskip

For the rest of the proof we may assume $\chi=1$. Consider $\beta^\lor=(1,m,1)$ and thus $\alpha=\alpha_{m-1}$. By \eqref{eq:B2} $\Phi_w^\lor$ contains all coroots of type $(r,t)$. Therefore, we can choose a reduced expression of $w$ with no operations \ref{itm:b'} carried out.  Note that \ref{itm:I} fails for $w_1$ in this case. Suppose $\ell<m<n$.  We can write $w_1=s_{m}w'$ where $s_{m}$ corresponds to the coroot $(1,m+1,1)$. We have
\begin{align*}
    &w^{\prime-1}\alpha_{m-1}=w_1^{-1}\alpha_{m-1}-w^{\prime-1}\alpha_{m}\\
    &=\left(2\sum_{j=m}^{n}\alpha_j+\sum_{j=1}^{m-1}\alpha_j\right)-\left(2\sum_{j=m+1}^{n}\alpha_j+\sum_{j=1}^{m}\alpha_j\right)=\alpha_m.
\end{align*}
Since $f^{(s)}$ is left $\iota_{\alpha_m}(\SL_2)$-invariant and the coroot $(1,m+1,1)$ has $h$-value $2n-m-1$, by \ref{itm:3}
\begin{align*}
    M'_{w_1}(1_s)f^{(s)}\big|_{s+s_k+1=2n-m-1} \textrm{ is left } \iota_{\alpha_{m-1}}(\SL_2)\textrm{-invariant}.
\end{align*}

Suppose $m=n$. We break down the discussion into two cases.

\textbf{Case $n<2\ell$:} Let $0\le c< n-\ell$ be the largest integer such that $(1,n-c,1+c)\in \Phi_w^\lor$. Rewrite $w=s_{n-1-2c}s_{n-2-2c}w'$ where $s_{n-2-2c}$ and $s_{n-1-2c}$ correspond to coroots $(1,n-c,2+c)$ and $(1,n-c,1+c)$ respectively. We claim
\begin{align*}
     M'_{s_{n-2-2c}w'}(1_s)f^{(s)}\big|_{s+s_k+1=n-1} \textrm{ is left } \iota_{\alpha_{n-1-2c}}(\SL_2)\textrm{-invariant}
\end{align*}
We have
\begin{align*}
    w^{\prime-1}\alpha_{n-1-2c}&=(s_{n-2-2c}w')^{-1}\alpha_{n-1-2c}-w^{\prime-1}\alpha_{n-2-2c}\\
    &=\left(2\sum_{j=n-c}^{n}\alpha_j+\sum_{j=1+c}^{n-c-1}\alpha_j\right)-\left(2\sum_{j=n-c}^{n}\alpha_j+\sum_{j=2+c}^{n-c-1}\alpha_j\right)=\alpha_{c+1}
\end{align*}
Since the coroot $(1,n-c,2+c)$ has $h$-value $n-1$ and $f^{(s)}$ is left $\iota_{\alpha_{c+1}}(\SL_2)$-invariant, our claim follows from \ref{itm:3}.

\textbf{Case $n\ge 2\ell$:}  Suppose there exists $0\le c<\ell$ such that $(1,n-c-1,2+c)\not\in\Phi_w^\lor$. Choose $c$ to be minimal, and rewrite  $w=s_{n-1-2c}s_{n-2-2c}w'$ where $s_{n-2-2c}$ and $s_{n-1-2c}$ correspond to coroots $(1,n-c,2+c)$ and $(1,n-c,1+c)$ respectively. The holomorphy of $M'_w(1_s)$ can be justified similarly as the previous case. Therefore, we assume $(1,n-c,1+c)\in\Phi_w^\lor$ for all $c$. We claim
\begin{align*}
     M'_{w_1}(1_s)f^{(s)}\big|_{s+s_k+1=n-1} \textrm{ is left } \iota_{\alpha_{n-1}}(\SL_2)\textrm{-invariant}.
\end{align*}
\begin{Lem}
$M'_{w_1}(1_s)f^{(s)}\big|_{s+s_k+1=n-1}$ is left $w''w'\iota_{\alpha_n}(\SL_2)$-invariant. 
\end{Lem}

\begin{proof}
 Write $w_1=w''s_{n-\ell}w'$ where $s_{n-\ell}$ corresponds to the coroot $(n-1,1)$ and $\Phi_{s_{n-\ell}w'}^\lor$ consists of all coroots of type $(r,t)$. Since the coroot $(n-1,1)$ has $h$-value $n-1$, by \ref{itm:3}
\begin{align*}
    M'_{s_{n-\ell}w'}(1_s)f^{(s)}\big|_{s+s_k+1=n-1} \textrm{ is left } w'\iota_{\alpha_n}(\SL_2)\textrm{-invariant}.
\end{align*}
Write $w''=v'v$ where $\Phi_{vs_{n-\ell}w'}^\lor-\Phi_{s_{n-\ell}w'}^\lor$ consists of all coroots of type $(2,r,t)$. Since $(2,1,1)$ is the only coroot of type $(2,r,t)$ with $h$-value $2(n-1)+1$ and coroots $(2,2,2)$ and $(2,2,1)$ are in $\Phi_{vs_{n-\ell}w'}^\lor-\Phi_{s_{n-\ell}w'}^\lor$, it follows from the proof of the case $\beta^\lor=(2,1,1)$ that $$\frac{a_{s_{n-\ell}w'}(1_s)}{a_{vs_{n-\ell}w'}(1_s)}M_{v}(1_s^{s_{n-\ell}w'})$$ is holomorphic at $s+s_k+1=n-1$. Therefore, the lemma follows once we show $$\frac{a_{vs_{n-\ell}w'}(1_s)}{a_{w_1}(1_s)}M_{v'}(1_s^{vs_{n-\ell}w'})$$ is holomorphic at $s+s_k+1=n-1$. Note that $\Phi_{w_1}^\lor-\Phi_{vs_{n-\ell}w'}^\lor$ consists of coroots of type $(1,r,t)$. By \eqref{eq:R3} the number of coroots in $\Phi_{w_1}^\lor-\Phi_{vs_{n-\ell}w'}^\lor$ with $h$-value $n$ is at most that with $h$-value $n-1$. Therefore, the holomorphy follows from the proof of type $A$.
\end{proof}
Clearly, $M'_{w_1}(1_s)f^{(s)}$ is left $w_1\iota_{\alpha_n}(\SL_2)$-invariant.
Since $(1,n-c,1+c)\in \Phi_w^\lor$ for all $c$, we have $(1,n,r)\in \Phi_w^\lor$ for all $r$, and thus a direct computation gives  \begin{align*}
    w_1\alpha_n=\alpha_n+\alpha_{n-1},\quad w''w'\alpha_n=-\alpha_n.
\end{align*}
As $(\alpha_n+\alpha_{n-1})+(-\alpha_n)=\alpha_{n-1},$ our claim follows from the above lemma.

\bigskip

Consider now $\beta^\lor=(r,1)$. By \ref{itm:1}, \eqref{eq:R1}, \eqref{eq:R3} and \eqref{eq:B2}, we may assume every coroot with $\lambda$-value $1$ in $\Phi_w^\lor$ has $h$-value at most $r$. 

\textbf{Case $n-1\ge r>\ell$:} Note that \ref{itm:I} fails for $w_1$ in this case. If $(1,r+1,\ell)\not\in \Phi_w^\lor$, then $\alpha=\alpha_{r-\ell+1}$ and we can write $w_1=s_{r-\ell}w'$ where $s_{r-\ell}$ corresponds to the coroot $(r-1,1)$. We have
\begin{align*}
    w^{\prime -1}\alpha_{r-\ell+1}=w_1^{-1}\alpha_{r-\ell+1}-w^{\prime -1}\alpha_{r-\ell}=\sum_{j=1}^{r} \alpha_j-\sum_{j=1}^{r-1} \alpha_j=\alpha_r.
\end{align*}
Since the coroot $(r-1,1)$ has $h$-value $r-1$ and $f^{(s)}$ is left $\iota_{\alpha_{r}}(\SL_2)$-invariant, by \ref{itm:3}
\begin{align*}
    M'_{w_1}(1_s)f^{(s)}\big|_{s+s_k+1=r-1} \textrm{ is left } \iota_{\alpha_{r}}(\SL_2)\textrm{-invariant}.
\end{align*}
Therefore, we assume $(1,r+1,\ell)\in \Phi_w^\lor$. Let $\ell>c\ge 1$ be the number of coroots $(1,r+1,t)$ in $\Phi_w^\lor$. Suppose the coroot $(1,r+1,\ell+1-c)$ has $h$-value less than $r$. If $(1,r,\ell+1-c)\not\in \Phi_w^\lor,$ then we can rewrite $w=s'w'$ such that $s'$ corresponds to the coroot $(1,r+1,\ell+1-c)$. The holomorphy of $M'_{w}(1_s)=M'_{s'w'}(1_s)$ follows by the induction hypothesis and \ref{itm:1}. If $(1,r,\ell+1-c)\in \Phi_w^\lor,$ let $m\le r$ be the smallest integer such that $(1,m,\ell+1-c)\in \Phi_w^\lor$. Rewrite $w=s_is_{i+1}w'$ ($i<n-1$) such that $s_{i+1}$ and $s_i$ correspond to coroots $(1,m+1,\ell+1-c)$ and $(1,m,\ell+1-c)$ respectively. The holomorphy of $M'_w(1_s)$ in this case follows by a similar argument as in the case $\beta^\lor=(1,m,1)$.

Consequently, we assume the coroot $(1,r+1,\ell+1-c)$ has $h$-value $r$, i.e., $2n-r-\ell-1+c=r$.
If $c\ge 2$ and $(1,r,\ell+2-c)\not\in \Phi_w^\lor$, then we rewrite $w=s_{i-1}s_{i}w'$ ($i<n$) where $s_i$ and $s_{i-1}$ corresponds to coroots $(1,r+1,\ell+2-c)$ and $(1,r+1,\ell+1-c)$ respectively. We have
\begin{align*}
    w^{\prime-1}\alpha_{i-1}&=(s_iw_1)^{-1}\alpha_{i-1}-w^{\prime-1}\alpha_i\\
    &=\left(2\sum_{j=r+1}^{n} \alpha_j+\sum_{j=\ell+1-c}^{r} \alpha_j\right)-\left(2\sum_{j=r+1}^{n} \alpha_j+\sum_{j=\ell+2-c}^{r} \alpha_j\right)=\alpha_{\ell+1-c}.
\end{align*}
Since the coroot $(1,r+1,\ell+1-c)$ has $h$-value $r-1$ and $f^{(s)}$ is left $\iota_{\alpha_{\ell+1-c}}(\SL_2)$-invariant, we have
\begin{align*}
    M'_{s_iw'}(1_s)f^{(s)}\big|_{s+s_k+1=r-1} \textrm{ is left } \iota_{\alpha_{i-1}}(\SL_2)\textrm{-invariant}.
\end{align*}
Suppose either $c=1$ or $c\ge 2$ and $(1,r,\ell+2-c)\in \Phi_w^\lor$. Let $s_i$ be a reflection such that $\ell(s_iw)=1+\ell(w)$ and the coroot corresponding to $s_i$ is $(1,r,\ell+1-c)$. Since the coroot $(1,r,\ell+1-c)$ has $h$-value $r+1$, by \ref{itm:1} and \ref{itm:5} $M'_{w}(1_s)$ is holomorphic at $s+s_k+1=r-1$ iff $M'_{s_iw}(1_s)$ is. To see $M'_{s_iw}(1_s)$ is holomorphic at $s+s_k+1=r-1$, rewrite $s_iw=s_{r-\ell+c+1}s_{r-\ell+c}w'$ where $s_{r-\ell+c}$ and $s_{r-\ell+c+1}$ correspond to coroots $(r,1)$ and $(r-1,1)$ respectively. The holomorphy follows from a similar argument as the case $(1,r+1,\ell)\not\in \Phi_w^\lor.$
 
\textbf{Case $r=\ell, 2(n-\ell)<\ell$:} Note that in this case \ref{itm:I} fails for $w_1$.  For the case that $(1,\ell+1+t,2n-2\ell-t)\in\Phi_w^\lor$ for some $t$, the holomorphy in this case follows from an argument analogous to that of the previous case when $(1,r+1,\ell+1-c)\in \Phi_w^\lor$ with $c\ge 2$ but $(1,r,\ell+2-c)\not\in\Phi_w^\lor$. Therefore, suppose $(1,\ell+1+t,2n-2\ell-t)\not\in\Phi_w^\lor$ for any $t$. If there is no coroot of type $(1,m,t)$ in $\Phi_w^\lor$, then we can write $w=w''w'$, where $\Phi_{w'}^\lor$ consists of all coroots of type $(m,t)$ in $\Phi_w^\lor$. The holomorphy then follows from the proof of the case $\beta^\lor=(2,m,1)$. If $\Phi_w^\lor$ contains some coroot of type $(1,m,t),$ we can rewrite $w=s_iw'$ where $s_i$ corresponds to a coroot with $(h,\lambda)$-value $(c,1)$ where $c<\ell$. Then by the induction hypothesis and \ref{itm:1} $M'_w(1_s)=M'_{s_iw'}(1_s)$ is holomorphic at $s+s_k+1=\ell-1$.

\textbf{Case $r=\ell, 2(n-\ell)\ge \ell$:} By the same argument as the previous case, it suffices to consider when $n-\ell<\ell$ and  $(1,2(n-\ell)+1+t,\ell-t)\in\Phi_w^\lor$ for all $t$. Since $2(n-\ell)\ge \ell$, for \ref{itm:I} to fail for $w_1$, there exists smallest $n-\ell> c\ge 0$ such that $(\ell+c+1,2+c)\not\in \Phi_w^\lor$ and $(\ell+c,1+c)\in \Phi_w^\lor$. Since $n<2\ell$, we can rewrite $w=s_is_{i+1}w'$ $(i<n-1)$ where $s_{i+1}$ and $s_{i}$ correspond to coroots $(\ell+c,2+c)$ and $(\ell+c,1+c)$. We have
\begin{align*}
    w^{\prime -1}\alpha_i=(s_{i+1}w')^{-1}\alpha_i- w^{\prime -1}\alpha_{i+1}=\sum_{j=1+c}^{\ell+c}\alpha_j-\sum_{j=2+c}^{\ell+c}\alpha_j=\alpha_{c+1}.
\end{align*}
Since the coroot $(\ell+c,2+c)$ has $h$-value $\ell-1$ and $f^{(s)}$ is left $\iota_{\alpha_{c+1}}(\SL_2)$-invariant, by \ref{itm:3}
\begin{align*}
    M'_{s_{i+1}w'}(1_s)f^{(s)}\big|_{s+s_k+1=\ell-1} \textrm{ is left } \iota_{\alpha_{i}}(\SL_2)\textrm{-invariant}.
\end{align*}
This completes the proof for the case $\ell<n.$ 

\bigskip

For $\ell=n\ge 2$, a reduced expression of $w_0$ is 
\begin{align}\label{eq:Bnw0n}
    (s_n)\cdots (s_{n-r+1}\cdots s_{n})\cdots (s_1\cdots s_n);
\end{align}
corresponding coroots $\tilde{\alpha}_{(i)}^\lor$ are
\begin{center}
\begin{tabular}{>{\tiny}c>{\tiny}c>{\tiny}c>{\tiny}c>{\tiny}c}
$\alpha_n^\lor+2\sum_{j=n}^{n-1}\alpha^\lor_j,$  & $\cdots$ & $\alpha_n^\lor+2\sum_{j=n}^{n-1}\alpha^\lor_j+\sum_{j=t}^{n-1}\alpha^\lor_j,$ &  $\cdots$ & $\alpha_n^\lor+2\sum_{j=n}^{n-1}\alpha^\lor_j+\sum_{j=1}^{n-1}\alpha^\lor_j,$   \\
                                 &  $\ddots$        & $\vdots$   &     $\ddots$     & $\vdots$                                                                               \\
   &    & $\alpha_n^\lor+2\sum_{j=r}^{n-1}\alpha^\lor_j+\sum_{j=t}^{r-1}\alpha^\lor_j,$     & $\cdots$   & $\alpha_n^\lor+2\sum_{j=r}^{n-1}\alpha^\lor_j+\sum_{j=1}^{r-1}\alpha^\lor_j,$\\
           &          &               &     $\ddots$     & $\vdots$                                                                               \\
   &          &                                                                                     &          &                      $\alpha_n^\lor+\sum_{j=2}^{n-1}\alpha^\lor_j+\alpha_1^\lor.$                                    
\end{tabular}
\end{center}
\smallskip

The inequalities \eqref{eq:comb-} and \eqref{eq:comb+} and the holomorphy of $M'_{w}(\chi_s)$ (especially $\chi=1$) follow from the same (and actually simpler) argument as the previous case for type $(2,r,t)$ coroots. 
\qed

\subsection{Type $C_n$ ($n\ge 2$)}\label{ssec:Cn}

For $1\le \ell<n$, coroots $\tilde{\alpha}_{(i)}^\lor$ corresponding to the reduced expression \eqref{eq:Bnw0} are

\begin{align*}
    \begin{tabular}{ccccc}
$\alpha_\ell^\lor,$                & $\cdots$ & $\sum_{j=t}^\ell \alpha_{j}^\lor,$ & $\cdots$ & $\sum_{j=1}^\ell \alpha_{j}^\lor,$ \\
                $\vdots$                       &           & $\vdots$                          &           &     $\vdots$                                   \\
$\sum_{j=\ell}^r \alpha_{j}^\lor,$ & $\cdots$ & $\sum_{j=t}^r \alpha_{j}^\lor,$    & $\cdots$ & $\sum_{j=1}^{r} \alpha_{j}^\lor,$  \\
                    $\vdots$                   &           & $\vdots$                          &           &                 $\vdots$                       \\
$\sum_{j=\ell}^{n-1} \alpha_{j}^\lor,$ & $\cdots$ & $\sum_{j=t}^{n-1} \alpha_{j}^\lor,$    & $\cdots$ & $\sum_{j=1}^{n-1} \alpha_{j}^\lor,$   
\end{tabular}
\end{align*}
\begin{center}
\begin{tabular}{>{\tiny}c>{\tiny}c>{\tiny}c>{\tiny}c>{\tiny}c>{\tiny}c}
$\sum_{j=\ell}^n \alpha^\lor_j$&$2\sum_{j=\ell}^{n}\alpha^\lor_j+\alpha_{\ell-1}^\lor,$  & $\cdots$ & $2\sum_{j=\ell}^{n}\alpha^\lor_j+\sum_{j=t}^{\ell-1}\alpha^\lor_j,$ & $\cdots$ & $2\sum_{j=\ell}^{n}\alpha^\lor_j+\sum_{j=1}^{\ell-1}\alpha^\lor_j,$    \\
&$2\alpha_n^\lor+\sum_{j=\ell}^{n-1} \alpha_j^\lor, $ & $\cdots$ & $2\sum_{j=m}^{n} \alpha_j^\lor+\sum_{j=\ell}^{m-1} \alpha_j^\lor, $ & $\cdots$ & $2\sum_{j=\ell+1}^{n} \alpha_j^\lor+ \alpha_\ell^\lor, $               \\
$\vdots$ & $\vdots$                                            &    $\ddots$      & $\vdots$                                                                            &    $\ddots$      & $\vdots$                                                                               \\
$\sum_{j=r}^n \alpha^\lor_j$&    &  & $2\sum_{j=r}^{n}\alpha^\lor_j+\sum_{j=t}^{r-1}\alpha^\lor_j,$     & $\cdots$ & $2\sum_{j=r}^{n}\alpha^\lor_j+\sum_{j=1}^{r-1}\alpha^\lor_j,$          \\
&$2\alpha_n^\lor+\sum_{j=r}^{n-1} \alpha_j^\lor, $    & $\cdots$ & $2\sum_{j=m}^{n} \alpha_j^\lor+\sum_{j=r}^{m-1} \alpha_j^\lor, $    & $\cdots$ & $2\sum_{j=\ell+1}^{n} \alpha_j^\lor+\sum_{j=r}^{\ell} \alpha_j^\lor, $ \\
$\vdots$ & $\vdots$                                            &          & $\vdots$                                                                            &    $\ddots$      & $\vdots$                                                                               \\
$\sum_{j=2}^n \alpha^\lor_j$&   &          &                                                                                     &          &      $2\sum_{j=2}^{n}\alpha^\lor_j+\alpha_1^\lor,$                                                                                    \\
&$2\alpha_n^\lor+\sum_{j=2}^{n-1} \alpha_j^\lor, $    & $\cdots$ & $2\sum_{j=m}^{n} \alpha_j^\lor+\sum_{j=2}^{m-1} \alpha_j^\lor, $    & $\cdots$ & $2\sum_{j=\ell+1}^{n} \alpha_j^\lor+\sum_{j=2}^{\ell} \alpha_j^\lor, $   \\
$\sum_{j=1}^n \alpha^\lor_j$ & & & & & \\
&$2\alpha_n^\lor+\sum_{j=1}^{n-1} \alpha_j^\lor, $    & $\cdots$ & $2\sum_{j=m}^{n} \alpha_j^\lor+\sum_{j=1}^{m-1} \alpha_j^\lor, $    & $\cdots$ & $2\sum_{j=\ell+1}^{n} \alpha_j^\lor+\sum_{j=1}^{\ell} \alpha_j^\lor. $ 
\end{tabular}
\end{center}
\medskip

By duality, the proof for type $B_n$ carries over with minor modification. The major difference is that $i$th $s_n$ in any reduced expression of $w_0$ corresponds to the coroot $\sum_{j=\ell-i+1}^n \alpha_j^\lor$, which has $\lambda$-value $1$ instead of $2$. In an analogous setup as Lemma \ref{lem:ind}, if the last coroot is $\sum_{j=1}^n \alpha_j^\lor$, one modifies the argument of the case $\beta^\lor=(1,n,1)$ in \S\ref{ssec:Bn}. Some explanation is given in the case $\ell=n$ below. If the last coroot is $2\alpha_{n}^\lor+\sum_{j=1}^{n-1} \alpha_j^\lor$ one applies the argument of the case $\beta^\lor=(1,r,1)$ with $r<n$ in \S\ref{ssec:Bn}. We leave the details of the other cases to the reader.

\medskip 

For $\ell=n$, coroots $\tilde{\alpha}_{(i)}^\lor$ corresponding to the reduced expression \eqref{eq:Bnw0n} are
\begin{center}
\begin{tabular}{>{\tiny}c>{\tiny}c>{\tiny}c>{\tiny}c>{\tiny}c>{\tiny}c}
$\alpha_n^\lor,$ & $2\alpha^\lor_n +\alpha_{n-1}^\lor,$  & $\cdots$ & $2\sum_{j=n}^{n}\alpha^\lor_j+\sum_{j=t}^{n-1}\alpha^\lor_j,$ & $\cdots$ & $2\sum_{j=n}^{n}\alpha^\lor_j+\sum_{j=1}^{n-1}\alpha^\lor_j,$   \\
$\vdots$ &                                          &    $\ddots$      & $\vdots$                                                               &  $\ddots$         & $\vdots$                                                                               \\
$\sum_{j=r}^n \alpha_j^\lor,$ & & &  $2\sum_{j=r}^{n}\alpha^\lor_j+\sum_{j=t}^{r-1}\alpha^\lor_j,$  &   $\cdots$& $2\sum_{j=r}^{n}\alpha^\lor_j+\sum_{j=1}^{r-1}\alpha^\lor_j,$   \\
$\vdots$ &                                          &          &                                                                       &   $\ddots$     & $\vdots$                                                                               \\
$\sum_{j=2}^n \alpha_j^\lor,$ &   & & & &$2\sum_{j=2}^n\alpha^\lor_j +\alpha_{1}^\lor,$\\
 $\sum_{j=1}^n \alpha_j^\lor$.&     &          &                                                                                     &          &                   
 \end{tabular}
\end{center}
\smallskip

As in the case $\ell<n$, except for the case where the last coroot is $\sum_{j=1}^n \alpha_j^\lor$, a similar inductive proof as in \S\ref{ssec:Bn} for type $(2,r,t)$ coroots justifies the holomorphy in this case. Therefore, we only explain how to prove the holomorphy of $M'_{w_0}(1_s)$ assuming $M'_w(\chi_s)$ is holomorphic for any $w\neq w_0$ and  $\chi$.

Rewrite $w_0=s_ns_{n-1}s_{n-2}s_nw'$ by switching the order of coroots $\sum_{j=2}^n \alpha_j^\lor$ and $2\sum_{j=3}^n \alpha_j^\lor+\sum_{j=1}^2 \alpha_j^\lor.$ We need to show
\begin{align*}
    M'_{s_{n-1}s_{n-2}s_nw'}(1_s)f^{(s)}\big|_{s+s_k+1=n-1} \textrm{ is left } \iota_{\alpha_n}(\SL_2)\textrm{-invariant}.
\end{align*}
We have $w'\alpha_{n-1}=s_{n-1}\alpha_{n-2}$ and
\begin{align*}
    s_{n}s_{n-1}\alpha_{n-2}+s_{n-1}\alpha_{n-2}=\alpha_{n}+2\alpha_{n}+2\alpha_{n-2}=s_{n-2}s_{n-1}\alpha_n.
\end{align*}
Since $\sum_{j=2}^n \alpha_j^\lor$ has $h$-value $n-1$ and $f^{(s)}$ is left $\iota_{\alpha_{n-1}}(\SL_2)$-invariant, by \ref{itm:3}
\begin{align*}
    M'_{s_nw'}(1_s)f^{(s)}\big|_{s+s_k+1=n-1} \textrm{ is left } s_{n-2}s_{n-1}\iota_{\alpha_n}(\SL_2)\textrm{-invariant}.
\end{align*}
Therefore, to justify the holomorphy it suffices to show
\begin{align*}
    \frac{a_{s_{n}w'}(1_s)}{a_{s_{n-1}s_{n-2}s_{n}w'}(1_s)}M_{s_{n-1}s_{n-2}}(1_s^{s_nw'})
\end{align*}
is holomorphic at $s+s_k+1=n-1$. Since coroots corresponding to $s_{n-2}$ and $s_{n-1}$ have $(h,\lambda)$-value $(2(n-1),2)$ and $(2(n-1)+1,2)$ respectively, the holomorphy follows again from \ref{itm:3}.
\qed

\subsection{Type $D_n$ ($n\ge 4$)}\label{ssec:Dn}

For $1\le \ell<n-1$, a reduced expression of $w_0$ is
\begin{align}\label{eq:Dnw0}
\begin{split}
    &(s_\ell\cdots s_{n-3}s_{n-2}s_{n-(\ell+1 \mod 2)})\cdot \\
    &\big((s_{\ell-1}\cdots s_{n-2})(s_{n-(\ell \mod 2)})\big)\cdot\\
    &\vdots\\
    &\big((s_4\cdots s_{n-\ell+3})(s_{n -\ell+4}\cdots s_{n-3}s_{n-2}s_{n-1})\big)\cdot\\
    &\big((s_3\cdots s_{n-\ell+2})(s_{n -\ell+3}\cdots s_{n-3} s_{n-2}s_n)\big)\cdot\\
    &\big((s_2\cdots s_{n-\ell+1})(s_{n-\ell+2}\cdots s_{n-3}s_{n-2}s_{n-1})\big)\cdot\\
    &\big((s_1\cdots s_{n-\ell})(s_{n-\ell+1}\cdots s_{n-3} s_{n-2}s_n)\big)\cdot\\
    &(s_{n-\ell}\cdots s_{n-1})\cdots (s_2\cdots s_{\ell}s_{\ell+1})(s_1\cdots s_{\ell-1}s_\ell);
\end{split}
\end{align}
corresponding coroots $\tilde{\alpha}_{(i)}^\lor$ are
\begin{align}
    \begin{tabular}{ccccc}
$\alpha_\ell^\lor,$                & $\cdots$ & $\sum_{j=t}^\ell \alpha_{j}^\lor,$ & $\cdots$ & $\sum_{j=1}^\ell \alpha_{j}^\lor,$ \\
                $\vdots$            &           & $\vdots$  &           &     $\vdots$                                   \\
$\sum_{j=\ell}^r \alpha_{j}^\lor,$ & $\cdots$ & $\sum_{j=t}^r \alpha_{j}^\lor,$    & $\cdots$ & $\sum_{j=1}^{r} \alpha_{j}^\lor,$  \\
    $\vdots$    &           & $\vdots$    &           &                 $\vdots$                       \\
$\sum_{j=\ell}^{n-1} \alpha_{j}^\lor,$ & $\cdots$ & $\sum_{j=t}^{n-1} \alpha_{j}^\lor,$    & $\cdots$ & $\sum_{j=1}^{n-1} \alpha_{j}^\lor,$  
\end{tabular}\tag{$\ast'$}\label{tab:upper1'}
\end{align}
\begin{center}
\begin{tabular}{>{\tiny}c>{\tiny}c>{\tiny}c>{\tiny}c>{\tiny}c}
$\sum_{j=\ell}^{n-2} \alpha_j^\lor+\sum_{j=\ell-1}^n\alpha_{j}^\lor,$ & $\cdots$ & $\sum_{j=\ell}^{n-2} \alpha_j^\lor+\sum_{j=t}^{n}\alpha_{j}^\lor,$ & $\cdots$ & $\sum_{j=\ell}^{n-2} \alpha_j^\lor+\sum_{j=1}^{n}\alpha_{j}^\lor,$     \\
  $\alpha_n^\lor+ \sum_{j=\ell}^{n-2}\alpha_{j}^\lor,$  & $\cdots$ & $\sum_{j=m}^n \alpha_j^\lor+ \sum_{j=\ell}^{n-2}\alpha_{j}^\lor,$   & $\cdots$ & $\sum_{j=\ell+1}^n \alpha_j^\lor+ \sum_{j=\ell}^{n-2}\alpha_{j}^\lor,$\\
$\vdots$     &   $\ddots$       & $\vdots$     & $\ddots$         & $\vdots$              \\
     &     & $\sum_{j=r}^{n-2} \alpha_j^\lor+\sum_{j=t}^{n}\alpha_{j}^\lor,$    & $\cdots$ & $\sum_{j=r}^{n-2} \alpha_j^\lor+\sum_{j=1}^{n}\alpha_{j}^\lor,$           \\
$\alpha_n^\lor+ \sum_{j=r}^{n-2}\alpha_{j}^\lor,$  & $\cdots$ & $\sum_{j=m}^n \alpha_j^\lor+ \sum_{j=r}^{n-2}\alpha_{j}^\lor,$   & $\cdots$ & $\sum_{j=\ell+1}^n \alpha_j^\lor+ \sum_{j=r}^{n-2}\alpha_{j}^\lor,$ \\
$\vdots$  &          & $\vdots$&   $\ddots$       & $\vdots$       \\
      &  &                  &          &   $\sum_{j=2}^{n-2} \alpha_j^\lor+\sum_{j=1}^n\alpha_{j}^\lor,$               \\
$\alpha_n^\lor+ \sum_{j=2}^{n-2}\alpha_{j}^\lor,$  & $\cdots$ & $\sum_{j=m}^n \alpha_j^\lor+ \sum_{j=2}^{n-2}\alpha_{j}^\lor,$   & $\cdots$ & $\sum_{j=\ell+1}^n \alpha_j^\lor+ \sum_{j=2}^{n-2}\alpha_{j}^\lor,$               \\
  &          &               &          &                           \\
$\alpha_n^\lor+ \sum_{j=1}^{n-2}\alpha_{j}^\lor,$  & $\cdots$ & $\sum_{j=m}^n \alpha_j^\lor+ \sum_{j=1}^{n-2}\alpha_{j}^\lor,$   & $\cdots$ & $\sum_{j=\ell+1}^n \alpha_j^\lor+ \sum_{j=1}^{n-2}\alpha_{j}^\lor,$                  
\end{tabular}
\end{center}
\medskip
\smallskip
\noindent Here \eqref{tab:upper1'} corresponds to the last row of \eqref{eq:Dnw0}. For rows of coroots below \eqref{tab:upper1'}, the $i$th odd (resp. even) rows correspond to the former (resp. latter) parenthesis in the $(i+1)$th big parenthesis (counting from bottom to top and from right to left).

We retain the terminology in \S \ref{ssec:An} and follow the idea in \S \ref{ssec:Bn}. Denote coroots
\begin{align*}
    \sum_{j=r}^{n-2} \alpha_j^\lor+\sum_{j=t}^{n}\alpha_{j}^\lor=\alpha_n^\lor+\alpha_{n-1}^\lor+2\sum_{j=r}^{n-2}\alpha_j^\lor+\sum_{j=t}^{r-1} \alpha_j^\lor& \quad \textrm{ for } 1\le t< r\le \ell, \textrm{ and } \\
    \sum_{j=m}^n\alpha_j^\lor+\sum_{j=r}^{n-2} \alpha_j^\lor& \quad \textrm{ for } 1\le r\le \ell<m\le n
\end{align*}
by $(2,r,t)$ and $(1,m,r)$ respectively. As the sum of two coroots of the same type is not a coroot, so we still have rules \eqref{eq:R1}, \eqref{eq:R2}, \eqref{eq:R3}. The rule \eqref{eq:R2} implies \eqref{eq:comb-} and \eqref{eq:comb+} for $\lambda=2$ similarly.

\begin{Rem}\label{rem:Rprime1}
Rules of coroots are stable under the symmetry of the Dynkin diagram $D_n$. For instance, \eqref{eq:R1} implies the  coroot $(1,n,t)$ must come after the coroot $(n-2,t)$ for any $1\le t\le \ell$.
\end{Rem}

Observe that the coroot $(2,r,t)$ can be written as a sum of coroots
\begin{align*}
    (1,m',r')+(r'',t'') 
\end{align*}
if and only if either $r'=r$ and $t''=t$ or $r'=t$ and $t''=r$. In both cases, $m'=r''+1$.

\begin{Lem}
The operation \ref{itm:b'} cannot be applied to the triple of coroots $(2,r,t), (1,m,t), (m-1,r)$ if (and only if) $m<n$.
\end{Lem}

\begin{proof}
By \eqref{eq:R2} and \eqref{eq:R3}, it suffices to show for $r=t+1, m=n-1$. By \eqref{eq:R1}, an operation \ref{itm:b'} that reverses the order of coroots $(n-1,t+1)$ and $(2,t+1,t)$ needs to be carried out first. However, in this case the coroot $(1,n,t+1)$ comes before $(2,t+1,t)$, but comes after $(n-2,t+1)$ by Remark \ref{rem:Rprime1}.
\end{proof}

\noindent Consequently, together with operations \ref{itm:a'}, we have
\begin{align}
    &(2,r,t)\longleftrightarrow (r',t') \textrm{ if and only if } t'<r \textrm{ or } r'=n-1, \tag{D1}\label{eq:D1}\\
    &(1,m,r)\longleftrightarrow (r',t') \textrm{ if and only if } t'<r \textrm{ or } m=r'+1=n, \tag{D2}\label{eq:D2}\\
    &(2,r,t)\longleftrightarrow (1,m',r') \textrm{ if and only if }  r<r' \textrm{ or } t<r=r' \textrm{ or }m'=n. \tag{D3}\label{eq:D3}
\end{align}
Inequalities \eqref{eq:comb-} and \eqref{eq:comb+} for $\lambda=1$ follow similarly from \eqref{eq:R1}, \eqref{eq:R3}, and \eqref{eq:D2}. 

As mentioned in Remark \ref{rem:Rprime1}, the relaxation of rules, compared to rules of type $B$, arises from the symmetry of $D_n$. More precisely, the Dynkin diagram of $D_{n}$ folds into that of $B_{n-1}$. This is the reason why we have named our coroots the same way as in \S\ref{ssec:Bn}. We explain how one can modify the inductive proof of holomorphy of $M'_w(\chi_s)$ of type $B$ to that of type $D$.

By the induction hypothesis we have an analogue of Lemma \ref{lem:ind} that asserts that $w$ can be written as $s_{\alpha}w_1$ where the corresponding coroot $\beta^\lor$ of $s_{\alpha}$ is either $(r,1), (1,m,1),$ or $(2,r',1)$. For the case $\beta^\lor=(r,1)$ for $r\neq n-1$, $\beta^\lor=(1,m,1)$ for $m\neq n,$ or $\beta^\lor=(2,r',1)$ for $1<r'$, a similar proof for the same type of roots considered in \S \ref{ssec:Bn} proves the holomorphy. Therefore, by the symmetry of $D_n$, it suffices to consider the case where $\beta^\lor=(n-1,1)$ and $\chi=1$. We can assume any coroot in $\Phi_w^\lor$ with $\lambda$-value $1$ has $h$-value at most $n-1$. Note that \ref{itm:I} fails if $\ell\ge 2$ or $\ell=1$ and $(1,n,1)\not\in \Phi_w^\lor$. If $(1,n,1)\not\in \Phi_w^\lor$, then the holomorphy follows from the argument in \S\ref{ssec:Bn} of the case $\beta^\lor=(n-1,1)$. Therefore, we assume $(1,n,1)\in \Phi_w^\lor$. Since \ref{itm:I} holds for $w_1$ if $\ell=1$, it suffices to consider $\ell\ge 2$. In this case, a similar argument as in \S\ref{ssec:Bn} of the case $\beta^\lor=(1,n,1)$ justifies the holomorphy.

\medskip

For $\ell=n, n-1$, by symmetry it suffices to deal with either case. Let $\ell=n$. A reduced expression of $w_0$ is 
\begin{align*}
    &(s_{n-(\ell+1 \mod 2)})\cdot \\
    &(s_{n-2}s_{n-(\ell \mod 2)})\cdot \\
    &\vdots\\
    &(s_4\cdots s_{n-3}s_{n-2}s_{n-1})\cdot\\
    &(s_3\cdots s_{n-3} s_{n-2}s_n)\cdot\\
    &(s_2\cdots s_{n-3}s_{n-2}s_{n-1})\cdot\\
    &(s_1\cdots s_{n-3} s_{n-2}s_n);
\end{align*}
corresponding coroots $\tilde{\alpha}_{(i)}^\lor$ are
\begin{center}
    \begin{tabular}{>{\tiny}c>{\tiny}c>{\tiny}c>{\tiny}c>{\tiny}c>{\tiny}c}
$\alpha_n^\lor$  &   $\cdots$                                                                   & $\cdots$ & $\alpha_n^\lor+\sum_{j=t}^{n-2} \alpha_j^\lor,$                                                         & $\cdots$ & $\alpha_n^\lor+\sum_{j=1}^{n-2} \alpha_j^\lor$                                                          \\
&$\sum_{j=n-2}^n\alpha_{j}^\lor,$                                                    & $\cdots$ & $\sum_{j=t}^{n} \alpha_j^\lor,$                                       & $\cdots$ & $\sum_{j=1}^{n} \alpha_j^\lor,$                \\                                                         &   &    $\ddots$      & $\vdots$ &      $\ddots$    & $\vdots$                          \\
& & & $\sum_{j=r}^{n-2}\alpha_{j}^\lor+\sum_{j=t}^{n}\alpha_j^\lor,$ & $\cdots$ & $\sum_{j=r}^{n-2}\alpha_{j}^\lor+\sum_{j=1}^{n}\alpha_j^\lor,$ \\
&    & &     &     $\ddots$   & $\vdots$                   \\
 & &          &                  &          &                 $\sum_{j=2}^{n-2}\alpha_{j}^\lor+\sum_{j=1}^n\alpha_{j}^\lor.$                                    
\end{tabular}
\end{center}
\smallskip

As mentioned above, the proof of holomorphy for coroots of type $(2,r,t)$ with $r>t$ in \S \ref{ssec:Bn} can be modified to justify the holomorphy. We leave the details to the reader. \qed

\subsection{Modification for archimedean cases}\label{ssec:modarch}

Let $F$ be an archimedean local field.

\begin{Lem}\label{lem:vanish}
Let $G=\SL_2$ and $w_0=\begin{psmatrix}
 & 1\\
-1 & 
\end{psmatrix}$. Let $\partial, \overline{\partial}$ be partial derivations in the first variable of $F^2$. Let $J\subseteq \mathcal{S}(F^2-\{0\})$ be an $\SL_2(F)$-submodule closed under $\partial,\overline{\partial}$ and $$M_{w_0}(1_s)f_{1_s}\bigg|_{s=-1}=0$$ for all $f\in J$. Then for $\chi=\mu^\ell,$
$$M_{w_0}(\chi_s)f_{\chi_s}\bigg|_{s=-1-\frac{2j+|\ell|}{[F:\rr]}}=0$$ for all $f\in J,\ell\in H$ and $j\in \zz_{\ge 0}$.
\end{Lem}

\begin{proof}
Since $J$ is an $\SL_2(F)$-module, it suffices to check the lemma at $\mathrm{Id}=\begin{psmatrix}
    1 & \\
     & 1
\end{psmatrix}.$ Let $f\in \mathcal{S}(F^2-\{0\})$. Using the identity
\begin{align*}
   \begin{psmatrix}
        & 1\\
      -1 & 
    \end{psmatrix}\begin{psmatrix}
       1 & x\\
       & 1
    \end{psmatrix}= \begin{psmatrix}
    -x^{-1}& 1\\
     & -x
    \end{psmatrix}\begin{psmatrix}
    1& \\
    x^{-1} & 1
    \end{psmatrix} \textrm{ for } x\in F^\times,
\end{align*} 
we can write for $\mathrm{Re}(s)>0$
\begin{align*}
    M_{w_0}(\chi_s)f_{\chi_s}(\mathrm{Id})&=\int_{F} f_{\chi_s}\left(\begin{psmatrix}
        & 1\\
      -1 & 
    \end{psmatrix}\begin{psmatrix}
       1 & x\\
       & 1
    \end{psmatrix}\right) dx\\
    &=\chi(-1)\int_{F^\times}\bar{\chi}(x)|x|^{-s-1} f_{\chi_s}\begin{psmatrix}
    1& \\
    x^{-1} & 1
    \end{psmatrix}dx\\
    &=\chi(-1)\int_{F^\times}\chi(x)|x|^{s-1} f_{\chi_s}\begin{psmatrix}
    1& \\
    x & 1
    \end{psmatrix}dx\\
    &=\chi(-1)\int_{F^\times}
    \int_{F^\times}\chi(ax)|x|^{s-1}|a|^{s+1} f(ax,a)d^\times a dx
\end{align*}
Note that the integral is absolutely convergent and thus we can change variables $x\mapsto a^{-1}x$ to obtain
\begin{align*}
     M_{w_0}(\chi_s)f_{\chi_s}(\mathrm{Id})=\chi(-1)\int_{F^\times}
    \int_{F^\times}\chi(x)|x|^{s-1}|a| f(x,a)d^\times adx.
\end{align*}

Consider first $F=\rr$. For $\chi=1$ applying integration by parts twice in the $x$ variable, one has
\begin{align*}
    &\int_{F^\times}
    \int_{F^\times} |x|^{s-1}|a| f(x,a)d^\times adx=\frac{1}{s(s+1)}\int_{F^\times}
    \int_{F^\times} |x|^{s+1}|a| \partial^2f(x,a)d^\times adx.
\end{align*}
 Therefore, for any $f\in J$ 
\begin{align*}
    \frac{1}{(s+1)^2}\int_{F}
    \int_{F^\times} |x|^{s+1}|a| \partial^2f(x,a)d^\times adx
\end{align*}
is holomorphic at $s=-1$.
Similarly, by integration by parts the assertion  $$M_{w_0}(1_s)f_{1_s}(\mathrm{Id})\bigg|_{s=-1-2j} =0 \quad \textrm{ and } \quad M_{w_0}(\mu_s)f_{\mu_s}(\mathrm{Id})\bigg|_{s=-2-2j}=0$$
for all $j\in \zz_{\ge 0}$ is equivalent to
\begin{align*}
    \frac{1}{(s+1)^2}\int_{F}
    \int_{F^\times} |x|^{s+1}|a| \partial^{j+2}f(x,a)d^\times adx.
\end{align*}
is holomorphic at $s=-1$ for all $j\in \zz_{\ge 0}$. Since $J$ is closed under taking $\partial$, the assertion follows.

Now consider $F=\cc$. Applying integration by parts, we have
\begin{align*}
   \int_{F^\times}
    \int_{F^\times}|x|^{s-1}|a| f(x,a)d^\times adx= \frac{1}{s^2(s+1)^2}
    \int_{F^\times}\int_{F^\times}|x|^{s+1}|a|\partial^2\overline{\partial}^2f(x,a)d^\times adx.
\end{align*}
Therefore, for any $f\in J$ 
\begin{align*}
    \frac{1}{(s+1)^3} \int_{F^\times}\int_{F^\times} |x|^{s+1}|a|^{2} \partial^2\overline{\partial}^2 f(x,a)d^\times adx 
\end{align*}
is holomorphic at $s=-1$. 
The assertion of the lemma is equivalent to 
\begin{align*}
    \frac{1}{(s+1)^3}\int_{F^\times}\int_{F^\times} |x|^{s+1}|a|^{2} \partial^{2+j_1}\overline{\partial}^{2+j_2} f(x,a)d^\times adx
\end{align*}
is holomorphic at $s=-1$ for any $j_1,j_2\in \zz_{\ge 0}$, which follows from the assumption that $J$ is closed under taking $\partial,\overline{\partial}$.
\end{proof}

\begin{Prop}
Suppose $G$ is either classical or $G_2$. Then $M'_w(\chi_s)$ is holomorphic for all $w\in W/W_M$ and character $\chi$.
\end{Prop}

\begin{proof}

Observe that combinatoric inequalities  \eqref{eq:comb-} and \eqref{eq:comb+} depend only on $G$ and $P$ but not the field $F$. In particular, these inequalities hold if $G$ is either classical or $G_2$.

We proceed the proof by induction on the length of $w$. Suppose $M'_w(\chi_s)$ is holomorphic for all $\ell(w)\le n$ and character $\chi$. Let $w$ be of length $n+1$ and write $w=s_\alpha w'$ for some $\alpha$ and $w'\in W/W_M$ such that $\ell(w')=n$. As argued in Lemma \ref{lem:holo}, it suffices to treat the case when $m_{w'}(h,\lambda)<m_{w'}(h-1,\lambda).$ Let
\begin{align*}
    \widetilde{J}:=\{ \tilde{f}: \textrm{ there is } f\in \mathcal{S}(X_P^\circ(F)) \textrm{ such that } \tilde{f}_{(\chi_s)^{w'}}=M_{w'}'(\chi_s)(f_{\chi_s}) \textrm{ for all } \chi\}.
\end{align*}
By the induction hypothesis and the Mellin inversion $\widetilde{J}$ is a $G(F)$-submodule of $\mathcal{S}(X_{w'Pw^{\prime -1}}^\circ(F))$. Let 
\begin{align*}
    J:=\left\{ \iota_\alpha^\ast\tilde{f}:  \tilde{f}\in \widetilde{J}\right\}.
\end{align*}
Then $J$ is an $\SL_2(F)$-submodule of $\mathcal{S}(F^2-\{0\}),$ and by the Leibniz integral rule it is closed under taking $\partial,\overline{\partial}$. By the argument in the nonarchimedean case and an analogue of Lemma \ref{lem:holo} for the trivial character, we have $M'_{ w}(\chi_s)$ is holomorphic at $\lambda(s+s_k+1)=h-1$ provided $\chi^\lambda=1$. Thus $J$ satisfies the assumption in Lemma \ref{lem:vanish}, so $M'_{w}(\chi_s)$ is holomorphic for all $\chi$.
\end{proof}

\begin{Rem}
It follows from the argument above that if for $G$ of type $E$ or $F$ Theorem \ref{thm:poles} can be proved analogously in the nonarchimedean case, then it also holds for the archimedean case.
\end{Rem}

\section{Locality and Weyl algebras}\label{sec:Weyl}

Throughout this section $F$ is an archimedean local field unless otherwise specified. We prove the Schwartz space $\mathcal{S}(X_P(F))$ is local and study the ring of differential operators on $X_P(F)$.

\subsection{Locality of Archimedean Schwartz spaces}\label{ssec:localarch}
 Since $G$ is split, we can choose $v_P$ and a basis of weight vectors $\{v_i\}$ of $V_P(\qq)$ such that $v_1=v_{P}$. Let $\{v_i^{\lor}\}$ be the dual basis in $V_P^\lor(\qq)$, so $v_1^\lor=v_{P^{\mathrm{op}}}^\ast.$ Define functions
\begin{align*}
    x_i(x)&:=\langle x, v_i^\lor \rangle, \quad x\in X_P^\circ(F),\\
    x_i^{\mathrm{op}}(x')&:=\langle v_i,x' \rangle, \quad x'\in X_{P^{\mathrm{op}}}^\circ(F).
\end{align*}
Note that $\bar{x}_i(x)=\langle \bar{x},v_i^\lor\rangle$ and $\{x_i,\bar{x}_i\}_{1\le i\le \dim V_P}$ generates the $\cc$-algebra $F[X_P]\otimes_\rr \cc.$ 

We will often view $p\in \cc[\mathrm{Res}_{F/\rr} X_P]=F[X_P]\otimes_\rr \cc$ as a linear operator via function multiplication
\begin{align}\label{eq:multp}
\begin{split}
    C^\infty(X_P^\circ(F))&\to C^\infty(X_P^\circ(F))\\
     f&\mapsto p\cdot f,
\end{split}
\end{align}
which we continue to denote as $p$.

\begin{Lem}\label{lem:mult:general}
    The operator \eqref{eq:multp} descends to a bounded linear endomorphism of $\mathcal{S}_{\Lambda}(X_P^\circ(F))$.
\end{Lem}

\begin{proof}
     Since $\{x_i,\bar{x}_i\}_{1\le i \le \dim V_P}$ generates $F[X_P]\otimes_\rr \cc$ and $V_P^\lor$ is an irreducible $G$-representation, by an analogue of Lemma \ref{lem:cc} it suffices to show for every $f\in \mathcal{S}_{\Lambda}(X_P^\circ(F))$ we have $x_1f\in \mathcal{S}_{\Lambda}(X_P^\circ(F)),$ and the operator $f\to x_1 f$ is continuous. Let $h:=x_1f.$ For a character $\chi,$ note that
\begin{align*}
    h_{\chi_s}(x)= \langle x,v_{P^{\mathrm{op}}}^\ast\rangle f_{\mu\chi_{s+\frac{1}{[F:\rr]}}}(x),\quad x\in X_P^\circ(F).
\end{align*}
By definition functions
\begin{align*}
    \frac{f_{\mu\chi_{s+\frac{1}{[F:\rr]}}}}{d\left(\mu\chi_{s+\frac{1}{[F:\rr]}}\right)}\quad  \textrm{ and }\quad \frac{d\left(\mu\chi_{s+\frac{1}{[F:\rr]}}\right)}{d(\chi_s)}
\end{align*}
are holomorphic, and so is $d(\chi_s)^{-1}h_{\chi_s}$. Furthermore, given real numbers $\sigma_1<\sigma_2$, a polynomial $Q(s)\in \cc[s]$ such that $Q(s)d(\chi_s)$ is holomorphic on $V_{\sigma_1,\sigma_2}$ for all $\chi\in\widehat{K}_{\GG_m},$ and a compact subset $\Omega\subset X_P^\circ(F),$ 
\begin{align*}
    |h|_{\sigma_1,\sigma_2,Q,\mathrm{Id},\Omega}\le \bigg(\sup_{x\in \Omega}|x_1(x)|\bigg )|f|_{\sigma_1+\frac{1}{[F:\rr]},\sigma_2+\frac{1}{[F:\rr]},Q',\mathrm{Id},\Omega},
\end{align*}
where $Q'(s):=Q(s-\frac{1}{[F:\rr]}).$  For general $D\in U(\mathfrak{m}^{\mathrm{ab}}\oplus \mathfrak{g}),$ $|h|_{\sigma_1,\sigma_2,Q,D,\Omega}$ is dominated by a finite sum of the form $\bigg(\sup_{x\in \Omega}|x_i(x)|\bigg)|f|_{A,B,Q',D',\Omega}$. This completes the proof.
\end{proof}

\begin{Prop}\label{prop:archlocal}
The space $\mathcal{S}(X_P(F))$ is stable under the multiplication by $F[X_P]\otimes_\rr \cc$.
\end{Prop}

\begin{proof}
By Proposition \ref{prop:holomorphy}, Lemma \ref{lem:mult:general} and the Iwasawa decomposition, it suffices to show for $f\in \mathcal{S}(X_P(F))$ we have
\begin{align}\label{eq:goalk}
    \frac{\mathcal{R}_{P|P^{\mathrm{op}}} (x_1f)_{\chi_s}^{\mathrm{op}}}{a_{w_0}(\chi_s)}(k)
\end{align}
is an entire function for all $k\in K$ and characters $\chi$. By definition
\begin{align*}
   & \mathcal{R}_{P|P^{\mathrm{op}}}(x_1f)_{\chi_s}^{\mathrm{op}}(k)\\
   &=\int_{M^{\mathrm{ab}}(F)} \delta_{P^{\mathrm{op}}}^{1/2}(m)\chi_{s}(\omega_P(m))\int_{N_P^{\mathrm{op}}(F)} \langle  v_Pm^{-1}uk,v_{P^{\mathrm{op}}}^\ast \rangle f(m^{-1}uk)dudm\\
   &=\int_{M^{\mathrm{ab}}(F)} \delta_{P^{\mathrm{op}}}^{1/2}(m)\mu\chi_{s+\frac{1}{[F:\rr]}}(\omega_P(m))\int_{N_P^{\mathrm{op}}(F)} \langle  v_Puk,v_{P^{\mathrm{op}}}^\ast \rangle f(m^{-1}uk)dudm. 
\end{align*}
In the case $k=\mathrm{Id}$, we have
$\langle  v_Puk,v_{P^{\mathrm{op}}}^\ast \rangle= \langle  v_P,v_{P^{\mathrm{op}}}^\ast \rangle=1$ and thus 
\begin{align*}
    \mathcal{R}_{P|P^{\mathrm{op}}}(x_1f)_{\chi_s}^{\mathrm{op}}(\mathrm{Id})=\mathcal{R}_{P|P^{\mathrm{op}}}(f)_{\mu\chi_{s+\frac{1}{[F:\rr]}}}^{\mathrm{op}}(\mathrm{Id}).
\end{align*}
 Since $f\in \mathcal{S}(X_P(F)),$ by definition 
\begin{align*}
    \frac{ \mathcal{R}_{P|P^{\mathrm{op}}}(f)_{\mu\chi_{s+\frac{1}{[F:\rr]}}}^{\mathrm{op}}(\mathrm{Id})}{a_{w_0}\left(\mu\chi_{s+\frac{1}{[F:\rr]}}\right)}\quad  \textrm{ and }\quad \frac{a_{w_0}\left(\mu\chi_{s+\frac{1}{[F:\rr]}}\right)}{a_{w_0}(\chi_s)}
\end{align*}
are holomorphic, so we have $\eqref{eq:goalk}$ is holomorphic for $k=\mathrm{Id}$.

To treat general $k\in K$, note that there is a nonzero entire function $z(s)$ depending only on $\chi$ such that the function
\begin{align*}
     z(s)\frac{\mathcal{R}_{P|P^{\mathrm{op}}} (x_1 f)_{\chi_s}^{\mathrm{op}}}{a_{w_0}(\chi_s)}(k)
\end{align*}
is entire for all $f$ and $k$. Assume first $f$ is $K$-finite. For each $s\in \cc$, let $V_s$ be the finite-dimensional complex $K$-representation generated by
\begin{align*}
    z(s)\frac{\mathcal{R}_{P|P^{\mathrm{op}}} (x_1 f)_{\chi_s}^{\mathrm{op}}}{a_{w_0}(\chi_s)}.
\end{align*}
Then $V_s$ is an analytic representation of  $\widetilde{K}_\cc$, the universal cover of the complexification of $K$. We denote the action by $k_c.v$ for $k_c\in \widetilde{K}_\cc$ and $v\in V_s$. Therefore, the function defined on $\cc\times \widetilde{K}_\cc$ given by
\begin{align*}
    J(s,k_c):=   \frac{1}{z(s)}k_c.z(s)\frac{\mathcal{R}_{P|P^{\mathrm{op}}} (x_1 f)_{\chi_s}^{\mathrm{op}}}{a_{w_0}(\chi_s)}(\mathrm{Id})
\end{align*}
is meromorphic. If $\Omega\subset \cc$ is the discrete set of zeros of $z(s)$, then all poles of $J$ lie in $\Omega\times \widetilde{K}_\cc$. As \eqref{eq:goalk} is entire for $k=\mathrm{Id},$ $J$ is holomorphic on $\cc\times \{\mathrm{Id}\}$ and thus poles of $J$ must be a closed analytic subvariety of $\cc\times \widetilde{K}_\cc$ of codimension at least $2$. It follows by Hartog's extension theorem \cite[Chapter 2 Theorem 5B]{Whitney} that $J$ is entire. Consequently, \eqref{eq:goalk} is entire.

Finally, consider the continuous map
\begin{align*}
    \mathcal{I}:\mathcal{S}(X_P(F))\times K\longrightarrow M(\cc,\delta)
\end{align*}
given by sending $(f,k)$ to the function \eqref{eq:goalk}, where $\delta:\Omega\to \zz_{>0}$ is given by $\delta(s_0):=\mathrm{ord}_{s=s_0} z(s)^{-1}$. Let $\mathcal{S}(X_P(F))_{\mathrm{fin}}$ be the space of $K$-finite functions in $\mathcal{S}(X_P(F))$. We have shown that the set $\mathcal{S}(X_P(F))_{\mathrm{fin}}\times K$ is contained in the closed subspace $\mathcal{I}^{-1}(H(\cc))$. Since $\mathcal{S}(X_P(F))_{\mathrm{fin}}$ is dense in  $\mathcal{S}(X_P(F))$, it follows that $\mathcal{I}^{-1}(H(\cc))=\mathcal{S}(X_P(F))\times K$. This justifies the assertion.
\end{proof}

In the rest of the section, we assume $G$ is either classical or $G_2$. Recall that by Lemma \ref{lem:germpoint} we have the following exact sequence of Fr\'echet spaces
\begin{align}\label{eq:ESexact}
    0\longrightarrow \mathcal{S}(X_P^\circ(F))\longrightarrow \mathcal{S}_{\mathrm{ES}}(X_P(F))\longrightarrow F[[X_P]]\otimes_\rr \cc\longrightarrow 0.
\end{align}
Here $F[[X_P]]\otimes_\rr \cc:=\widehat{F[X_P]}_{0}\otimes_\rr \cc$ admits a natural expression: Write 
$$F[X_P]\otimes_\rr \cc=\bigoplus_{n=0}^\infty V_{-n\omega_P}^\lor\otimes_\rr F$$
where $V_{-n\omega_P}^\lor$ is the dual of the complex irreducible representation of $G$ of highest weight $-n\omega_P$ equipped with the usual topology of finite-dimensional $\cc$-vector spaces. Then 
$F[[X_P]]\otimes_\rr \cc=\prod_{n=0}^\infty V_{-n\omega_P}^\lor\otimes_\rr F$ is equipped with the product topology.

\begin{Cor}
The Schwartz space $\mathcal{S}(X_P(F))$ is a continuous $\mathcal{S}_{\mathrm{ES}}(X_P(F))$-module. In particular,  $\mathcal{S}_{\mathrm{ES}}(X_P(F))$ is a closed subspace of $\mathcal{S}(X_P(F))$, and $\mathcal{S}(X_P(F))$ is local.
\end{Cor}

\begin{proof}
 Since $\mathcal{S}(X_P^\circ(F))$ is a topological algebra under function multiplication, by \eqref{eq:ESexact} and Theorem \ref{thm:exactarch} it suffices to check the assertion for germs at the origin. By Proposition \ref{prop:archlocal}, $A_{X_P(F)}=\mathcal{S}(X_P(F))/\mathcal{S}(X_P^\circ(F))$ is a continuous $V_{-n\omega_P}^\lor\otimes_\rr F$-module under function multiplication for all $n$. By the definition of the product topology, $A_{X_P(F)}$ is a continuous $F[[X_P]]\otimes_\rr \cc$-module.

By Corollary \ref{cor:trivcont} the trivial representation $\cc$ is a submodule of $A_{X_P(F)},$ and thus $F[[X_P]]\otimes_\rr \cc$ is a closed subspace of $A_{X_P}(F),$ and $\mathcal{S}_{\mathrm{ES}}(X_P(F))$ is a closed subspace of $\mathcal{S}(X_P(F))$. Since $\mathcal{S}(X_P^\circ(F))$ is local and $\mathcal{S}(X_P(F))$ is a continuous $\mathcal{S}_{\mathrm{ES}}(X_P(F))$-module,  $\mathcal{S}(X_P(F))$ is local.
\end{proof}

\subsection{Weyl algebra}\label{ssec:Weyl}

By Proposition \ref{prop:archlocal} we can view $x_i$ as a bounded linear endomorphism on $\mathcal{S}(X_P(F))$ by function multiplication. Furthermore, the operator
\begin{align*}
    \partial_i:= \mathcal{F}_{P^{\mathrm{op}}|P} \circ  x_{i}^{\mathrm{op}} \circ \mathcal{F}_{P|P^{\mathrm{op}}} 
\end{align*}
is a continuous linear endomorphism of $\mathcal{S}(X_P(F))$, which satisfies
\begin{align}\label{eq:partialf}
    \mathcal{F}_{P|P^{\mathrm{op}}} \circ \partial_i=x_{i}^{\mathrm{op}} \circ \mathcal{F}_{P|P^{\mathrm{op}}} \quad \textrm{ and } \quad \partial_i \circ \mathcal{F}_{P^{\mathrm{op}}|P}= \mathcal{F}_{P^{\mathrm{op}}|P} \circ  x_{i}^{\mathrm{op}}.
\end{align}
We define $\overline{\partial}_i:= \mathcal{F}_{P^{\mathrm{op}}|P} \circ  \bar{x}_{i}^{\mathrm{op}} \circ \mathcal{F}_{P|P^{\mathrm{op}}}$ and operators $\partial_i^{\mathrm{op}}$ and $\overline{\partial}_i^{\mathrm{op}}$ on $\mathcal{S}(X_{P^{\mathrm{op}}}(F))$ analogously. 

\begin{Def}
Let $F$ be an archimedean local field. Define \textbf{the Weyl algebra of $X_P(F)$} to be the $\cc$-subalgebra $W_{X_P(F)}$ of $\mathrm{End}_{\cc}(\mathcal{S}(X_P(F)))$ generated by $x_i,\partial_i,\overline{x}_i,\overline{\partial}_i$.
\end{Def}

We note that there is a canonical linear isomorphism $W_{X_P(F)}\to W_{X_{P^{\mathrm{op}}}(F)}$ induced by the Fourier transform. Furthemore, since $F[X_P]\otimes_\rr \cc$ is a $G(F)$-module, $W_{X_P(F)}$ admits a natural $G(F)$-module structure and thus we have a group $G(F)\ltimes W_{X_P(F)}$.

Recall that for a character $\chi,$ $\mu_{\Lambda}(\chi)=\prod_{i=1}^k \gamma(-s_i,\chi^{\lambda_i},\psi)$. 
\begin{Lem}\label{lem:shift}
Let $f\in \mathcal{S}(X_P(F))$. Then for any character $\chi$
\begin{align*}
    (\partial_{1}f)_{\chi_s}(\mathrm{Id})=\frac{\mu_{\Lambda}(\chi_s^{-1})}{\mu_{\Lambda}\left(\mu\chi^{-1}_{s-\frac{1}{[F:\rr]}}\right)}f_{\mu^{-1}\chi_{s-\frac{1}{[F:\rr]}}}(\mathrm{Id}).
\end{align*}
\end{Lem}

\begin{proof}
We have by Theorem \ref{Thm: Fourier formula}
 \begin{align*}
   (\partial_{1}f)_{\chi_s}(\mathrm{Id})&=
   (\mathcal{F}_{P^{\mathrm{op}}|P}\circ  x_1^{\mathrm{op}}\circ \mathcal{F}_{P|P^{\mathrm{op}}}(f))_{\chi_s}(\mathrm{Id})\\
    &=\mu_{\Lambda}(\chi_s^{-1})(\mathcal{R}_{P^{\mathrm{op}}|P} \circ x_1^{\mathrm{op}} \circ \mathcal{F}_{P|P^{\mathrm{op}}}(f))_{\chi_s}(\mathrm{Id})\\
    &=\mu_{\Lambda}(\chi_s^{-1})\left(\mathcal{R}_{P^{\mathrm{op}}|P} (\mathcal{F}_{P|P^{\mathrm{op}}}(f))\right)_{\mu^{-1}\chi_{s-\frac{1}{[F:\rr]}}}(\mathrm{Id})\\
    &=\frac{\mu_{\Lambda}(\chi_s^{-1})}{\mu_{\Lambda}\left(\mu\chi_{s-\frac{1}{[F:\rr]}}^{-1}\right)}f_{\mu^{-1}\chi_{s-\frac{1}{[F:\rr]}}}(\mathrm{Id}).
\end{align*}
\end{proof}

\begin{Lem}\label{lem:subweyl}
The module $\mathcal{S}(X_P^\circ(F))$ is a $G(F)\ltimes W_{X_P(F)}$-submodule of $\mathcal{S}(X_P(F))$.
\end{Lem}

\begin{proof}
It suffices to show $\partial_1\big(\mathcal{S}(X_P^\circ(F))\big)\subseteq \mathcal{S}(X_P^\circ(F)),$ i.e., $(\partial_1 f)_{\chi_s}$ is holomorphic for $f\in \mathcal{S}(X_P^\circ(F))$. The proof is similar to that of Proposition \ref{prop:archlocal}, in which case $(\partial_1 f)_{\chi_s}(\mathrm{Id})$ is easily checked to be entire using Lemma \ref{lem:shift}.
\end{proof}

Lemma \ref{lem:subweyl} implies we have an induced $W_{X_P(F)}$-action on $A_{X_P(F)}=\mathcal{S}(X_P(F))/\mathcal{S}(X_P^\circ(F))$.

\begin{Lem}\label{lem:end}
The module $F[X_P]\otimes_\rr \cc$ is stable under the natural $W_{X_P(F)}$-action.
\end{Lem}

\begin{proof}
Each irreducible complex subrepresentation $V$ of $F[X_P]\otimes_\rr \cc$ is contained in the induced representation $\mathrm{Ind}_P^G{((\mu^\ell)_{-s-s_k-1}))}$ for some unique pair $(\ell,s)\in H\times \frac{1}{[F:\rr]}\zz_{\ge 0}$. Moreover by Lemma \ref{lem:shift}
\begin{align*}
    \sum_{i=1}^{\dim V_P}\partial_iV\subset \mathrm{Ind}_P^G{((\mu^{\ell+1})_{-s+\frac{1}{[F:\rr]}-s_k-1})}
\end{align*}
is a finite-dimensional $G(F)$-representation. If it is zero, then we are done. Otherwise, by \cite[Theorem 1]{finitemultfree} $\sum_{i=1}^{\dim V_P}\partial_iV$ is the unique irreducible subrepresentation of both $\mathrm{Ind}_P^G{((\mu^{\ell+1})_{-s+\frac{1}{[F:\rr]}-s_k-1})}$ and $F[X_P]\otimes_\rr \cc$. The case for $\overline{\partial}_i$ can be deduced by taking the complex conjugate.
\end{proof}

Let $j\in \zz_{>0}$. Choose a function $f\in \mathcal{S}_{\mathrm{ES}}(X_P(F))$ such that for $|x|\le 1, f(x)=x_1^j(x)$. Since $x_1^j$ is the lowest weight vector in $V_{-j\omega_P}^\lor$, the germ of $\partial_1f\in \mathcal{S}_{\mathrm{ES}}(X_P(F))$ at the origin is $p_jx_1^{j-1}$ for some $p_j\in \cc$. Let us compute $p_j$. Note that $f_{(\mu^{-\ell})_s}=0$ if $\ell \neq j$ in $H$, and for $\ell=j$ in $H$ we have
\begin{align*}
    f_{(\mu^{-j})_s}(\mathrm{Id})=\int_{|a|\le 1} |a|^{s+s_k+1+\frac{j}{[F:\rr]}}d^\times a +h(s)=C\frac{1}{s+s_k+1+\frac{j}{[F:\rr]}}+h(s)
\end{align*}
for some nonzero absolute constant $C$ and entire function $h$. Thus by Lemma \ref{lem:shift} $x_1^{j}$ is sent to $p_jx_1^{j-1}$ where
\begin{align*}
    p_j&=\frac{\mu_{\Lambda}(\chi_s^{-1})}{\mu_{\Lambda}\left(\mu\chi^{-1}_{s-\frac{1}{[F:\rr]}}\right)}\Bigg|_{\chi_s=(\mu^{-(j-1)})_{-s_k-1-\frac{j-1}{[F:\rr]}}} =\prod_{i=1}^k \frac{\gamma\left(-s_i,(\mu^{j-1})^{\lambda_i}_{s-\frac{1}{[F:\rr]}},\psi\right)}{\gamma\left(-s_i,(\mu^{j})^{\lambda_i}_{s},\psi\right)}\Bigg|_{s=s_k+1+\frac{j}{[F:\rr]}}.
\end{align*}
Assume $\psi(x)=e^{2\pi i \mathrm{tr}_{F/\rr}(x)}.$ Then
\begin{align*}
    \gamma(s,\mu^\ell,\psi)^{-1}=\begin{cases}
        2^{1-s}\pi^{-s}\cos(\frac{\pi s}{2})\Gamma(s) & \textrm{ if } F=\rr, \ell=0,\\
        -i2^{1-s}\pi^{-s}\sin(\frac{\pi s}{2})\Gamma(s) & \textrm{ if } F=\rr, \ell=1,\\
        (-i)^{|\ell|}(2\pi)^{1-2s}\frac{\Gamma(s+|\ell|/2)}{\Gamma(1-s+|\ell|/2)}& \textrm{ if } F=\cc, \ell\in \zz.
    \end{cases}
\end{align*}

Assume $F=\rr$. For $i$ such that $\lambda_i$ is even, each term in the product above is
\begin{align*}
&\frac{2^{1-\lambda_is+s_i}\pi^{-\lambda_is+s_i} \cos\left(\frac{\pi(\lambda_is-s_i)}{2}\right)\Gamma(\lambda_is-s_i)}{2^{1-\lambda_i(s-1)+s_i}\pi^{-\lambda_i(s-1)+s_i} \cos\left(\frac{\pi(\lambda_is-\lambda_i-s_i)}{2}\right)\Gamma(\lambda_i (s-1)-s_i)}\Bigg|_{s=s_k+1+j}\\
&=(2\pi)^{-\lambda_i}(-1)^{\lambda_i/2}\prod_{\ell=0}^{\lambda_i-1} (\lambda_i(s_k+j)-s_i+\ell)\\
&=\left(\frac{\lambda_i}{2\pi\sqrt{-1}}\right)^{\lambda_i}\prod_{\ell=0}^{\lambda_i-1} \left(s_k+j-\frac{s_i}{\lambda_i}+\frac{\ell}{\lambda_i}\right).
\end{align*}
For odd $\lambda_i$ and odd $j$, each term is
\begin{align*}
    &\frac{-\sqrt{-1}2^{1-\lambda_is+s_i}\pi^{-\lambda_i s+s_i} \sin\left(\frac{\pi(\lambda_i s-s_i)}{2}\right)\Gamma(\lambda_i s-s_i)}{2^{1-\lambda_i(s-1)+s_i}\pi^{-\lambda_i(s-1)+s_i} \cos\left(\frac{\pi(\lambda_is-\lambda_i-s_i)}{2}\right)\Gamma(\lambda_i (s-1)-s_i)}\Bigg|_{s=s_k+1+j}\\
    &=\left(\frac{\lambda_i}{2\pi\sqrt{-1}}\right)^{\lambda_i}\prod_{\ell=0}^{\lambda_i-1} \left(s_k+j-\frac{s_i}{\lambda_i}+\frac{\ell}{\lambda_i}\right).
\end{align*}
For odd $\lambda_i$ and even $j$, each term is
\begin{align*}
    &\frac{2^{1-\lambda_is+s_i}\pi^{-\lambda_i s+s_i} \cos\left(\frac{\pi(\lambda_i s-s_i)}{2}\right)\Gamma(\lambda_i s-s_i)}{-\sqrt{-1}2^{1-\lambda_i(s-1)+s_i}\pi^{-\lambda_i(s-1)+s_i} \sin\left(\frac{\pi(\lambda_is-\lambda_i-s_i)}{2}\right)\Gamma(\lambda_i (s-1)-s_i)}\Bigg|_{s=s_k+1+j}\\
    &=\left(\frac{\lambda_i}{2\pi\sqrt{-1}}\right)^{\lambda_i}\prod_{\ell=0}^{\lambda_i-1} \left(s_k+j-\frac{s_i}{\lambda_i}+\frac{\ell}{\lambda_i}\right).
\end{align*}
Now consider $F=\cc$. Then each term is
\begin{align*}
    &\frac{(-i)^{\lambda_i j}(2\pi)^{1-2(\lambda_is-s_i)}\frac{\Gamma(\lambda_is-s_i+\lambda_ij/2)}{\Gamma(1-\lambda_is+s_i+\lambda_ij/2)}}{(-i)^{\lambda_i(j-1)}(2\pi)^{1-2(\lambda_is-\lambda_i/2-s_i)}\frac{\Gamma(\lambda_is-\lambda_i/2-s_i+\lambda_i(j-1)/2)}{\Gamma(1-\lambda_is+\lambda_i/2+s_i+\lambda_i(j-1)/2)}}\Bigg|_{s=s_k+1+\frac{j}{2}}\\
    &=\left(\frac{\lambda_i}{2\pi\sqrt{-1}}\right)^{\lambda_i}\prod_{\ell=0}^{\lambda_i-1} \left(s_k+j-\frac{s_i}{\lambda_i}+\frac{\ell}{\lambda_i}\right).
\end{align*}

Therefore, in any case $p_j= c_\psi J(j)$ where $J$ is the polynomial of degree $$d_P:=\sum_{i=1}^k \lambda_i$$ given by
\begin{align*}
    J(s):=\prod_{i}^k  \lambda_i^{\lambda_i} \prod_{\ell=0}^{\lambda_i-1} \left(s+s_k-\frac{s_i}{\lambda_i}+\frac{\ell}{\lambda_i}\right),\quad c_\psi:=(2\pi\sqrt{-1})^{-d_P}.
\end{align*}
Note that the polynomial $J\in s\zz[s]$ has all but one roots in $\qq_{<0}$ by Proposition \ref{prop:lambda=1}.


For a vector $\underline{j}=(j_i)\in \zz_{\ge 0}^{\dim V_P},$ let $|\underline{j}|$ be the sum of its entries. We use the standard notation $\mathbf{x}^{\underline{j}}:=\prod_{i=1}^{\dim V_P} x_i^{j_i}.$

\begin{Thm}\label{thm:Weylrealize}
Each element in $W_{X_P(F)}$ has a representative in the standard Weyl algebra $W_{V_P}(F\otimes_\rr \cc)$. Explicitly, one can write
\begin{align}\label{eq:formula}
    c_\psi^{-1}\partial_1=c_1\frac{\partial}{\partial x_1}+\sum_{\ell=2}^{d_P} \bigg( c_\ell x_1^{\ell-1}\frac{\partial^\ell}{\partial x_1^\ell}+\sum_{\substack{\underline{j}\neq (\ell,0,\ldots,0)\\|\underline{j}|=\ell} } p_{\underline{j}} \frac{\partial^\ell}{\partial \mathbf{x}^{\underline{j}}}\bigg)
\end{align}
where $p_{\underline{j}}\in \qq[X_P]$, if nonzero, is a homogeneous polynomial of degree $|\underline{j}|-1$, and 
\begin{align*}
    c_\ell=\frac{1}{\ell!}\sum_{j=0}^{\ell} (-1)^{\ell-j}\binom{\ell}{j}J(j)\in \qq_{>0}.
\end{align*}
\end{Thm}

\begin{proof}
 We may assume $F=\rr$. By Lemma \ref{lem:end} $c_\psi^{-1}\partial_1$ restricts to an endomorphism of the graded $\cc$-module $\cc[X_P]$ of degree $-1$, which can be lifted to an endomorphism $\widetilde{\partial}$ of the graded $\cc$-module $\cc[V_P]$ of degree $-1$. By induction on the total degree, one can find homogeneous polynomials $p_{\underline{j}}$ of degree $|\underline{j}|-1$ in $\cc[V_P]$ successively such that
 \begin{align*}
     \widetilde{\partial}(\mathbf{x}^{\underline{n}})= p_{\underline{n}}\frac{\partial^{|\underline{n}|} \mathbf{x}^{\underline{n}}}{\partial \mathbf{x}^{\underline{n}}}+\sum_{\ell=1}^{|n|-1} \sum_{|\underline{j}|=\ell}p_{\underline{j}} \frac{\partial^\ell \mathbf{x}^{\underline{n}}}{\partial \mathbf{x}^{\underline{j}}}
 \end{align*}
 for all $\mathbf{x}^{\underline{n}}.$ Therefore, we can formally write
\begin{align}\label{eq:formalexp}
    \widetilde{\partial}=\sum_{\ell=1}^\infty \sum_{\substack{|\underline{j}|=\ell}} p_{\underline{j}} \frac{\partial^\ell}{\partial \mathbf{x}^{\underline{j}}}.
\end{align}
Now restricting further to $\cc[X_P]$ we have formally
\begin{align}\label{eq:formalexpX}
    c_\psi^{-1}\partial_1\bigg|_{\cc[X_P]}=\sum_{\ell=1}^\infty \sum_{\substack{|\underline{j}|=\ell}} p_{\underline{j}} \frac{\partial^\ell}{\partial \mathbf{x}^{\underline{j}}},
\end{align}
and the image of each $p_{\underline{j}}$ in $\cc[X_P]$ is unique. 

Assume for the moment that the sum over $\ell$ in the expression \eqref{eq:formalexp} is finite. By Nachbin's theorem \cite[Theorem 1.2.1]{L:approxdifffunctions} $\cc[X_P]$ is dense in $C^\infty(X_P^\circ(\rr))$ (with the topology of compact convergence for all derivatives). Now as the identity \eqref{eq:formalexpX} holds for germs at the origin and $G(\rr)$ acts transitively on $X_P^\circ(\rr),$ \eqref{eq:formalexpX} can be extended to $C^\infty(X_P^\circ(\rr))$. In particular, it holds on $\mathcal{S}(X_P(\rr)).$ 

Therefore, to prove the assertion, it remains to show that \eqref{eq:formalexp} is a finite sum with claimed coefficients as stated in the theorem. Say $x_i$ has weight $r_i$, so $r_1$ has the lowest weight. Note that $\cc[V_P]=\bigoplus_{\ell=0}^\infty \mathrm{Sym}^\ell V_{P}^\lor(\cc)$, and while $\mathrm{Sym}^\ell V_{P}^\lor(\cc)$ is in general not irreducible, $\cc^\times x_1^\ell$ are lowest weight vectors in $\mathrm{Sym}^\ell V_{P}^\lor(\cc)$ with weight $\ell r_1$. Let $\ell\ge 1$. Since $\partial_1(x_1^\ell)$ has weight $(\ell-1)r_1$, it must be $\partial_1(x_1^\ell)\in \cc x_1^{\ell-1}$. Therefore, by induction $p_{(\ell,0,0,\ldots,0)}=c_\ell x_1^{\ell-1}$ for some $c_\ell\in \cc$. Similarly, the coefficients of $\frac{\partial^\ell}{\partial x_1^\ell}$ is zero in the formal differential operator expression of  $\partial_i$ for every $i\neq 1$.

Let us compute $c_\ell$ and show that $c_\ell$ is zero if $\ell> d_P$. We have already shown that
\begin{align*}
    J(\ell)x_1^{\ell-1}=c_\psi^{-1}\partial_{1}(x^\ell_1)=\sum_{j=1}^{\ell} c_j\frac{\ell !}{(\ell-j)!}x_1^{\ell-1}.
\end{align*}
Therefore, for any $\ell\ge 1$
\begin{align*}
    J(\ell)=\sum_{j=1}^\ell j!c_j\binom{\ell}{j}.
\end{align*}
Since $J(s)\in s\zz[s]$ is of degree $d_P$, it follows that we have 
\begin{align*}
    J(s)=\sum_{j=1}^{d_P} j!c_j\binom{s}{j}.
\end{align*}
Consequently, $c_\ell=0$ for $\ell>d_P,$ and for $\ell\le d_P$ using binomial coefficients
\begin{align*}
   \ell!c_{\ell}=\sum_{j=1}^{\ell} (-1)^{\ell-j}\binom{\ell}{j}J(j).
\end{align*}
This is the $\ell$th forward difference of $J$ with step $1$. Since $J$ has all but one roots in $\qq_{<0},$ $c_\ell$ are positive by the mean value theorem.

We now show that $p_{\underline{j}}=0$ for all $|\underline{j}|>d_P$. For $\ell\ge 1$, consider the map
\begin{align}
\begin{split}\label{eq:themap}
    \mathrm{Span}\left\{ \partial_i\right\}\times V_{-\ell \omega_P}^\lor(\cc) &\to  V_{-(\ell-1) \omega_P}^\lor(\cc)\\
    (D,f)&\mapsto D(f).
\end{split}
\end{align}
It can be identified as a unique morphism $\varphi$ in
\begin{align}\label{eq:morphismid}
\begin{split}
    &\mathrm{Hom}_{G(\rr)}(V_P(\cc)\otimes V^\lor_{-\ell\omega_P}(\cc),V^\lor_{-(\ell-1)\omega_P}(\cc))\\
    &=\mathrm{Hom}_{G(\rr)}(V_P(\cc)\otimes V^\lor_{-\ell\omega_P}(\cc),\mathrm{Sym}^{\ell-1} V_P^\lor(\cc)/ I_{\ell-1}),
\end{split}
\end{align}
where $I_{\ell-1}< \mathrm{Sym}^{\ell-1} V_P^\lor(\cc)$ is the $G(\rr)$-module spanned by homogeneous polynomials of degree $\ell$ vanishing on $X_P$. On the other hand, for each $\ell$ the expression $\sum_{|\underline{j}|=\ell} p_{\underline{j}}\frac{\partial^\ell}{\partial \mathbf{x}^{\underline{j}}}$ defines a unique morphism in
\begin{align}\label{eq:morphismcoeff}
    \mathrm{Hom}_{G(\rr)}(V_P(\cc), V^\lor_{-(\ell-1)\omega_P}(\cc)\otimes \mathrm{Sym}^\ell V_P(\cc))
\end{align}
that maps $\partial_1$ to $\sum_{|\underline{j}|=\ell} p_{\underline{j}}\otimes\frac{\partial^\ell}{\partial \mathbf{x}^{\underline{j}}}.$ This morphism is injective if $c_\ell\neq 0$. Consequently, the expression $\sum_{|\underline{j}|<\ell} p_{\underline{j}}\frac{\partial^{|\underline{j}|}}{\partial \mathbf{x}^{\underline{j}}}$
also defines a unique morphism $\varphi'$ in \eqref{eq:morphismid} by considering the usual action of differential operators on functions as in \eqref{eq:themap}. If $c_\ell=0,$ then $(\varphi-\varphi')(V_P(\cc)\otimes x_1^\ell)=0$. Since $V^\lor_{-\ell \omega_P}$ is irreducible, we have $\varphi=\varphi'.$ This implies $p_{\underline{j}}=0$ in $\cc[X_P]$ for $|\underline{j}|=\ell>d_P$. 

Finally, since $c_\ell\in \qq$ and all representations can be made defined over $\qq,$ by our choices of $x_i$ we conclude that $p_{\underline{j}}\in \qq[X_P]$ by \eqref{eq:morphismcoeff}. Note that  \eqref{eq:morphismcoeff} also implies $p_{\underline{j}}=0$ if $|\underline{j}|=1$ and $\underline{j}\neq (1,0,\ldots,0).$
\end{proof}

Theorem \ref{thm:Weylrealize} implies  the Weyl algebra $W_{X_P(F)}$ can be defined over $\qq$. We denote it by $W_{X_P}$. Thus $W_{X_P}(F\otimes_\rr \cc)=W_{X_P(F)}$. The following is a direct consequence of Theorem \ref{thm:Weylrealize}.

\begin{Cor}
Let $F$ be a field of characteristic zero. Then $F[X_P]$ is a faithful simple $W_{X_P}(F)$-module. In particular, the Weyl algebra $W_{X_P}(F)$ is a primitive ring.\qed
\end{Cor}

We remark that our proof of Theorem \ref{thm:Weylrealize} is constructive, so one can explicitly compute $\partial_i$ case by case. We end this subsection with an illustration.

\begin{Ex}\label{ex:cone}
Consider the cone $X_n\subset V_n= \mathbb{A}^{2n}$ cut out by 
\begin{align*}
    x_1x_{2n}+\ldots+x_{n}x_{n+1}=0.
\end{align*}
In this case, a highest weight vector is $e_{2n}$ and $J(s)=s(s+n-2).$ Therefore, by Theorem \ref{thm:Weylrealize}
\begin{align*}
    \frac{1}{-4\pi^2}\partial_1=(n-1)\frac{\partial }{\partial x_1}+\sum_{|\underline{j}|=2} p_{\underline{j}} \frac{\partial^2}{\partial \mathbf{x}^{\underline{j}}}
\end{align*}
It follows from \eqref{eq:morphismcoeff}, coefficients $p_{\underline{j}}$ correspond to a $G$-equivariant map
\begin{align*}
    V_n\otimes V_{n}&\to \mathrm{Sym}^2 (V_n).
\end{align*}
Since $\mathrm{Sym}^2(V_n)$ decomposes into the trivial representation and the representation of highest weight $-2\omega_{P_n}$, we have 
\begin{align*}
    \sum_{|\underline{j}|=2} p_{\underline{j}} \frac{\partial^2}{\partial \mathbf{x}^{\underline{j}}}=\sum_{j=1}^{2n} x_j\frac{\partial^2}{\partial x_1\partial x_j}+Cx_{2n}\sum_{j=1}^n \frac{\partial^2}{\partial x_j\partial x_{2n+1-j}}.
\end{align*}
for some $C\in \qq$. Now using the condition $\partial_1(x_1x_{2n}+\ldots+ x_{n}x_{n+1})=0$, one can conclude $C=-1$.

Let us compare this with the operators defined in \cite{Kobayashi:Mano}. The defining equation of their space is
\begin{align*}
    y_1^2-y_{2n}^2+\ldots +y_{n}^2-y_{n+1}^2.
\end{align*}
We perform the change of variables $x_i=y_{i}+y_{2n+1-i},x_{2n-i}=y_{i}-y_{2n+1-i}, 1\le i\le n.$ Then
\begin{align*}
    \frac{\partial}{\partial x_i}&=\frac{1}{2}\frac{\partial}{\partial y_i}+\frac{1}{2}\frac{\partial}{\partial y_{2n+1-i}},& \frac{\partial}{\partial x_{2n+1-i}}&=\frac{1}{2}\frac{\partial}{\partial y_i}-\frac{1}{2}\frac{\partial}{\partial y_{2n+1-i}}.
\end{align*}
Therefore,
\begin{align*}
    \sum_{j=1}^{2n} x_j\frac{\partial^2}{\partial x_1\partial x_j}&=\frac{1}{4}\sum_{j=1}^{n} (y_j+y_{2n+1-j})\left(\frac{\partial^2}{\partial y_1\partial y_j}+\frac{\partial^2}{\partial y_1\partial y_{2n+1-j}}+\frac{\partial^2}{\partial y_{2n}\partial y_j}+\frac{\partial^2}{\partial y_{2n}\partial y_{2n+1-j}}\right)\\
    &+\frac{1}{4}\sum_{j=1}^{n} (y_j-y_{2n+1-j})\left(\frac{\partial^2}{\partial y_1\partial y_j}-\frac{\partial^2}{\partial y_1\partial y_{2n+1-j}}+\frac{\partial^2}{\partial y_{2n}\partial y_j}-\frac{\partial^2}{\partial y_{2n}\partial y_{2n+1-j}}\right)\\
    &=\frac{1}{2}\sum_{j=1}^{2n} y_j\left(\frac{\partial^2}{\partial y_1\partial y_j}+\frac{\partial^2}{\partial y_{2n}\partial y_j}\right).
\end{align*}
Thus $\frac{1}{-4\pi^2}\partial_1$ in this coordinate is
\begin{align*}
    \frac{n-1}{2}\left(\frac{\partial}{\partial y_1}+\frac{\partial}{\partial y_{2n}}\right)+\frac{1}{2}\sum_{j=1}^{2n} y_j\left(\frac{\partial^2}{\partial y_1\partial y_j}+\frac{\partial^2}{\partial y_{2n}\partial y_j}\right)+\frac{-1}{4}(y_1-y_{2n})\left(\sum_{j=1}^n \frac{\partial^2}{\partial y_j^2}-\frac{\partial^2}{\partial y_{n+j}^2}\right).
\end{align*}
In the new coordinate, using the notation in \cite{Kobayashi:Mano}, we have
\begin{align*}
    P_1+P_{2n}=-4c_\psi^{-1}\partial_1.
\end{align*}
Therefore, as a consequence of Theorem \ref{thm:uniqueF} below or \cite[Theorem 2.5.4]{Kobayashi:Mano}, $\mathcal{S}(X_n(\rr))$ is the unitarizable minimal representation of $\mathrm{SO}(n+1,n+1)$ considered in loc. cit, which by \cite[Theorem 3.4]{minrep:real} is a subrepresentation of $I_{P_{n+1}}^{\mathrm{SO}(n+1,n+1)}(1_{-1})$.

We claim this also implies $\mathcal{S}(X_{n}(\cc))$ is the unitarizable minimal representation of $\mathrm{SO}_{2n+2}(\cc)$ contained in $I_{P_{n+1}}^{\mathrm{SO}_{2n+2}}(1_{-1})$. As explained in \cite{minrep:real}, $\mathrm{SO}(n+1,n+1)$ can be realized as (up to a finite cover) the conformal group of the real Jordan algebra $\rr^{n,n}.$ Then the computation above implies the Bessel operators defined in loc. cit. agrees with $\partial_i$ up to a change of coordinates. On the other hand, $\mathrm{SO}_{2n+2}(\cc)$ is (up to a finite cover) the conformal group of the complex Jordan algebra $\cc^{2n},$ which is the complexification of $\rr^{n,n}$. The Bessel operators defined in loc. cit. for $\cc^{2n}$ are just complexifications of those of $\rr^{n,n}$. Since $W_{X_n(\cc)}=W_{X_n(\rr)}\oplus \overline{W}_{X_{n}(\rr)},$ we conclude that those Bessel operators are $\partial_i,\overline{\partial}_i.$ Now the claim follows from Theorem \ref{thm:uniqueF} or \cite[Proposition 3.17]{minrep:real}. 
\qed
\end{Ex}

\subsection{Applications of Weyl algebras on Schwartz spaces}\label{ssec:application:Weyl}

It is shown in \cite{Hsu:Weyl} that $W_{X_P}$ share several key properties with the usual rings of differential operators on smooth affine varieties. In this subsection, we use properties of the Weyl algebra $W_{X_P}$ established in loc. cit. to improve our understanding of $\mathcal{S}(X_P(F))$. 

\begin{Thm}\label{thm:Weylprop} \cite{Hsu:Weyl}
    The Weyl algebra $W_{X_P}$ (over any field $F$ of characteristic zero) enjoys the following properties.
    \begin{enumerate}
        \item It is a simple Noetherian domain with no nontrivial finite-dimensional representation.
        \item For any $m\in \zz_{\ge 0},$ the left ideal in $W_{X_P}(F)$ generated by 
        \begin{align*}
            \{ b(\mathbf{x}), b^{\mathrm{op}}(\mathbf{\partial}): b\in F[X_P], b^{\mathrm{op}}\in F[X_{P^{\mathrm{op}}}] \textrm{ homogeneous of degree } m\}
        \end{align*}
        is $W_{X_P}(F)$.
        \item It contains $\mathfrak{m}^{\mathrm{ab}}\oplus \mathfrak{g}.$
         \item It equals the ring of differential operators on $X_P$ (in the sense of Grothendieck).
    \end{enumerate}\qed
\end{Thm}

In below let $F$ be an archimedean local field. We introduce the following notion.
\begin{Def}
 A \textbf{graded Fr\'echet} $G(F)\ltimes W_{X_P(F)}$-\textbf{module} is a $\cc$-vector space $U=\prod_{n\in \zz_{\ge0}} U_n$ such that each $U_n$ is a Fr\'echet $G(F)$-module, and for all $i,n$, $x_iU_n\subseteq U_{n+1}$ and $\partial_iU_{n}\subseteq U_{n-1}$. If $F$ is complex, then additionally $\overline{x}_iU_n\subseteq U_{n+1}$ and $\overline{\partial}_iU_{n}\subseteq U_{n-1}$.
\end{Def} 

Recall that a smooth representation $V$ of $G(F)$ is quasisimple if the center $\mathfrak{Z}(\mathfrak{g})$ of the universal enveloping algebra of $\mathfrak{g}$ acts on $V$ via scalars.

\begin{Prop}\label{prop:finlength}
      Let $U=\prod_{n\in \zz_{\ge0}} U_n$ be a graded Fr\'echet $G(F)\ltimes W_{X_P(F)}$-module satisfying the following assumptions:
     \begin{enumerate}
       \item Each $U_n$ is a smooth (Fr\'echet) $G(F)$-representation of finite length. Moreover, if $F=\rr$
       \begin{align*}
          \sup_{n\ge 0} \mathrm{length}(U_n)<\infty,
       \end{align*}
       and if $F=\cc$
       \begin{align*}
          \sup_{n\ge 0} {\frac{\mathrm{length}(U_n)}{n+1}}<\infty.
       \end{align*}
       \item For $n\neq n',$ $U_n$ and $U_{n'}$ do not share a common composition factor. When $F=\cc$, each $U_n$ is a (finite) direct sum of quasisimple representations.
       \end{enumerate}
    Then $U$ is a graded Fr\'echet $G(F)\ltimes  W_{X_P(F)}$-module of length at most
    \begin{align*}
        C(U):=\begin{cases}
            {\displaystyle\liminf_{n\to\infty }} \,\mathrm{length}(U_n) & \textrm{ if } F=\rr,\\
            \bigg\lfloor{\displaystyle\liminf_{n\to\infty}} \, {\frac{\mathrm{length}(U_n)}{n+1}} \bigg\rfloor& \textrm{ if } F=\cc.
        \end{cases}
    \end{align*}
\end{Prop} 

\begin{proof}
    We work in the category of smooth Fr\'echet $G(F)$-modules of finite length. Assumptions imply every (nonzero) graded Fr\'echet $G(F)\ltimes W_{X_P(F)}$-submodule $N$ of $U$ is of the form $\prod_{n\ge 0} N_n,$ where each $N_n$ is a $G(F)$-submodule of $U_n$. For a graded Fr\'echet submodule $N,$ it is clear that both $N$ and $U/N$ satisfy both assumptions (1)(2). We claim
    \begin{itemize}
        \item [$(\ast)$]    For any $n\ge 0$ and each nonzero $u\in U_{n}$, there is at least one $i$ such that $x_iu\neq 0$. If $F=\cc$, there is also some $j$ such that $\overline{x}_ju\neq 0$.
    \end{itemize}
    It suffices to prove for $F=\rr$. Suppose on the contrary that there is a nonzero $u$ such that $x_iu=0$ for all $i$. Since $u\in U_{n},$ $b^{\mathrm{op}}(\partial)u=0$ for all homogeneous polynomials $b^{\mathrm{op}}$ of degree $n+1$ in $\cc[X_{P^{\mathrm{op}}}]$. Consequently, the left ideal
    \begin{align*}
        I=\{ D\in W_{X_P}(\cc): Du=0\}
    \end{align*}
    contains $b(\mathbf{x})$ and $b^{\mathrm{op}}(\mathbf{\partial})$ for all homogeneous polynomials $b\in \cc[X_P]$ and $b^{\mathrm{op}}\in \cc[X_{P^{\mathrm{op}}}]$ of degree $n+1$. It follows by Theorem \ref{thm:Weylprop}(2) that $1\in I$, which implies $u=0,$ a contradiction.

    As a consequence of $(\ast),$ if $N\neq 0$ then $C(N)\ge 1$. This is clear when $F=\rr,$ and  when $F=\cc$ it follows from assumption (2) as each irreducible $G(\cc)$-subrepresentation of $\cc[\mathrm{Res}_{\cc/\rr}X_P]$ has distinct infinitesimal character.

    Clearly, $C(U/N)+C(N)\le C(U)$ for any graded Fr\'echet submodule $N$. Furthermore, if  $C(N)=C(U)$, then $N=U$. Indeed, if $U/N\neq 0$ then one has $1\le C(U/N)\le C(U)-C(N)=0$, a contradiction. Therefore by induction on $C(U)$, it suffices to prove the assertion when $C(U)=1$, which is clear.
\end{proof}

We can now prove Theorem \ref{thm:main:arch}.

\begin{Cor}\label{cor:arch:finite}
The graded Fr\'echet $G(F)\ltimes  W_{X_P(F)}$-module $A_{X_P(F)}$ is of finite length.
\end{Cor}

\begin{proof}
   One observes that for all $i,r,n,\ell$, 
   \begin{align*}
       x_iA_{r,n,\ell}\subseteq A_{r+\frac{1}{[F:\rr]},n,\ell+1},\quad   \bar{x}_iA_{r,n,\ell}\subseteq A_{r+\frac{1}{[F:\rr]},n,\ell-1},\\
       \partial_{i}A_{r,n,\ell}\subseteq A_{r-\frac{1}{[F:\rr]},n,\ell-1},\quad \bar{\partial}_{i}A_{r,n,\ell}\subseteq A_{r-\frac{1}{[F:\rr]},n,\ell+1}.
   \end{align*}
   Let $A_n:=\oplus_{m\le n}\prod_{r,\ell} A_{r,m,\ell}.$ By Theorem \ref{thm:Weylrealize} and the definition of $A_{r,m,\ell}$, $A_n$ is a $G(F)\ltimes  W_{X_P(F)}$-submodule of $A_{X_P(F)}$. The $G(F)$-equivariant map $A_{n}/A_{n-1}\hookrightarrow A_{1}$ in Theorem \ref{thm:exactarch} is also $ W_{X_P(F)}$-equivariant by Theorem \ref{thm:Weylrealize} and the product rule. Therefore to prove the assertion, it suffices to show $A_1$ is a graded Fr\'echet $G(F)\ltimes  W_{X_P(F)}$-module $A_{X_P(F)}$ of finite length.

   Choose a finite set $S\subset \cup_{\ell\in H}\mathrm{Supp}(L_\ell)$ such that
   \begin{align*}
       \bigcup_{\ell\in H}\mathrm{Supp}(L_\ell)=\bigsqcup_{r\in S} r+\frac{1}{[F:\rr]}\zz_{\ge 0}.
   \end{align*}
    For a given $r$, we let
   \begin{align*}
       U(r):=\prod_{n\in \zz_{\ge 0 }}\bigoplus_{\ell:r+\frac{n}{[F:\rr]}\in\mathrm{Supp}(L_\ell)}A_{r+\frac{n}{[F:\rr]},1,\ell}.
   \end{align*}
   Then we have $A_1=\oplus_{r\in S} U(r),$ and each $U(r)$ is graded Fr\'echet. We check that $U(r)$ satisfies assumptions in Proposition \ref{prop:finlength}. For $\mathrm{Re}(\chi)<0,$ $I_P(\chi)$ are quasisimple of finite length with distinct infinitesimal characters \cite[Lemma 4.1.8]{V:realrepbook}, so $U(r)$ satisfies assumption (2). That $U(r)$ satisfies assumption (1) is a consequence of \eqref{eq:normalizedchi}, \cite[Theorem 4.1.4]{V:realrepbook} and \cite[\S 18]{AvLTV}.
\end{proof}

To study the structure of $A_{X_P(F)}$ further, let us recall the notion of \textbf{Gelfand-Kirillov dimension} of a finitely generated $U(\mathfrak{g})$-module $V$ \cite{Vogan:size, noncomm}. Let $(U^{\le n}(\mathfrak{g}))_{n\ge 0}$ be the PBW filtration of $U(\mathfrak{g})$. Choose a finite-dimensional complex vector subspace $V_0\subseteq V$ such that $U(\mathfrak{g})V_0=V$. Define $V_n:=U^{\le n}(\mathfrak{g})V_0$ for $n\ge 0$. Then there is a nonnegative integer $\mathrm{GK}(V)$ and $c>0$ such that 
\begin{align*}
    \dim V_{n}= cn^{\mathrm{GK}(V)}+O(n^{\mathrm{GK}(V)-1}).
\end{align*}
The integer $\mathrm{GK}(V)$ is the Gelfand-Kirillov dimension of $V.$ It is independent of the choice of $V_0$. Let us apply this to composition factors of $A_{X_P(F)}$. 

Let $U=\prod_{n\in \zz_{\ge 0}} U_n$ be a composition factor of $A_{X_P(F)}$ as graded Fr\'echet $G(F)\ltimes W_{X_P(F)}$-modules. Let $U_{\mathrm{fin}}:=\oplus_{n\in \zz_{\ge 0}} (U_n)_{\mathrm{fin}}$ where $(U_n)_{\mathrm{fin}}$ is the subspace of $K$-finite vectors of $U_n$. By \cite[Theorem 3.4.12]{WallachRG1} each $(U_n)_{\mathrm{fin}}$ is an admissible $(\fg,K)$-modules of finite length and hence a finitely generated $U(\mathfrak{g})$-module.

\begin{Prop}
    Let $W$ be a nonzero Fr\'echet $G(F)$-subrepresentation of $U_j$ for some $j\ge 0$. Then $\mathrm{GK}(W_{\mathrm{fin}})$ is independent of $j$ and $W.$
\end{Prop}

\begin{proof}
    It suffices to show when $W$ is irreducible, for any $n\ge 0$
    \begin{align*}
        \mathrm{GK}((U_n)_{\mathrm{fin}})\le  \mathrm{GK}(W_{\mathrm{fin}}).
    \end{align*}
    By Theorem \ref{thm:Weylprop}(3) $W_{X_P(F)}W_{\mathrm{fin}}$ is a $(\mathfrak{g},K)$-module. Since $U$ is an irreducible $G(F)\ltimes W_{X_P(F)}$-module, we have $W_{X_P(F)}W_{\mathrm{fin}}=U_{\mathrm{fin}}$. Choose a finite subset $S\subset W_{X_P(F)}$ such that $SW_{\mathrm{fin}}\subseteq (U_n)_{\mathrm{fin}}$ generates $(U_n)_{\mathrm{fin}}$ as an $U(\fg)$-module. By enlarging $S$ if necessary, we may assume $\mathrm{Span}(S)$ is stable under $[X,\cdot]$ for $X\in \mathfrak{g}$. Then
    \begin{align*}
        (U_n)_{\mathrm{fin}}=U(\fg)\mathrm{Span}(S)W_{\mathrm{fin}}=\mathrm{Span}(S)U(\fg)W_{\mathrm{fin}}=\mathrm{Span}(S)W_{\mathrm{fin}}.
    \end{align*}
    Since $S$ is finite, we have $\mathrm{GK}((U_n)_{\mathrm{fin}})\le \mathrm{GK}(W_{\mathrm{fin}})$ by \cite[Lemma 2.2]{Vogan:size}.
\end{proof}


Finally, we show that $\mathcal{S}(X_P(F))$ coincides with the function space defined in \cite[\S 6.1]{BKnormalized}: 
\begin{align*}
    \mathcal{S}_{\mathcal{D}}(X_P(F)):=\bigg\{f\in L^2(X_P(F)): Df\in L^2(X_P(F))  \textrm{ for every } D\in W_{X_P(F)}\bigg\}.
\end{align*}
It has a natural Fr\'echet topology given by the seminorms $ \norm{Df}_2$.
We remark that 
\begin{align}\label{eq:preserve}
    \mathcal{F}_{P|P^{\mathrm{op}}}(\mathcal{S}_{\mathcal{D}}(X_P(F)))=\mathcal{S}_{\mathcal{D}}(X_{P^{\mathrm{op}}}(F)).
\end{align}
This justifies \cite[Conjecture 6.2]{BKnormalized} by Theorem \ref{thm:Weylprop}(4).

\begin{Thm}
We have $\mathcal{S}(X_P(F))=\mathcal{S}_{\mathcal{D}}(X_P(F))$ as Fr\'echet spaces.
\end{Thm}

\begin{proof}
    The inclusion $\mathcal{S}(X_P(F))\hookrightarrow \mathcal{S}_{\mathcal{D}}(X_P(F))$ is continuous by Lemma \ref{lem:mult:general} and Proposition \ref{prop:archlocal}. Conversely, suppose $f\in \mathcal{S}_{\mathcal{D}}(X_P(F))$. Then $f\in C^\infty(X_P^\circ(F))$ rapidly decays at infinity and $|f(x)|\ll_\epsilon |x|^{-s_k-1-\epsilon}$ for all $x\in X_P^\circ(F)$ and $\epsilon>0$.
   
    Let $\varphi\in F[X_P]$ be a homogeneous function of degree $n>0$. Then for any $f\in \mathcal{S}_{\mathcal{D}}(X_P(F))$ and $\ell\in H,$ by the Iwasawa decomposition the local zeta function
    \begin{align*}
       Z_\varphi(f,(\mu^\ell)_s):=\int_{X_P(F)} f(x)(\mu^\ell)_s(\varphi(x)) dx
    \end{align*}
    converges absolutely for $\mathrm{Re}(s)>-\frac{s_k+1}{n}$. Furthermore, using the theory of Bernstein polynomials on $X_P$ \cite[\S 4.6]{Hsu:Weyl}, one can argue as in \cite[\S 5]{Igusa:zeta} to conclude that there is a discrete subset $S$ of $\cc$ depending only on $\varphi$ such that $Z_\varphi(f,(\mu^\ell)_s)$ extends to a meromorphic function on $\cc$ with possible poles in $S$, and for any real numbers $\sigma_2>\sigma_1$ and $\epsilon>0$
    \begin{align}\label{eq:zetazero}
       \sup_{\substack{\sigma_1<\mathrm{Re}(s)<\sigma_2\\|s-z|>\epsilon,z\in S}} |Z_\varphi(f,(\mu^\ell)_s)|<\infty.
    \end{align}

    Consider first $f$ such that the $K$-representation generated by $f$ is irreducible, and the (left) action of $K_{\GG_m}$ acts by $\mu^\ell$ for some $\ell\in H$. By the Peter-Weyl theorem and the Iwasawa decomposition, there is $\varphi$ homogeneous of degree $n>0$ such that 
    \begin{align}\label{eq:zeta:pinduc}
        Z_\varphi(f,(\mu^\ell)_s)=\int_{M^\mathrm{ab}(F)}  \delta_P(m)(\mu^\ell)_{ns}(\omega_P(m))f(m^{-1})dm=f_{\delta_P^{1/2}(\mu^\ell)_{ns}}(1).
    \end{align}
    We conclude by linearity and continuity that for $f\in \mathcal{S}_{\mathcal{D}}(X_P(F))$ the sections $f_{\chi_s}$ are meromorphic for each character $\chi,$ and possibles poles of $f_{\chi_s}$ are in the left half-plane $\mathrm{Re}(s)\le 0,$ and so are poles of $\mathcal{F}_{P|P^{\mathrm{op}}}(f)_{\chi_s}$ by \eqref{eq:preserve}.
    
    Suppose $f_{\chi_s}/d(\chi_s)$ has a pole at some $s_0$ such that $\mathrm{Re}(s_0)<0$. Then we have
    \begin{align*}
        \frac{\mathcal{F}_{P|P^{\mathrm{op}}}(f_{\chi_s})}{d(\chi_s)}=\frac{\mathcal{F}_{P|P^{\mathrm{op}}}(f)_{\chi_s^{-1}}}{d(\chi_s)}
    \end{align*}
    has a pole at $s_0$. This would imply $\mathcal{F}_{P|P^{\mathrm{op}}}(f)_{(\chi^{-1})_s}$ has a pole at $-s_0$, a contradiction. Therefore, possible poles of $f_{\chi_s}/d(\chi_s)$ are on the imaginary axis. Then using \eqref{eq:zetazero} and \eqref{eq:zeta:pinduc}, one can use the asymptotics in \S \ref{sec:Sch:arch} to conclude $\mathcal{S}_{\mathcal{D}}(X_P(F))/\mathcal{S}(X_P(F))$ if nonzero is a Fr\'echet $G(F)$-representation of finite length with each irreducible subquotients contained in $I(\mu)$ for some (unitary) character $\mu$. Since $\mathcal{S}(X_P(F))$ is closed under the action of $W_{X_P(F)}$, for any given $f\in \mathcal{S}_{\mathcal{D}}(X_P(F)),$ $x_if,\partial_if, \bar{x}_if,\bar{\partial}_if$ all lie in $\mathcal{S}(X_P(F))$ for all $i$. Hence the left ideal 
    \begin{align*}
        I=\{ D\in W_{X_P(F)}: Df\in \mathcal{S}(X_P(F))\}
    \end{align*}
    contains generators $x_i,\overline{x}_i,\partial_i,\overline{\partial}_i$ of $W_{X_P(F)}$. By Theorem \ref{thm:Weylprop}(2) we have $1\in I$ and thus $f\in \mathcal{S}(X_P(F))$.

\end{proof}

\subsection{Harmonic analysis on $X_P$}\label{ssec:harm}

\begin{Lem}\label{lem:rigidity}
 If $\mathcal{I}:\mathcal{S}(X_P^\circ(F))\longrightarrow \mathcal{S}(X_P(F))$ is a continuous $G(F)\ltimes (F[X_P]\otimes_\rr \cc)$-equivariant linear operator, then $\mathcal{I}$ is up to a scalar the natural inclusion.
\end{Lem}

\begin{proof}
We may assume $\mathcal{I}$ is nonzero. Let $f\in C^\infty_c(X_P^\circ(F)).$ For any $x\in X_P^\circ(F)$ and $h\in C^\infty_c(X_P^\circ(F))$ such that $h(x)=1$, by Nachbin's theorem \cite[Theorem 1.2.1]{L:approxdifffunctions} there is a sequence $\{p_i\}\subset F[X_P]\otimes_\rr\cc$ converging to $h$ in $C^\infty(X_P^\circ(F))$. Then for any $h'\in C^\infty_c(X_P^\circ(F))$ $$\lim_{i\to \infty} p_ih'=hh'$$ in $\mathcal{S}(X_P^\circ(F)).$ Since $\mathcal{I}$ is $(F[X_P]\otimes_\rr \cc)$-equivaraint, we have
\begin{align*}
    \mathcal{I}(f)(x)=h^2\mathcal{I}(f)(x)=\lim_{i\to \infty} p_ih\mathcal{I}(f)(x)=\lim_{i\to \infty} h\mathcal{I}(p_if)(x)=I(hf)(x).
\end{align*}
Therefore $f(x)$ uniquely determines $\mathcal{I}(f)(x).$ Since $G(F)$ acts transitively on $X_P^\circ(F),$ $C^\infty_c(X_P^\circ(F))$ is dense in $\mathcal{S}(X_P^\circ(F)),$ and $\mathcal{I}$ is nonzero and $G(F)$-equivariant, there is a constant $c\in \cc^\times $ such that for any $f\in C^\infty_c(X_P^\circ(F))$ and $x\in X_P^\circ(F),$
\begin{align*}
    \mathcal{I}(f)(x)=cf(x).
\end{align*}
The identity extends to all $\mathcal{S}(X_P^\circ(F))$ by continuity.
\end{proof}

\begin{Thm}\label{thm:uniqueF}
Let  $\mathcal{G}:\mathcal{S}(X_P(F))\to \mathcal{S}(X_{P^{\mathrm{op}}}(F))$ be a continuous  linear operator satisfying $\partial_i^{\mathrm{op}}\circ \mathcal{G}  =\mathcal{G}\circ x_{i}$ and $\bar{\partial}_i^{\mathrm{op}}\circ \mathcal{G}  =\mathcal{G}\circ \bar{x}_{i}$.
Suppose one of the following  holds:
\begin{enumerate}
    \item The operator $\mathcal{G}$ satisfies additionally $ \mathcal{G} \circ \partial_i=x_{i}^{\mathrm{op}} \circ \mathcal{G}$ and $\mathcal{G} \circ \bar{\partial}_i=\bar{x}_{i}^{\mathrm{op}} \circ \mathcal{G}$.
    \item The operator $\mathcal{G}$ is $G(F)$-equivariant and extends to a bounded linear operator $L^2(X_P(F))\to L^2(X_{P^{\mathrm{op}}}(F))$.
\end{enumerate}
 Then $\mathcal{G}$ is a constant multiple of $ \mathcal{F}_{P|P^{\mathrm{op}}}$.
\end{Thm}

\begin{proof}
 Consider the operator 
\begin{align*}
    \mathcal{I}:= \mathcal{F}_{P^{\mathrm{op}}|P}\circ \mathcal{G}.
\end{align*}
We need to show $\mathcal{I}$ is a constant multiple of the identity. Assumption (1) implies $\mathcal{I}$ is $W_{X_P(F)}$-equivariant. Since $\mathfrak{g}\subset W_{X_P(F)}$ by Theorem \ref{thm:Weylprop}, both assumptions (1) and (2) imply $\mathcal{I}$ is a continuous $G(F)\ltimes (F[X_P]\otimes_\rr \cc)$-equivariant endomorphism of $\mathcal{S}(X_P(F))$. Therefore by Lemma \ref{lem:rigidity} there is some $c\in \cc$ such that $\mathcal{I}(f)=cf$ for all $f\in \mathcal{S}(X_P^\circ(F))$. By replacing $\mathcal{G}$ with $\mathcal{G}-c\mathcal{F}_{P|P^{\mathrm{op}}}$, we may assume $\mathcal{S}(X_P^\circ(F))\le \mathrm{ker}\, \mathcal{I}$. Since $\mathcal{S}(X_P^\circ(F))$ is dense in $L^2(X_P(F)),$ if (2) holds we have $\mathcal{I}=0$ by continuity. Now suppose (1) is satisfied. Then we have an induced continuous  $W_{X_P(F)}$-equivariant map $$\mathcal{I}':A_{X_P(F)}\longrightarrow \mathcal{S}(X_P(F)).$$
Since $\mathfrak{m}^{\mathrm{ab}}\subset W_{X_P(F)}$ by Theorem \ref{thm:Weylprop}, the map $\mathcal{I}'$ is $M^{\mathrm{ab}}(F)$-equivariant. This implies $\mathcal{I}'$=0.
\end{proof}

 Let $s\in \zz_{\ge 0}.$ In the monomial $\cc$-basis of $F[V_P^\lor]\otimes_\rr \cc$, let $\mathcal{B}_s(F)$ be the subset consisting of monomials of degree at most $s$. Define a Sobolev space on $X_P(F)$
\begin{align*}
    H^{s}(X_P(F)):=\bigg\{ f\in L^2(X_P(F)): p\mathcal{F}_{P|P^{\mathrm{op}}}(f)\in L^2(X_{P^{\mathrm{op}}}(F)) \textrm{ for any } p\in \mathcal{B}_s(F)\bigg\}
\end{align*}
equipped with the inner product
\begin{align*}
   \langle f,h\rangle_{H^s(X_P(F))}:=\sum_{p\in \mathcal{B}_s(F)} \langle p\mathcal{F}_{P|P^{\mathrm{op}}}(f), p\mathcal{F}_{P|P^{\mathrm{op}}}(h)\rangle _{L^2}.
\end{align*}
One can define similarly $H^{s}(X_{P^\mathrm{op}}(F))$. 

\begin{Prop}
For $s\in \zz_{\ge 0},$ $H^s(X_P(F))$ is a separable Hilbert  space and $C^\infty_c(X_P^\circ(F))$ is dense in $H^s(X_P(F))$.\qed
\end{Prop}

\begin{proof}
    Completeness of the inner product space $H^s(X_P(F))$ follows from a standard argument (see e.g., \cite[Lemma 5.2]{T:Sobolev}). Consider the map
    \begin{align*}
        H^s(X_P(F))&\longrightarrow L^2(X_{P^{\mathrm{op}}}(F))^{|\mathcal{B}_s(F)|}\\
        f &\longmapsto (p\mathcal{F}_{P|P^{\mathrm{op}}}(f))_{p\in \mathcal{B}_s(F)}.
    \end{align*}
It is an isometry. The space $L^2(X_{P^{\mathrm{op}}}(F))$ is separable, and so is  $H^s(X_P(F))$. 

For the last assertion, it suffices to show given $h\in L^2(X_{P^{\mathrm{op}}}(F))$ such that $ph\in L^2(X_{P^{\mathrm{op}}}(F))$ for all $p\in \mathcal{B}_s(F)$, there is a sequence of functions $f_i\in C^\infty_c(X_P^\circ(F))$ such that for every $p\in \mathcal{B}_s(F),$ $pf_i$ converges in $L^2$ to $ph,$ which is well known. 
\end{proof}


The following is straightforward by using smooth cut-off functions and the realization of $\partial_i$ as differential operators.

\begin{Prop}\label{prop:Sobolev}
For any $f\in \mathcal{S}(X_P(F))$ and $s\in \zz_{\ge 0}$, there is a sequence of functions $f_i\in C^\infty_c(X_P^\circ(F))$ such that
\begin{align*}
    \lim_{i\to \infty} \norm{f_i- f}_{H^s(X_P(F))}= \lim_{i\to \infty} \norm{\mathcal{F}_{P|P^{\mathrm{op}}}(f_i-f)}_{H^s(X_{P^\mathrm{op}}(F))}=0.
\end{align*}
\qed
\end{Prop}

\begin{Rem}
 Proposition \ref{prop:Sobolev} implies that the assertion in \cite[Lemma 10.2]{Getz:Hsu:Leslie} holds for archimedean $F,$ and so does Theorem 10.1 in loc.~ cit. This is the initial motivation for us to study the Weyl algebra on $X_P$.
\end{Rem}

\section{Poisson summation formulae}\label{sec:Poisson}

Let $E$ be a global field and $\mathbb{A}_E$ be its ring of adele. In this section, we establish the Poisson summation formula on $X_P$. It has already been discussed generally in \cite{Choie:Getz} under a mild assumption. Suppose $G$ is either classical or $G_2.$ We verify their assumption and state the explicit form of the summation formula.

Fix a nontrivial adelic character $\psi=\otimes_v \psi_v:E\backslash \mathbb{A}_E\to \cc^\times$. For each place $v$ of $E$, fix a maximal compact subgroup $K_v$ of $G(E_v)$ such that $G(E_v)=B(E_v)K_v$. If $v$ is finite, choose $K_v=G(\calo_v)$. Let
\begin{align*}
    \mathcal{S}(X_P(E_\infty))&:=\widehat{\bigotimes_{v|\infty}}\, \mathcal{S}(X_P(E_v))
\end{align*}
be the completed projective topological tensor product. Define
\begin{align*}
    \mathcal{S}(X_P(\mathbb{A}_E^\infty))&:={\bigotimes_{v\nmid \infty}}^\prime \mathcal{S}(X_P(E_v))
\end{align*}where the restricted tensor product is taken with respect to the basic function $b_v$. Here $b_v$ is the right $K_v$-invariant function in $\mathcal{S}(X_P(E_v))$ whose Mellin transform (along the trivial character) is $d(1_{vs})\phi_{vs}$, where $\phi_{vs}\in I_P(1_{vs})$ is the unique right $K_v$-invariant section such that $\phi_{vs}\big|_{K_v}=1$. Define the adelic Schwartz space
\begin{align*}
    \mathcal{S}(X_P(\mathbb{A}_E))&:=\mathcal{S}(X_P(E_\infty))\otimes \mathcal{S}(X_P(\mathbb{A}_E^\infty)).
\end{align*}
By Theorem \ref{thm:exact}, for $v\nmid\infty $ we can equip $\mathcal{S}(X_P(E_v))$ with the natural algebraic topology in the sense of \cite[Definition 4.1]{Li:Schwartz}. Then we endow $\mathcal{S}(X_P(\mathbb{A}_E^\infty))$ and $\mathcal{S}(X_P(\mathbb{A}_E))$ with the inductive tensor product topology. 

\begin{Lem}
The adelic Schwartz space $\mathcal{S}(X_P(\mathbb{A}_E))$ is nuclear, barrelled and separable, and the adelic Fourier transform $\mathcal{F}_{P|P^{\mathrm{op}}}$ is continuous.
\end{Lem}

\begin{proof}
Topological properties of  $\mathcal{S}(X_P(\mathbb{A}_E))$ follow from Lemma \ref{lem:nuclear}, \cite[Lemma 4.1.3]{Li:Schwartz} and its proof. Continuity of $\mathcal{F}_{P|P^{\mathrm{op}}}$ follows from Theorem \ref{Thm: Fourier formula} and the nature of algebraic topological spaces.
\end{proof}

For $f\in \mathcal{S}(X_P(\mathbb{A}_E))$ and an idelic (unitary) character $\omega=\otimes_v \omega_v:\A_E^\times\to \cc^\times$, the degenerate Eisenstein series
\begin{align*}
    \textrm{Eis}(g;f_{\omega_s}):=\sum_{\gamma\in P\backslash G(E)} f_{\omega_s}(\gamma g)
\end{align*}
converges absolutely for $g\in G(\mathbb{A}_E)$ and $\mathrm{Re}(s)> \beta$ for some constant $\beta\in \rr_{>0}$ independent of $f$ and $\omega$. For ease of notation, we let $\textrm{Eis}(f_{\omega_s}):=\textrm{Eis}(\mathrm{Id};f_{\omega_s})$. 

Recall notations introduced in \S \ref{sec:casebycase}. For $w\in W/W_M,$ let
\begin{align*}
    d(\omega_s):=\prod_v d(\omega_{vs}),\quad c_w(\omega_s):=\prod_v c_{w}(\omega_{vs}).
\end{align*}
By Theorem \ref{thm:poles} $d(\omega_s)$ and $d(\omega_s)c_{w}(\omega_s)$ are finite products of $L$-functions and thus define meromorphic functions on the complex plane. Define
\begin{align*}
    e(\omega_{s_0}):=\max_{w\in W/W_M}\mathrm{ord}_{s=s_0} d(\omega_{s})c_w(\omega_{s}).
\end{align*}
Consider the following finite set of idelic quasi-characters
 $$J(G,P,E):=\bigg\{\omega_{s_0}: e(\omega_{s_0})>0 \bigg\}.$$
We first verify Conjecture 1.4 in \cite{Choie:Getz}. More precisely, we show that for $K$-finite $f,$ possible poles of the Eisenstein series $\mathrm{Eis}(f_{\omega_s})$ are at $\omega_{s_0}\in J(G,P,E)$ with orders at most $e(\omega_{s_0})$.

    According to \cite{Langlands:eisen}, $\textrm{Eis}(f_{\omega_s})$ has the same poles (counting multiplicities) as its constant term \begin{align}\label{eq:constpole}
    \sum_{w\in W/W_M} M_w(\omega_s) f_{\omega_s}.
\end{align}
By Gindikin-Karpelevi\v{c} formula \cite[Proposition 4.6]{Lai} for finite places $v$ such that the conductor of $\psi_v$ is $\mathcal{O}_v$
\begin{align*}
    \left(d(1_s)c_w(1_s)\right)^{-1}M_w(1_s) (b_v)_{1_s}|_{G(\calo_v)}=1.
\end{align*} 
Therefore 
\begin{align*}
    \left(d(\omega_s)c_w(\omega_s)\right)^{-1}M_w(\omega_s) f_{\omega_s}
\end{align*}
is entire for all $w\in W/W_M$ by Corollary \ref{cor:more}, so $\mathrm{ord}_{s=s_0}\textrm{Eis}(f_{\omega_s})\le e(\omega_{s_0})$. Now we can apply \cite[Theorem 4.4]{Choie:Getz} to conclude

\begin{Thm}\label{thm:poissonweak}
        Suppose $G$ is either classical or $G_2$. Let $f\in \mathcal{S}(X_P(\mathbb{A}_E))$. Then $\mathrm{Eis}(f_{\omega_s})$ admits a meromorphic continuation and satisfies a functional equation
\begin{align*}
    \mathrm{Eis}(f_{\omega_s})=\mathrm{Eis}(\mathcal{F}_{P|P^{\mathrm{op}}}(f)_{\omega_{s}}^{\mathrm{op}})=\mathrm{Eis}(\mathcal{F}_{P|P^{\mathrm{op}}}(f)_{\omega_{s}^{-1}}).
\end{align*}
Possible poles of $\mathrm{Eis}(f_{\omega_s})$ are contained in $J(G,P,E)$ and their orders are at most $e(\omega_{s})$. Furthermore,
\begin{align*}
        &\sum_{x\in X_P^\circ(E)} f(x)-\frac{1}{\kappa_E}\sum_{\substack{\omega_{s_0}\in J(G,P,E) \\ \mathrm{Re}(s_0)\le 0}} \left(1-\frac{1}{2}\delta_{\mathrm{Re}(s_0)}\right)\mathrm{Res}_{s=s_0} \mathrm{Eis}(f_{\omega_s})\\
        &=\sum_{x\in X_{P^{\mathrm{op}}}^\circ(E)} \mathcal{F}_{P|P^{\mathrm{op}}}(f)(x)-\frac{1}{\kappa_E}\sum_{\substack{\omega_{s_0}\in J(G,P,E)\\ \mathrm{Re}(s_0)\le 0}} \left(1-\frac{1}{2}\delta_{\mathrm{Re}(s_0)}\right)\mathrm{Res}_{s=s_0}   \mathrm{Eis}(\mathcal{F}_{P|P^{\mathrm{op}}}(f)_{\omega_{s}}),
    \end{align*}
and each term above is continuous in $f$.
Here \begin{align*}
    \kappa_E:=\mathrm{Res}_{s=1}\, \zeta_E(s),
\end{align*}
and $\delta$ is the Kronecker delta.\qed
\end{Thm}

We remark that there may be some $\omega_{s_0}\in J(G,P,E)$ such that $\mathrm{Eis}(f_{\omega_s})$ is holomorphic at $s=s_0$ for all $f$. This is due to the cancellation occurred in the sum \eqref{eq:constpole}, which is not taken into account of in the proof above. A more precise version of the summation formula is predicted by the following.

\begin{Conj}\label{conj:vanish}
    Let $v_1,v_2$ be two distinct places of $E$. Let $f=\otimes f_v \in \mathcal{S}(X_P(\mathbb{A}_E))$ such that $f_{v_1}\in \mathcal{S}(X_P^\circ(E_{v_1}))$ and $\mathcal{F}_{P|P^{\mathrm{op}}}(f_{v_2})\in \mathcal{S}(X_{P^{\mathrm{op}}}^\circ(E_{v_2})).$ Then
    \begin{align*}
        \sum_{\gamma\in X_P^\circ(E)}f(\gamma)=\sum_{\gamma\in X_{P^{\mathrm{op}}}^\circ(E)} \mathcal{F}_{P|P^{\mathrm{op}}}(f)(\gamma).
    \end{align*}
\end{Conj}

   Conjecture \ref{conj:vanish} was already claimed in \cite[Theorem 6.4]{BKnormalized} whose proof only made use of functional equations of $\mathrm{Eis}(f_{\omega_s}).$ On the other hand, by a standard argument Conjecture \ref{conj:vanish} is equivalent to claiming $\mathrm{Eis}(f_{\omega_s})$ is entire for all $\omega$, which does not follow from functional equations of degenerate Eisenstein series. Therefore, it seems to us more explanations are required to prove  Conjecture \ref{conj:vanish}.

Consider the (finite) set of idelic characters
\begin{align*}
     R(G,P,E)&:=\big\{\omega_{s_0}: \mathrm{ord}_{s=s_0} d(\omega_{vs})>0 \textrm{ for all places } v \big\}\\
     &=\bigg\{\omega_{s_0}: \sum_{i=1}^k\mathrm{ord}_{s=s_0} L_v(1+s_i,\omega_{s}^{\lambda_i})>0 \textrm{ for all places } v \bigg\}.
\end{align*}

\begin{Prop}\label{Prop:Poisson}
    Conjecture \ref{conj:vanish} is equivalent to the following statement: Let $f\in \mathcal{S}(X_P(\mathbb{A}_E))$.
 The only possible poles of $\mathrm{Eis}(f_{\omega_s})$ on the left half-plane $\mathrm{Re}(\omega_{s_0})\le 0$ are contained in $R(G,P,E)$ and their orders are at most 
\begin{align*}
\min_v \ord_{s=s_0} d(\omega_{vs})=\min_v\sum_{i=1}^k\mathrm{ord}_{s=s_0} L_v(1+s_i,\omega_{s}^{\lambda_i}).
\end{align*}
\end{Prop}

\begin{proof}
    Let $v_1,v_2$ be two distinct places of $E$. By linearity and continuity, we may assume $f=\otimes_v f_v$ is $K$-finite. Let $f'=f_{v_1}'f_{v_2}'f^{v_1v_2}\in \mathcal{S}(X_P(\mathbb{A}_E))$ be the function such that for all characters $\chi$
\begin{align*}
    (f'_{v_1})_{\chi_s}&=\frac{(f_{v_1})_{\chi_s}}{d(\chi_s)},\quad \mathcal{F}_{P|P^{\mathrm{op}}}(f'_{v_2})_{\chi_s}=\frac{\mathcal{F}_{P|P^{\mathrm{op}}}(f_{v_2})_{\chi_s}}{d(\chi_s)}.
\end{align*}
Then $f'_{v_1}\in \mathcal{S}(X_P^\circ(E_{v_1}))$ and $\mathcal{F}_{P|P^{\mathrm{op}}}(f'_{v_2})\in \mathcal{S}(X_{P^{\mathrm{op}}}^\circ(E_{v_2}))$. Then Conjecture \ref{conj:vanish} is equivalent to the assertion that
\begin{align*}
    \frac{\mathrm{Eis}(f_{\omega_s})}{d(\omega_{v_1s})d(\omega_{v_2s}^{-1})}
\end{align*}
is entire for all places $v_1,v_2$. 
\end{proof}

The rest of the section is devoted to proving
\begin{Thm}\label{thm:poisson}
    Conjecture \ref{conj:vanish} holds if $G$ is either classical or $G_2$. In particular,
    \begin{align*}
        &\sum_{x\in X_P^\circ(E)} f(x)-\frac{1}{\kappa_E}\sum_{\omega_{s_0}\in R(G,P,E)} \mathrm{Res}_{s=s_0}\mathrm{Eis}(f_{\omega_s})\\
        &=\sum_{x\in X_{P^{\mathrm{op}}}^\circ(E)} \mathcal{F}_{P|P^{\mathrm{op}}}(f)(x)-\frac{1}{\kappa_E}\sum_{\omega_{s_0}\in R(G,P,E)}  \mathrm{Res}_{s=s_0}\mathrm{Eis}(\mathcal{F}_{P|P^{\mathrm{op}}}(f)_{\omega_s}).
    \end{align*}
\end{Thm}

This Poisson summation formula follows from Theorem \ref{thm:poissonweak} and Proposition \ref{Prop:Poisson}. Our proof of Conjecture \ref{conj:vanish} is modified from the proof of \cite[Proposition 1.6]{Ikeda:poles:triple}. While the proof works for all global fields, in below we assume $E$ is a number field and leave the function field case to the reader.

The case $G=G_2$ is proved in Appendix \ref{appendix:B}. In the rest of the section $G$ is classical, and we continue to use the Bourbaki numbering for Dynkin diagrams. Consider first the case $P=P_1$. The assertion is clear when $G$ is of type $A$ or $C$. For $G$ of type $D$, it is proved in \cite{GK:auto} when $E$ is a function field (see also \cite{Getz:Quad} for a geometric form of the Poisson summation formulae over a number field). We defer the proof for $G$ of type $B$ to Appendix \ref{appendix:B}. One can prove the assertion for type $D$ in a similar fashion,  which is analogous to the proof of type $B,$ so we leave it to the reader.

\subsection{Proof of Conjecture \ref{conj:vanish} for classical types}

Suppose $G$ is classical and $P=P_\ell$ where $\ell>1$. We may assume $f=\otimes f_v\in \mathcal{S}(X_P(\A_E))$ is $K$-finite. Let $S$ be a finite set of places of $E$ such that if $v\not\in S$ then $v$ is finite, $\psi_v$ has conductor $\mathcal{O}_v$ and $f_v=b_v$. For an idelic character $\omega=\otimes \omega_v$, by the functional equation of the degenerate Eisenstein series and Proposition \ref{Prop:Poisson}, it suffices to study poles of $\mathrm{Eis}(f_{\omega_s})$ on $\mathrm{Re}(s)\ge 0$, and since $G$ is classical we may further assume $\omega^2=1$. As $d(\omega_{vs})$ is holomorphic on $\mathrm{Re}(s)\ge 0$ for any place $v$, we may assume $f_v\in \mathcal{S}(X_P^\circ(E_v))$ for $v\in S$.

Let $P_1$ be the standard maximal parabolic subgroup of $G$ associated to the first node and let $P':=P_1\cap P$. We have a natural map $X_{P'}^\circ\to X_P^\circ$. Let $v$ be a place of $E$ and $h\in C^\infty(X_{P'}^\circ(E_v))$. For $t,s\in \cc,$ $\chi$ a character and $g\in G(E_v)$, we let
\begin{align*}
    h_{t,\chi_s}(g):=\int_{P'^{\mathrm{ab}}(E_v)} \delta_{P'}^{1/2}(m)|\omega_1(m)|^t \chi_s(\omega_P(m))h(m^{-1}g) dm,
\end{align*}
whenever the integral converges absolutely in an open region of $\cc^2$ and extends meromorphically.

For $v\in S$ choose a $K$-finite function $\widetilde{f}_v\in \mathcal{S}(X_{P'}^\circ(E_v))$ such that for $m\in P'(E_v)$ and $k\in K$
\begin{align*}
    (\widetilde{f}_{v})_{t,\chi_s}(mk)=\delta_{P'}^{1/2}(m)|\omega_1(m)|^t\chi_s(\omega_P(m))(f_{v})_{\chi_{s+c/2}}(k),
\end{align*}
where 
\begin{align*}
    c:=\begin{cases}
        2 & \textrm{if } G=\mathrm{Spin}_{2n} \textrm{ or } \mathrm{Spin}_{2n+1}, \textrm{ and } \ell=n,\\
        1 & \mathrm{otherwise}.
    \end{cases}.
\end{align*}
Note that $(\widetilde{f}_{v})_{t,\chi_s}$ is a holomorphic section of $\mathrm{Ind}_{P'}^G(|\omega_1|^t\chi_s(\omega_P))$.
 For $v\not\in S$, let $\widetilde{b}_v$ be the basic function on $X_{P'}(E_v)$ (denoted as $c_{P',0}$ in \cite{BKnormalized}). Let $\widetilde{f}:=\otimes_{v\in S} \widetilde{f}_v \otimes (\otimes_{v\not\in S} \widetilde{b}_v)$.
Consider the Eisenstein series
\begin{align*}
    \mathrm{Eis}(\widetilde{f}_{t,\omega_{s}}):=\sum_{\gamma\in P\backslash G(E)} \widetilde{f}_{t,\omega_{s}}(\gamma).
\end{align*}
It converges absolutely and defines a holomorphic function on 
\begin{align*}
    \Omega_1:=\{(t,s):\mathrm{Re}(s),\mathrm{Re}(t) \gg 1\},
\end{align*}
and it extends to a meromorphic function on $\cc^2$.

We can write
\begin{align}\label{eq:innerM}
     \mathrm{Eis}(\widetilde{f}_{t,\omega_{s}})=\sum_{\gamma'\in P\backslash G(E)}\sum_{\gamma\in P'\backslash P(E)} \widetilde{f}_{t,\omega_{s}}(\gamma \gamma').
\end{align}
The inner sum is a degenerate Eisenstein series of $\mathrm{GL}_{\ell}$ attached to the first node. It can also be written as
\begin{align}\label{eq:innerM1}
     \mathrm{Eis}(\widetilde{f}_{t,\omega_{s}})=\sum_{\gamma'\in P_1\backslash G(E)}\sum_{\gamma\in P'\backslash P_1(E)} \widetilde{f}_{t,\omega_{s}}(\gamma \gamma').
\end{align}
Let $M_1$ be the Levi subgroup of $P_1$ such that $T\le M_1$. Then the inner sum above is a degenerate Eisenstein series of $M_1$ attached to the $(\ell-1)$th node. Observe that $M_1$ and $G$ are of the same type and this is where we use the induction hypothesis.

In below we prove the case for $G$ of type $A$ or $B$. The rest of the cases are similar, so we only include necessary discussions and leave the details to the reader.

\subsubsection{Type $A_n$} We may assume $\omega=1$ and $2\le \ell\le (n+1)/2$ by symmetry. We need to show $\mathrm{Eis}(f_{1_s})$ only has possible simple poles at $s=\frac{n+1}{2},\ldots,\frac{n-2\ell+3}{2}$ for $\mathrm{Re}(s)\ge 0$. In this case
\begin{align*}
    (\widetilde{b}_v)_{t,1_s}(\mathrm{Id})=\zeta_v\left(t+\frac{\ell}{2}\right)\zeta_v\left(t+s+\frac{n-\ell}{2}+1\right)\prod_{j=0}^{\ell-2} \zeta_v\left(s+\frac{n}{2}-j\right).
\end{align*}
By the functional equation of the inner Eisenstein series in \eqref{eq:innerM}, $\mathrm{Eis}(\widetilde{f}_{t,1_s})$ is also holomorphic on 
\begin{align*}
    \Omega_2:=\bigg\{(t,s) : \mathrm{Re}(t)\ll -1,\, \mathrm{Re}(t+s)\gg 1\bigg\}
\end{align*}
and extends meromorphically to the convex closure of $\Omega_1\cup \Omega_2$ with singularities along  $t=\pm \tfrac{\ell}{2}$. Note that the expression \eqref{eq:innerM} implies $\mathrm{Res}_{t=\frac{\ell}{2}}\mathrm{Eis}(\widetilde{f}_{t,1_s})$ is $\mathrm{Eis}(f_{1_{s+\frac{1}{2}}})$ up to a nonzero constant.

We can further rewrite the Eisenstein series above (after applying the functional equation) as a double sum, whose inner sum is a degenerate Eisenstein series of $\mathrm{GL}_{n-\ell+2}$ attached to the first node. Thus by the functional equation of this inner Eisenstein series, we see that $\mathrm{Eis}(\widetilde{f}_{t,1_s})$ is also holomorphic on 
\begin{align*}
    \Omega_3:=\bigg\{(t,s) : \mathrm{Re}(t)\ll -1, \mathrm{Re}(t+s)\ll -1,\mathrm{Re}(s)\gg 1 \bigg\}
\end{align*}
and extends meromorphically to the convex closure of $\Omega_2\cup \Omega_3$ with singularities along $t+s=\pm\frac{n-\ell+2}{2}$. 

On the other hand, the inner sum in \eqref{eq:innerM1} is
\begin{align*}
    \zeta^S\left(t+\frac{\ell}{2}\right)\zeta^S\left(t+s+\frac{n-\ell}{2}+1\right)
\end{align*}
times a degenerate Eisenstein series on $\mathrm{GL}_{n}$ associated to a good section of $I_{\mathrm{GL}_n\cap P_\ell}^{\mathrm{GL}_n}(1_s)$. By the induction hypothesis, the pole of this Eisenstein series has at most simple poles at
\begin{align}\label{eq:poleA}
    s= \frac{n}{2}, \ldots, \frac{n-2\ell+4}{2}.
\end{align}
By the functional equation of this Eisenstein series, $\mathrm{Eis}(\widetilde{f}_{t,1_s})$ is also holomorphic on  
\begin{align*}
    \Omega_4:=\{(t,s) : \mathrm{Re}(t+s)\gg 1, \mathrm{Re}(s)\ll -1\}
\end{align*}
and extends meromorphically to the convex closure of $\Omega_1\cup \Omega_4$ with possible singularities \eqref{eq:poleA}.

 We can conclude that singularities of $\mathrm{Eis}(\widetilde{f}_{t,1_s})$ in the convex closure $\Omega$ of $\cup_{i=1}^4 \Omega_i$ are \eqref{eq:poleA}, $t=\pm \frac{\ell}{2}$ and $t+s=\pm\frac{n-\ell+2}{2}.$ Note that $\Omega$ contains the line  $t=\frac{\ell}{2}$. It follows that on $\mathrm{Re}(s)\ge -\frac{1}{2}$, $\mathrm{Eis}(f_{1_{s+\frac{1}{2}}})$ has possible simple poles at $s\in \{\frac{n}{2},\ldots ,\frac{n-2\ell+2}{2}\}$. This completes the proof.\qed

\subsubsection{Type $B_n$ $(n\ge 3)$}

Suppose $1<\ell<n$ and $\omega^2=1$. In this case we can assume $G=\mathrm{SO}_{2n+1}$, and we have
\begin{align*}
     (\widetilde{b}_v)_{t,\omega_{vs}}(\mathrm{Id})&=\zeta_v(2t+2s+1)\zeta_v\left(t+\frac{\ell}{2}\right)\zeta_v\left(t+2s+\frac{\ell}{2}\right)L_v\left(t+n-\ell+\frac{1}{2},\omega_s\right)\\
    &\times\prod_{i=0}^{\min(\ell-2,2n-2\ell-1)}L_v\left(\frac{2n-\ell-1}{2}-i,\omega_s\right)\prod_{i=0}^{\lfloor \ell/2\rfloor-1}\zeta_v(2s+\ell-1-2i).
\end{align*}
By the functional equation of the inner Eisenstein series in \eqref{eq:innerM}, $\mathrm{Eis}(\widetilde{f}_{t,\omega_s})$ is also holomorphic on 
\begin{align*}
    \Omega_2:=\bigg\{(t,s) : \mathrm{Re}(t)\ll -1,\, \mathrm{Re}(t+s)\gg 1,\mathrm{Re}(s)\gg 1\bigg\}
\end{align*}
and extends meromorphically to the convex closure of $\Omega_1\cup \Omega_2$ with possible singularities $t=\pm \frac{\ell}{2}.$ Then $\mathrm{Res}_{t=\frac{\ell}{2}}\mathrm{Eis}(\widetilde{f}_{t,1_s})$ is up to a nonzero constant 
\begin{align}\label{eq:Bres}
\begin{split}
    \zeta^S(2s+\ell)\begin{cases}
         \mathrm{Eis}(f_{\omega_{s+\frac{1}{2}}}) & \textrm{if }\ell-2< 2n-2\ell-1, 2\nmid \ell,\\
        L^S(\frac{-2n+3\ell+1}{2},\omega_s)\mathrm{Eis}(f_{\omega_{s+\frac{1}{2}}}) & \textrm{if }\ell-2\ge  2n-2\ell-1, 2\nmid \ell,\\
        \zeta^S(2s+1)\mathrm{Eis}(f_{\omega_{s+\frac{1}{2}}}) & \textrm{if }\ell-2< 2n-2\ell-1, 2|\ell,\\
       \zeta^S(2s+1)L^S(\frac{-2n+3\ell+1}{2},\omega_s)\mathrm{Eis}(f_{\omega_{s+\frac{1}{2}}}) & \textrm{if }\ell-2\ge  2n-2\ell-1, 2|\ell.
    \end{cases}
\end{split}
\end{align}

We can further rewrite the Eisenstein series above (after applying the functional equation) as a double sum, whose inner sum is a degenerate Eisenstein series of $\mathrm{SO}_{2n-2\ell+3}$ attached to the first node. Thus we see that $\mathrm{Eis}(\widetilde{f}_{t,\omega_s})$ is holomorphic on 
\begin{align*}
    \Omega_3:=\bigg\{(t,s) : \mathrm{Re}(t+2s)\gg 1,  \mathrm{Re}(t+s)\ll -1\bigg\}
\end{align*}
and extends meromorphically to the convex closure of $\Omega_2\cup \Omega_3$ with possible singularities $t+s=\pm(n-\ell+\frac{1}{2})$ if $\omega=1$, and $t+s=\pm \frac{1}{2}$. Rewriting further the Eisenstein series as another double sum, whose inner sum is a degenerate Eisenstein series of $\mathrm{GL}_{\ell}$ attached to the first node, we see that $\mathrm{Eis}(\widetilde{f}_{t,\omega_s})$ is holomorphic on 
\begin{align*}
    \Omega_4:=\bigg\{(t,s) : \mathrm{Re}(t+2s)\ll -1, \mathrm{Re}(s)\gg 1\bigg\}
\end{align*}
and extends meromorphically to the convex closure of $\Omega_3\cup \Omega_4$ with possible singularities at $t+2s=\pm \frac{\ell}{2}$. 

On the other hand, the inner sum of \eqref{eq:innerM1} is
\begin{align*}
    \zeta^S(2t+2s+1)\zeta^S\left(t+\frac{\ell}{2}\right)\zeta^S\left(t+2s+\frac{\ell}{2}\right)L^S\left(t+n-\ell+\frac{1}{2},\omega_s\right)
\end{align*}
times a degenerate Eisenstein series of $\mathrm{SO}_{2n-1}$ associated to a good section of $I^{\mathrm{SO}_{2n-1}}_{P\cap \mathrm{SO}_{2n-1}}(\omega_s)$. By the induction hypothesis, the possible poles of this Eisenstein series are (counting multiplicities)
\begin{align}\label{eq:Bpole}
\begin{split}
    &\left\{ \frac{2n-\ell-1}{2} ,\ldots,  \frac{2n-\ell-1-2\min(\ell-2,2n-2\ell-1)}{2} \right\}\\
    &+ \left\{ \frac{\ell-1}{2}, \ldots,  \frac{\ell+1-2\lfloor \ell/2\rfloor}{2}\right\}\quad \textrm{ if } \omega=1, \textrm{ and }\\
    &\left\{ \frac{\ell-1}{2}, \ldots,  \frac{\ell+1-2\lfloor \ell/2\rfloor}{2}\right\} \quad \textrm{ if } \omega\neq 1.
\end{split}
\end{align}
By the functional equation of the inner Eisenstein series, $\mathrm{Eis}(\widetilde{f}_{t,\omega_s})$ is also holomorphic on  
\begin{align*}
    \Omega_5:=\bigg\{(t,s) :  \mathrm{Re}(t+2s)\gg 1, \mathrm{Re}(s)\ll -1\bigg\}
\end{align*}
and extends meromorphically to the convex closure of $\Omega_1\cup \Omega_5$ with possible singularities \eqref{eq:Bpole}.

 We conclude that possible singularities of $\mathrm{Eis}(\widetilde{f}_{t,\omega_s})$ in the convex closure $\Omega$ of $\cup_{i=1}^5 \Omega_i$ are \eqref{eq:Bpole}, $t=\pm \frac{\ell}{2}, t+s=\pm \frac{1}{2}, t+2s=\pm\frac{\ell}{2},$ and $t+s=\pm \left(n-\ell+\frac{1}{2}\right)$ if $\omega=1$. Note that $\Omega$ contains the line  $t=\frac{\ell}{2}$. It follows that on $\mathrm{Re}(s)\ge -1/2$, the possible poles of $\mathrm{Res}_{t=\frac{\ell}{2}}\mathrm{Eis}(\widetilde{f}_{t,\omega_s})$ are the sum of the multisets \eqref{eq:Bpole}, $\{0\}$,
 and 
 \begin{align*}
  \left\{n-\tfrac{3\ell}{2}+\tfrac{1}{2}\right\}, \quad\textrm{ if }\omega=1, n-\tfrac{3\ell}{2}\ge -1.
 \end{align*}

To prove the assertion, we need to show the following.
\begin{enumerate}    
     \item Suppose $\omega=1$ and $\ell-2\ge 2n-2\ell-1$. Then $\mathrm{Eis}(f_{1_s})$ is holomorphic at $s=n-\frac{3\ell}{2}+1$.

     \item When $\ell$ is even $\mathrm{Eis}(f_{\omega_s})$ is holomorphic at $s=\frac{1}{2}$.
\end{enumerate}

For (1) observe that $\ell>2$ and $n-\frac{3\ell}{2}+1$ is either $0$ or $\frac{1}{2}$
depending on whether $\ell$ is even or odd. Since $L^S(\frac{-2n+3\ell+1}{2},1_s)$ has a simple pole at $s=n-\frac{3\ell-1}{2},$ and when $\ell$ is even $\zeta^S(2s+1)$ is holomorphic at $s=-\frac{1}{2},$ assertion (1) follows from \eqref{eq:Bres}. Since $\zeta^S(2s+1)$ has a simple pole at $s=0$, (2) follows from the proof of (1) and \eqref{eq:Bres}.

Now suppose $\ell=n$. We may assume $\omega=1$. Then
\begin{align*}
     (\widetilde{b}_v)_{t,1_{s}}&=\zeta_v\left(t+\frac{n}{2}\right)\zeta_v(2t+s+1)\zeta_v\left(t+s+\frac{n}{2}\right)\prod_{i=0}^{\lfloor n/2\rfloor -1} \zeta_v(s+n-1-2i)
\end{align*}
By the functional equation of the inner Eisenstein series in \eqref{eq:innerM}, $\mathrm{Eis}(\widetilde{f}_{t,1_s})$ is also holomorphic on 
\begin{align*}
    \Omega_2:=\bigg\{(t,s) : \mathrm{Re}(t)\ll -1,\, \mathrm{Re}(2t+s)\gg 1\bigg\}
\end{align*}
and extends meromorphically to the convex closure of $\Omega_1\cup \Omega_2$ with possible singularities $t=\pm \tfrac{n}{2}.$ Then $\mathrm{Res}_{t=\frac{n}{2}}\mathrm{Eis}(\widetilde{f}_{t,1_s})$ is up to a nonzero constant
\begin{align}\label{eq:Bresn}
    \zeta^S(s+n)\begin{cases}
         \mathrm{Eis}(f_{1_{s+1}}) &  \textrm{ if } 2\nmid n,\\
        \zeta^S(s+1)\mathrm{Eis}(f_{1_{s+1}}) &  \textrm{ if }2 \mid n.
    \end{cases}
\end{align}
We can further rewrite the Eisenstein series above (after applying the functional equation) as a double sum, whose inner sum is a degenerate Eisenstein series of $\mathrm{GSpin}_{3}=\mathrm{GL}_2$ attached to the unique node. Thus we see that $\mathrm{Eis}(\widetilde{f}_{t,1_s})$ is holomorphic on 
\begin{align*}
    \Omega_3:=\bigg\{(t,s):\mathrm{Re}(2t+s)\ll -1, \mathrm{Re}(t+s)\gg 1,\mathrm{Re}(s)\gg 1\bigg\}
\end{align*}
and extends meromorphically to the convex closure of $\Omega_2\cup \Omega_3$ with possible singularities $2t+s=\pm 1$. Rewriting further the Eisenstein series as another double sum, whose inner sum is a degenerate Eisenstein series of $\mathrm{GL}_{n}$ attached to the first node, we see that $\mathrm{Eis}(\widetilde{f}_{t,1_s})$ is holomorphic on 
\begin{align*}
    \Omega_4:=\bigg\{(t,s):\mathrm{Re}(t+s)\ll -1,\mathrm{Re}(s)\gg 1\bigg\},
\end{align*}
and extends meromorphically to the convex closure of $\Omega_3\cup \Omega_4$ with possible singularities $t+s=\pm \frac{n}{2}$.

On the other hand, the inner sum of \eqref{eq:innerM1} is
\begin{align*}
\zeta^S\left(t+\frac{n}{2}\right)\zeta^S(2t+s+1)\zeta^S\left(t+s+\frac{n}{2}\right)
\end{align*}
times a degenerate Eisenstein series of $\mathrm{GSpin}_{2n-1}$ associated to a good section of $I^{\mathrm{GSpin}_{2n-1}}_{P\cap \mathrm{GSpin}_{2n-1}}(1_s)$. By the induction hypothesis, possible poles of this Eisenstein series are
\begin{align}\label{eq:Bpolen}
 \left\{ n-1, \ldots,  n+1-2\lfloor n/2 \rfloor\right\}.
\end{align}
By the functional equation of the inner Eisenstein series, $\mathrm{Eis}(\widetilde{f}_{t,1_s})$ is also holomorphic on  
\begin{align*}
    \Omega_5:=\bigg\{(t,s) :  \mathrm{Re}(t+s)\gg 1, \mathrm{Re}(s)\ll -1\bigg\}
\end{align*}
and extends meromorphically to the convex closure of $\Omega_1\cup \Omega_5$ with possible singularities \eqref{eq:Bpolen}.

 We conclude that possible singularities of $\mathrm{Eis}(\widetilde{f}_{t,1_s})$ in the convex closure $\Omega$ of $\cup_{i=1}^5 \Omega_i$ are \eqref{eq:Bpolen}, $t=\pm \frac{n}{2}, 2t+s=\pm 1$ and $t+s=\pm\frac{n}{2}$. Note that $\Omega$ contains the line  $t=\frac{n}{2}$. It follows that on $\mathrm{Re}(s)\ge -1$, possible poles of $\mathrm{Res}_{t=\frac{n}{2}}\mathrm{Eis}(\widetilde{f}_{t,1_s})$ are the sum of the sets \eqref{eq:Bpolen} and $\{0\}$. We need to show when $n$ is even, $\mathrm{Eis}(f_{1_s})$ is holomorphic at $s=1$. Since $\zeta^S(s+1)$ has a simple pole at $s=0,$ the assertion follows from \eqref{eq:Bresn}.\qed

\subsubsection{Type $C_n$} Let $\ell>1$ and $\omega^2=1$. Then
\begin{align*}
(\widetilde{b}_v)_{t,\omega_{vs}}&=\zeta_v\left(t+\frac{\ell}{2}\right)\zeta_v\left(t+2s+\frac{\ell}{2}\right)L_v(t+n-\ell+1,\omega_s)\\
    &\times\prod_{i=0}^{\min(\ell-2,2n-2\ell)}L_v\left(\frac{2n-\ell}{2}-i,\omega_s\right)\prod_{i=1}^{\lfloor (\ell-1)/2\rfloor}\zeta_v(2s+\ell-2i).
\end{align*}
The Eisenstein series $\mathrm{Eis}(\widetilde{f}_{t,\omega_s})$ is holomorphic on the regions
\begin{align*}
    \Omega_2&:=\bigg\{(t,s) : \mathrm{Re}(t)\ll -1,\, \mathrm{Re}(t+s)\gg 1,\mathrm{Re}(s)\gg 1\bigg\},\\
    \Omega_3&:=\bigg\{(t,s) : \mathrm{Re}(t+2s)\gg 1,  \mathrm{Re}(t+s)\ll 1\bigg\},\\
    \Omega_4&:=\bigg\{(t,s) : \mathrm{Re}(t+2s)\ll -1,\mathrm{Re}(s)\gg 1\bigg\},\\
    \Omega_5&:=\bigg\{(t,s) :  \mathrm{Re}(t+2s)\gg 1, \mathrm{Re}(s)\ll -1\bigg\}.
\end{align*}
In the convex closure of $\cup_{i=1}^5 \Omega_i$, its possible singularities are divisors in the $s$-variable
\begin{align*}
      &\left\{ \frac{2n-\ell}{2} ,\ldots,  \frac{2n-\ell-2\min(\ell-2,2n-2\ell)}{2} \right\}\\
    &+ \left\{ \frac{\ell-2}{2}, \ldots,  \frac{\ell-2\lfloor (\ell-1)/2\rfloor}{2}\right\}\quad \textrm{ if } \omega=1, \textrm{ and }\\
    & \left\{ \frac{\ell-2}{2}, \ldots,  \frac{\ell-2\lfloor (\ell-1)/2\rfloor}{2}\right\} \quad \textrm{ if } \omega\neq 1
\end{align*}
together with $t=\pm \frac{\ell}{2},t+2s=\pm\frac{\ell}{2},$ and $t+s=n-\ell+1$ if $\omega=1$. The residue $\mathrm{Res}_{t=\frac{\ell}{2}}\mathrm{Eis}(\widetilde{f}_{t,\omega_s})$ is up to a nonzero constant 
\begin{align*}
   \begin{cases}
         \mathrm{Eis}(f_{\omega_{s+1/2}}) & \textrm{if } \ell-2< 2n-2\ell, 2\mid \ell,\\
        L^S(-n+\frac{3\ell}{2},\omega_s)\mathrm{Eis}(f_{\omega_{s+1/2}}) & \textrm{if }\ell-2\ge  2n-2\ell, 2\mid \ell,\\
        \zeta^S(2s+1)\mathrm{Eis}(f_{\omega_{s+1/2}}) & \textrm{if }\ell-2< 2n-2\ell, 2\nmid\ell,\\
       \zeta^S(2s+1)L^S(-n+\frac{3\ell}{2},\omega_s)\mathrm{Eis}(f_{\omega_{s+1/2}}) & \textrm{if }\ell-2\ge  2n-2\ell, 2\nmid\ell.
    \end{cases}
\end{align*}
\qed

\subsubsection{Type $D_n$ $(n\ge 4)$} 

Suppose $n-1>\ell>1$ and $\omega^2=1$. In this case we can assume $G=\mathrm{SO}_{2n},$ and we have
\begin{align*}
     (\widetilde{b}_v)_{t,\omega_{vs}}&=\zeta_v\left(t+\frac{\ell}{2}\right)\zeta_v\left(t+2s+\frac{\ell}{2}\right)L_v\left(t+n-\ell,\omega_s\right)L_v\left(t+1,\omega_s\right)\\
    &\times L_v\left(\frac{\ell}{2},\omega_s\right)\prod_{i=0}^{\min(\ell-2,2n-2\ell-2)}L_v\left(\frac{2n-\ell-2}{2}-i,\omega_s\right)\prod_{i=0}^{\lfloor (\ell-3)/2\rfloor}\zeta_v(2s+\ell-2i-2).
\end{align*}
The Eisenstein series $\mathrm{Eis}(\widetilde{f}_{t,\omega_s})$ is holomorphic on the regions
\begin{align*}
    \Omega_2&:=\bigg\{(t,s) : \mathrm{Re}(t)\ll-1,\mathrm{Re}(t+s)\gg 1,\mathrm{Re}(s)\gg 1\bigg\},\\
    \Omega_3&:=\bigg\{(t,s) : \mathrm{Re}(t+2s)\gg 1,  \mathrm{Re}(t+s)\ll -1\bigg\},\\
    \Omega_4&:=\bigg\{(t,s) : \mathrm{Re}(t+2s)\ll -1,\mathrm{Re}(s)\gg 1\bigg\},\\
    \Omega_5&:=\bigg\{(t,s) :  \mathrm{Re}(t+2s)\gg 1, \mathrm{Re}(s)\ll -1\bigg\}.
\end{align*}
In the convex closure of $\cup_{i=1}^5 \Omega_i$, its possible singularities are divisors in the $s$-variable
\begin{align*}
      &\left\{ \frac{2n-\ell-2}{2} ,\ldots,  \frac{2n-\ell-2-2\min(\ell-2,2n-2\ell-2)}{2} \right\}+\left\{\frac{\ell}{2}\right\}\\
    &+ \left\{ \frac{\ell-2}{2}, \ldots,  \frac{\ell-2\lfloor (\ell-1)/2\rfloor}{2}\right\}\quad \textrm{ if } \omega=1, \textrm{ and }\\
    & \left\{ \frac{\ell-2}{2}, \ldots,  \frac{\ell-2\lfloor (\ell-1)/2\rfloor}{2}\right\} \quad \textrm{ if } \omega\neq 1
\end{align*}
together with $t=\pm \frac{\ell}{2},t+2s=\pm\frac{\ell}{2},$ and $t+s=1,n-\ell$ if $\omega=1$. 
The residue $\mathrm{Res}_{t=\frac{\ell}{2}}\mathrm{Eis}(\widetilde{f}_{t,\omega_s})$ is up to a nonzero constant 
\begin{align*}
    L^S\left(\frac{\ell}{2},\omega_{s}\right)\begin{cases}
         \mathrm{Eis}(f_{\omega_{s+\frac{1}{2}}}) & \textrm{if } \ell-2< 2n-2\ell-2, 2\mid \ell,\\
        L^S(\frac{-2n+3\ell+2}{2},\omega_s)\mathrm{Eis}(f_{\omega_{s+\frac{1}{2}}}) & \textrm{if } \ell-2\ge  2n-2\ell-2, 2\mid \ell,\\
        \zeta^S(2s+1)\mathrm{Eis}(f_{\omega_{s+\frac{1}{2}}}) & \textrm{if } \ell-2< 2n-2\ell-2, 2\nmid \ell,\\
       \zeta^S(2s+1)L^S(\frac{-2n+3\ell+2}{2},\omega_s)\mathrm{Eis}(f_{\omega_{s+\frac{1}{2}}}) & \textrm{if } \ell-2\ge  2n-2\ell-2, 2\nmid \ell.
    \end{cases}
\end{align*}

By symmetry the remaining case is $\ell=n$ with $n\ge 5$, in which case we can assume $\omega=1$. We have
\begin{align*}
     (\widetilde{b}_v)_{t,1_{s}}&=\zeta_v\left(t+\frac{n}{2}\right)\zeta_v\left(t+s+\frac{n}{2}\right)\prod_{i=0}^{\lfloor \frac{n-3}{2}\rfloor} \zeta_v(s+n-2-2i).
\end{align*}
The Eisenstein series $\mathrm{Eis}(\widetilde{f}_{t,1_s})$ is holomorphic on the regions
\begin{align*}
    \Omega_2&:=\bigg\{(t,s) : \mathrm{Re}(t)\ll -1, \mathrm{Re}(t+s)\gg 1,\mathrm{Re}(s)\gg 1\bigg\},\\
    \Omega_3&:=\bigg\{(t,s) : \mathrm{Re}(t+s)\ll -1,  \mathrm{Re}(s)\gg 1\bigg\},\\
    \Omega_4&:=\bigg\{(t,s) :  \mathrm{Re}(t+s)\gg 1, \mathrm{Re}(s)\ll -1\bigg\}.
\end{align*}
In the convex closure of $\cup_{i=1}^4 \Omega_i$, its possible singularities are divisors in the $s$-variable
\begin{align*}
      &\left\{ n-2 ,\ldots,  n-2-2\lfloor \frac{n-3}{2}\rfloor \right\}
\end{align*}
together with $t=\pm \frac{n}{2},t+2s=\pm\frac{n}{2}$. The residue $\mathrm{Res}_{t=\frac{n}{2}}\mathrm{Eis}(\widetilde{f}_{t,1_s})$ is up to a nonzero constant 
\begin{align*}
   \begin{cases}
         \mathrm{Eis}(f_{1_{s+1}}) & \textrm{if } 2\mid n,\\
        \zeta^S(s+1)\mathrm{Eis}(f_{1_{s+1}}) &  \textrm{if }2\nmid n.
    \end{cases}
\end{align*}
\qed 

\newpage

\appendix

\section{Asymptotics Data}\label{appendix}

Let $G$ be a split, simply connected, almost simple algebraic group. Let $P_\ell$ be the maximal parabolic associated to the $\ell$th node of the Dynkin diagram of $G,$ using the Bourbaki numbering. In below we list $s_k$ and the set 
$\left\{\frac{s_i+1}{\lambda_i}-s_k-1:\lambda_i=\lambda\right\}$ for each $\lambda$ obtained from the multiset $\Lambda=\{(s_i,\lambda_i)\}$. 

\begin{center}

    \begin{tabular}{ccc}
\multicolumn{3}{c}{$A_n$} \\ \hline
Node         & $\lambda=1$            & $s_k$ \\  \hline
$\ell$ &
  \begin{tabular}[c]{@{}c@{}}$\left\{-i: 0\le i \le \min(\ell-1,n-\ell) \right\}$\end{tabular}& $\tfrac{n-1}{2}$ \\ \hline
    \end{tabular}
\end{center}

\medskip

\begin{center}
    \begin{tabular}{@{}cccc@{}}
\multicolumn{4}{c}{$B_n\; (n\ge 2)$}                           \\ \midrule
Node & $\lambda=1$                                      & $\lambda=2$  & $s_k$   \\ \midrule
$1\le \ell< n$ & $\{-i :0\le i\le \min(\ell-1,2n-2\ell-1)\}$ & $\{\ell-n-i : 0\le i\le \lfloor (\ell-1)/2\rfloor\}$ &  $\tfrac{2n-\ell-2}{2}$ \\ \midrule
$n$  & $\{-2i : 0\le i\le \lfloor(n-1)/2\rfloor\}$ &  & $n-1$          \\ \bottomrule
\end{tabular}
\end{center}

\medskip

\begin{center}
    \begin{tabular}{@{}cccc@{}}
\multicolumn{4}{c}{$C_n\; (n\ge 2)$} \\ \midrule
Node      & $\lambda=1$      & $\lambda=2$ & $s_k$    \\ \midrule
$1$       & $\{0\}$     &      & $n-1$      \\ \midrule
$1<\ell $ & $\{-i :0\le i \le \min(\ell-1,2n-2\ell)\}$ & $\{\ell-n-i: 1\le i\le \lfloor \ell/2 \rfloor\}$  & $\tfrac{2n-\ell-1}{2}$\\ \bottomrule
\end{tabular}
\end{center}

\medskip 

\begin{center}
\scalebox{0.85}{
    \begin{tabular}{@{}cccc@{}}
\multicolumn{4}{c}{$D_n\; (n\ge 3)$}                            \\ \midrule
Node     & $\lambda=1$                                      & $\lambda=2$ & $s_k$ \\ \midrule
$1$      & $\{0, 2-n\}$                                &      & $n-2$  \\ \midrule
$1<\ell< n-1$ & $\{-i, \ell-n+1: 0\le i\le \min(\ell-1,2n-2\ell-2)\}$ & $\{\ell-n-i : 0\le i\le \lfloor (\ell-2)/2\rfloor\}$ & $\tfrac{2n-\ell-3}{2}$ \\ \midrule
$n-1, n$ & $\{ -2i : 0\le i\le \lfloor{(n-2)/2}\rfloor \}$ &    & $n-2$    \\ \bottomrule
\end{tabular}
}
\end{center}

\medskip 

\begin{center}
    \begin{tabular}{@{}ccccc@{}}
\multicolumn{5}{c}{$E_6$}                                          \\ \midrule
Node  & $\lambda=1$              & $\lambda=2$                    & $\lambda=3$   & $s_k$\\ \midrule
$1,6$ & $\{0,-3\}$          &                           &     & $5$     \\ \midrule
$2$   & $\{0,-2,-3\}$       & $\{-5\}$                  &     & $\frac{9}{2}$     \\ \midrule
$3,5$ & $\{0,-1,-2,-3\}$    & $\{-3\}$                  &     & $\frac{7}{2}$     \\ \midrule
$4$   & $\{0,-1,-1,-2,-2\}$ & $\{-2,-\tfrac{5}{2},-3\}$ & $\{-3\}$  & $\frac{5}{2}$\\ \bottomrule
\end{tabular}
\end{center}

\medskip

\begin{center}
    \begin{tabular}{@{}cccccc@{}}

\multicolumn{6}{c}{$E_7$}                                       \\ \midrule
Node & $\lambda=1$              & $\lambda=2$         & $\lambda=3$   & $\lambda=4$ & $s_k$\\ \midrule
$1$  & $\{0,-3,-5\}$       & $\{-8\}$       &          &     &  $\frac{15}{2}$ \\ \midrule
$2$  & $\{0,-2,-3,-4,-6\}$ & $\{-5\}$       &          &      & $6$ \\ \midrule
$3$  & $\{0,-1,-2,-3,-4\}$ & $\{-3,-4,-5\}$ & $\{-5\}$ &     &  $\frac{9}{2}$ \\ \midrule
$4$ & $\{0,-1,-1,-2,-2,-3\}$ & $\{-2,-\tfrac{5}{2},-3,-3\}$ &  $\{-3,-\tfrac{10}{3}\}$ & $\{-\tfrac{7}{2}\}$  & $3$\\ \midrule
$5$ & $\{0,-1,-2,-2,-3,-4\}$ & $\{-3,-\tfrac{7}{2},-4\}$    & $\{-4\}$                &              &    $4$   \\ \midrule
$6$  & $\{0,-1,-3,-4\}$    & $\{-4,-6\}$    &          &    &   $\frac{11}{2}$ \\ \midrule
$7$  & $\{0, -4, -8\}$     &                &          &     &   $8$\\ \bottomrule
\end{tabular}
\end{center}

\medskip

\begin{center}
\scalebox{0.88}{
\begin{tabular}{@{}cccccccc@{}}
\multicolumn{8}{c}{$E_8$}                                                                                                                                   \\ \midrule
Node & $\lambda=1$                                                         & $\lambda=2$                       & $\lambda=3$      & $\lambda=4$               & $\lambda=5$ & $\lambda=6$ & $s_k$\\ \midrule
$1$  & $\left\{\begin{array}{l}0,-3,-5,\\-6,-9\end{array}\right\}$    & $\{-8,-11\}$                 &             &                      &        &      & $\frac{21}{2}$  \\ \midrule
$2$  & $\left\{\begin{array}{l}0,-2,-3,\\-4,-5,-6\end{array}\right\}$ & $\{-5,-6,-7,-8\}$            & $\{-7\}$    &                      &        &    &$\frac{15}{2}$    \\ \midrule
$3$ &
  $\left\{\begin{array}{l}0,-1,-2,\\-3,-4,-5\end{array}\right\}$ &
  $\left\{\begin{array}{l}-3,-4,-\tfrac{9}{2},\\-5,-6\end{array}\right\}$ &
  $\{-5,-\tfrac{16}{3}\}$ &
  $\{-\tfrac{11}{2}\}$ &
   & & $\frac{11}{2}$
   \\ \midrule
$4$ &
  $\left\{\begin{array}{l}0,-1,-1,\\-2,-2,-3\end{array}\right\}$ &
  $\left\{\begin{array}{l}-2,-\tfrac{5}{2},-3,\\-3,-\tfrac{7}{2},-4\end{array}\right\}$ &
  $\left\{\begin{array}{l}-3,-\tfrac{10}{3},\\-\tfrac{11}{3},-4\end{array}\right\}$ &
  $\{-\tfrac{7}{2},-\tfrac{15}{4},-4\}$ &
  $\{-4,-\tfrac{21}{5}\}$ &
  $\{-4\}$ & $\frac{7}{2}$\\  \midrule
$5$ &
  $\left\{\begin{array}{l}0,-1,-2,-2,\\-3,-3,-4\end{array}\right\}$ &
  $\left\{\begin{array}{l}-3,-\tfrac{7}{2},-4,\\-4,-\tfrac{9}{2},-5\end{array}\right\}$ &
  $\left\{\begin{array}{l}-4,-\tfrac{13}{3},\\-\tfrac{14}{3},-5\end{array}\right\}$ &
  $\{-\tfrac{9}{2},-5\}$ &
  $\{-5\}$ &
   & $\frac{9}{2}$\\ \midrule
$6$  & $\left\{\begin{array}{l}0,-1,-2,\\-3,-4,-5\end{array}\right\}$ & $\{-4,-\tfrac{9}{2},-5,-6\}$ & $\{-5,-6\}$ & $\{-\tfrac{13}{2}\}$ &        &      & $6$  \\ \midrule
$7$  & $\left\{\begin{array}{l}0,-1,-4,\\-5,-8\end{array}\right\}$    & $\{-5,-7,-9\}$               & $\{-9\}$    &                      &        &     &$\frac{17}{2}$   \\ \midrule
$8$  & $\left\{\begin{array}{l}0,-5,-9\end{array}\right\}$            & $\{-14\}$                    &             &                      &        &       & $\frac{27}{2}$ \\ \bottomrule
\end{tabular}
}
\end{center}

\medskip

\begin{center}
    \begin{tabular}{@{}cccccc@{}}
\multicolumn{6}{c}{$F_4$}                                                \\ \midrule
Node & $\lambda=1$        & $\lambda=2$                      & $\lambda=3$   & $\lambda=4$ & $s_k$   \\ \midrule
$1$  & $\{0,-3\}$    & $\{-2\}$                    &          &   & $3$        \\ \midrule
$2$  & $\{0,-1\}$    & $\{-1, -\tfrac{3}{2}, -2\}$ & $\{-2\}$ & $\{-2\}$ & $\frac{3}{2}$ \\ \midrule
$3$  & $\{0,-1,-2\}$ & $\{-2, -3\}$                & $\{-3\}$ &   &    $\frac{5}{2}$   \\ \midrule
$4$  & $\{0,-3\}$    & $\{-5\}$                    &          &    &    $\frac{9}{2}$  \\ \bottomrule
\end{tabular}
\end{center}

\medskip

\begin{center}
    \begin{tabular}{@{}ccccc@{}}
\multicolumn{5}{c}{$G_2$}        \\ \midrule
Node & $\lambda=1$  & $\lambda=2$ & $\lambda=3$ & $s_k$ \\ \midrule
$1$  & $\{0\}$ & $\{-2\}$ &       & $\frac{3}{2}$ \\ \midrule
$2$  & $\{0\}$ & $\{-1\}$ & $\{-1\}$ & $\frac{1}{2}$\\ \bottomrule
\end{tabular}
\end{center}
\newpage 

\section{Remaining Cases of Theorem \ref{thm:poisson}}\label{appendix:B}
 Let $E$ be a global field. We prove Theorem \ref{thm:poisson} for $G=G_2$ and $G=B_n$ with $P=P_1$. We mimic the proof of the base case for induction in \cite[Proposition 1.6]{Ikeda:poles:triple}, which looks into the cancellation in the constant terms of the degenerate Eisenstein series
 \begin{align*}
     \sum_{w\in W/W_M} M_wf_{\omega_s}.
 \end{align*}
 The key input is the following statement \cite[Lemma 1.5] {Ikeda:poles:triple}.
\begin{Lem}\label{lem:sl2global}
    Let $G=\mathrm{SL}_2$ and $w_0=\begin{psmatrix}
         & 1\\
         -1 &
    \end{psmatrix}$. Then the global intertwining operator $M_{w_0}:I_B(1_s)\to I_{B}(1_{-s})$ is holomorphic at $s=0$, and is equal to the scalar multiplication by $-1$ at $s=0$.\qed
\end{Lem}

We assume $E$ is a number field and leave the  function field case to the reader.

\subsection{Type $B$ and $\ell=1$} 

 In this case $w_0$ has a unique reduced expression
\begin{align*}
    w_0=s_1s_2\cdots s_{n-1}s_ns_{n-1}\cdots s_2s_1.
\end{align*}
For ease of notation let $w_j:=s_j\cdots s_2s_1$ and $w_j':=s_j\cdots s_{n-1}$ for $1\le j\le n-1$. Then
\begin{align*}
    d(\omega_s)c_{\mathrm{Id}}(\omega_s)&=L\left(\tfrac{2n-1}{2},\omega_s\right)L(1,\omega^2_s),\\
    d(\omega_s)c_{w_j}(\omega_s)&=L\left(\tfrac{2n-1}{2}-j,\omega_s\right)L(1,\omega^2_s),\\
    d(\omega_s)c_{s_nw_{n-1}}(\omega_s)&=L\left(\tfrac{1}{2},\omega_s\right)L(0,\omega^2_s),\\
    d(\omega_s)c_{w_j's_nw_{n-1}}(\omega_s)&=L\left(\tfrac{1}{2}-(n-j),\omega_s\right)L(0,\omega^2_s).
\end{align*}
Therefore to prove the assertion we can assume $\omega^2=1,$ and it follows that possible poles of the constant term $\sum_{w\in W/W_M} M_{w}f_{\omega_s}$ are simple except for $s=\pm 1/2$ when $\omega=1$.

We first show the constant term is holomorphic at $s=0$. By Lemma \ref{lem:sl2global}
\begin{align*}
    s(M_{s_n}+1)M_{w_{n-1}}f_{\omega_s}\bigg|_{s=0}=0,
\end{align*}
and thus
\begin{align*}
    s(M_{w_j's_nw_{n-1}}+M_{w_{j-1}})f_{\omega_s}\bigg|_{s=0}&=0,\quad 2\le j\le n-1,\\
    s(M_{w_0}+1)f_{\omega_s}\bigg|_{s=0}&=0.
\end{align*}
This proves that $s=0$ is not a pole.

By the functional equation, we are left to show when $\omega=1$ the constant term is holomorphic at $s=-\frac{2n-1}{2}+j$ for $1\le j\le n-2$ and has at most a simple pole at $s=-\frac{1}{2}$. If $j=1,$ then among terms $M_{w}f_{1_s}$ all but $f_{1_s},M_{s_1}f_{1_s}$ are holomorphic at $s=-\frac{2n-3}{2},$ and by Lemma \ref{lem:sl2global}
\begin{align*}
    \left(s+\tfrac{2n-3}{2}\right)(M_{s_1}+1)f_{1_s}\bigg|_{s=-\frac{2n-3}{2}}=0.
\end{align*}
For $2\le j\le n-2,$ all but $M_{w_{j-1}}f_{1_s},M_{w_j}f_{1_s}$ are holomorphic at $s=-\frac{2n-1-2j}{2},$ and by Lemma \ref{lem:sl2global}
\begin{align*}
    \left(s+\tfrac{2n-1-2j}{2}\right)(M_{s_j}+1)M_{w_{j-1}}f_{1_s}\bigg|_{s=-\frac{2n-1-2j}{2}}=0.
\end{align*}
Finally, all but $M_{w_{n-2}}f_{1_s},M_{w_{n-1}}f_{1_s}$ have at most a simple pole at $s=-\frac{1}{2}$, and by Lemma \ref{lem:sl2global}
\begin{align*}
    \left(s+\tfrac{1}{2}\right)^2(M_{s_{n-1}}+1)M_{w_{n-2}}f_{1_s}\bigg|_{s=-\frac{1}{2}}=0.
\end{align*}
This completes the justification.\qed

\subsubsection{Type $G_2$}

Suppose $\ell=1$. We have 
\begin{align*}
    d(\omega_s)c_{\mathrm{Id}}(\omega_s)=L(\tfrac{5}{2},\omega_s)L(1,\omega^2_s),\\
    d(\omega_s)c_{s_1}(\omega_s)=L(\tfrac{3}{2},\omega_s)L(1,\omega^2_s),\\
    d(\omega_s)c_{s_2s_1}(\omega_s)=L(\tfrac{1}{2},\omega_s)L(1,\omega^2_s),\\
    d(\omega_s)c_{s_1s_2s_1}(\omega_s)=L(\tfrac{1}{2},\omega_s)L(0,\omega^2_s),\\
    d(\omega_s)c_{s_2s_1s_2s_1}(\omega_s)=L(-\tfrac{1}{2},\omega_s)L(0,\omega^2_s),\\
    d(\omega_s)c_{w_0}(\omega_s)=L(-\tfrac{3}{2},\omega_s)L(0,\omega^2_s).
\end{align*}
Therefore, we may assume $\omega^2=1$. It follows that possible poles of $\mathrm{Eis}(f_{\omega_s})$ on $\mathrm{Re}(s)\le 0$ are the multiset
\begin{align*}
    \left\{-\tfrac{5}{2},-\tfrac{3}{2},-\tfrac{1}{2},-\tfrac{1}{2},0\right\}
\end{align*}
if $\omega=1$ and $\{-\tfrac{1}{2},0\}$ if $\omega\neq 1$. We need to show $\mathrm{Eis}(f_{\omega_s})$ is holomorphic at $s=-\tfrac{3}{2},0,$ and has at most a simple pole at $s=-\tfrac{1}{2}$ when $\omega=1$.

Assume first $\omega=1$. Among terms $M_{w}f_{1_s}$ all but $f_{1_s},M_{s_1}f_{1_s}$ are holomorphic at $s=-\tfrac{3}{2}$ and by Lemma \ref{lem:sl2global}
\begin{align*}
   (s+\tfrac{3}{2})(M_{s_1}+1)f_{1_s}\bigg|_{s=-\frac{3}{2}}=0.
\end{align*}
All but $M_{s_1}f_{1_s},M_{s_2s_1}f_{1_s}$ have at most a simple pole at $s=-\tfrac{1}{2}$, and by Lemma \ref{lem:sl2global}
\begin{align*}
    (s+\tfrac{1}{2})^2(M_{s_2}+1)M_{s_1}f_{1_s}\bigg|_{s=-\frac{1}{2}}=0.
\end{align*}
Suppose $\omega^2=1$. We have by Lemma \ref{lem:sl2global}
\begin{align*}
    s(M_{s_1}+1)M_{s_2s_1}f_{\omega_s}\bigg|_{s=0}=s(M_{s_2s_1s_2}+1)M_{s_1}f_{\omega_s}\bigg|_{s=0}=s(M_{w_0}+1)f_{\omega_s}\bigg|_{s=0}=0.
\end{align*}
This completes the justification for $\ell=1$.

Suppose $\ell=2$. We have 
\begin{align*}
    d(\omega_s)c_{\mathrm{Id}}(\omega_s)=L(\tfrac{3}{2},\omega_s)L(1,\omega^2_s)L(\tfrac{3}{2},\omega_s^3),\\
    d(\omega_s)c_{s_2}(\omega_s)=L(\tfrac{1}{2},\omega_s)L(1,\omega^2_s)L(\tfrac{3}{2},\omega_s^3),\\
    d(\omega_s)c_{s_1s_2}(\omega_s)=L(\tfrac{1}{2},\omega_s)L(1,\omega^2_s)L(\tfrac{1}{2},\omega_s^3),\\
    d(\omega_s)c_{s_2s_1s_2}(\omega_s)=L(\tfrac{1}{2},\omega_s)L(0,\omega^2_s)L(\tfrac{1}{2},\omega_s^3),\\
    d(\omega_s)c_{s_1s_2s_1s_2}(\omega_s)=L(\tfrac{1}{2},\omega_s)L(0,\omega^2_s)L(-\tfrac{1}{2},\omega_s^3),\\
    d(\omega_s)c_{w_0}(\omega_s)=L(-\tfrac{1}{2},\omega_s)L(0,\omega^2_s)L(-\tfrac{1}{2},\omega_s^3).
\end{align*}
It follows that we can assume $\omega$ has order at most $3,$ and possible poles of $\mathrm{Eis}(f_{\omega_s})$ on $\mathrm{Re}(s)\le 0$ are the multiset
\begin{align*}
    \left\{-\tfrac{3}{2},-\tfrac{1}{2},-\tfrac{1}{2},0,-\tfrac{1}{2},-\tfrac{1}{6}\right\}
\end{align*}
if $\omega=1,$ $\{-\frac{1}{2},0\}$ if $\omega$ is of order $2$, and
$\{-\frac{1}{2},-\frac{1}{6}\}$ if $\omega$ is of order $3.$
We need to show $\mathrm{Eis}(f_{\omega_s})$ is holomorphic at $s=-\frac{1}{6},0$ and has at most a pole of order $2$ at $s=-\frac{1}{2}$ when $\omega=1$.

Suppose $\omega=1$. All but $f_{1_s},M_{s_2}f_{1_s}$ have at most a pole of order $2$ at $s=-\frac{1}{2}$, and by Lemma \ref{lem:sl2global}
\begin{align*}
    (s+\tfrac{1}{2})^3(M_{s_2}+1)f_{1_s}\bigg|_{s=-\frac{1}{2}}=0.
\end{align*}
Suppose $\omega^3=1.$ Then $M_{s_1s_2s_1s_2}f_{\omega_s}$ and $M_{w_0}f_{\omega_s}$ are holomorphic at $s=-1/6$, and by Lemma \ref{lem:sl2global}
\begin{align*}
    (s+\tfrac{1}{6})(M_{s_1}+1)M_{s_2}f_{\omega_s}\bigg|_{s=-\frac{1}{6}}&=0,\\
    (s+\tfrac{1}{6})(M_{s_2s_1s_2}+1)f_{\omega_s}\bigg|_{s=-\frac{1}{6}}&=0.
\end{align*}
For $\omega^2=1$ and $s=0$, the proof is similar to that in the case $\ell=1$ above. This completes the verification.\qed

\bibliographystyle{alpha}

\bibliography{bibs}
\end{document}